\documentclass[12pt]{article}
\usepackage{amssymb,amsmath,amsthm,tikz,multirow,nccrules,wasysym}
\usetikzlibrary{calc, arrows, arrows.meta, math}

\title{Tilings of the Sphere by Congruent Quadrilaterals or  Triangles}
\author{Ho Man Cheung, Hoi Ping Luk, Min Yan\thanks{Research was supported by Hong Kong RGC General Research Fund 16303515 and 16305920.} \\
Hong Kong University of Science and Technology}

\usepackage[hidelinks, hyperindex]{hyperref}

\hyperref[sec:function]{}

\newcommand{\sub}{\subset}

\newcommand{\mc}{\mathcal}
\newcommand{\bb}{\mathbb}
\newcommand{\ssum}{\textstyle \sum}

\newcommand{\arcThroughThreePoints}[4][]{
\coordinate (middle1) at ($(#2)!.5!(#3)$);
\coordinate (middle2) at ($(#3)!.5!(#4)$);
\coordinate (aux1) at ($(middle1)!1!90:(#3)$);
\coordinate (aux2) at ($(middle2)!1!90:(#4)$);
\coordinate (center) at ($(intersection of middle1--aux1 and middle2--aux2)$);
\draw[#1] 
 let \p1=($(#2)-(center)$),
      \p2=($(#4)-(center)$),
      \n0={veclen(\p1)},       
      \n1={atan2(\y1,\x1)}, 
      \n2={atan2(\y2,\x2)},
      \n3={\n2>\n1?\n2:\n2+360}
    in (#2) arc(\n1:\n3:\n0);
}

\newcommand{\dash}{\hspace{0.1em}\dashrule{0.7}{2.4 1 2.4 1 2.4}\hspace{0.1em}} 
\newcommand{\thin}{\hspace{0.1em}\rule{0.7pt}{0.8em}\hspace{0.1em}}
\newcommand{\thick}{\hspace{0.1em}\rule{1.5pt}{0.8em}\hspace{0.1em}}

\newtheorem{theorem}{Theorem}
\newtheorem{lemma}[theorem]{Lemma}

\newtheorem{proposition}[theorem]{Proposition}
\newtheorem*{theorem*}{Theorem}

\theoremstyle{definition}
\newtheorem*{definition*}{Definition}
\newtheorem*{case*}{Case}
\newtheorem*{subcase*}{Subcase}

\theoremstyle{remark}

\numberwithin{equation}{section}

\begin{document}

\maketitle

\begin{abstract}
We completely classify edge-to-edge tilings of the sphere by congruent quadrilaterals. As part of the classification, we also present a modern version of the classification of edge-to-edge tilings of the sphere by congruent triangles. Together with our series of papers that classifies edge-to-edge tilings of the sphere by congruent pentagons, we complete the classification of edge-to-edge tilings of the sphere by congruent polygons.

\medskip

\end{abstract}

\tableofcontents

\section{Introduction}

The history of 2-dimensional spherical tilings can be traced as early as Plato, Archimedes, and Kepler. However, compared with the abundance of works about tilings of the plane \cite{rao,zong}, works about tilings of the sphere are relatively rare. One recent breakthrough is the complete classification of tilings of the sphere by regular polygons \cite{aehj,johnson, zal}. Another fundamental problem with long history is tilings of the sphere by congruent polygons. 

In an {\em edge-to-edge} tiling of the sphere by congruent polygons, the polygon must be a triangle, quadrilateral, or pentagon. The classification of edge-to-edge tilings of the sphere by congruent triangles was started in 1924 by Sommerville \cite{so}, somewhat completed by Davis \cite{davisH}, and fully completed in 2002 by Ueno and Agaoka \cite{ua}. In a series of papers \cite{ay1,awy,cly,gsy,wy1,wy2,wy3}, we  completely classified edge-to-edge tilings of the sphere by congruent pentagons. This paper completely classifies edge-to-edge tilings of the sphere by congruent quadrilaterals. Therefore the fundamental problem of edge-to-edge tilings of the sphere by congruent polygons is finally solved.

Not much is known about non-edge-to-edge tilings of the sphere by congruent polygons. Dawson and Doyal \cite{dawson1,dawson2,dd1,dd2,dd3} did some partial works on triangular tilings. We hope the method developed for edge-to-edge tilings, the algorithms developed by Rao \cite{rao}, and the trigonometric Diophantine analysis can be used to solve the problem.

The following is the main result of this paper.

\begin{theorem*}
Edge-to-edge tilings of the sphere by congruent quadrilaterals are the following:
\begin{enumerate}
\item Platonic: Cube $P_6$, quadrilateral subdivisions $Q_{\square}P_4,Q_{\square}P_6,Q_{\square}P_{12}$ of the Platonic solids, quadricentric subdivision $C_{\square}P_{12}$ of the Platonic solids, flip modifications $FQ_{\square}P_6,FQ_{\square}P_8$ of the quadrilateral subdivisions of the cube and octahedron.
\item Earth map: Two families $E_{\square}1,E_{\square}2$, their flip modifications $FE_{\square}1$, $F_1E_{\square}2$, $F_2E_{\square}2$, and a rearrangement $RE_{\square}1$.
\item Sporadic: $S_{12\square}1,S_{16\square}1,S_{16\square}2,S_{16\square}3,FS_{16\square}3,S_{16\square}4,S_{36\square}5,S_{36\square}6$.
\end{enumerate}
\end{theorem*}

We denote the Platonic solids by $P_f$, where $f$ is the number of tiles. The quadrilateral subdivision $Q_{\square}$ and the quadricentric subdivision $C_{\square}$ are uniform constructions that change Platonic solids to tilings by congruent quadrilaterals. The theorem actually includes all quadrilateral and quadricentric subdivisions because we have $Q_{\square}P_4=C_{\square}P_6=C_{\square}P_8$, $Q_{\square}P_6=Q_{\square}P_8$, $Q_{\square}P_{12}=Q_{\square}P_{20}$, and $C_{\square}P_{12}=C_{\square}P_{20}$. Moreover, the tilings $P_6,Q_{\square}P_6$ are deformed in the sense they allow free parameters. 

The earth map tilings are infinite families with no limit on the number of tiles. The flip modifies part (or several parts) of tilings to get new tilings. The pentagonal version of the earth map tiling was first introduced in \cite{yan} as a combinatorial concept, and then appeared as tilings by congruent almost equilateral pentagons \cite{cly}.  

Some quadrilateral tilings are intertwined with triangular tilings. As part of the classification argument, we also obtain the classification of triangular tilings in a more modern and efficient way than Davis \cite{davisH} and Ueno and Agaoka \cite{ua}.

\begin{theorem*}
Edge-to-edge tilings of the sphere by congruent triangles are the following:
\begin{enumerate}
\item Platonic: Tetrahedron $P_4$, octahedron $P_8$, icosahedron $P_{20}$, triangular subdivisions $T_{\triangle}P_4,T_{\triangle}P_6,T_{\triangle}P_8,T_{\triangle}P_{12},T_{\triangle}P_{20}$ of the Platonic solids, barycentric subdivisions $B_{\triangle}P_6,B_{\triangle}P_{12}$ of the Platonic solids, flip modification $FB_{\triangle}P_8$ of the barycentric subdivision of the octahedron.
\item Earth map: Three families $E_{\triangle}1,E_{\triangle}2,E_{\triangle}3$, and their flip modifications $FE_{\triangle}1,F'E_{\triangle}1,FE_{\triangle}2,FE_{\triangle}3$.
\item Sporadic: Simple triangular subdivisions $S_{\triangle}P_6$ of the cube.
\end{enumerate}
\end{theorem*}

The triangular subdivision $T_{\triangle}$ and the barycentric subdivision $B_{\triangle}$ are uniform constructions that change Platonic solids to tilings by congruent triangles. Again the theorem includes all triangular and barycentric subdivisions because we have $B_{\triangle}P_4=T_{\triangle}P_6$, $B_{\triangle}P_6=B_{\triangle}P_8$, and $B_{\triangle}P_{12}=B_{\triangle}P_{20}$.

The prime in the notation $F'E_{\triangle}1$ means further modification on the flip modification $FE_{\triangle}1$. The simple triangular subdivision $S_{\triangle}P_6$ of the cube independently divides each square face of the regular cube into half. There are seven such tilings, and five of the seven are already included in the first two categories. We may also regard $S_{\triangle}P_6$ as the Platonic type.

Edge-to-edge tilings of the sphere by congruent quadrilaterals can have four possible edge combinations: general $a^2bc$, kite $a^2b^2$, almost equilateral $a^3b$, and rhombus $a^4$ (see Figure \ref{quad} and Lemma \ref{edge_combo}). Kite and rhombus tilings can be derived from triangular tilings. Akama, van Cleemput, and Sakano \cite{akama1,akama2,ac,sakano-akama} used the relation to obtain the classification.

Therefore the main result of this paper is the classification of general and almost equilateral tilings. These are much more difficult than kite and rhombus tilings. In fact, Ueno and Agaoka \cite{ua2} tried the almost equilateral tilings and highlighted the difficulty in classifying the case. We use the insight from our work on pentagonal tilings, especially the almost equilateral $a^4b$ tilings \cite{cly}. We develop various new technical tools, especially the global counting argument, to completely classify quadrilateral tilings.

The two difficult cases can also be classified by the existing techniques, in less efficient and less insightful way. Liao, Qian, Wang, and Xu \cite{lqwx1} classified general $a^2bc$ tilings by using the older idea of special tile in \cite{wy1,wy2}. Moreover, Myerson \cite{myerson} and Coolsaet \cite{coolsaet} classified all the convex almost equilateral quadrilaterals, such that all four angles are rational multiples of $\pi$. The classification was based on the trigonometric Diophantine analysis by Conway and Jones \cite{cj}, and was intended as a way of classifying almost equilateral $a^3b$ tilings. Cheung and Luk \cite{cl} and Liao, Qian, Wang, and Xu \cite{lw,lqwx2} extended the work to non-convex quadrilaterals, and finished the classification of almost equilateral tilings (including the cases the angles are not rational multiples of $\pi$). 

The alternative approaches are more specialised.  The special tile method cannot be used for classifying almost equilateral quadrilateral tilings. The trigonometric Diophantine analysis is much more complicated for almost equilateral pentagons, and it remains to be seen whether the method can lead to the classification of almost equilateral pentagonal tilings. In \cite{cly}, we developed more sophisticated and systematic techniques to classify almost equilateral pentagonal tilings. These techniques are transferred to triangular and quadrilateral tilings, and become the key techniques in this paper.

Finally, we outline the content of this paper.

Section \ref{construction} gives detailed description of all the tilings. We first describe Platonic solids and earth map tilings, which are the primary tilings. Then we apply subdivision and flip modification operations to get new tilings from the primary tilings. We end the section with the list of sporadic tilings. We note that some sporadic tilings are also earth map tilings, except only three or four timezones are allowed. Therefore sporadic earth map tilings are not parts of infinite families.

Section \ref{technical} is the technical preparation. The first major technique is counting, which is combinatorial and global. The second major technique is geometrical, and uses straight edges (part of great arcs) and simple boundary. In particular, we give an updated treatment of the equality obtained by Coolsaet \cite{coolsaet}, but drop the convexity requirement. Moreover, we also show that the equality effectively implies the existence of the quadrilateral. The existence is very important because there are many tilings  that are combinatorially valid but are actually impossible due to geometrical reasons. The third major technique is the adjacent angle deduction, which was first developed for pentagonal tilings. It is a convenient language to describe tile arrangements around vertices, and can be regarded as mini-tiling pictures. The adjacent angle deduction argument saves us from lots of pictures (there are already 76 pictures in this paper!), and even enables certain argument that cannot be easily done by pictures. 

The subsequent sections are the classification proof. In Section \ref{rhombustiling}, we classify rhombus tilings, and then use this to classify isosceles triangular tilings. In Section \ref{kitetiling}, we classify general triangular tilings, and then use this to classify kite tilings. In Section \ref{generalquad}, we classify general quadrilateral tilings. In Section \ref{almostquad}, we classify almost equilateral quadrilateral tilings. 

The proof requires some calculations, especially for almost equilateral quadrilaterals. All calculations can be done symbolically, and give exact values. Moreover, when we present approximate values such as $\sin b=0.3246$ and $\alpha=0.4195\pi$, we mean $0.3246<\sin b<0.3247$ and $0.4195\pi<\alpha<0.4196\pi$.

\section{Construction}
\label{construction}

In general, there are three edge length combinations for triangles. See Figure \ref{triangle}. All three combinations are suitable for tiling the sphere. In all the tiling pictures, we indicate the edge lengths $a,b,c$ by normal, thick, and dashed lines. In this section, we draw triangular tiling pictures without indicating the angles. The reader can fill the angles by referring to Figure \ref{triangle}.

\begin{figure}[htp]
\centering
\begin{tikzpicture}[>=latex]


\draw
	(90:0.8) -- (210:0.8);
\draw[line width=1.2]
	(90:0.8) -- (-30:0.8);
\draw[dashed]
	(-30:0.8) -- (210:0.8);

\node at (90:0.4) {\small $\alpha$};
\node at (210:0.4) {\small $\beta$};
\node at (-30:0.4) {\small $\gamma$};

\node at (150:0.6) {\small $a$};
\node at (30:0.6) {\small $b$};
\node at (-90:0.6) {\small $c$};

\node at (0,-0.97) {general $abc$};


\begin{scope}[xshift=3cm]

\draw
	(-30:0.8) -- (90:0.8) -- (210:0.8);
	
\draw[line width=1.2]
	(210:0.8) -- (-30:0.8);

\node at (90:0.4) {\small $\alpha$};
\node at (210:0.4) {\small $\beta$};
\node at (-30:0.4) {\small $\beta$};

\node at (150:0.6) {\small $a$};
\node at (30:0.6) {\small $a$};
\node at (-90:0.6) {\small $b$};

\node at (0,-0.93) {isosceles $a^2b$};

\end{scope}


\foreach \a in {0,1,2}
{
\begin{scope}[xshift=6cm, rotate=120*\a]

\draw
	(-30:0.8) -- (90:0.8);

\node at (90:0.4) {\small $\alpha$};

\node at (30:0.6) {\small $a$};

\end{scope}
}

\node at (6,-0.95) {equilateral $a^3$};

\end{tikzpicture}
\caption{All possible triangles, with distinct $a,b,c$.}
\label{triangle}
\end{figure}
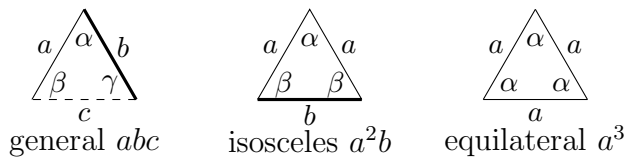

For quadrilaterals, there are four edge length combinations that are suitable for tiling the sphere. See Figure \ref{quad} and Lemma \ref{edge_combo}. The quadrilateral with the edge combination $a^2bc$ is the most {\em general} among the four. Moreover, the rhombus is exactly the equilateral quadrilateral.

\begin{figure}[htp]
\centering
\begin{tikzpicture}[>=latex,scale=1]


\draw
	(-0.6,0.6) -- (0.6,0.6) -- (0.6,-0.6);

\draw[line width=1.2]
	(-0.6,0.6) -- (-0.6,-0.6);

\draw[dashed]
	(-0.6,-0.6) -- (0.6,-0.6);	

\node at (0.4,0.4) {\small $\alpha$};
\node at (-0.4,0.4) {\small $\beta$};
\node at (-0.4,-0.4) {\small $\gamma$};
\node at (0.4,-0.4) {\small $\delta$};

\node at (0,0.8) {\small $a$};
\node at (0.8,0) {\small $a$};
\node at (-0.8,0) {\small $b$};
\node at (0,-0.8) {\small $c$};

\node at (0,-1.34) {general $a^2bc$};


\begin{scope}[xshift=3cm]

\draw
	(-0.6,0.6) -- (0.6,0.6) -- (0.6,-0.6);

\draw[line width=1.2]
	(-0.6,0.6) -- (-0.6,-0.6) -- (0.6,-0.6);	

\node at (0.4,0.4) {\small $\alpha$};
\node at (-0.4,0.4) {\small $\beta$};
\node at (0.4,-0.4) {\small $\beta$};
\node at (-0.4,-0.4) {\small $\gamma$};

\node at (0,0.8) {\small $a$};
\node at (0.8,0) {\small $a$};
\node at (-0.8,0) {\small $b$};
\node at (0,-0.8) {\small $b$};

\node at (0,-1.3) {kite $a^2b^2$};

\end{scope}


\begin{scope}[xshift=6cm]

\draw
	(-0.6,-0.6) -- (-0.6,0.6) -- (0.6,0.6) -- (0.6,-0.6);

\draw[line width=1.2]
	(-0.6,-0.6) -- (0.6,-0.6);	

\node at (0.4,0.4) {\small $\alpha$};
\node at (-0.4,0.4) {\small $\beta$};
\node at (-0.4,-0.4) {\small $\gamma$};
\node at (0.4,-0.4) {\small $\delta$};

\node at (0,0.8) {\small $a$};
\node at (0.8,0) {\small $a$};
\node at (-0.8,0) {\small $a$};
\node at (0,-0.8) {\small $b$};

\node at (0,-1.15) {almost};

\node at (0,-1.55) {equilateral $a^3b$};

\end{scope}


\begin{scope}[xshift=9cm]

\draw
	(-0.6,-0.6) rectangle (0.6,0.6);

\foreach \a in {0,...,3}
\node at (90*\a :0.8) {\small $a$};

\node at (0,-1.3) {rhombus $a^4$};

\node at (0.4,0.4) {\small $\alpha$};
\node at (-0.4,-0.4) {\small $\alpha$};
\node at (-0.4,0.4) {\small $\beta$};
\node at (0.4,-0.4) {\small $\beta$};

\end{scope}
	
\end{tikzpicture}
\caption{Quadrilaterals suitable for tiling, with distinct $a,b,c$.}
\label{quad}
\end{figure}
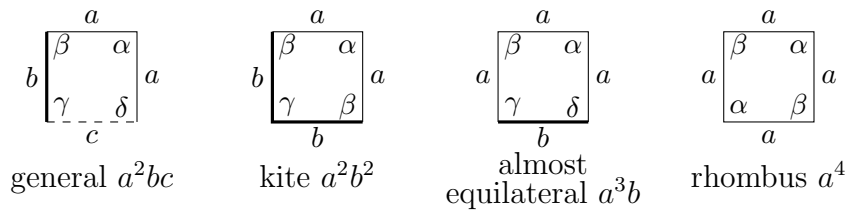

In this section, we also draw quadrilateral tiling pictures without indicating the angles. For general quadrilateral tilings and kite tilings, the reader can fill the angles by referring to Figure \ref{quad}. For almost equilateral quadrilateral tilings, however, we need to know the orientation of $\alpha\to\beta\to\gamma\to\delta$ in order to fill the angles. Our convention is that tiles without indication have the counterclockwise orientation, and tiles with ``$-$'' indication have the clockwise orientation. See the first and second of Figure \ref{quad_angle}. For rhombus tilings, we use $\bullet$ to indicate the $\beta$ angle. See the third of Figure \ref{quad_angle}. 

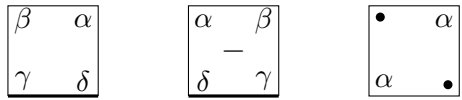
\begin{figure}[htp]
\centering
\begin{tikzpicture}[>=latex,scale=1]

\foreach \a in {1,-1}
{

\draw[xshift=1.2*\a cm]
	(-0.6,-0.6) -- (-0.6,0.6) -- (0.6,0.6) -- (0.6,-0.6);

\draw[line width=1.2, xshift=1.2*\a cm]
	(-0.6,-0.6) -- (0.6,-0.6);	

\begin{scope}[xshift=1.2*\a cm, xscale=\a]

\node at (-0.4,0.4) {\small $\alpha$};
\node at (0.4,0.4) {\small $\beta$};
\node at (0.4,-0.4) {\small $\gamma$};
\node at (-0.4,-0.4) {\small $\delta$};

\end{scope}

}

\node at (1.2,0) {$-$};

\begin{scope}[xshift=3.6cm]

\draw
	(-0.6,-0.6) rectangle (0.6,0.6);

\node at (0.4,0.4) {\small $\alpha$};
\node at (-0.4,-0.4) {\small $\alpha$};

\fill
	(0.45,-0.45) circle (0.06)
	(-0.45,0.45) circle (0.06);
	
\end{scope}

\end{tikzpicture}
\caption{Angle indications when not determined by edges.}
\label{quad_angle}
\end{figure}

The general triangle has several symmetries, such as exchanging $(\alpha,b)$ with $(\beta,c)$. The general quadrilateral has the symmetry of exchanging $(\beta,b)$ with $(\delta,c)$. The kite has the symmetry of exchanging $(\alpha,a)$ with $(\gamma,b)$. The almost equilateral quadrilateral has the symmetry of exchanging $(\alpha,\delta)$ with $(\beta,\gamma)$. The rhombus has the symmetry of exchanging $\alpha$ with $\beta$. If one tiling can be obtained from the other by an exchange symmetry, then we consider two tilings to be the same. When we say ``up to symmetry'', we mean up to one of the exchange symmetries above.

\subsection{Platonic Solids and Earth Map Tilings}

There are two primary sources of tilings by congruent polygons. The first is the five {\em Platonic solids}, given by Figure \ref{platonic_solids}. Note that the tiles are congruent, but are not required to be regular. In other words, these are {\em deformed} Platonic solids. In fact, the tetrahedron $P_4$, the cube $P_6$, and the dodecahedron $P_{12}$ allow two free parameters \cite{ay1,wy3}, and the octahedron $P_8$ allows one free parameter. The icosahedron $P_{20}$ cannot be deformed. 

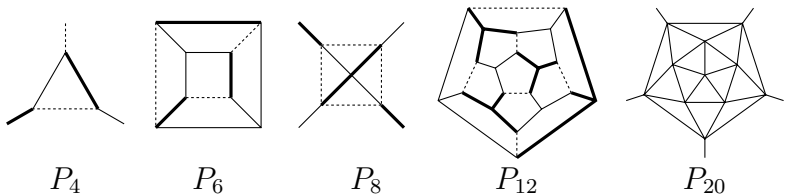
\begin{figure}[htp]
\centering
\begin{tikzpicture}[>=latex]


\begin{scope}[yshift=-0.2cm]

\draw
	(90:0.5) -- (210:0.5)
	(-30:0.5) -- (-30:0.9);
\draw[line width=1.2]
	(90:0.5) -- (-30:0.5)
	(210:0.5) -- (210:0.9);
\draw[dash pattern=on 1pt off 1pt]
	(-30:0.5) -- (210:0.5)
	(90:0.5) -- (90:0.9);

\node at (0,-1.2) {$P_4$};

\end{scope}


\begin{scope}[xshift=1.9cm]

\draw
	(0.3,0.3) -- (-0.3,0.3) -- (-0.3,-0.3)
	(-0.3,0.3) -- (-0.7,0.7) 
	(-0.7,-0.7) -- (0.7,-0.7) -- (0.7,0.7)
	(0.3,-0.3) -- (0.7,-0.7) ;

\draw[line width=1.2]
	(0.3,-0.3) -- (0.3,0.3)
	(-0.7,0.7) -- (0.7,0.7)
	(-0.3,-0.3) -- (-0.7,-0.7);

\draw[dash pattern=on 1pt off 1pt]
	(-0.3,-0.3) -- (0.3,-0.3)
	(-0.7,0.7) -- (-0.7,-0.7)
	(0.3,0.3) -- (0.7,0.7);
	
\node at (0,-1.4) {$P_6$};

\end{scope}	
	

\foreach \a in {1,-1}
{
\begin{scope}[xshift=3.8cm, scale=\a]

\draw[dash pattern=on 1pt off 1pt]
	(-0.4,0.4) -- (0.4,0.4) -- (0.4,-0.4);
	
\draw[line width=1.2]
	(0,0) -- (0.4,0.4)
	(-0.4,0.4) -- (-0.7,0.7);
	
\draw
	(0,0) -- (-0.4,0.4)
	(0.4,0.4) -- (0.7,0.7);

\end{scope}
}

\node at (4,-1.4) {$P_8$};


\begin{scope}[xshift=6cm]

\draw
	(90:0.3) -- (162:0.3) -- (234:0.3)
	(162:0.3) -- (162:0.56)
	(18:0.56) -- (54:0.76) -- (90:0.56)
	(54:0.76) -- (54:1.1)
	(-18:0.76) -- (-54:0.56) -- (-90:0.76)
	(-54:0.3) -- (-54:0.56)
	(126:1.1) -- (198:1.1) -- (-90:1.1)
	(198:0.76) -- (198:1.1);

\draw[line width=1.2]
	(90:0.3) -- (18:0.3) -- (-54:0.3)
	(18:0.3) -- (18:0.56)
	(162:0.56) -- (126:0.76) -- (90:0.56)
	(126:0.76) -- (126:1.1)
	(198:0.76) -- (234:0.56) -- (-90:0.76)
	(234:0.3) -- (234:0.56)
	(54:1.1) -- (-18:1.1) -- (-90:1.1)
	(-18:0.76) -- (-18:1.1);

\draw[dash pattern=on 1pt off 1pt]
	(234:0.3) -- (-54:0.3)
	(90:0.3) -- (90:0.56)
	(18:0.56) -- (-18:0.76)
	(162:0.56) -- (198:0.76)
	(-90:0.76) -- (-90:1.1)
	(54:1.1) -- (126:1.1);

\node at (0,-1.4) {$P_{12}$};

\end{scope}


\foreach \a in {0,...,4}
\draw[xshift=8.5cm, rotate=72*\a]
	(0,0) -- (18:0.45) -- (90:0.45)
	(-18:0.85) -- (54:0.85) -- (54:1.1)
	(18:0.45) -- (54:0.85) -- (90:0.45);

\node at (8.5,-1.4) {$P_{20}$};
	
\end{tikzpicture}
\caption{Platonic solids.}
\label{platonic_solids}
\end{figure}

The opposite of deformation is {\em reduction}. The $a=b=c$ reductions of $P_4,P_6,P_8,P_{12}$ are the regular Platonic solids. The $a=b$ reduction of $P_4$ is isosceles triangular tiling. The $a=b$ and $b=c$ reductions of $P_6$ are almost equilateral and kite tilings. 

Table \ref{platonic_data} gives the data for the triangular and quadrilateral Platonic solids, where $f$ is the number of tiles. 

\renewcommand{\arraystretch}{1.2}

\begin{table}[htp]
\centering
\scalebox{1}{
\begin{tabular}{|c|l|l|c|} 
\hline
Tiling
& \multicolumn{1}{|c|}{Angle}
& \multicolumn{1}{|c|}{Vertex}
& $f$ \\
\hline  \hline 
$P_4$
& $\alpha+\beta+\gamma=2\pi$ 
& $\alpha\beta\gamma$
& $4$  \\
\hline
$P_6$ 
& $\alpha=\frac{2}{3}\pi,\; \beta+\gamma+\delta=2\pi$ 
& $\alpha^3,\beta\gamma\delta$ 
& $6$ \\
\hline
$P_8$
& $\alpha=\tfrac{1}{2}\pi,\;\beta+\gamma=\pi$ 
& $\alpha^4,\beta^2\gamma^2$ 
& $8$ \\
\hline
$P_{20}$
& $\alpha=\frac{2}{5}\pi$ 
& $\alpha^5$ 
& $20$ \\
\hline
\end{tabular}
}
\caption{Angles and vertices for Platonic solids, $P_{12}$ not included. }
\label{platonic_data}
\end{table}

The second source of tilings is the {\em earth map tiling}. Figure \ref{emt} gives all the earth map tilings by congruent triangles or quadrilaterals. There are three triangular earth map tilings $E_{\triangle}1,E_{\triangle}2,E_{\triangle}3$ and two quadrilateral earth map tilings $E_{\square}1,E_{\square}2$. 

The term earth map tiling is inspired by our usual way of presenting the map of the world. It is a repetition of {\em timezones}, which is a combination of several tiles and is indicated by the shaded tiles. Three or four timezones are presented in Figure \ref{emt}. There is no limit on the number of timezones, which means there is no limit on the number of tiles. The vertical edges at the top converge to a vertex, and the vertical edges at the bottom converge to another vertex. The two vertices are the two {\em poles} of the earth map tiling. The boundary of a timezone is two paths connecting the two poles. Each such path is a {\em meridian}.

We note that the choice of timezones is not unique. For example, in Figure \ref{flip9}, we will choose a different timezone for $E_{\triangle}3$ in order to give another description of the flip modification. The choice of the timezone for $E_{\triangle}3$ in Figure \ref{emt} is intended to be compatible with the timezone for $E_{\square}1$, because $E_{\triangle}3=T_{\triangle}E_{\square}^R1$ is the triangular subdivision (see Section in \ref{division}) of $E_{\square}^R1$.

\begin{figure}[htp]
\centering
\begin{tikzpicture}[>=latex]



\begin{scope}

\node at (-0.5,0) {$E_{\triangle}1$};

\fill[gray!50]
	(0,0.6) -- (0,-0.6) -- (1.2,-0.6) -- (1.2,0.6);
	
\foreach \a in {0,2,4}
{
\draw[xshift=0.6*\a cm]
	(0,0) -- (0,0.6);

\draw[xshift=0.6*\a cm, line width=1.2]
	(0,0) -- (0,-0.6);
}

\foreach \a in {1,3}
{
\draw[xshift=0.6*\a cm]
	(0,0) -- (0,-0.6);

\draw[xshift=0.6*\a cm, line width=1.2]
	(0,0) -- (0,0.6);
}
	
\draw[dash pattern=on 1pt off 1pt]
	(0,0) -- (2.4,0);

\end{scope}


\begin{scope}[yshift=-1.5cm]

\node at (-0.5,0) {$E_{\triangle}^I1$};

\fill[gray!50]
	(0.6,0.6) -- (0.6,-0.6) -- (0,-0.6) -- (0,0.6);
	
\foreach \a in {0,...,4}
\draw[xshift=0.6*\a cm]
	(0,-0.6) -- (0,0.6);
	
\draw[line width=1.2]
	(0,0) -- (2.4,0);

\end{scope}


\begin{scope}[yshift=-3cm]

\node at (-0.5,0) {$E_{\triangle}^J1$};

\fill[gray!50]
	(0,0.6) -- (0,-0.6) -- (1.2,-0.6) -- (1.2,0.6);
	
\foreach \a in {0,2,4}
{
\draw[xshift=0.6*\a cm]
	(0,0) -- (0,0.6);

\draw[xshift=0.6*\a cm, line width=1.2]
	(0,0) -- (0,-0.6);
}

\foreach \a in {1,3}
{
\draw[xshift=0.6*\a cm]
	(0,0) -- (0,-0.6);

\draw[xshift=0.6*\a cm, line width=1.2]
	(0,0) -- (0,0.6);
}
	
\draw
	(0,0) -- (2.4,0);

\end{scope}


\begin{scope}[yshift=-4.5cm]

\node at (-0.5,0) {$E_{\triangle}2$};

\fill[gray!50]
	(0.4,0.6) -- (0.4,0.2) -- (0,-0.2) -- (0,-0.6) -- (0.8,-0.6) -- (0.8,-0.2) -- (1.2,0.2) -- (1.2,0.6);

\foreach \a in {0,1,2,3}
\draw[xshift=0.8*\a cm]
	(0.4,0.6) -- (0.4,0.2) -- (0,-0.2) -- (0,-0.6);

\foreach \a in {0,1,2}
\draw[xshift=0.8*\a cm]
	(0.4,0.2) -- (0.8,-0.2);
	
\draw[line width=1.2]
	(0.4,0.2) -- (2.8,0.2)
	(0,-0.2) -- (2.4,-0.2);

\end{scope}
	

\begin{scope}[yshift=-6cm]

\node at (-0.5,0) {$E_{\triangle}3$};

\fill[gray!50]
	(0.4,0.6) -- (0.4,0.2) -- (0,-0.2) -- (0,-0.6) -- (0.8,-0.6) -- (0.8,-0.2) -- (1.2,0.2) -- (1.2,0.6);
	
\foreach \a in {0,1,2,3}
\draw[xshift=0.8*\a cm]
	(0.4,0.6) -- (0.4,0.2) -- (0,-0.2) -- (0,-0.6);

\foreach \a in {0,1,2}
{
\draw[xshift=0.8*\a cm]
	(0.4,0.2) -- (0.8,-0.2);

\draw[xshift=0.8*\a cm, line width=1]
	(0.4,0.2) -- (0.4,-0.6)
	(0.8,-0.2) -- (0.8,0.6);
}	

\draw[dash pattern=on 1pt off 1pt]
	(0.4,0.2) -- (2.8,0.2)
	(0,-0.2) -- (2.4,-0.2);

\end{scope}


\begin{scope}[xshift=5cm]

\foreach \a in {0,...,3}
\fill[yshift=-1.5*\a cm, gray!50]
	(0,0.6) -- (0,0) -- (-0.4,0) -- (-0.4,-0.6) -- (0.4,-0.6) -- (0.4,0) -- (0.8,0) -- (0.8,0.6);
	

\begin{scope}

\node at (-0.9,0) {$E_{\square}1$};

\foreach \a in {0,...,4}
{
\draw
	(0.8*\a,0) -- ++(0,0.6)
	(-0.4+0.8*\a,0) -- ++(0,-0.6);
	
\draw[dash pattern=on 1pt off 1pt]
	(-0.4+0.8*\a,0) -- ++(0.4,0);
}

\foreach \a in {0,...,3}
\draw[line width=1.2]
	(0.8*\a,0) -- ++(0.4,0);

\end{scope}


\begin{scope}[yshift=-1.5cm]

\node at (-0.9,0) {$E_{\square}^A1$};

\foreach \a in {0,...,4}
\draw
	(0.8*\a,0.6) -- ++(0,-0.6)
	(-0.4+0.8*\a,0) -- ++(0,-0.6)
	(-0.4+0.8*\a,0) -- ++(0.4,0);

\foreach \a in {0,...,3}
{
\draw
	(0.8*\a,0) -- ++(0.4,0);

\draw[line width=1.2]
	(0.8*\a,0) -- ++(0.4,0);
}

\end{scope}


\begin{scope}[yshift=-3cm]

\node at (-0.9,0) {$E_{\square}^K1$};

\foreach \a in {0,...,4}
\draw
	(0.8*\a,0.6) -- ++(0,-0.6)
	(-0.4+0.8*\a,0) -- ++(0,-0.6);

\draw[line width=1.2]
	(-0.4,0) -- (3.2,0);

\end{scope}


\begin{scope}[yshift=-4.5cm]

\node at (-0.9,0) {$E_{\square}^R1$};

\foreach \a in {0,...,4}
{
\draw
	(0.8*\a,0.6) -- ++(0,-0.6)
	(-0.4+0.8*\a,0) -- ++(0,-0.6);

}

\draw
	(-0.4,0) -- (3.2,0);
	
\foreach \a in {0,...,3}
{
\begin{scope}[xshift=0.8*\a cm]

\fill
	(0.7,0.1) circle (0.04)
	(0.1,0.1) circle (0.04)
	(0.3,-0.1) circle (0.04)
	(-0.3,-0.1) circle (0.04);

\end{scope}
}

\end{scope}
	

\begin{scope}[yshift=-6cm]

\fill[gray!50]
	(-0.4,0.6) -- (-0.4,-0.6) -- (1.2,-0.6) -- (1.2,0.6);

\node at (-0.9,0) {$E_{\square}2$};

\foreach \a in {0,1,2,3}
{
\begin{scope}[xshift=-0.4 cm+ 1.6*\a cm]

\draw
	(0,0.2) -- (0,0.6);
	
\draw[dash pattern=on 1pt off 1pt]
	(0,0.2) -- (0,-0.6);
	
\end{scope}
}

\foreach \a in {0,1,2}
{
\begin{scope}[xshift=-0.4 cm+ 1.6*\a cm]

\draw
	(0.8,-0.2) -- (0.8,-0.6)
	(0.4,-0.2) -- (0.4,0.2)
	(1.2,-0.2) -- (1.2,0.2);

\draw[dash pattern=on 1pt off 1pt]
	(0.8,-0.2) -- (0.8,0.6);

\draw[line width=1.2]
	(0,-0.2) -- (0.4,-0.2)
	(1.6,-0.2) -- (1.2,-0.2)
	(0.4,0.2) -- (1.2,0.2);
		
\end{scope}
}

\draw
	(-0.4,-0.2) -- (4.4,-0.2)
	(-0.4,0.2) -- (4.4,0.2);
		
\end{scope}

\end{scope}

\end{tikzpicture}
\caption{Triangular and quadrilateral earth map tilings.}
\label{emt}
\end{figure}
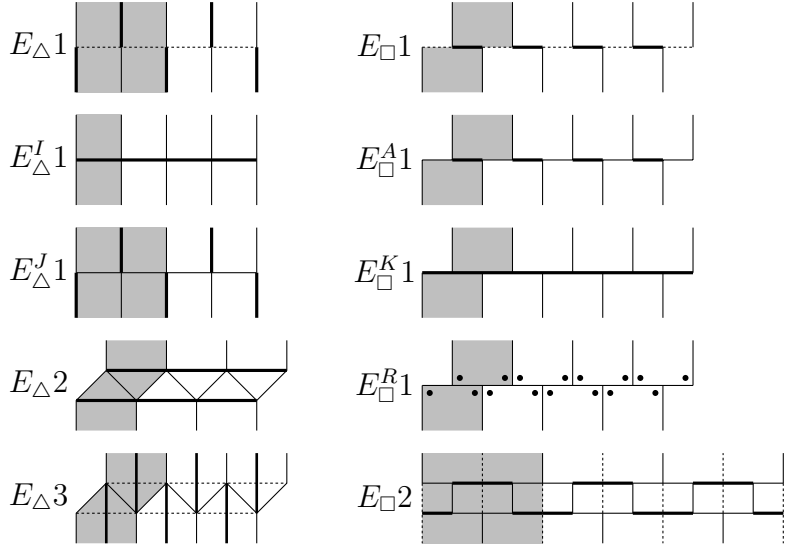

Like the Platonic solids, the edge lengths $a,b,c$ (corresponding to normal, thick and dashed edges) in the earth map tilings are not assumed to be distinct. For the convenience of subsequent discussions, we introduce additional decorations in case some edge lengths become equal. We denote the two isosceles reductions $a=b$ and $a=c$ of $E_{\triangle}1$ by $E_{\triangle}^I1$ and  $E_{\triangle}^J1$. In $E_{\triangle}1$ and $E_{\triangle}^J1$, the number $f$ of triangles in the tiling needs to be a multiple of $4$. In $E_{\triangle}^I1$, the number $f$ only needs to be even, which causes slight conflict with the word ``reduction''. The further reduction $a=b=c$ of $E_{\triangle}1$ happens only for $f=8$ and is the regular octahedron $P_8$.

The tiling $E_{\square}1$ has edge combination $a^2bc$. We denote the reductions $a=c$ (almost equilateral), $b=c$ (kite), and $a=b=c$ (rhombus) of $E_{\square}1$ by $E_{\square}^A1,E_{\square}^K1,E_{\square}^R1$. Moreover, $\bullet$ in $E_{\square}^R1$ indicates the $\beta$ angles.

We also use $E_{f\triangle}i$ and $E_{f\square}i$ to indicate the number $f$ of tiles. Then $E_{6\square}1, E_{8\triangle}1,E_{20\triangle}2$ are the Platonic solids $P_6,P_8,P_{20}$ in Figure \ref{platonic_solids}, and $E_{24\triangle}3$ is the triangular subdivision $T_{\triangle}P_6$ of the cube.

Table \ref{emt_data} gives the geometrical data for the triangular and quadrilateral earth map tilings. The angle values omit $\pi$, and $1-\gamma$ and $1-\gamma-\delta$ in the table mean $\beta+\gamma=\pi$ and $\beta+\gamma+\delta=\pi$. The parameter $p$ is the number of timezones. 

\renewcommand{\arraystretch}{1.2}

\begin{table}[h]
\centering
\scalebox{1}{
\begin{tabular}{|c|c|c|c|c|l|c|} 
\hline
Tiling
& $\alpha$ & $\beta$ & $ \gamma$ & $\delta$
& \multicolumn{1}{|c|}{Vertex}
& $f$ \\
\hline  \hline 
$E_{\triangle}1$
& $\frac{4}{f}$ & $1-\gamma$ & &
& $\beta^2\gamma^2,\alpha^{2p}$
& $4p$ \\
\hline
$E_{\triangle}^I1$
& $\frac{4}{f}$ & $\tfrac{1}{2}$  & &
& $\beta^4,\alpha^p$ 
& $2p$ \\
\hline
$E_{\triangle}2$
& $\frac{8}{f}$ & $\tfrac{1}{2}-\tfrac{2}{f}$ & &
& $\alpha\beta^4,\alpha^p$ 
& $4p$ \\
\hline
$E_{\triangle}3$
& $\frac{8}{f}$ & $\tfrac{1}{2}-\tfrac{4}{f}$ & $\tfrac{1}{2}$ &
& $\alpha^2\beta^4,\gamma^4,\alpha^{2p}$ 
& $8p$ \\
\hline
$E_{\square}1$
& $\frac{4}{f}$ & $1-\gamma-\delta$ & &
& $\beta\gamma\delta,\alpha^p$  
& $2p$ \\
\hline
$E_{\square}2$
& $1-\tfrac{8}{f}$ & $\tfrac{1}{2}+\tfrac{4}{f}$ & $\tfrac{1}{2}$ & $\tfrac{8}{f}$ 
& $\alpha\beta^2,\alpha^2\delta^2,\gamma^4,\delta^{2p}$
& $8p$  \\
\hline
\end{tabular}
}
\caption{Angles ($\pi$ omitted) and vertices for earth map tilings.}
\label{emt_data}
\end{table}

\renewcommand{\arraystretch}{1}

\subsection{Subdivision}
\label{division}

A {\em subdivision} is an operation of dividing polygons into smaller $n$-gons, where $n$ is fixed. Moreover, the operation should be applied to all the tiles in any edge-to-edge tiling, such that the divisions of the two tiles on the two sides of any edge are compatible. The triangular and quadrilateral subdivisions needed for constructing tilings in this paper are given in Figure \ref{subdivision}, where the original tiling is given by the thick edges, and the subdivisions are given by adding the normal edges. We remark that there are also two pentagonal subdivisions: \underline{p}entagonal subdivision $P_{\pentagon}P_i$, and \underline{d}ouble pentagonal subdivision $D_{\pentagon}P_i$. They are introduced in \cite[Section 3]{wy1}, and their combinatorial properties are studied in \cite{yan2}.

\begin{figure}[htp]
\centering
\begin{tikzpicture}[>=latex,scale=0.9]


\foreach \a/\b in {0,1,2,3}
\draw[xshift=3*\a cm, line width=1.2]
	(-0.6,-0.6) rectangle (0.6,0.6)
	(-0.6,0.6) -- (0,1.8) -- (0.6,0.6) -- (1.8,0) -- (0.6,-0.6)
	(0,1.8) -- (1,1.8) -- (1.8,1) --  (1.8,0)
	;

\draw[gray, dotted, xshift=12 cm]
	(-0.6,-0.6) rectangle (0.6,0.6)
	(-0.6,0.6) -- (0,1.8) -- (0.6,0.6) -- (1.8,0) -- (0.6,-0.6)
	(0,1.8) -- (1,1.8) -- (1.8,1) --  (1.8,0)
	;		


\foreach \a in {1,3,4}
\draw[xshift=3*\a cm]
	(-0.6,-0.6) -- (1,1)
	(0,1) -- (-0.6,0.6) -- (0.6,-0.6) -- (1,0)
	(0,1.8) -- (0,1) -- (0.6,0.6) -- (1,0) -- (1.8,0)
	(0,1.8) -- (1,1) -- (1.8,0)
	(1,1.8) -- (1,1) -- (1.8,1)
	;

\foreach \a in {2,3}
\draw[xshift=3*\a cm]
	(-0.6,0) -- (1,0) -- (1.2,-0.3)
	(0,-0.6) -- (0,1) -- (-0.3,1.2)
	(1,0) -- (1.2,0.3) -- (1,1) -- (0.3,1.2) -- (0,1)
	(0.5,1.8) -- (1,1) -- (1.8,0.5)
	(1,1) -- (1.4,1.4)
	;

\node at (0.6,-1) {original};
\node at (3.6,-1) {triangular};
\node at (6.6,-1) {quadrilateral};
\node at (9.6,-1) {barycentric};
\node at (12.6,-1) {quadricentric};

\end{tikzpicture}
\caption{Triangular and quadrilateral subdivisions.}
\label{subdivision}
\end{figure}
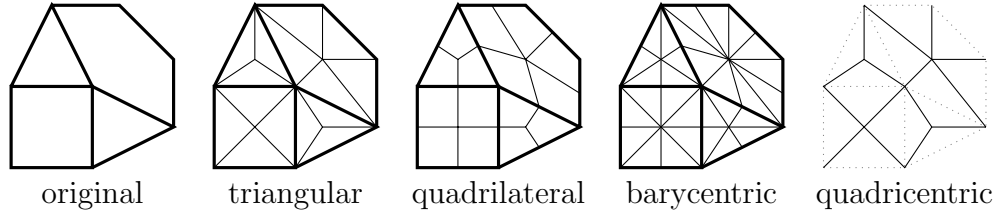

We keep original edges in the triangular subdivision, and subdivide the original edges in the quadrilateral and barycentric subdivisions. We drop the original edges in the quadricentric subdivision, which causes slight conflict with the word ``subdivision''. 

The barycentric, quadrilateral, quadricentric subdivisions of a tiling are the same as the subdivisions of the dual tiling. Moreover, applying the quadricentric subdivision twice gives the quadrilateral subdivision.

Each regular Platonic solid is a tiling by congruent regular polygons. If we apply the subdivisions uniformly to the regular faces of a Platonic solid, then we get tilings of the sphere by congruent polygons. Figure \ref{subdivision_platonic} gives all the triangular and quadrilateral subdivisions of Platonic solids, where the original Platonic solids are indicated by the thick edges, and the subdivisions are indicated by the normal and dashed edges.

\begin{figure}[htp]
\centering
\begin{tikzpicture}[>=latex, scale=0.75]

\node at (3,-1.8) {triangular subdivision $T_{\triangle}P_i$};
\node at (3,-6.3) {barycentric subdivision $B_{\triangle}P_i$};
\node at (3,-10.8) {quadrilateral subdivision $Q_{\square}P_i$};
\node at (3,-15.3) {quadricentric subdivision $C_{\square}P_i$};


\foreach \b in {0,1,2}
{
\begin{scope}[line width=1.2, yshift=-4.5*\b cm]


\foreach \a in {0,1,2}
\draw[rotate=120*\a]
	(-30:0.8) -- (90:0.8) -- (90:1.2);


\foreach \a in {0,1,2,3}
\draw[xshift=2.5cm, rotate=90*\a]
	(-0.4,0.4) -- (0.4,0.4) -- (1,1)
	(1,1) -- (-1,1);


\foreach \a in {0,1,2,3}
\draw[xshift=5.5cm, rotate=90*\a]
	(0.7,0.7) -- (-0.7,0.7)
	(0,0) -- (1.3,1.3);


\foreach \a in {0,...,4}
\draw[xshift=9.3cm, rotate=72*\a]
	(90:0.5) -- (18:0.5) -- (18:1) -- (54:1.4) -- (90:1)
	(54:1.4) -- (54:2) -- (-18:2) -- (-18:1.4);


\foreach \a in {0,...,4}
\draw[shift={(14cm,-0.2cm)}, rotate=72*\a]
	(0,0) -- (18:0.9) -- (90:0.9)
	(-18:1.8) -- (54:1.8) -- (54:2.2)
	(18:0.9) -- (54:1.8) -- (90:0.9);
	
\end{scope}	
}


\foreach \b in {0,1,3}
{
\begin{scope}[yshift=-4.5*\b cm]


\foreach \a in {0,1,2}
\draw[rotate=120*\a]
	(0,0) -- (-30:0.8) -- (30:0.8) -- (90:0.8)
	(30:0.8) -- (30:1.2);


\foreach \a in {0,1,2,3}
\draw[xshift=2.5cm, rotate=90*\a]
	(0,0) -- (0.4,0.4) -- (0.7,0) -- (0.4,-0.4)
	(1,1) -- (0.7,0) -- (1,-1)
	(1,1) -- (1.3,1.3);

	
\foreach \a in {0,1,2,3}
\draw[xshift=5.5cm, rotate=90*\a]
	(0,0) -- (0.4,0)
	(0.7,0.7) -- (0.4,0) -- (0.7,-0.7)
	(0.7,0.7) -- (1,0) -- (0.7,-0.7)
	(1,0) -- (1.3,0);


\foreach \a in {0,...,4}
\draw[xshift=9.3cm, rotate=72*\a]
	(0,0) -- (18:0.5) -- (54:0.9) -- (90:0.5)
	(18:1) -- (54:0.9) -- (90:1)
	(54:0.9) -- (54:1.4) -- (18:1.35) -- (-18:1.4)
	(18:1) -- (18:1.35)
	(54:2) -- (18:1.35) -- (-18:2) -- (-18:2.4);


\foreach \a in {0,...,4}
\draw[shift={(14cm,-0.2cm)}, rotate=72*\a]
	(0,0) -- (54:0.45)
	(18:0.9) -- (54:0.45) -- (90:0.9)
	(18:0.9) -- (54:1.1) -- (90:0.9)
	(54:1.1) -- (54:1.8)
	(-18:1.8) -- (18:1.2) -- (54:1.8)
	(18:0.9) -- (18:1.2)
	(-18:1.8) -- (18:1.8) -- (54:1.8)
	(18:2.2) -- (18:1.8);

\end{scope}
}

\begin{scope}[yshift=-9cm]


\foreach \a in {0,1,2}
\draw[rotate=120*\a]
	(0,0) -- (30:0.8) -- (90:1.05) -- (150:0.8);


\foreach \a in {0,1,2,3}
\draw[xshift=2.5cm, rotate=90*\a]
	(0,0) -- (1.3,0)
	(0.7,0.7) -- (0.7,-0.7);


\foreach \a in {0,1,2,3}
\draw[xshift=5.5cm, rotate=90*\a]
	(-0.4,0.4) -- (0.4,0.4) 
	(1,1) -- (-1,1)
	(0.4,0) -- (1,0);


\foreach \a in {0,...,4}
\draw[xshift=9.3cm, rotate=72*\a]
	(0,0) -- (54:0.9) -- (-18:0.9) -- (0:1.11) -- (18:1.35) -- (36:1.11) -- (54:0.9)
	(18:1.35) -- (54:1.7) -- (90:1.35) -- (90:2);


\foreach \a in {0,...,4}
\draw[shift={(14cm,-0.2cm)}, rotate=72*\a]
	(90:0.45) -- (54:0.45) -- (18:0.45)
	(54:0.45) -- (54:1.1) -- (42:1.3) -- (18:1.2) -- (-6:1.3) -- (-18:1.1)
	(18:1.2) -- (18:1.8) -- (54:2.05) -- (90:1.8);
	
\end{scope}


\begin{scope}[dash pattern=on 1pt off 1pt, yshift=-4.5cm]


\foreach \a in {0,1,2}
\draw[rotate=120*\a]
	(0,0) -- (30:0.8) -- (90:1.05) -- (150:0.8);


\foreach \a in {0,1,2,3}
\draw[xshift=2.5cm, rotate=90*\a]
	(0,0) -- (1.3,0)
	(0.7,0.7) -- (0.7,-0.7);


\foreach \a in {0,1,2,3}
\draw[xshift=5.5cm, rotate=90*\a]
	(-0.4,0.4) -- (0.4,0.4) 
	(1,1) -- (-1,1)
	(0.4,0) -- (1,0);


\foreach \a in {0,...,4}
\draw[xshift=9.3cm, rotate=72*\a]
	(0,0) -- (54:0.9) -- (-18:0.9) -- (0:1.11) -- (18:1.35) -- (36:1.11) -- (54:0.9)
	(18:1.35) -- (54:1.7) -- (90:1.35) -- (90:2);


\foreach \a in {0,...,4}
\draw[shift={(14cm,-0.2cm)}, rotate=72*\a]
	(90:0.45) -- (54:0.45) -- (18:0.45)
	(54:0.45) -- (54:1.1) -- (42:1.3) -- (18:1.2) -- (-6:1.3) -- (-18:1.1)
	(18:1.2) -- (18:1.8) -- (54:2.05) -- (90:1.8);
	
\end{scope}


\begin{scope}[yshift=-13.5cm]

\foreach \a in {0,1,2,3}
\fill[xshift=2.5cm, rotate=90*\a] 
	(0:0.15) circle (0.04)
	(0:0.58) circle (0.04)
	(0:0.8) circle (0.04)
	(15:0.7) circle (0.04)
	(-15:0.7) circle (0.04)
	(0:1.2) circle (0.04);

\foreach \a in {0,1,2,3}
\fill[xshift=5.5cm, rotate=90*\a] 
	(0.1,0.1) circle (0.04)
	(0.55,0.55) circle (0.04)
	(0.7,0.5) circle (0.04)
	(0.5,0.7) circle (0.04)
	(0.75,0.75) circle (0.04)
	(1,1) circle (0.04);

\foreach \a in {0,1,2,3,4}
\fill[xshift=9.3cm, rotate=72*\a] 
	(54:0.15) circle (0.04)
	(54:0.72) circle (0.04)
	(43:0.82) circle (0.04)
	(65:0.82) circle (0.04)
	(48:1) circle (0.04)
	(60:1) circle (0.04)
	(18:1.45) circle (0.04)
	(30:1.38) circle (0.04)
	(6:1.38) circle (0.04)
	(23:1.22) circle (0.04)
	(13:1.22) circle (0.04)
	(90:1.9) circle (0.04);
	
\foreach \a in {0,1,2,3,4}
\fill[shift={(14cm,-0.2cm)}, rotate=72*\a] 
	(18:0.15) circle (0.04)
	(18:0.7) circle (0.04)
	(7:0.82) circle (0.04)
	(29:0.82) circle (0.04)
	(12:1) circle (0.04)
	(24:1) circle (0.04)
	(54:1.9) circle (0.04)
	(63:1.65) circle (0.04)
	(45:1.65) circle (0.04)
	(57.5:1.55) circle (0.04)
	(50.5:1.55) circle (0.04)
	(54:2.2) circle (0.04);
		
\end{scope}

\end{tikzpicture}
\caption{Triangular and quadrilateral subdivisions of Platonic solids.}
\label{subdivision_platonic}
\end{figure}
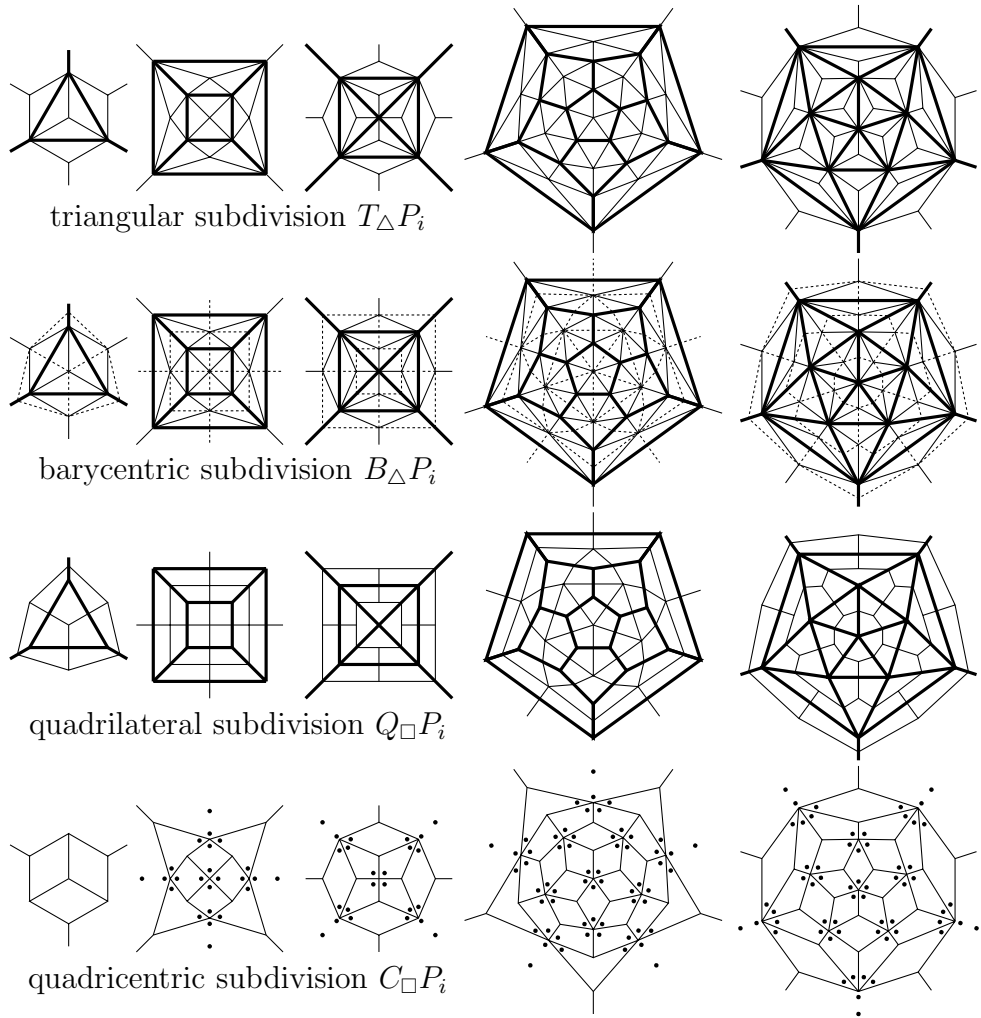

We denote the \underline{t}riangular, \underline{b}arycentric, \underline{q}uadrilateral, and quadri\underline{c}entric subdivisions of Platonic solids by $T_{\triangle}P_i,B_{\triangle}P_i,Q_{\square}P_i,C_{\square}P_i$. We note that the exchange of $(\alpha,b)$ with $(\beta,c)$ identifies $B_{\triangle}P_6$ with $B_{\triangle}P_8$, and identifies $B_{\triangle}P_{12}$ with $B_{\triangle}P_{20}$. Similarly, the exchange of $(\alpha,a)$ with $(\gamma,b)$ gives $Q_{\square}P_6=Q_{\square}P_8$ and $Q_{\square}P_{12}=Q_{\square}P_{20}$. Moreover, we have $b=c$ in $B_{\triangle}P_4$, and $B_{\triangle}P_4=T_{\triangle}P_6$, which is also the reduction $b=c$ of $E_{24\triangle}3$.

We note that $C_{\square}P_4$ is the regular cube $P_6$, with square (rhombus with $\alpha=\beta$) tiles. Moreover, $C_{\square}P_6,C_{\square}P_8,C_{\square}P_{12},C_{\square}P_{20}$ are (non-square) rhombus tilings, and we use $\bullet$ to indicate the $\beta$ angles. Since the quadricentric subdivision of a tiling and its dual tiling are the same, we have $C_{\square}P_6=C_{\square}P_8$ and $C_{\square}P_{12}=C_{\square}P_{20}$.  Moreover, we have $a=b$ in $Q_{\square}P_4$, making $Q_{\square}P_4$ into a rhombus tiling. In fact, we have $Q_{\square}P_4=C_{\square}P_6=C_{\square}P_8$.

The quadrilateral subdivision $Q_{\square}P_6$ ($=Q_{\square}P_8$) can be deformed, with one free parameter, as in Figure \ref{deformq6}. Figure \ref{flip1} shows another view of $Q_{\square}P_6$ that puts $\alpha^3$ at the center.

\begin{figure}[htp]
\centering
\begin{tikzpicture}[>=latex]


\foreach \a in {1,-1}
{
\begin{scope}[scale=\a]

\draw
	(-1,1) -- (1,1) -- (1,-1)
	(-0.4,0.4) -- (0.4,0.4) -- (0.4,-0.4)
	(0.4,0.4) -- (1,1)
	(-0.4,0.4) -- (-1,1);

\draw[line width=1.2]
	(0.7,0.7) -- (-0.7,0.7)
	(0,0.4) -- (0,0)
	(0.4,0) -- (1,0)
	(0,1) -- (0,1.3);
	
\draw[dash pattern=on 1pt off 1pt]
	(0.7,0.7) -- (0.7,-0.7)
	(0.4,0) -- (0,0)
	(0,0.4) -- (0,1)
	(1,0) -- (1.3,0);

\end{scope}
}

\node at (0,-1.5) {cube view};


\foreach \a in {1,-1}
{
\begin{scope}[xshift=3.5cm, scale=\a]

\draw
	(-1,1) -- (1,1) -- (1,-1)
	(-0.4,0.4) -- (0.4,0.4) -- (0.4,-0.4)
	(0.4,0) -- (1,0)
	(0,0.4) -- (0,1);

\draw[line width=1.2]
	(0,0) -- (-0.4,-0.4)
	(0,0.7) -- (0.7,0.7) -- (0.7,0)
	(-1,1) -- (-0.4,0.4)
	(1,1) -- (1.3,1.3);
	
\draw[dash pattern=on 1pt off 1pt]
	(0,0) -- (-0.4,0.4)
	(0,0.7) -- (-0.7,0.7) -- (-0.7,0)
	(1,1) -- (0.4,0.4)
	(-1,1) -- (-1.3,1.3);

\end{scope}
}

\node at (3.5,-1.5) {octahedron view};
	
\end{tikzpicture}
\caption{Deformed $Q_{\square}P_6$.}
\label{deformq6}
\end{figure}
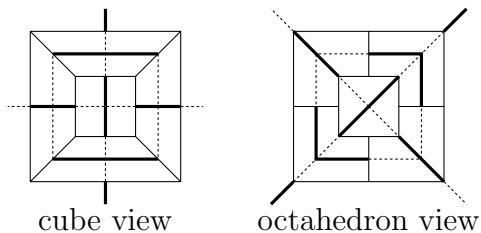

Table \ref{platonic_division_data} gives the data for the triangular and quadrilateral subdivisions of Platonic solids. If several tilings are equal, then the data is given for the first one. Again the angle values omit $\pi$, and $1-\delta$ in the table means $\beta+\delta=\pi$. For $Q_{\square}P_6$, the data are given for the deformed version. 

\renewcommand{\arraystretch}{1.2}

\begin{table}[htp]
\centering
\scalebox{1}{
\begin{tabular}{|c|c|c|c|l|c|} 
\hline
Tiling
& $\alpha$ & $\beta$ & $ \gamma$ 
& \multicolumn{1}{|c|}{Vertex}
& $f$ \\
\hline  \hline 
$T_{\triangle}P_4$
& $\frac{2}{3}$ & $\frac{1}{3}$ &
& $\alpha^3,\beta^6$ 
& $12$ \\
\hline
$T_{\triangle}P_6,B_{\triangle}P_4$
& $\frac{1}{2}$ & $\frac{1}{3}$ &
& $\alpha^4,\beta^6$ 
& $24$  \\
\hline
$T_{\triangle}P_8$
& $\frac{2}{3}$ & $\frac{1}{4}$ &
& $\alpha^3,\beta^8$ 
& $24$ \\
\hline
$T_{\triangle}P_{12}$
& $\frac{2}{5}$ & $\frac{1}{3}$ &
& $\alpha^5,\beta^6$ 
& $60$ \\
\hline
$T_{\triangle}P_{20}$
& $\frac{2}{3}$ & $\frac{1}{5}$ &
& $\alpha^3,\beta^{10}$ 
& $60$ \\
\hline
$B_{\triangle}P_6,B_{\triangle}P_8$
& $\frac{1}{3}$ & $\frac{1}{4}$ & $\frac{1}{2}$ 
& $\alpha^6,\beta^8,\gamma^4$ 
& $48$ \\
\hline
$B_{\triangle}P_{12},B_{\triangle}P_{20}$
& $\frac{1}{3}$ & $\frac{1}{5}$ & $\frac{1}{2}$ 
& $\alpha^6,\beta^{10},\gamma^4$ 
& $120$ \\
\hline
$Q_{\square}P_4,C_{\square}P_6,C_{\square}P_8$
& $\frac{2}{3}$ & $\frac{1}{2}$ &
& $\alpha^3,\beta^4$ 
& $12$ \\
\hline
$Q_{\square}P_6,Q_{\square}P_8$
& $\frac{2}{3}$ & $1-\delta$ & $\frac{1}{2}$ 
& $\alpha^3,\beta^2\delta^2,\gamma^4$ 
& $24$ \\
\hline
$Q_{\square}P_{12},Q_{\square}P_{20}$
& $\frac{2}{5}$ & $\frac{1}{2}$ & $\frac{2}{3}$ 
& $\alpha^5,\beta^4,\gamma^3$ 
& $60$ \\
\hline
$C_{\square}P_{12},C_{\square}P_{20}$
& $\frac{2}{3}$ & $\frac{2}{5}$ &
& $\alpha^3,\beta^5$ 
& $30$ \\
\hline
\end{tabular}
}
\caption{Angles ($\pi$ omitted) and vertices for subdivisions of Platonic solids. The data for multiple identical tilings are given for the first one. The data for $Q_{\square}P_6$ is given according to Figure \ref{deformq6}. }
\label{platonic_division_data}
\end{table}

\renewcommand{\arraystretch}{1}

Finally, we note that $E_{\triangle}3$ is the triangular subdivision of $E_{\square}^R1$. We denote the fact by writing $E_{\triangle}3=T_{\triangle}E_{\square}^R1$.  

\subsection{Simple Subdivision}
\label{simpledivision}

Section \ref{division} describes ``universal'' subdivisions that can be applied to any tiling. We also have {\em simple subdivisions} in Figure \ref{simple_subdivision}, that are applied to specific types of tilings.

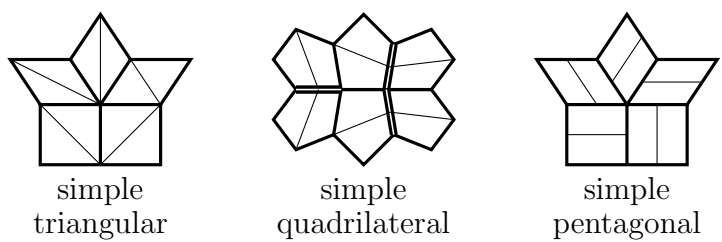
\begin{figure}[htp]
\centering
\begin{tikzpicture}[>=latex]


\draw[line width=1.2]
	(-0.8,0) rectangle (0.8,-0.8)
	(0,0) -- (0,-0.8)
	(-0.8,0) -- (-1.2,0.6) -- (-0.4,0.6) -- (0,0) -- (0.4,0.6) -- (1.2,0.6) -- (0.8,0)
	(-0.4,0.6) -- (0,1.2) -- (0.4,0.6);

\draw
	(-0.8,0) -- (0,-0.8) -- (0.8,0) -- (0.4,0.6)
	(-1.2,0.6) -- (0,0) -- (0,1.2);

\node at (0,-1.15) {simple};
\node at (0,-1.6) {triangular};


\begin{scope}[shift={(3.5cm,0.2cm)}]

\foreach \b in {1,-1}
{
\begin{scope}[yscale=\b]

\begin{scope}[line width=1.2]

\draw
	(-0.4,0.6) -- (-0.9,0.8) -- (-1.2,0.4) -- (-0.9,0) -- (0.9,0) -- (1.2,0.4) -- (0.9,0.8) -- (0.4,0.6)
	(0.3,0) -- (0.4,0.6) -- (0,1) -- (-0.4,0.6) -- (-0.3,0);

\draw[double, line width=1.2]
	(-0.9,0) -- (-0.3,0)
	(0.3,0) -- (0.4,0.6);

\end{scope}

\draw
	(-0.9,0.8) -- (-0.6,0)
	(-0.4,0.6) -- (0.35,0.3) -- (1.2,0.4);

\end{scope}
}

\node at (0,-1.35) {simple};
\node at (0,-1.8) {quadrilateral};

\end{scope}


\begin{scope}[xshift=7cm]

\draw[line width=1.2]
	(-0.8,0) rectangle (0.8,-0.8)
	(0,0) -- (0,-0.8)
	(-0.8,0) -- (-1.2,0.6) -- (-0.4,0.6) -- (0,0) -- (0.4,0.6) -- (1.2,0.6) -- (0.8,0)
	(-0.4,0.6) -- (0,1.2) -- (0.4,0.6);

\draw
	(-0.8,-0.4) -- (0,-0.4) 
	(0.4,0) -- (0.4,-0.8)
	(0.2,0.3) -- (1,0.3)
	(-0.4,0) -- (-0.8,0.6)
	(-0.2,0.3) -- (0.2,0.9);

\node at (0,-1.15) {simple};
\node at (0,-1.6) {pentagonal};

\end{scope}

\end{tikzpicture}
\caption{Simple subdivisions.}
\label{simple_subdivision}
\end{figure}

In the first of Figure \ref{simple_subdivision}, the {\em simple triangular subdivision} $S_{\triangle}{\mc T}$ of a quadrilateral tiling ${\mc T}$ divides each quadrilateral into two triangles by drawing a diagonal. For all tiles in $S_{\triangle}{\mc T}$ to be congruent, ${\mc T}$ can be the following:
\begin{itemize}
\item Suppose ${\mc T}$ is a tiling by congruent kites (second of Figure \ref{quad}). Then we may divide each tile by connecting the $\alpha$-vertex and the $\gamma$-vertex. For example, we may write $E_{\triangle}1=S_{\triangle}E_{\square}^K1$, to indicate $E_{\triangle}1$ is the simple triangular subdivision of the kite earth map tiling $E_{\square}^K1$. Moreover, all tiles in $Q_{\square}P_i$ are kites, and we have $B_{\triangle}P_i=S_{\triangle}Q_{\square}P_i$.
\item Suppose ${\mc T}$ is a tiling by congruent rhombi (fourth of Figure \ref{quad}, with $\alpha\ne\beta$). Then we may divide each tile by connecting the two $\alpha$-vertices, or divide each tile by connecting the two $\beta$-vertices. We get two simple triangular subdivision tilings. For example, we have $E_{\triangle}^J1,E_{\triangle}2=S_{\triangle}E_{\square}^R1$, and $T_{\triangle}P_i=S_{\triangle}C_{\square}P_i$.
\item The tiling by congruent squares (fourth of Figure \ref{quad}, with $\alpha=\beta$) is the regular cube $P_6$. We may divide each tile by drawing the diagonal independently. Altogether we get seven non-equivalent simple subdivision tilings $S_{\triangle}P_6$ by $12$ isosceles triangles, given by Figure \ref{simple_subdivision_cube} in two views. The non-equivalence can be detected by the graphs formed by the normal edges. Three of the seven tilings can be identified with the pentagonal subdivision $T_{\triangle}P_4$ and the earth map tilings $E_{12\triangle}^J1,E_{12\triangle}2$. The remaining four do not have such interpretations, and will be regarded as sporadic tilings in Section \ref{sporadic}. 
\end{itemize}

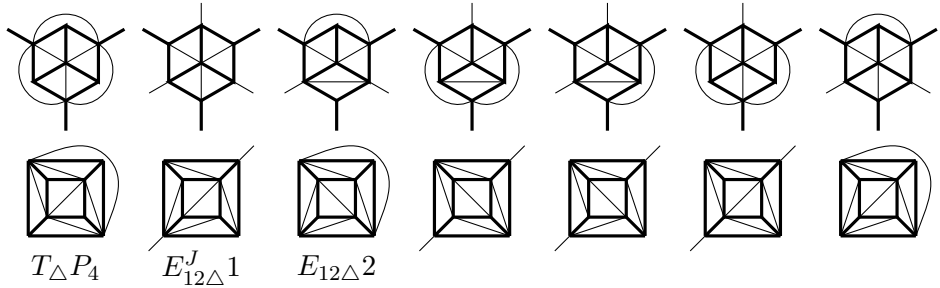
\begin{figure}[htp]
\centering
\begin{tikzpicture}[>=latex]


\foreach \a in {0,1,2}
\foreach \b in {0,...,6}
\draw[xshift=1.8*\b cm, rotate=120*\a, line width=1.2]
	(0,0) -- (-30:0.5) -- (30:0.5) -- (90:0.5)
	(30:0.5) -- (30:0.9);
	

c\foreach \a in {0,1,2}
\foreach \b in {0,1,5,6}
\draw[xshift=1.8*\b cm, rotate=120*\a]
	(0,0) -- (30:0.5);
	
\foreach \b in {2,3,4}
\draw[xshift=1.8*\b cm]
	(150:0.5) -- (0,0) -- (30:0.5)
	(210:0.5) -- (-30:0.5);

\foreach \a in {0,1,2}
\draw[rotate=120*\a]
	(30:0.5) to[out=90,in=0] 
	(90:0.65) to[out=180,in=90] 
	(150:0.5);

\foreach \a in {0,1,2}
\draw[xshift=1.8 cm, rotate=120*\a]
	(90:0.5) -- (90:0.8);	

\draw[xshift=3.6 cm]
	(30:0.5) to[out=90,in=0] 
	(90:0.65) to[out=180,in=90] 
	(150:0.5)
	(-30:0.5) -- (-30:0.8)
	(210:0.5) -- (210:0.8);
	
\foreach \a in {1,2}
\draw[xshift=3*1.8 cm, rotate=120*\a]
	(30:0.5) to[out=90,in=0] 
	(90:0.65) to[out=180,in=90] 
	(150:0.5);
\draw[xshift=3*1.8 cm]
	(90:0.5) -- (90:0.8);

\draw[xshift=4*1.8 cm, rotate=240]
	(30:0.5) to[out=90,in=0] 
	(90:0.65) to[out=180,in=90] 
	(150:0.5)
	(-30:0.5) -- (-30:0.8)
	(210:0.5) -- (210:0.8);
		
\foreach \a in {-1,1}
\draw[xshift=5*1.8 cm, rotate=120*\a]
	(30:0.5) to[out=90,in=0] 
	(90:0.65) to[out=180,in=90] 
	(150:0.5);
\draw[xshift=5*1.8 cm]
	(90:0.5) -- (90:0.8);

\draw[xshift=6*1.8 cm]
	(30:0.5) to[out=90,in=0] 
	(90:0.65) to[out=180,in=90] 
	(150:0.5)
	(-30:0.5) -- (-30:0.8)
	(210:0.5) -- (210:0.8);


\begin{scope}[yshift=-1.8cm]

\foreach \a in {0,1,2,3}
\foreach \b in {0,...,6}
\draw[xshift=1.8*\b cm, rotate=90*\a, line width=1.2]
	(-0.5,0.5) -- (0.5,0.5)
	(-0.25,0.25) -- (0.25,0.25) -- (0.5,0.5);
	
\foreach \b in {0,...,6}
\draw[xshift=1.8*\b cm] 
	(-0.5,0.5) -- (0.25,0.25) -- (0.5,-0.5);
	
\foreach \b in {0,1,5,6}
\draw[xshift=1.8*\b cm]
	(-0.25,-0.25) -- (0.25,0.25);

\foreach \b in {2,3,4}
\draw[xshift=1.8*\b cm]
	(-0.25,0.25) -- (0.25,-0.25);
	
\foreach \b in {0,3,5}
\draw[xshift=1.8*\b cm]
	(-0.5,0.5) -- (-0.25,-0.25) -- (0.5,-0.5);
	
\foreach \b in {0,2,6}
\draw[xshift=1.8*\b cm]
	(-0.5,0.5) to[out=20, in=135] 
	(0.55,0.55) to[out=-45, in=70] (0.5,-0.5);

\foreach \b in {1,3,4,5}
\draw[xshift=1.8*\b cm]
	(0.5,0.5) -- (0.7,0.7)
	(-0.5,-0.5) -- (-0.7,-0.7);

\foreach \b in {1,2,6}
\draw[xshift=1.8*\b cm]
	(-0.25,0.25) -- (-0.5,-0.5)
	(0.25,-0.25) -- (-0.5,-0.5);

\draw[xshift=4*1.8cm]
	(-0.5,0.5) -- (-0.25,-0.25)
	(-0.5,-0.5) -- (0.25,-0.25);
	
\end{scope}

\node at (0,-2.75) {\small $T_{\triangle}P_4$};
\node at (1.8,-2.75) {\small $E_{12\triangle}^J1$};
\node at (3.6,-2.75) {\small $E_{12\triangle}2$};

\end{tikzpicture}
\caption{Simple triangular subdivisions $S_{\triangle}P_6$ of the cube.}
\label{simple_subdivision_cube}
\end{figure}

A pentagonal tiling ${\mc T}$ is {\em paired}, if some edges of ${\mc T}$ are selected, such that each tile has exactly one selected edge. See the second of Figure \ref{simple_subdivision}, where the selected edges are indicated by thick double lines. The {\em simple quadrilateral subdivision} $S_{\square}{\mc T}$ divides each pair into four quadrilaterals by drawing a line perpendicular to the shared edge. For all tiles in $S_{\triangle}{\mc T}$ to be congruent, the pentagon in ${\mc T}$ needs to be symmetric. There are two pentagonal tilings with symmetric tiles:
\begin{itemize}
\item ${\mc T}$ is the pentagonal earth map tiling $E_{\pentagon}1$ in \cite[Proposition 1]{cly}. The simple quadrilateral subdivision of $E_{\pentagon}1$ is $E_{\square}2$, and we write $E_{\square}2=S_{\square}E_{\pentagon}1$. See Figure \ref{pemt}.

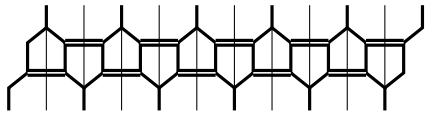
\begin{figure}[htp]
\centering
\begin{tikzpicture}[>=latex,scale=1]

\foreach \a in {0,...,5}
\draw[line width=1.2, xshift=\a cm]
	(-0.5,-0.7) -- (-0.5,-0.4) -- (-0.25,-0.2) -- (-0.25,0.2) -- (0,0.4) -- (0,0.7);
	
\foreach \a in {0,...,4}
\draw[line width=1.2, xshift=\a cm]
	(-0.25,-0.2) -- (0.25,-0.2) -- (0.5,-0.4) 
	(0,0.4) -- (0.25,0.2) -- (0.75,0.2)
	(0.25,0.2) -- (0.25,-0.2);	

\foreach \a in {0,...,4}
\draw[double, line width=1.2, xshift=\a cm]
	(-0.25,-0.2) -- (0.25,-0.2)
	(0.25,0.2) -- (0.75,0.2);

\foreach \a in {0,...,4}
\draw[xshift=\a cm]
	(0,0.4) -- (0,-0.7)
	(0.5,-0.4) -- (0.5,0.7);

\end{tikzpicture}
\caption{$E_{\square}2$ is the simple subdivision of $E_{\pentagon}1$.}
\label{pemt}
\end{figure}

\item The dodecahedron tiling $P_{12}$ can be paired in five ways, given by the thick part of Figure \ref{simple_subdivision_dodecahedron}. Geometrically (thick edge has length $a$ and double thick edge has length $b$), the first is the reduction $a=b$ of the deformed dodecahedron in Figure \ref{platonic_solids}, and allows one free parameter. By \cite{ay1,gsy}, the other four in Figure \ref{simple_subdivision_dodecahedron} are the regular dodecahedron.

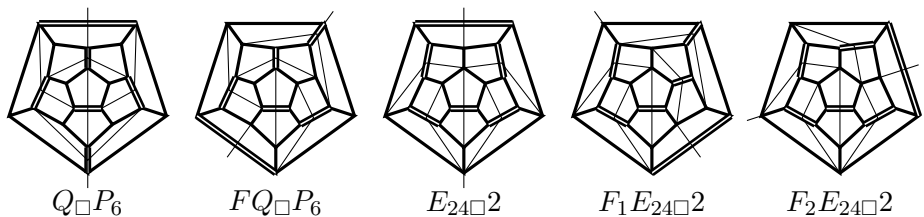
\begin{figure}[htp]
\centering
\begin{tikzpicture}

\foreach \a in {0,...,4}
\foreach \b in {0,...,4}
\draw[line width=1.2, xshift=2.5*\b cm, rotate=72*\a]
	(90:0.3) -- (18:0.3) -- (18:0.56) -- (54:0.76) -- (90:0.56)
	(54:0.76) -- (54:1.1) -- (-18:1.1);


\draw[double, line width=1]
	(234:0.3) -- (-54:0.3)
	(90:0.3) -- (90:0.56)
	(162:0.56) -- (198:0.76)
	(18:0.56) -- (-18:0.76)
	(-90:0.76) -- (-90:1.1)
	(126:1.1) -- (54:1.1);

\draw
	(-90:1.3) -- (90:1.1)
	(18:0.56) -- (90:0.43) -- (162:0.56)
	(54:1.1) -- (0:0.62) -- (-54:0.3)
	(126:1.1) -- (180:0.62) -- (234:0.3)
	(-18:0.76) -- (-90:0.93) -- (198:0.76);

\node at (0,-1.5) {\small $Q_{\square}P_6$};


\begin{scope}[xshift=2.5cm]

\draw[double, line width=1]
	(234:0.3) -- (-54:0.3)
	(90:0.3) -- (90:0.56)
	(162:0.56) -- (198:0.76)
	(-54:0.56) -- (-18:0.76)
	(54:0.76) -- (54:1.1)
	(198:1.1) -- (-90:1.1);

\draw
	(-90:0.76) -- (90:0.3)
	(18:0.56) -- (90:0.43) -- (162:0.56)
	(-90:1.1) -- (-36:0.62) -- (18:0.3)
	(126:1.1) -- (180:0.62) -- (234:0.3) -- (234:1.1)
	(126:0.76) -- (54:0.93) -- (-18:0.76)
	(54:1.1) -- (54:1.3);

\node at (0,-1.5) {\small $FQ_{\square}P_6$};

\end{scope}


\begin{scope}[xshift=5cm]

\draw[double, line width=1]
	(234:0.3) -- (-54:0.3)
	(18:0.56) -- (54:0.76)
	(162:0.56) -- (126:0.76)
	(-54:0.56) -- (-18:0.76)
	(234:0.56) -- (198:0.76)
	(126:1.1) -- (54:1.1);

\draw
	(-90:0.76) -- (90:1.1)
	(-90:1.3) -- (-90:1.1) -- (-36:0.62) -- (18:0.3)
	(-90:1.1) -- (216:0.62) -- (162:0.3)
	(-18:1.1) -- (36:0.62) -- (90:0.3)
	(198:1.1) -- (144:0.62) -- (90:0.3);

\node at (0,-1.5) {\small $E_{24\square}2$};

\end{scope}


\begin{scope}[xshift=7.5cm]
	
\draw[double, line width=1] 
	(234:0.3) -- (-54:0.3)
	(162:0.56) -- (126:0.76)
	(234:0.56) -- (198:0.76)
	(18:0.3) -- (18:0.56)
	(54:0.76) -- (54:1.1)
	(-18:1.1) -- (-90:1.1);

\draw
	(-90:0.76) -- (90:0.3)
	(-90:1.1) -- (216:0.62) -- (162:0.3)
	(198:1.1) -- (144:0.62) -- (90:0.3)
	(90:0.56) -- (18:0.43) -- (-54:0.56)
	(126:1.3) -- (126:0.76) -- (54:0.93) -- (-18:0.76)
	(-54:0.56) -- (-54:1.1);

\node at (0,-1.5) {\small $F_1E_{24\square}2$};
	
\end{scope}


\begin{scope}[xshift=10cm]

\draw[double, line width=1]
	(234:0.3) -- (-54:0.3)
	(90:0.56) -- (54:0.76)
	(162:0.56) -- (126:0.76)
	(-54:0.56) -- (-18:0.76)
	(234:0.56) -- (198:0.76)
	(-18:1.1) -- (54:1.1);

\draw
	(-90:0.76) -- (90:0.3)
	(-90:1.1) -- (-36:0.62) -- (18:0.3)
	(-90:1.1) -- (216:0.62) -- (162:0.3)
	(198:1.3) -- (198:1.1) -- (144:0.62) -- (90:0.3)
	(126:1.1) -- (72:0.62) -- (18:0.3)
	(18:0.56) -- (18:1.1);

\node at (0,-1.5) {\small $F_2E_{24\square}2$};

\end{scope}

\end{tikzpicture}
\caption{Simple quadrilateral subdivisions of the dodecahedron.}
\label{simple_subdivision_dodecahedron}
\end{figure}

The corresponding simple subdivisions give five quadrilateral tilings. The first and third are $Q_{\square}P_6,E_{\square}2$, and the remaining three are the flip modifications $FQ_{\square}P_6,F_1E_{\square}2,F_2E_{\square}2$ to be introduced in Section \ref{flip}. 
\end{itemize}

Finally, in the third of Figure \ref{simple_subdivision}, the {\em simple pentagonal subdivision} $S_{\pentagon}{\mc T}$ of a quadrilateral tiling ${\mc T}$ divides each quadrilateral tile by drawing a line connecting the middle points of two opposite edges. The lines are drawn in compatible way to yield a pentagonal tiling. The simple subdivision is discussed in \cite{yan2}, and underlies the pentagonal subdivision and the double pentagonal subdivision.

\subsection{Modification}
\label{flip}

If certain part ${\mc A}$ of a tiling ${\mc T}$ has the property that flipping ${\mc A}$ with respect to some axis does not change the angle values along the boundary of ${\mc A}$, then the flip gives another tiling $F{\mc T}$, called a {\em flip modification} of ${\mc T}$.

\medskip

\noindent {\bf Flip of Platonic type tilings}

\medskip

The first of Figure \ref{flip8} is another way of drawing the tiling $Q_{\square}P_6$. We can easily see the tiling is divided into the inner and outer halves. The second picture shows the angles along the boundary between the two halves. If $\alpha=\beta$, then the third picture shows the angle values of the inner half, where $\bar{\alpha}=\pi-\alpha$. There are many flips (or rotations) of the inner half that preserve the angle values, and one is the flip with respect to the gray line. All these give the same tiling $FQ_{\square}P_6$ in the fourth picture. We remark that $FQ_{\square}P_6$ is also the second of Figure \ref{simple_subdivision_dodecahedron}.

\begin{figure}[htp]
\centering
\begin{tikzpicture}

\foreach \a in {0,1,2}
\foreach \b in {0,1}
{
\begin{scope}[xshift=10*\b cm, rotate=120*\a]

\draw
	(90:0.8) -- (60:0.8) -- (30:0.8) -- (0:0.8) -- (-30:0.8)
	(30:0.8) -- (30:1.4);

\draw[dash pattern=on 1pt off 1pt]
	(60:0.8) -- (90:1.1) -- (150:1.1);

\draw[line width=1.2]
	(120:0.8) -- (90:1.1) -- (30:1.1);

\end{scope}

\begin{scope}[rotate=120*\a]

\draw
	(0,0) -- (90:0.8);
	
\draw[dash pattern=on 1pt off 1pt]
	(90:0.4) -- (0:0.8);

\draw[line width=1.2]
	(-30:0.4) -- (60:0.8);
			
\end{scope}

\begin{scope}[xshift=10cm, rotate=120*\a]

\draw
	(0,0) -- (60:0.8);
	
\draw[dash pattern=on 1pt off 1pt]
	(-60:0.4) -- (30:0.8);

\draw[line width=1.2]
	(60:0.4) -- (-30:0.8);
				
\end{scope}

}

\node at (0,-1.7) {\small $Q_{\square}P_6$};

\node at (10,-1.7) {\small $FQ_{\square}P_6$};


\begin{scope}[xshift=3.3cm]

\foreach \b in {0,1,2}
{
\begin{scope}[rotate=120*\b]

\foreach \a in {0,...,3}
\draw[rotate=30*\a]
	(0:1.4) -- (30:1.4);

\draw[line width=1.2]
	(60:1.4) -- (60:1)
	(0:1.4) -- (0:1.8);

\draw
	(-30:1.4) -- (-30:1)
	(30:1.4) -- (30:1.8);

\draw[dash pattern=on 1pt off 1pt]
	(60:1.4) -- (60:1.8)
	(0:1.4) -- (0:1);
		
\node at (90:1.55) {\small $\alpha$};
\node at (82:1.2) {\small $\alpha$};
\node at (98:1.2) {\small $\alpha$};

\node at (30:1.2) {\small $\alpha$};
\node at (24:1.55) {\small $\alpha$};
\node at (36:1.55) {\small $\alpha$};

\node at (54:1.55) {\small $\delta$};
\node at (66:1.55) {\small $\delta$};

\node at (-8:1.55) {\small $\beta$};
\node at (8:1.55) {\small $\beta$};

\end{scope}
}

\node at (52:1.15) {\small $\beta$};
\node at (68:1.15) {\small $\beta$};

\node at (-52:1.2) {\small $\beta$};
\node at (-70:1.2) {\small $\beta$};

\node at (170:1.15) {\small $\beta$};
\node at (190:1.15) {\small $\beta$};

\node at (-8:1.2) {\small $\delta$};
\node at (8:1.2) {\small $\delta$};

\node at (112:1.2) {\small $\delta$};
\node at (128:1.2) {\small $\delta$};

\node at (-113:1.2) {\small $\delta$};
\node at (-127:1.2) {\small $\delta$};

\end{scope}


\begin{scope}[xshift=6.8cm]

\foreach \b in {0,1,2}
{
\begin{scope}[rotate=120*\b]

\foreach \a in {0,...,3}
\draw[rotate=30*\a]
	(0:1.4) -- (30:1.4);
		
\node at (60:1.2) {\small $\bar{\alpha}$};
\node at (90:1.2) {\small $\bar{\alpha}$};
\node at (30:1.2) {\small $\alpha$};
\node at (0:1.2) {\small $\alpha$};

\end{scope}
}

\draw[gray]
	(75:1.6) -- (75:-1.6);

\end{scope}

\end{tikzpicture}
\caption{Flip modification $FQ_{\square}P_6$ of $Q_{\square}P_6$ in case $\alpha=\beta$.}
\label{flip8}
\end{figure}
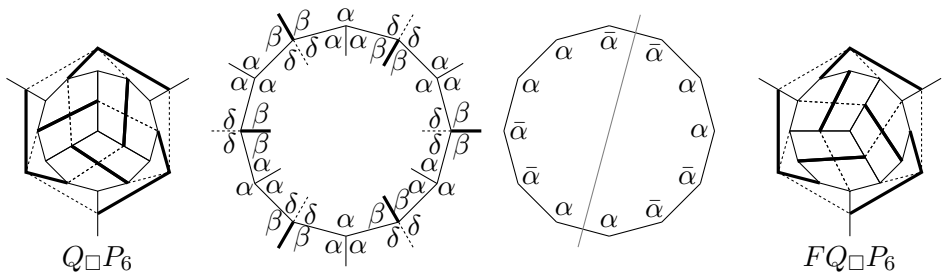

Figure \ref{flip1} shows the flip modifications $FB_{\triangle}P_8,FQ_{\square}P_8$ of $B_{\triangle}P_8,Q_{\square}P_8$. Again, $B_{\triangle}P_8$ and $Q_{\square}P_8$ consist of two half sphere tilings. The right of Figure \ref{flip1} shows the angle values to be $\pi$ along the boundary between the two halves. Therefore the two halves are actually genuine half spheres. All the rotations and flip modifications of the inner half give the same tiling as the flip with respect to the gray line.

\begin{figure}[htp]
\centering
\begin{tikzpicture}[>=latex]

\foreach \a in {0,...,3}
\foreach \b in {0,1}
{
\begin{scope}[xshift=3.5*\b cm, rotate=90*\a]

\draw
	(45:0.8) -- (0:1.1) -- (-45:0.8)
	(0:1.1) -- (0:1.4);

\draw[line width=1.2]
	(45:0.8) -- (45:1.4)
	(0:0.8) -- (45:0.8) -- (90:0.8);

\draw[dash pattern=on 1pt off 1pt]
	(0:0.8) -- (0:1.1)
	(45:1.1) -- (0:1.1) -- (-45:1.1);

\end{scope}

\begin{scope}[xshift=3.5*\b cm, rotate=45*\b+90*\a]

\draw
	(0,0) -- (0.45,0)
	(45:0.8) -- (0.45,0) -- (-45:0.8);

\draw[line width=1.2]
	(0,0) -- (45:0.8);

\draw[dash pattern=on 1pt off 1pt]
	(0:0.8) -- (0:0.45) -- (45:0.45) -- (90:0.45);

\end{scope}

}

\node at (1.6,-0.5) {\small $B_{\triangle}P_8$};
\node at (5.3,-0.5) {\small $FB_{\triangle}P_8$};


\begin{scope}[yshift=-3cm]

\foreach \a in {0,...,3}
\foreach \b in {0,1}
{
\begin{scope}[xshift=3.5*\b cm, rotate=90*\a]

\draw
	(0:0.8) -- (0:1.1)
	(45:1.1) -- (0:1.1) -- (-45:1.1);

\draw[line width=1.2]
	(45:0.8) -- (45:1.4)
	(0:0.8) -- (45:0.8) -- (90:0.8);

\end{scope}

\begin{scope}[xshift=3.5*\b cm, rotate=45*\b+90*\a]

\draw
	(0:0.8) -- (0:0.45) -- (45:0.45) -- (90:0.45);

\draw[line width=1.2]
	(0,0) -- (45:0.8);

\end{scope}
}

\node at (1.6,-0.5) {\small $Q_{\square}P_8$};
\node at (5.3,-0.5) {\small $FQ_{\square}P_8$};

\end{scope}


\foreach \a in {0,...,7}
{
\begin{scope}[shift={(7cm, -1.5cm)}]

\draw[gray]
	(22.5:1.2) -- (22.5:-1.2);
	
\draw[line width=1.2, rotate=45*\a]
	(0:1) -- (45:1);

\node at (45*\a:0.8) {\small $\pi$};
				
\end{scope}
}

\end{tikzpicture}
\caption{Flip modifications $FB_{\triangle}P_8,FQ_{\square}P_8$ of $B_{\triangle}P_8,Q_{\square}P_8$.}
\label{flip1}
\end{figure}
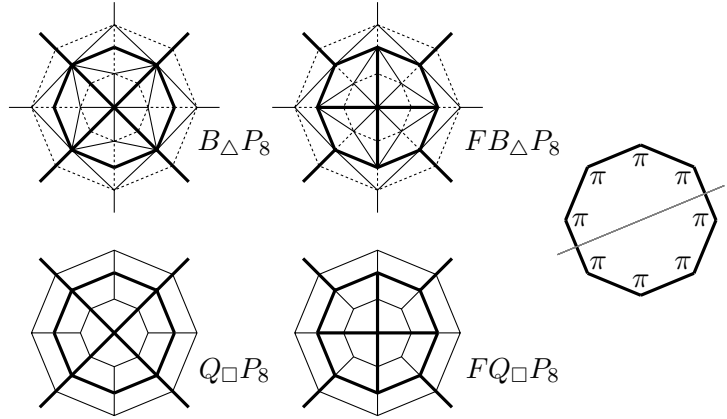

We remark that, despite $Q_{\square}P_6=Q_{\square}P_8$, the flip modifications $FQ_{\square}P_6$ and $FQ_{\square}P_8$ are different tilings. The former is obtained by flipping a half cube (therefore the subscript $6$), and the later is obtained by flipping a half octahedron (therefore the subscript $8$). 


\medskip

\noindent {\bf Modification of $E_{\square}1$, and subsequent flips of $E_{\square}^R1,E_{\triangle}^J1,E_{\triangle}2,E_{\triangle}3$}

\medskip

The general earth map tiling $E_{\square}1$ and the kite earth map tiling $E_{\square}^K1$ do not have modifications. Therefore the flip modification of $FE_{\square}1$ is only about $E_{\square}^A1$ and its $a=b$ reduction $E_{\square}^R1$. In a very special case, we also have a rearrangement $RE_{\square}^A1$, which we denote by $RE_{\square}1$.

The earth map tiling $E_{\square}^A1$ consists of  $p=\frac{f}{2}$ timezones. If $\alpha=s\beta$ for an integer $s$ (necessarily $>1$), then we collect $s$ timezones to get a {\em partial earth map tiling} ${\mc A}_s$, consisting of $2s$ tiles and with $\beta^s$ at the two ends. The first of Figure \ref{flip5} shows the partial earth map tiling ${\mc A}_7$. The second of Figure \ref{flip5} shows the angles along the boundary of ${\mc A}_s$. Then we may flip ${\mc A}_s$ with respect to the first gray line and get a new tiling. This turns the poles $\beta^p$ of the earth map tiling into $\alpha\beta^{p-s}$. In fact, if $t\alpha=ts\beta<2\pi$, then we may take $t$ non-overlapping copies of ${\mc A}_s$ and flip these ${\mc A}_s$ simultaneously. Then we get a new tiling which turns $\beta^p$ into $\alpha^t\beta^{p-st}$.  

Similarly, if $\gamma+\delta=s\beta$, then we may flip ${\mc A}_s$ with respect to the second gray line. Again we may flip $t$ non-overlapping copies of ${\mc A}_s$ simultaneously in case $t(\gamma+\delta)=ts\beta<2\pi$, which turns $\beta^p$ into $\beta^{p-st}\gamma^t\delta^t$.

\begin{figure}[htp]
\centering
\begin{tikzpicture}

\begin{scope}[xshift=-4cm, xscale=2]

\foreach \b in {1,-1}
{
\begin{scope}[scale=\b]

\draw
	(-90:1) -- (-30:1) -- (30:1) -- (90:1)
	(0,-1) -- (0.62,-0.47) -- (0.62,0.47) -- (0,1)
	(0,-1) -- (0.37,-0.44) -- (0.37,0.44) -- (0,1)
	(0,-1) -- (0.12,-0.41) -- (0.12,0.41) -- (0,1);

\draw[line width=1.2]
	(0.62,0.47) -- (-30:1)
	(0.37,0.44) -- (0.62,-0.47)
	(0.12,0.41) -- (0.37,-0.44)
	(0.12,-0.41) -- (0,0);

\end{scope}
}

\node at (0,-1.4) {${\mc A}_7$};

\end{scope}


\draw[gray]
	(-60:1.2) -- (120:1.2)
	(60:1.2) -- (240:1.2);

\node at (60:1.4) {\small $1$};
\node at (120:1.4) {\small $2$};

\foreach \b in {1,-1}
{
\begin{scope}[scale=\b]

\foreach \a in {0,...,5}
\draw[rotate=60*\a]
	(30:1) -- (90:1);

\draw[line width=1.2]
	(150:1) -- (150:0.6);
	
\node at (92:0.7) {\small $\beta^s$};
\node at (30:0.8) {\small $\alpha$};
\node at (133:0.7) {\small $\gamma$};
\node at (167:0.7) {\small $\delta$};

\node at (0,0) {\small ${\mc A}_s$};		

\end{scope}
}

\node at (0,-1.4) {$FE_{\square}^A1$};

\foreach \b in {1,-1}
{
\begin{scope}[xshift=3cm, scale=\b]

\draw[gray]
	(60:1.2) -- (60:-1.2);
	
\foreach \a in {0,...,5}
\draw[rotate=60*\a]
	(30:1) -- (90:1);

\draw
	(-30:1) -- (-30:0.6);
	
\node at (90:0.7) {\small $\alpha^q$};
\node at (30:0.75) {\small $\beta$};
\node at (-17:0.7) {\small $\alpha$};
\node at (-47:0.7) {\small $\beta$};		

\end{scope}
}

\node at (3,-1.4) {$FE_{\square}^R1$};

\end{tikzpicture}
\caption{Flip modifications $FE_{\square}^A1$ and $FE_{\square}^R1$.}
\label{flip5}
\end{figure}
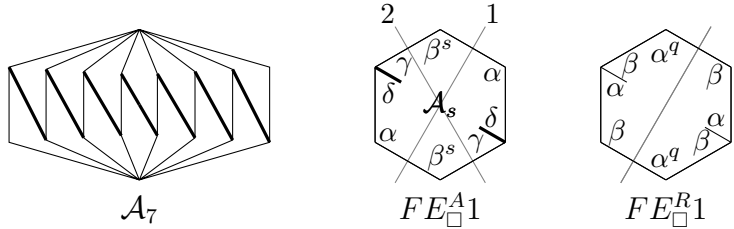

We remark that the necessary and sufficient condition for $s,t$ is given by \eqref{acdeqB}. The condition implies $t\le 3$. Moreover, by $\alpha+\gamma+\delta=2\pi$, we know $\alpha=s\beta$ if and only if $\gamma+\delta=s'\beta$, with $s+s'=p$. Then the earth map tiling $E_{\square}^A1$ is the disjoint union of ${\mc A}_s$ and ${\mc A}_{s'}$. The first flip of ${\mc A}_s$ and the second flip of ${\mc A}_{s'}$ actually give the same tiling. Therefore all we can do is the following:
\begin{enumerate}
\item If $\alpha=s\beta\le\pi$, then we may apply the first flip up to three times.
\item If $\gamma+\delta=s\beta<\pi$, which means $\alpha>\pi$, then we may apply the second flip up to three times.
\end{enumerate}
We may denote the flip modification by $FE_{\square}^A1$, without ambiguity on which flip is used.

For a very special combination of angle values
\[
\alpha=(\tfrac{4}{3}-\tfrac{4}{3f})\pi,\;
\beta=\tfrac{4}{f}\pi,\;
\gamma=(\tfrac{2}{3}-\tfrac{2}{3f})\pi,\;
\delta=\tfrac{2}{f}\pi,
\]
the tiling $E_{\square}^A1$ has yet another modification. Let $f=6q+4$, then the earth map tiling $E_{\square}^A1$ is three copies of ${\mc A}_q$ plus 4 tiles. Then we may glue three copies of ${\mc A}_q$ together as in the first of Figure \ref{flip7}. The complement of the union is two congruent hexagons, indicated by the grey shade. The hexagon is actually a union of two tiles. We may divide each of the two hexagons into two tiles by drawing the thick edge in any one of three ways. All the possible ways of dividing the two hexagons give three non-equivalent tilings that we call the {\em rearrangements} of $E_{\square}^A1$, and denote by $RE_{\square}^A1$. For example, the second of Figure \ref{flip7} gives the tiling with $q=4$, or $28$ tiles. For more details, see Figure \ref{acdD}.

\begin{figure}[htp]
\centering
\begin{tikzpicture}

\foreach \a in {0,1,2}
\fill[gray!50, rotate=120*\a] 
	(0,0) -- (60:0.8) -- (0:0.3) -- (-60:0.8)
	(60:1.4) -- (0:1.4) -- (-60:1.4) -- (-60:1.7) -- (0:1.7) -- (60:1.7);
	
\foreach \a in {0,1,2}
{
\begin{scope}[rotate=120*\a]

\draw
	(60:0.8) -- (0:0.3) -- (-60:0.8)
	(60:0.8) -- (60:1.4) -- (0:1.4) -- (-60:1.4);

\draw[line width=1.2]
	(0.3,0) -- (0.6,0)
	(1.4,0) -- (1.1,0);

\node at (0.85,0) {\small ${\mc A}_q$};	
		
\end{scope}
}

\draw[line width=1.2]
	(0.3,0) -- (-0.8,0)
	(1.4,0) -- (1.7,0)
	(-1.4,0) -- (-1.7,0);
	
\node at (0.1,-0.2) {$-$};
\node at (90:1.4) {$-$};

\node at (0,-1.7) {$RE_{\square}^A1$};


\begin{scope}[xshift=4.5cm]

\draw[line width=1.2, gray!50]
	(1,0) -- (1.3,0);
	
\foreach \a in {0,1,2}
{
\begin{scope}[rotate=120*\a]

\foreach \b in {0,...,3}
\draw[line width=1.2]
	(0.3+0.5*\b,0) -- ++(0.2,0);

\foreach \b in {0,1,2}
\draw
	(0.8+0.5*\b,0) -- (-60:0.8)
	(0.5+0.5*\b,0) -- (60:1.2);

\draw
	(60:0.8) -- (0:0.3) -- (-60:0.8) -- (-60:1.2) -- (0:2) -- (60:1.2)
	(0:0.3) -- (0:2);
	
\end{scope}
}

\draw[line width=1.2]
	(-0.8,0) -- (0.3,0)
	(2,0) -- (2.3,0)
	(-1.2,0) -- (-1.5,0);

\node at (0.1,-0.2) {$-$};
\node at (60:1.5) {$-$};

\node at (0.5,-1.7) {$RE_{28\square}1$};

\end{scope}

\end{tikzpicture}
\caption{Rearrangement $RE_{\square}^A1$.}
\label{flip7}
\end{figure}
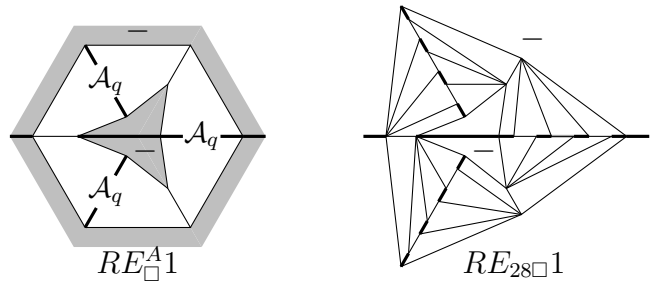 

The rhombus earth map tiling $E_{\square}^R1$ is the reduction $a=b$ of the almost equilateral earth map tiling $E_{\square}^R1$. The angles $\alpha=\frac{4}{f}\pi,\beta=(1-\frac{2}{f})\pi$ of the rhombus become angles $\alpha'=\beta,\beta'=\alpha,\gamma'=\beta,\delta'=\alpha$ of the almost equilateral quadrilateral. Then $\alpha'=q\beta'$ ($q=s$) for the flip modification of $E_{\square}^A1$ becomes $\beta=(1-\frac{2}{f})\pi=q\alpha=\frac{4q}{f}\pi$. This means $f=4q+2$. The flip $FE_{\square}^A1$ then becomes the flip $FE_{\square}^R1$ in the third of Figure \ref{flip5}. Since $\alpha'=\beta$ is close to $\pi$, we get $t=1$ and there is only one $FE_{\square}^R1$.

Figure \ref{flip3} gives another view of $FE_{\square}^R1$. The tiling $E_{\square}^R1$ is the union of two identical {\em half earth map tilings} in the first of Figure \ref{flip3}, each with $\alpha^q$ and $\alpha^{q+1}$ at the two ends. The second of Figure \ref{flip3} is the angles along the boundary of the (right) half earth map tiling. Then the flip with respect to the gray line also gives $FE_{\square}^R1$. The flip in Figure \ref{flip3} can also be obtained by a rotation, which is the viewpoint in \cite{ua}.

\begin{figure}[htp]
\centering
\begin{tikzpicture}[>=latex]

\foreach \b in {1,-1}
{
\begin{scope}[scale=\b]

\draw
	(0.2,0.5) -- (0.2,0) -- (-0.2,0) -- (-0.2,-0.5)-- (3.2,-0.5) -- (3.2,0) -- (2.8,0) -- (2.8,0.5) -- (0.2,0.5);
	
\node at (1.5,0.7) {\small $\alpha^q$};
\node at (1.5,-0.7) {\small $\alpha^{q+1}$};

\node at (1.5,0) {\small half EMT};

\end{scope}
}


\begin{scope}[xshift=5cm]

\draw[gray]
	(30:1.2) -- (30:-1.2);
	
\foreach \a in {0,...,5}
\draw[rotate=60*\a]
	(30:1) -- (90:1);

\foreach \b in {1,-1}
{
\begin{scope}[xscale=\b]

\draw
	(30:1) -- (30:0.6);

\node at (90:0.75) {\small $\alpha^q$};
\node at (-90:0.6) {\small $\alpha^{q+1}$};
\node at (-30:0.75) {\small $\beta$};
\node at (17:0.7) {\small $\alpha$};
\node at (45:0.7) {\small $\beta$};		

\end{scope}
}

\end{scope}

\end{tikzpicture}
\caption{Flip modification $FE_{\square}^R1$ of $E_{\square}^R1$, second viewpoint.}
\label{flip3}
\end{figure}
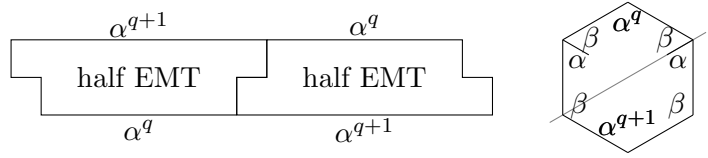

Since $E_{\triangle}^J1,E_{\triangle}2$ are the two simple triangular subdivisions of $E_{\square}^R1$, and $E_{\triangle}3$ is the triangular subdivision of $E_{\square}^R1$, the flip modification $FE_{\square}^R1$ induces flip modifications $FE_{\triangle}^J1,FE_{\triangle}2,FE_{\triangle}3$. We note that in this viewpoint, the timezone for $E_{\triangle}^J1$ follows from the timezone for $E_{\square}^R1$, and therefore consists of the shaded tiles in the first of Figure \ref{flip9}, which is different from the timezone indicated in Figure \ref{emt}.

\begin{figure}[htp]
\centering
\begin{tikzpicture}[>=latex]

\begin{scope}[xshift=-4cm]

\fill[gray!50]
	(1.2,0.6) -- (1.2,0) -- (0.6,0) -- (0.6,-0.6) -- (1.8,-0.6) -- (1.8,0) -- (2.4,0) -- (2.4,0.6);
	
\foreach \a in {0,2,4}
{
\draw[xshift=0.6*\a cm]
	(0,0) -- (0,0.6);

\draw[xshift=0.6*\a cm, line width=1.2]
	(0,0) -- (0,-0.6);
}

\foreach \a in {1,3}
{
\draw[xshift=0.6*\a cm]
	(0,0) -- (0,-0.6);

\draw[xshift=0.6*\a cm, line width=1.2]
	(0,0) -- (0,0.6);
}
	
\draw
	(0,0) -- (2.4,0);

\node at (1.4,-1) {\small timezone of $E_{\triangle}^J1$};
\node at (1.4,-1.5) {\small for flip modification};

\end{scope}


\fill[gray!50]
	(0,0.6) -- (0,-0.6) -- (0.8,-0.6) -- (0.8,0.6);

\draw[dash pattern=on 1pt off 1pt]
	(0,0.2) -- (2.8,0.2)
	(0,-0.2) -- (2.8,-0.2);

\draw
	(0.4,0.2) -- (0.8,-0.2);

\foreach \a in {0,1,3}
{

\draw[xshift=0.8*\a cm]
	(0,-0.6) -- (0,-0.2) -- (0.4,0.2) -- (0.4,0.6);

\draw[xshift=0.8*\a cm, line width=1.2]
	(0,0.6) -- (0,-0.2)
	(0.4,-0.6) -- (0.4,0.2);
}

\node at (1.5,0.85) {\small $\alpha^{2q+1}$};
\node at (1.5,-0.85) {\small $\alpha^{2q+1}$};

\node at (1.4,-1.5) {\small half EMT in $E_{\triangle}3$};


\begin{scope}[xshift=5cm]

\draw[gray]
	(60:1.4) -- (60:-1.4);

\foreach \b in {1,-1}
{
\begin{scope}[scale=\b]

\draw
	(30:1.2) -- (90:1.2)
	(30:1.2) -- (30:0.8);

\draw[line width=1.2]
	(30:1.2) -- (-30:1.2) -- (-90:1.2);
	
\node at (95:0.82) {\small $\alpha^{2q+1}$};

\node at (-30:0.85) {\small $\gamma^2$};

\node at (20:0.9) {\small $\alpha$};
\node at (45:0.8) {\small $\beta^2$};			

\end{scope}
}

\node at (0,-1.5) {$FE_{\triangle}3$};

\end{scope}

\end{tikzpicture}
\caption{Flip modifications $FE_{\triangle}^J1,FE_{\triangle}3$, with $\beta^2=\beta\dash\beta$ and $\gamma^2=\gamma\thin\gamma$.}
\label{flip9}
\end{figure}
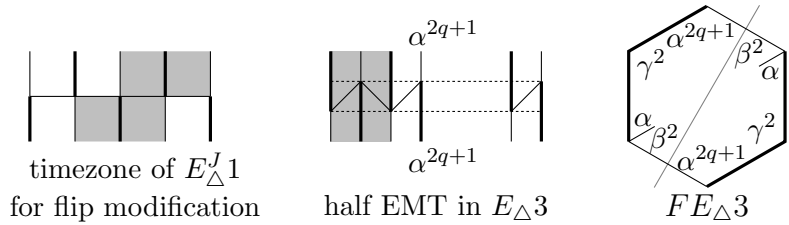

The second of Figure \ref{flip9} shows a choice of the timezone of $E_{\triangle}3$ different from the one in Figure \ref{emt}, and the corresponding half earth map tiling. Then the flip modification $FE_{\triangle}3$ is also obtained by the flip in the third of Figure \ref{flip9}.

\medskip

\noindent {\bf Flip of $E_{\square}2$}

\medskip


Figure \ref{pemt} shows $E_{\square}2$ is the simple pentagonal subdivision of the (symmetric) pentagonal earth map tiling $E_{\pentagon}1$. The pentagonal tiling has two flip modifications $F_1E_{\pentagon}1,F_2E_{\pentagon}1$ \cite[Figure 4]{cly}. Correspondingly, if $f=16q+8$, then $E_{\square}2$ is the union of two half earth map tilings, each with $\delta^{2q}=(\thin\delta\dash\delta\thin)^q$ and $\delta^{2q+2}=(\thin\delta\dash\delta\thin)^{q+1}$ at the two ends. The angles along the boundary of one half earth map tiling is given by Figure \ref{flip10}, and the flips with respect to the two gray lines give two flip modifications $F_1E_{\square}2,F_2E_{\square}2$. We remark that for $f=24$, $E_{24\square}2,F_1E_{24\square}2,F_2E_{24\square}2$ are the last three tilings in Figure \ref{simple_subdivision_dodecahedron}.

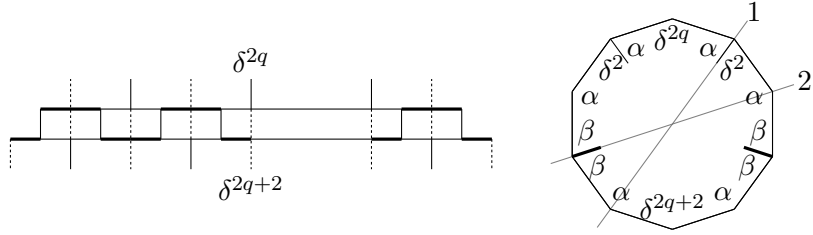
\begin{figure}[htp]
\centering
\begin{tikzpicture}[>=latex]

\foreach \a in {0,1,3}
{
\begin{scope}[xshift=1.6*\a cm]

\draw[dash pattern=on 1pt off 1pt]
	(0,-0.2) -- (0,0.6);

\draw
	(0,-0.2) -- (0,-0.6)
	(-0.4,0.2) -- (-0.4,-0.2)
	(0.4,0.2) -- (0.4,-0.2);

\draw[line width=1.2]
	(-0.4,0.2) -- (0.4,0.2)
	(-0.8,-0.2) -- (-0.4,-0.2)
	(0.8,-0.2) -- (0.4,-0.2);
	
\end{scope}
}

\foreach \a in {1,3,5}
{
\begin{scope}[xshift=0.8*\a cm]

\draw[dash pattern=on 1pt off 1pt]
	(0,0.2) -- (0,-0.6);

\draw
	(0,0.2) -- (0,0.6);
	
\end{scope}
}

\draw[dash pattern=on 1pt off 1pt]
	(-0.8,-0.2) -- (-0.8,-0.6)
	(5.6,-0.2) -- (5.6,-0.6);
	
\draw
	(-0.4,0.2) -- (5.2,0.2)
	(-0.8,-0.2) -- (5.6,-0.2);

\node at (2.4,0.85) {\small $\delta^{2q}$};
\node at (2.4,-0.85) {\small $\delta^{2q+2}$};


\begin{scope}[xshift=8cm]

\draw[gray]
	(54:1.7) -- (54:-1.7)
	(18:1.7) -- (18:-1.7);

\node at (54:1.85) {\small $1$};
\node at (18:1.85) {\small $2$};

\foreach \b in {1,-1}
{
\begin{scope}[xscale=\b]

\foreach \a in {0,...,9}
\draw[rotate=36*\a]
	(-18:1.4) -- (18:1.4);

\draw
	(54:1.4) -- (54:1);

\draw[line width=1.2]
	(-18:1.4) -- (-18:1);

\node at (18:1.15) {\small $\alpha$};
\node at (64:1.1) {\small $\alpha$};
\node at (43:1.1) {\small $\delta^2$};	
\node at (-6:1.15) {\small $\beta$};	
\node at (-30:1.15) {\small $\beta$};
\node at (-54:1.15) {\small $\alpha$};		

\end{scope}
}

\node at (90:1.15) {\small $\delta^{2q}$};
\node at (-90:1.1) {\small $\delta^{2q+2}$};

\end{scope}

\end{tikzpicture}
\caption{Flip modifications $F_1E_{\square}2,F_2E_{\square}2$, with $\delta^2=\delta\dash\delta$.}
\label{flip10}
\end{figure}

\medskip

\noindent {\bf Flip of $E_{\triangle}1$}

\medskip


If $f=4p=8q+4$, then $E_{\triangle}1$ is the union of two half earth map tilings in the first of Figure \ref{flip4}, each with $\alpha^p=\alpha^{2q+1}$ at the two ends. The second picture shows the angles along the boundary of the half earth map tiling. By $p\alpha=\beta+\gamma$, we may flip with respect to the gray line and get the flip modification $FE_{\triangle}1$. 

The reduction (by $a=b$) $E_{\triangle}^I1$ of $E_{\triangle}1$ only requires $f=2p$, and the induced flip modification $FE_{\triangle}^I1$ only requires $f=2p=4q$ to be a multiple of $4$. 

The reduction (by $a=c$) $E_{\triangle}^J1$ of $E_{\triangle}1$ is also the simple triangle subdivision of $E_{\square}^R1$. The flip modifications of $E_{\triangle}1$ and  $E_{\square}^R1$ induce the same flip modification $FE_{\triangle}^J1$. Therefore there is no ambiguity (especially in the main theorem in the introduction) to use $FE_{\triangle}1$ to denote the flip modification of $E_{\triangle}1$, including the flip modifications of the reductions $E_{\triangle}^I1,E_{\triangle}^J1$.

\begin{figure}[htp]
\centering
\begin{tikzpicture}[>=latex]

\begin{scope}[xshift=-5cm]

\draw[dash pattern=on 1pt off 1pt]
	(0,0) -- (3.2,0);

\foreach \a in {0,1,2,3.5}
{
\draw[xshift=0.8*\a cm]
	(0,0.4) -- (0,0);

\draw[xshift=0.8*\a cm, line width=1.2]
	(0,-0.4) -- (0,0);
}

\foreach \a in {0,1,3.5}
{
\draw[xshift=0.8*\a cm]
	(0.4,-0.4) -- (0.4,0);

\draw[xshift=0.8*\a cm, line width=1.2]
	(0.4,0.4) -- (0.4,0);
}

\node at (1.75,0.65) {\small $\alpha^{2q+1}$};
\node at (1.75,-0.65) {\small $\alpha^{2q+1}$};

\node at (1.6,-1.4) {\small half EMT in $E_{\triangle}1$};

\end{scope}
	

\draw[gray]
	(0.7,0.7) -- (-0.7,-0.7);

\foreach \b in {0,1}
{
\begin{scope}[xshift=3.5*\b cm]

\draw
	(0,1) -- (-1,0)
	(0,-1) -- (1,0);

\draw[line width=1.2]
	(1,0) -- (0,1)
	(-1,0) -- (0,-1);

\end{scope}
}

\draw[dash pattern=on 1pt off 1pt]
	(1,0) -- (0.6,0)
	(-1,0) -- (-0.6,0);

\foreach \b in {1,-1}
{
\begin{scope}[scale=\b]

\node at (-0.04,0.65) {\small $\alpha^p$};

\node at (0.6,0.15) {\small $\gamma$};
\node at (0.6,-0.2) {\small $\beta$};

\end{scope}
}

\node at (0,-1.4) {$FE_{\triangle}1$};


\begin{scope}[xshift=3.5cm]

\draw
	(-1.6,0) to[out=-90, in=180] (0,-1);

\draw[dash pattern=on 1pt off 1pt]
	(0,1) -- (0,0)
	(0,-1) -- (0,-0.6)
	(-1.6,0) -- (-1,0);

\draw[line width=1.2]
	(-1,0) -- (0,0)
	(-1.6,0) to[out=90, in=180] (0,1);

\node at (0.65,0) {\small $\alpha^p$};
\node at (-0.15,-0.6) {\small $\gamma$};
\node at (0.15,-0.6) {\small $\beta$};
\node at (0.15,0.6) {\small $\gamma$};
\node at (-0.35,-0.15) {\scriptsize $\alpha^{p-1}$};

\node at (-0.15,0.6) {\small $\beta$};
\node at (-0.15,0.15) {\small $\gamma$};
\node at (-0.65,0.15) {\small $\alpha$};

\node at (-0.4,0.8) {\small $\alpha$};
\node at (-1.4,0.2) {\small $\gamma$};
\node at (-1.05,0.2) {\small $\beta$};

\node at (-0.4,-0.8) {\small $\alpha$};
\node at (-1.4,-0.2) {\small $\beta$};
\node at (-1.05,-0.2) {\small $\gamma$};

\node[draw,shape=circle, inner sep=0.5] at (-0.4,0.3) {\footnotesize $1$};
\node[draw,shape=circle, inner sep=0.5] at (-0.8,0.6) {\footnotesize $2$};

\end{scope}


\begin{scope}[xshift=6cm]

\draw[gray]
	(0,0.9) -- (0,-0.9);

\foreach \b in {1,-1}
{
\begin{scope}[scale=\b]

\draw
	(0,0) -- (0.6,-0.6);

\draw[line width=1.2]
	(0.6,0.6) -- (0.6,-0.6);

\draw[dash pattern=on 1pt off 1pt]
	(0.6,0.6) -- (-0.6,0.6);

\node at (0.43,-0.2) {\small $\alpha$};
\node at (0.15,-0.42) {\small $\beta$};
\node at (0.43,0.4) {\small $\gamma$};

\end{scope}
}

\node[draw,shape=circle, inner sep=0.5] at (0.2,0.2) {\footnotesize $1$};
\node[draw,shape=circle, inner sep=0.5] at (-0.2,-0.2) {\footnotesize $2$};

\node at (0,-1.4) {$F'E_{\triangle}1$};

\end{scope}

\end{tikzpicture}
\caption{Flip modifications $FE_{\triangle}1$ and $F'E_{\triangle}1$.}
\label{flip4}
\end{figure}
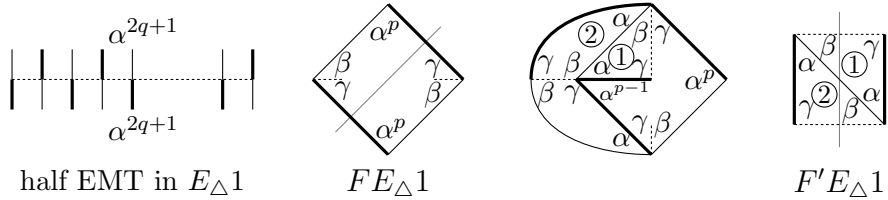

For the special case $\alpha+\beta=\gamma$, the flip modification $FE_{\triangle}1$ has further modification. The third of Figure \ref{flip4} shows the flipped inner half earth map tiling of $FE_{\triangle}1$, and two tiles in the unflipped outer half earth map tiling. By $\alpha+\beta=\gamma$, the two tiles $T_1,T_2$ form the rectangle in the fourth of Figure \ref{flip4}. Then we may flip the rectangle with respect to the gray line, and still get a tiling. We note that there is another rectangle opposite to $T_1,T_2$. We may flip one or two rectangles, and denote both further flip modifications by $F'E_{\triangle}1$.

For the special case $f=12$, the condition $\alpha+\beta=\gamma$ becomes $\alpha=\beta=\frac{1}{3}\pi$ and $\gamma=\frac{2}{3}\pi$. Therefore the triangle is isosceles, and removing the normal edges gives the regular cube. Then adding the normal edges back means simple triangular subdivisions of the regular cube in Figure \ref{simple_subdivision_cube}. If we flip one rectangle (actually square), then $F'E_{\triangle}1$ is the fourth of Figure \ref{simple_subdivision_cube}. If we flip two rectangles, then $F'E_{\triangle}1$ is the fifth of Figure \ref{simple_subdivision_cube}. 

Table \ref{flip_data} gives the data for all the modified tilings. 

\renewcommand{\arraystretch}{1.2}

\begin{table}[htp]
\centering
\begin{tabular}{|c|c|c|c|c|l|c|} 
\hline
Tiling
& $\alpha$ & $\beta$ & $ \gamma$ & $\delta$
& \multicolumn{1}{|c|}{Vertex}
& $f$ \\
\hline  \hline 
$FB_{\triangle}P_8$
& $\frac{1}{4}$ & $\frac{1}{3}$ & $\frac{1}{2}$ & 
& $\alpha^8,\beta^6,\alpha^4\gamma^2,\gamma^4$ 
& $48$ \\
\hline
$FQ_{\square}P_6$
& $\frac{2}{3}$ & $\frac{2}{3}$ & $\frac{1}{2}$ & $\frac{1}{3}$ 
& $\alpha^3,\alpha\beta^2,\alpha^2\delta^2,\beta^2\delta^2,\gamma^4$
& $24$  \\
\hline
$FQ_{\square}P_8$
& $\frac{2}{3}$ & $\frac{1}{2}$ & $\frac{1}{2}$ &
& $\alpha^3,\beta^4,\beta^2\gamma^2,\gamma^4$
& $24$ \\
\hline
$FE_{\triangle}1$
& $\frac{4}{f}$ & $1-\gamma$ & &
& $\beta^2\gamma^2,\alpha^{2q+1}\beta\gamma$
& \multirow{2}{*}{$8q+4$} \\
\cline{1-6} 
$F'E_{\triangle}1$
& $\frac{4}{f}$ & $\frac{1}{2}-\frac{2}{f}$ & $\frac{1}{2}+\frac{2}{f}$ & 
& $\beta^2\gamma^2,\alpha\beta^3\gamma, \alpha^{2q}\gamma^2,\alpha^{2q+1}\beta\gamma$
& \\
\hline
$FE_{\triangle}^I1$
& $\frac{4}{f}$ & $\frac{1}{2}$ & &
& $\beta^4,\alpha^q\beta^2$ 
& $4q$ \\
\hline
$FE_{\triangle}2$
& $\frac{8}{f}$ & $\tfrac{1}{2}-\tfrac{2}{f}$ & &
& $\alpha\beta^4,\alpha^{q+1}\beta^2$ 
& $8q+4$ \\
\hline
$FE_{\triangle}3$
& $\frac{8}{f}$ & $\frac{1}{2}-\tfrac{4}{f}$ & $\frac{1}{2}$ &
& $\alpha^2\beta^4,\gamma^4,\alpha^{2q+2}\beta^2$
& $16q+8$ \\
\hline
\multirow{2}{*}{$FE_{\square}1$}
& $\tfrac{4s}{f}$ & $\tfrac{4}{f}$ & $\tfrac{4s'}{f}-\delta$ &
& $\alpha\gamma\delta,\alpha^t\beta^{p-st},\beta^s\gamma\delta$
& \multirow{2}{*}{$2p$} \\
\cline{2-6}
& \multicolumn{4}{|c|}{$s+s'=p$}
& $\alpha\gamma\delta,\alpha\beta^{s'},\beta^{p-s't}\gamma^t\delta^t$
& \\
\hline
$RE_{\square}1$
& $\tfrac{4}{3}-\tfrac{4}{3f}$ & $\tfrac{4}{f}\pi$ & $\tfrac{2}{3}-\tfrac{2}{3f}$ & $\tfrac{2}{f}$ 
& $\alpha\gamma\delta,\gamma^3\delta,\alpha\beta^{q+1},\alpha\beta^q\delta^2$
& $6q+4$  \\
\hline
$F_1E_{\square}2$
& \multirow{2}{*}{$1-\tfrac{8}{f}$} 
& \multirow{2}{*}{$\tfrac{1}{2}+\tfrac{4}{f}$} 
& \multirow{2}{*}{$\tfrac{1}{2}$}
& \multirow{2}{*}{$\tfrac{8}{f}$}
& $\alpha\beta^2,\alpha^2\delta^2,\gamma^4,\alpha\delta^{2q+2},\beta^2\delta^{2q}$
& \multirow{2}{*}{$16q+8$} \\
\cline{1-1} \cline{6-6}
$F_2E_{\square}2$
& & & & 
& $\alpha\beta^2,\alpha^2\delta^2,\gamma^4,\alpha\delta^{2q+2}$
& \\
\hline
\end{tabular}
\caption{Angles ($\pi$ omitted) and vertices for modified tilings. }
\label{flip_data}
\end{table}

\renewcommand{\arraystretch}{1}

\subsection{Sporadic Tilings}
\label{sporadic}

Besides tilings of Platonic and earth map types, there are four triangular sporadic tilings, and eight quadrilateral sporadic tilings. 

The sporadic triangular tilings are actually the simple subdivisions of the regular cube. There are seven such tilings, given by Figure \ref{simple_subdivision_cube}, in which four are not assigned notations. We call these four {\em triangular sporadic tilings}. Figure \ref{sporadic_triangular} gives the 3D pictures of these tilings. We note that normal and thick edges are exchanged in order to get the isosceles triangle in Figure \ref{triangle}. We also note that the first and second tilings (fourth and fifth of Figure \ref{simple_subdivision_cube}) can also be interpreted as $F'E_{12\triangle}1$. 

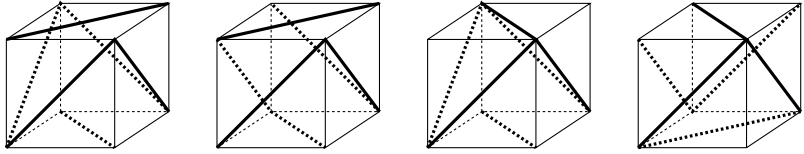
\begin{figure}[htp]
\centering
\begin{tikzpicture}[>=latex,scale=0.8]

\foreach \x in {0,...,3}
{
\begin{scope}[xshift=3.5*\x cm]

\coordinate (A0X\x) at (0.9,0.6);
\coordinate (A1X\x) at (2.7,0.6);
\coordinate (B0X\x) at (0,0);
\coordinate (B1X\x) at (1.8,0);

\coordinate (C0X\x) at (0.9,-1.2);
\coordinate (C1X\x) at (2.7,-1.2);
\coordinate (D0X\x) at (0,-1.8);
\coordinate (D1X\x) at (1.8,-1.8);

\draw
	(B0X\x) -- (B1X\x) -- (A1X\x) -- (A0X\x) -- (B0X\x) -- (D0X\x) -- (D1X\x) -- (C1X\x) -- (A1X\x)
	(B1X\x) -- (D1X\x);

\draw[dash pattern=on 1pt off 1pt]
	(D0X\x) -- (C0X\x) 
	(A0X\x) -- (C0X\x) -- (C1X\x);
	
\end{scope}
}

\begin{scope}[line width=1.2]

\draw
	(C1X0) -- (B1X0) -- (D0X0) 
	(B0X0) -- (A1X0)
	(C1X1) -- (B1X1) -- (D0X1) 
	(B0X1) -- (A1X1)
	(A0X2) -- (B1X2) -- (D0X2) 
	(C1X2) -- (B1X2)
	(A0X3) -- (B1X3) -- (D0X3) 
	(C1X3) -- (B1X3);

\draw[dash pattern=on 1pt off 1pt]
	(D0X0) -- (A0X0) -- (C1X0)
	(C0X0) -- (D1X0)
	(B0X1) -- (C0X1) -- (D1X1)
	(C1X1) -- (A0X1)
	(D0X2) -- (A0X2) -- (C1X2)
	(C0X2) -- (D1X2)
	(B0X3) -- (C0X3) -- (A1X3)
	(D0X3) -- (C1X3);
	
\end{scope}

\end{tikzpicture}
\caption{Sporadic triangular tilings $S_{\triangle}P_6$.}
\label{sporadic_triangular}
\end{figure}

Figure \ref{sporatic_tiling} gives the eight {\em quadralateral sporadic tilings}. The tilings $S_{12\square}1$, $S_{16\square}1$, $S_{16\square}2$, $S_{16\square}3$, $FS_{16\square}3$, $S_{36\square}5$ are actually earth map tilings with three or four timezones. Presumably, these belong to several infinite families of earth map tilings. Although such earth map tilings are combinatorially possible, the quadrilaterals exist only for the specific values of $f$, due to geometrical reason.

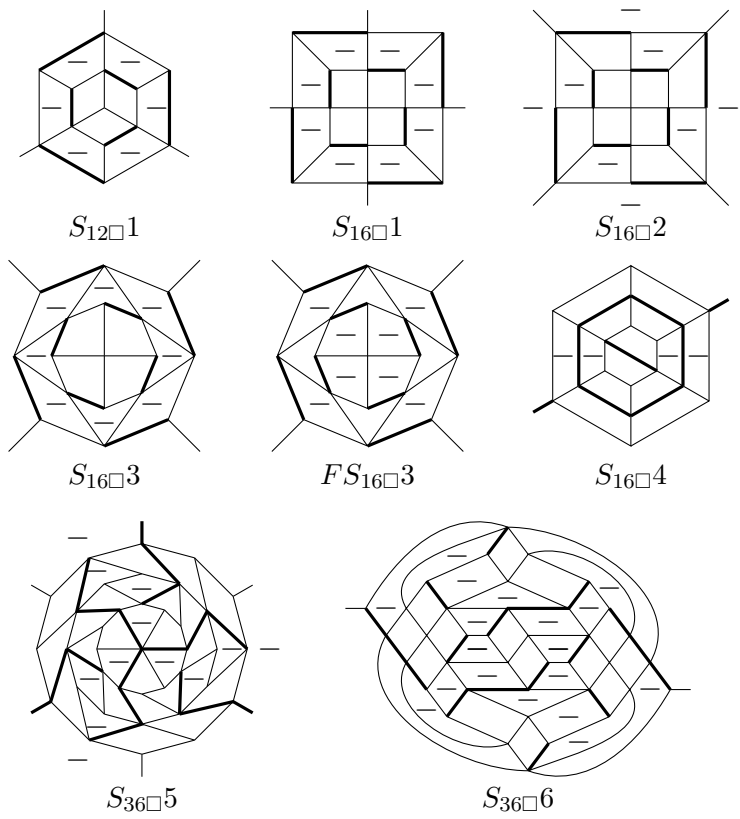
\begin{figure}[htp]
\centering
\begin{tikzpicture}[>=latex]


\foreach \a in {0,1,2}
{
\begin{scope}[rotate=120*\a]

\draw
	(0,0) -- (90:1.3)
	(-30:0.5) -- (30:0.5) -- (30:1) -- (90:1);

\draw[line width=1.2]
	(-30:1) -- (30:1)
	(30:0.5) -- (90:0.5);

\node at (60:0.7) {$-$};
\node at (120:0.7) {$-$};

\end{scope}
}

\node at (0,-1.6) {\small $S_{12\square}1$};


\foreach \a in {0,1,2,3}
{
\begin{scope}[xshift=3.5cm, rotate=90*\a]

\draw
	(0,0) -- (1.3,0)
	(0.5,0) -- (0.5,0.5) -- (1,1) -- (0,1);

\draw[line width=1.2]
	(0.5,0.5) -- (0,0.5)
	(1,0) -- (1,1);

\node at (0.75,0.3) {$-$};
\node at (0.75,-0.3) {$-$};

\end{scope}
}

\node at (3.5,-1.6) {\small $S_{16\square}1$};


\foreach \a in {0,1,2,3}
{
\begin{scope}[xshift=7cm,rotate=90*\a]

\draw
	(0,0) -- (1,0)
	(0.5,0) -- (0.5,0.5) -- (1.3,1.3)
	(1,1) -- (0,1);

\draw[line width=1.2]
	(0.5,0.5) -- (0,0.5)
	(1,0) -- (1,1);

\node at (0.3,0.75) {$-$};
\node at (1.3,0) {$-$};

\end{scope}
}

\node at (7,-1.6) {\small $S_{16\square}2$};


\begin{scope}[yshift=-3.3cm]

\foreach \a in {0,...,3}
{

\foreach \b in {0,1}
{
\begin{scope}[xshift=3.5*\b cm, rotate=90*\a]

\draw
	(1.2,0) -- (45:1.2) -- (0,1.2)
	(1.2,0) -- (45:0.7) -- (0,1.2)
	(90:0.7) -- (45:0.7) -- (0:0.7) -- (0,0)
	(45:1.2) -- (45:1.8);

\draw[line width=1.2]
	(45:1.2) -- (0:1.2);
	
\node at (45:0.9) {$-$};

\end{scope}
}

\begin{scope}[line width=1.2, rotate=90*\a]

\draw
	(45:0.7) -- (90:0.7);
	
\node at (0:0.9) {$-$};	

\end{scope}

\begin{scope}[xshift=3.5cm, line width=1.2, rotate=90*\a]

\draw[line width=1.2]
	(45:0.7) -- (0:0.7);

\node at (45:0.4) {$-$};

\end{scope}

}

\node at (0,-1.6) {\small $S_{16\square}3$};
\node at (3.5,-1.6) {\small $FS_{16\square}3$};


\begin{scope}[xshift=7cm]

\foreach \a in {0,...,5}
{
\begin{scope}[rotate=60*\a]

\draw
	(-30:0.4) -- (30:0.4) -- (30:1.2) -- (-30:1.2) ;
	
\draw[line width=1.2]
	(-30:0.8) -- (30:0.8);

\end{scope}
}

\draw[line width=1.2]
	(150:0.4) -- (-30:0.4)
	(30:1.5) -- (30:1.2)
	(210:1.5) -- (210:1.2);
	
\node at (0.5,0) {$-$};
\node at (0.9,0) {$-$};	
\node at (-0.5,0) {$-$};
\node at (-0.9,0) {$-$};

\node at (0,-1.6) {\small $S_{16\square}4$};

\end{scope}

\end{scope}


\begin{scope}[yshift=-7.2cm]

\begin{scope}[xshift=0.5 cm]

\foreach \a in {0,...,11}
\draw[rotate=30*\a]
	(0:0.6) -- (30:0.6)
	(0:1) -- (30:1)
	(0:1.4) -- (30:1.4);

\foreach \a in {0,1,2}
{
\begin{scope}[rotate=120*\a]

\draw
	(0,0) -- (60:0.6)
	(30:0.6) -- (60:1)
	(-30:0.6) -- (0:1)
	(-30:1) -- (0:1.4)
	(30:1) -- (60:1.4)
	(30:1.4) -- (30:1.7);
	
\draw[line width=1.2]
	(0,0) -- (0:0.6) -- (30:1) -- (0:1.4)
	(90:0.6) -- (60:1)
	(60:1) -- (90:1.4) -- (90:1.7);

\node at (90:0.35) {$-$};
\node at (90:0.8) {$-$};
\node at (0:1.2) {$-$};
\node at (0:1.7) {$-$};

\end{scope}
}

\node at (0,-2) {\small $S_{36\square}5$};

\end{scope}


\begin{scope}[xshift=5.5cm]

\foreach \a in {1,-1}
{
\begin{scope}[scale=0.45*\a]

\draw
	(0,0) -- (0.3,0.4) -- (2.1,0.4) 
	(-0.3,1.2) -- (0.9,-0.4) -- (1.5,-0.4) -- (2.1,-1.2) -- (3.3,0.4)
	(2.1,-1.2) -- (0.3,-1.2)
	(2.7,-2) -- (3.3,-1.2) -- (1.5,1.2)
	(2.1,0.4) -- (2.7,1.2) -- (2.1,2)
	(-1.5,-1.2) -- (0.3,-2) -- (2.1,-1.2)
	(2.7,-2) -- (0.9,-2.8) -- (0.3,-2) -- (-0.3,-2.8) -- (-2.1,-2)
	(0.3,-3.6) -- (-0.3,-2.8)
	(3.3,-1.2) -- (3.9,-0.4)
	(4.5,-1.2) -- (5.1,-1.2)
	(3.3,0.4) to[out=60,in=20] (0.3,2.8)
	(3.9,-0.4) to[out=60,in=0] (-0.3,3.6)
	(3.3,-1.2) to[out=-60,in=20] (0.3,-3.6)
	(4.5,-1.2) to[out=-120,in=-20] (0.3,-3.6)
	;

\draw[line width=1.2]
	(2.1,2) -- (1.5,1.2) -- (-0.3,1.2) -- (-0.9,0.4)
	(1.5,-0.4) -- (2.1,0.4) 
	(2.7,-2) -- (2.1,-1.2)
	(0.3,-3.6) -- (0.9,-2.8)
	(2.7,1.2) -- (4.5,-1.2);

\node at (-1.2,0) {$-$};
\node at (-1.2,0.8) {$-$};
\node at (-0.3,1.5) {$-$};
\node at (-1.5,2) {$-$};
\node at (0.9,0.8) {$-$};
\node at (1.2,0) {$-$};
\node at (2.1,1.2) {$-$};
\node at (2.7,1.8) {$-$};
\node at (-1.8,2.75) {$-$};
\node at (-3.9,1.2) {$-$};

\end{scope}
}

\node at (0,-2) {\small $S_{36\square}6$};

\end{scope}

\end{scope}

\end{tikzpicture}
\caption{Sporadic quadrilateral tilings.}
\label{sporatic_tiling}
\end{figure}

Both $S_{12\square}1,S_{16\square}1$ are labelled as the first sporadic tiling because they belong to the same combinatorial family of earth map tilings. Moreover, $FS_{16\square}3$ can be obtained by a flip of the interior half of $S_{16\square}3$. 

The tilings $S_{16\square}4,S_{36\square}6$ have two fold symmetry, and are not earth map tilings. Moreover, some pictures in Figure \ref{sporatic_tiling} look different from the ones appearing in the later proof (especially $S_{36\square}6$ and Figure \ref{add-abbbB}). They are actually the same tilings, and we draw different pictures just to give different views.

Table \ref{special_data} gives the data for all the sporadic tilings. Note that some angle values are not rational multiples of $\pi$, and we can only specify them by the cosine or tangent values:
\begin{align}
S_{12\square}1 &\colon \tan\gamma=-\tfrac{\sqrt{3}}{\sqrt{5}}. \label{seq1} \\
S_{16\square}1 &\colon \tan\gamma=2-\sqrt{5}-{\textstyle \sqrt{7-3\sqrt{5}}}. \label{seq2} \\
S_{16\square}2 &\colon \cos\delta=\tfrac{1-\sqrt{2}}{2}. \label{seq3} \\
S_{16\square}4 &\colon \tan\gamma=2+\sqrt{2}. \label{seq4}
\end{align}

\renewcommand{\arraystretch}{1.2}

\begin{table}[htp]
\centering
\scalebox{1}{
\begin{tabular}{|c|c|c|c|c|l|c|} 
\hline
Tiling
& $\alpha$ & $\beta$ & $ \gamma$ & $\delta$
& \multicolumn{1}{|c|}{Vertex}
& $f$ \\
\hline \hline 
$S_{\triangle}P_6$ 
& $\frac{2}{3}$ & $\frac{1}{3}$ & &
& $\alpha^3,\alpha^2\beta^2,\alpha\beta^4,\beta^6$
& $12$ \\
\hline
$S_{12\square}1$
& $\tfrac{2}{3}$ & $2-2\gamma$ & \eqref{seq1} & $\gamma-\tfrac{1}{3}$
& $\alpha^3,\beta\gamma^2,\alpha\beta\delta^2$
& $12$ \\
\hline
$S_{16\square}1$
& $\tfrac{1}{2}$ & $2-2\gamma$ & \eqref{seq2} & $\gamma-\tfrac{1}{4}$
& $\alpha^4,\beta\gamma^2,\alpha\beta\delta^2$
& $16$  \\
\hline
$S_{16\square}2$
& $\tfrac{1}{2}$ & $1-\delta$ & $\tfrac{3}{4}$ & \eqref{seq3}
& $\alpha^4,\alpha\gamma^2,\beta^2\delta^2$
& $16$  \\
\hline
$S_{16\square}3$
& \multirow{2}{*}{$\tfrac{1}{2}$} 
& \multirow{2}{*}{$1$} 
& \multirow{2}{*}{$\tfrac{1}{4}$}
& \multirow{2}{*}{$\tfrac{1}{2}$}
& \multirow{2}{*}{$\alpha^4,\beta\delta^2,\alpha\beta\gamma^2$}  
& \multirow{2}{*}{$16$} \\
\cline{1-1} 
$FS_{16\square}3$
& & & & & & \\
\hline
$S_{16\square}4$
& $\tfrac{1}{2}$ & $\tfrac{3}{4}$ & \eqref{seq4} & $1-\gamma$ 
& $\alpha\beta^2,\alpha^2\gamma\delta,\gamma^2\delta^2$
& $16$ \\
\hline
$S_{36\square}5$
& $\tfrac{4}{9}$ & $\tfrac{7}{9}$ & $\tfrac{1}{3}$ & $\tfrac{5}{9}$
& $\alpha\beta^2,\alpha^2\delta^2,\alpha\gamma^3\delta,\gamma\delta^3,\gamma^6$
& $36$ \\
\hline
$S_{36\square}6$
& $\tfrac{1}{3}$ & $\tfrac{5}{9}$ & $\tfrac{7}{18}$ & $\tfrac{5}{6}$
& $\alpha\delta^2,\alpha\beta^3,\alpha^2\beta\gamma^2,\gamma^3\delta$
& $36$ \\
\hline
\end{tabular}
}
\caption{Angles ($\pi$ omitted) and vertices for sporadic tilings.}
\label{special_data}
\end{table}

\renewcommand{\arraystretch}{1}

\section{Technique}
\label{technical}

By a {\em tiling}, we mean an edge-to-edge tiling of the sphere by congruent polygons. In particular, all vertices have degree $\ge 3$. Unless otherwise stated, the assumptions will always be implicit in the discussions, lemmas, and propositions. 
 
All edge lengths and angle values are strictly between $0$ and $2\pi$, with Lemma \ref{geometry6} as the only exception. 
By \cite[Lemma 1]{gsy}, the polygon is simple, in the sense the boundary is a simple closed curve. This implies $a<\pi$ in the isosceles and equilateral triangles in Figure \ref{triangle}, and in all the quadrilaterals in Figure \ref{quad}. This also implies $b<\pi$ in the kite.

The tiling problem has the exchange symmetry given in the introduction of Section 
\ref{construction}. The technical results in this section are sometimes stated in one version, and the version obtained by applying the exchange symmetry is also valid, and may not be mentioned.

\subsection{Notations and Conventions}

Let $v,e,f$ be the numbers of vertices, edges, and tiles in a tiling by $n$-gons, where $n$ is fixed. Let $v_k$ be the number of vertices of degree $k$. Then we have
\begin{align*}
2
&=v-e+f, \\
nf=2e
&=\ssum_{k=3}^{\infty}kv_k=3v_3+4v_4+5v_5+\cdots,  \\
v
&=\ssum_{k=3}^{\infty}v_k=v_3+v_4+v_5+\cdots. 
\end{align*}
Canceling $e$ and $v_3$ from the three equations, we get
\begin{equation}\label{generalf}
(6-n)f=12+\ssum_{k\ge 4}2(k-3)v_k\ge 12.
\end{equation}
This implies $n=3,4,5$. In other words, the polygon that tiles the sphere must be triangle, quadrilateral, or pentagon.

For a {\em triangular} tiling, the equality \eqref{generalf} becomes the first of the following
\begin{align} 
f
&=4+\ssum_{k\ge 4}\frac{2}{3}(k-3)v_k
=4+\tfrac{2}{3}v_4+\tfrac{4}{3}v_5+v_6+\cdots \label{trivcountf4} \\
&=8+\ssum_{k\ge 3}(k-4)v_k
=8-v_3+v_5+2v_6+3v_7+\cdots \label{trivcountf3}.
\end{align}
The second equality is obtained by canceling $e$ and $v_4$ from the three equations. The first equality implies $f\ge 4$.

For a {\em quadrilateral} tiling, the equality \eqref{generalf} becomes the first of the following   
\begin{align} 
f
&=6+\ssum_{k\ge 4}(k-3)v_k
=6+v_4+2v_5+3v_6+\cdots, \label{quadvcountf} \\
v_3
&=8+\ssum_{k\ge 4}(k-4)v_k
=8+v_5+2v_6+3v_7+\cdots. \label{quadvcount3}
\end{align}
The second equality is obtained by canceling $e$ and $f$. The first equality implies $f\ge 6$, and the second equality implies $v_3\ge 8$. Therefore a quadrilateral tiling has at least $6$ tiles, and at least $8$ degree $3$ vertices.

The similar formulae are used in the classification of tilings of the sphere by congruent pentagons \cite{awy,cly,wy1,wy2}. In the subsequent discussion, we will omit mentioning similar results for pentagonal tilings.

Figure \ref{triangle} gives all the possible edge combinations for triangles. All are suitable for tiling the sphere. Unlike triangles, and similar to pentagons \cite[Lemma 9]{wy1}, not all quadrilaterals are suitable for tiling the sphere.

\begin{lemma}\label{edge_combo}
In an (edge-to-edge) tiling of the sphere by congruent quadrilaterals, the edge lengths of the quadrilateral are arranged in one of the four ways in Figure \ref{quad}, with distinct edge lengths $a,b,c$. 
\end{lemma}

\begin{proof}
Combinatorially, the possible edge combinations are $abcd,a^2bc,a^2b^2,a^3b,a^4$. For $abcd$, we may assume the edges are arranged as in the first of Figure \ref{edges2}. For $a^2bc$ and $a^2b^2$, depending on whether the two $a$-edges are adjacent or separated, we have two possibilities in the first and second of Figure \ref{quad} and the second and third of Figure \ref{edges2}. For $a^3b$ and $a^4$, we have the unique arrangement of edges in Figure \ref{quad}.

\begin{figure}[htp]
\centering
\begin{tikzpicture}[>=latex,scale=1]

\foreach \a in {0,1,2,4}
\draw[xshift=2*\a cm]
	(-0.6,-0.6) rectangle (0.6,0.6);

\node[fill=white, inner sep=0.5] at (-0.6,0) {\small $c$};
\node[fill=white, inner sep=0.5] at (0,0.6) {\small $b$};
\node[fill=white, inner sep=0.5] at (0.6,0) {\small $a$};
\node[fill=white, inner sep=0.5] at (0,-0.6) {\small $d$};

\begin{scope}[xshift=2cm]

\node[fill=white, inner sep=0.5] at (-0.6,0) {\small $a$};
\node[fill=white, inner sep=0.5] at (0,0.6) {\small $b$};
\node[fill=white, inner sep=0.5] at (0.6,0) {\small $a$};
\node[fill=white, inner sep=0.5] at (0,-0.6) {\small $c$};	

\end{scope}

\begin{scope}[xshift=4cm]

\node[fill=white, inner sep=0.5] at (-0.6,0) {\small $a$};
\node[fill=white, inner sep=0.5] at (0,0.6) {\small $b$};
\node[fill=white, inner sep=0.5] at (0.6,0) {\small $a$};
\node[fill=white, inner sep=0.5] at (0,-0.6) {\small $b$};	

\end{scope}

\begin{scope}[xshift=8cm]

\draw[dotted]
	(-0.6,-0.6) -- (0.6,-0.6)
	(-0.6,0.6) -- (0.6,0.6);

\node[fill=white, inner sep=0.5] at (-0.6,0) {\small $x$};
\node[fill=white, inner sep=0.5] at (0,0.6) {\small $y$};
\node[fill=white, inner sep=0.5] at (0.6,0) {\small $x$};
\node[fill=white, inner sep=0.5] at (0,-0.6) {\small $y$};	

\end{scope}

\end{tikzpicture}
\caption{Not suitable for tiling.}
\label{edges2}
\end{figure}
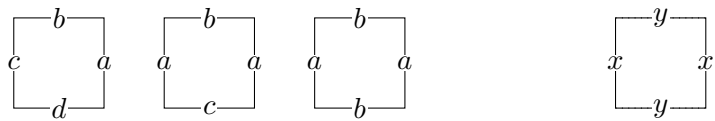

All three edge arrangements in Figure \ref{edges2} are of the form in the fourth of Figure \ref{edges2}, where each of $x,y$ may represent one or two edge lengths. Since two $x$ are separated, and two $y$ are separated, we know the edges at any vertex are alternatively arranged as $x,y,x,y,\dots$. In particular, all vertices have even degree, contradicting $v_3\ge 8$ by \eqref{quadvcount3}.
\end{proof}

Let $\alpha,\beta,\gamma,\delta$ be the four angles in a quadrilateral tiling. Then the area of the quadrilateral is $\alpha+\beta+\gamma+\delta-2\pi$, and is also the area $4\pi$ of the sphere divided by the number $f$ of tiles. This gives the {\em angle sum for quadrilateral} 
\begin{equation}\label{anglesum}
\alpha+\beta+\gamma+\delta
=(2 + \tfrac{4}{f})\pi.
\end{equation}
Similarly, we have the {\em angle sum for triangle}
\begin{equation}\label{3anglesum}
\alpha+\beta+\gamma=(1+\tfrac{4}{f})\pi.
\end{equation}
We remark that, by the proof of \cite[Lemma 4]{wy1}, the angle sum for polyhedron can be proved without using the area. Therefore the equalities \eqref{anglesum} and \eqref{3anglesum} are combinatorial facts.

In a quadrilateral tiling, by a vertex $\alpha^k\beta^l\gamma^m\delta^n$, we mean the numbers of $\alpha,\beta,\gamma,\delta$ at the vertex are respectively $k,l,m,n$. The {\em angle sum} of the vertex means the equality
\[
k\alpha+l\beta+m\gamma+m\delta=2\pi.
\]
A vertex $\alpha^2\beta\cdots$ means a vertex $\alpha^k\beta^l\gamma^m\delta^n$ with $k\ge 2$ and $l\ge 1$. For example, by the angle sum \eqref{anglesum} for quadrilateral, the angle sum of $\alpha\beta\gamma\delta\cdots$ is $>2\pi$. Therefore $\alpha\beta\gamma\delta\cdots$ is not a vertex. In other words, at least one of $k,l,m,n$ must be $0$. Moreover, the {\em remainder} $R(\alpha^2\beta)$ means either the remaining angle combination $\alpha^{k-2}\beta^{l-1}\gamma^m\delta^n$, or the value $R(\alpha^2\beta)=2\pi-2\alpha-\beta$ of this remaining combination.

If we know all the angle values, then we may find all the possible angle combinations by finding non-negative integer solutions of the angle sum formula. For example, a rhombus has angles $\alpha,\alpha,\beta,\beta$. The first case we consider is $\alpha=\tfrac{2}{3}\pi$ and $\beta=(\tfrac{1}{3}+\tfrac{2}{f})\pi$. The possible angle combinations are all non-negative integers $a,b$ satisfying $\tfrac{2}{3}a+(\tfrac{1}{3}+\tfrac{2}{f})b=2$. We clearly have $a\le 3$ and $b\le 5$. We may substitute various choices of $a,b$ satisfying $a\le 3$ and $b\le 5$, and see which ones yield integers $f\ge 6$ (by \eqref{quadvcountf}). The list of all vertices obtained at the end is the {\em anglewise vertex combination}, which we abbreviate as the AVC. For example, one list we get from $\alpha=\tfrac{2}{3}\pi$ and $\beta=(\tfrac{1}{3}+\tfrac{2}{f})\pi$ is $\{\alpha^3,\alpha\beta^3\}$ for $f=18$, and there are other possible lists. 

For another example, we have $\alpha=(1-\tfrac{2}{f})\pi$ and $\beta=\tfrac{4}{f}\pi$ for another case of the rhombus tiling. In this case, there is no upper bound for $l$ in $\alpha^k\beta^l$, and we get $\text{AVC}=\{\alpha^2\beta,\alpha\beta^k,\beta^k\}$. Here $k$ is a generic notation. In fact, we have the specific values $k=\frac{f+2}{4}$ in $\alpha\beta^k$, and $k=\frac{f}{2}$ in $\beta^k$. In this paper, we reserve $k,l$ for {\em generic} numbers. This means $k,l$ may appear in several vertices, and the values of $k,l$ may be different in different vertices. 

The discussion above about angle combinations at vertices also applies to triangular tilings.

\subsection{Counting}

Inspired by the parity argument in the proof of Lemma \ref{edge_combo}, we have the following result. The property is used so often that we give a specific name. Similar parity lemma for pentagonal tilings appeared in \cite[Lemma 1]{cly} and \cite[Lemma 10]{wy2}.

\begin{lemma}[Parity Lemma]\label{parity}
In a tiling of the sphere by congruent triangles or quadrilaterals, the numbers of the following angles have the same parity at any vertex:
\begin{enumerate}
\item $\alpha,\beta,\gamma$ in general triangle in the first of Figure \ref{triangle}.
\item $\beta,\gamma,\delta$ in general quadrilateral in the first of Figure \ref{quad}.
\item $\gamma,\delta$ in almost equilateral quadrilateral in the third of Figure \ref{quad}.
\end{enumerate} 
Moreover, for the isosceles triangle in the second of Figure \ref{triangle}, the number of $\beta$ at any vertex is even.
\end{lemma}

\begin{proof}
Let $\beta^k\gamma^l\cdots$ be a vertex in a tiling of the sphere by congruent general quadrilateral, with no $\beta,\gamma$ in the remainder. Then $k+l$ is twice of the number of $b$-edges at the vertex. Therefore $k,l$ have the same parity. By considering $c$-edges at a vertex, we also find the numbers of $\gamma$ and $\delta$ have the same parity at any vertex.

The proof for the other polygons are similar.
\end{proof}

By the edge length consideration, we clearly have the following.

\begin{lemma}\label{nod}
In a tiling of the sphere by congruent general quadrilaterals in the first of Figure \ref{quad}, $\alpha^k\gamma^l$ ($k,l>0$) is not a vertex. 
\end{lemma}

The following is another useful property that deserves a special name. It is \cite[Lemma 5]{cly}.
 
\begin{lemma}[Counting Lemma]\label{counting}
In a tiling of the sphere by congruent polygons, suppose two different angles $\theta,\rho$ appear the same number of times in the polygon. If all vertices are of the form $\theta^k\rho^l\cdots$, with $k\le l$ and no $\theta,\rho$ in the remainder, then all vertices are of the form $\theta^k\rho^k\cdots$, with no $\theta,\rho$ in the remainder. 
\end{lemma}

The assumption means that, at every vertex, the number of $\theta$ is no more than the number of $\rho$. The conclusion is that the two numbers must be the same at every vertex.

\begin{proof}
Let $\#_V\theta$ and $\#_V\rho$ be the numbers of $\theta$ and $\rho$ at a vertex $V$. Then the assumption $k\le l$ means 
\[
\#_V\theta\le \#_V\rho\text{ for all }V.
\]
Taking the sum over all $V$, we get
\[
\sum_{\text{all }V}\#_V\theta\le \sum_{\text{all }V}\#_V\rho.
\]
However, the two sides of the second inequality are the total numbers $\#\theta$ and $\#\rho$ of $\theta$ and $\rho$ in the tiling. Suppose both $\theta$ and $\rho$ appear $p$ times in the polygon, then the total number of $\theta$ in the tiling is $\#\theta=pf$, and the total number of $\rho$ is also $\#\rho=pf$. Therefore the second inequality is actually an equality. This implies the first inequalities are all equalities.
\end{proof}

We remark that the counting in the lemma (as well as the subsequent Lemmas \ref{balance_aabc}, \ref{balance_aaab}, \ref{angle_pi}, \ref{notatdeg4}, \ref{deg3miss}, \ref{count-aaa}, \ref{count-att}) relies on the criterion for distinguishing angles. The criterion can be angle values. For example, we know $\alpha\ne \gamma$ (angle equality is always about values) in the second of Figure \ref{quad}. But it is possible to have $\alpha=\beta$. The criterion can also be the bounding edge lengths. For example, the four angles in the first of Figure \ref{quad} are distinguished as $a^2$-angle $\alpha$, $ab$-angle $\beta$, $bc$-angle $\gamma$, $ac$-angle $\delta$. 

Combining the counting lemma and the parity lemma, we get the following useful results about tilings by congruent general quadrilaterals $a^2bc$ or congruent almost equilateral quadrilaterals $a^3b$. The pentagon versions of the lemmas are \cite[Lemma 6]{cly} and \cite[Lemma 11]{wy2}.

\begin{lemma}[Balance Lemma]\label{balance_aabc}
In a tiling of the sphere by congruent general quadrilaterals, one of $\beta^2\cdots,\gamma^2\cdots,\delta^2\cdots$ is a vertex if and only if all three are vertices. Moreover, if the three are not vertices, then the only vertices are $\alpha^k$ and $\beta\gamma\delta$.
\end{lemma}

Proposition \ref{bcd} further shows the tiling is $E_{\square}1$.

\begin{lemma}[Balance Lemma]\label{balance_aaab}
In a tiling of the sphere by congruent non-symmetric almost equilateral quadrilaterals, one of $\gamma^2\cdots,\delta^2\cdots$ is a vertex if and only if both are vertices. Moreover, if both are not vertices, then the only vertices are $\alpha^k\beta^l,\alpha^k\gamma\delta,\beta^l\gamma\delta$. 
\end{lemma}

In the lemma, non-symmetric means $\gamma\ne\delta$ (different values). 

\begin{proof}
We first prove the second balance lemma. Suppose $\gamma^2\cdots$ is not a vertex. Then any vertex is $\gamma^m\delta^n\cdots$, with $m=0,1$ and no $\gamma,\delta$ in the remainder. If $m=0$, then $m\le n$. If $m=1$, then by the parity lemma, we know $n$ is odd, which implies $m\le n$. Therefore we have $m\le n$ at every vertex. Applying the counting lemma to $\gamma,\delta$, we get $m=n\le 1$ at every vertex. This means $\alpha^k\beta^l$ and $\alpha^k\beta^l\gamma\delta$ are the only vertices. By the angle sum \eqref{anglesum} for quadrilateral, we must have $k=0$ or $l=0$ in $\alpha^k\beta^l\gamma\delta$.

For the first balance lemma, the same argument can be applied to any pair from $\beta,\gamma,\delta$. Then we conclude that one of $\beta^2\cdots,\gamma^2\cdots,\delta^2\cdots$ is a vertex if and only if all three are vertices. Moreover, if all three are not vertices, then by the (second part of) parity lemma (Lemma \ref{parity}), the only vertices are $\alpha^k,\alpha^k\beta\gamma\delta$. By the angle sum \eqref{anglesum} for quadrilateral, we must have $k=0$ in $\alpha^k\beta\gamma\delta$.
\end{proof}

\begin{lemma}\label{angle_pi}
In a tiling of the sphere by congruent quadrilaterals, at most one angle is $\ge \pi$.
\end{lemma}

\begin{proof}
Let $\theta_1,\theta_2$ be two angles $\ge \pi$. Then by the angle sum \eqref{anglesum} for quadrilateral, the other two angles $\rho_1,\rho_2\le\frac{4}{f}\pi<\pi$. By $\theta_i\ge \pi$, we know vertices with $\theta_i$ are $\theta_i\rho_1^k\rho_2^l$. Since all vertices have degree $\ge 3$, we have $k+l\ge 2>1$. Then the total number of $\theta_*$ at every vertex is strictly less than the total number of $\rho_*$. However, the total number of $\theta_*$ in the tiling is $2f$, and the total number of $\rho_*$ in the tiling is also $2f$. We get a contradiction.
\end{proof}

We remark that the proof counts $\theta_1,\theta_2$ as one angle, and counts $\rho_1,\rho_2$ as another angle. By $\theta_*\ge \pi$ and $\rho_*<\pi$, the two angles are distinguished. This is the reason the counting argument is valid.

\begin{lemma}\label{notatdeg4}
In a tiling of the sphere by congruent quadrilaterals, if an angle appears at least twice at every degree $3$ vertex, and at least once at every degree $4$ vertex, then the angle appears at least twice in the quadrilateral.
\end{lemma}

\begin{proof}
By \eqref{quadvcountf} and \eqref{quadvcount3}, we have 
\begin{align*}
2v_3+v_4-f
&=2(8+\ssum_{k\ge 4}(k-4)v_k)+v_4-6-\ssum_{k\ge 4}(k-3)v_k  \\
&=10+\ssum_{k\ge 5}(k-5)v_k>0. 
\end{align*}
Therefore the total number of times the angle appears in the tiling is $\ge 2v_3+v_4>f$. This implies the angle appears at least twice in the quadrilateral. 
\end{proof}

\begin{lemma}\label{deg3miss}
In a tiling of the sphere by congruent quadrilaterals, if two angles do not appear at any degree $3$ vertex, then the two angles together either appear at least three times at a degree $4$ vertex, or appear five times at a degree $5$ vertex.
\end{lemma}

The two angles can be two different angles, or can be one angle appearing twice in the quadrilateral.

\begin{proof}
Suppose the two angles do not appear at any degree $3$ vertex, appear at most twice at any degree $4$ vertex, and at most four times at any degree $5$ vertex. Then the total number of times the two angles appear in the tiling is $\le 2v_4+4v_5+\sum_{k\ge 6}kv_k$. By \eqref{quadvcountf}, the number is $<2f$. On the other hand, the total number of times the two angles appear in the tiling should be $\ge 2f$. 
\end{proof}

\begin{lemma}\label{count-aaa}
In a tiling of the sphere by congruent quadrilaterals, if an angle $\alpha$ appears once in the quadrilateral, and $\alpha^3$ is the only degree $3$ vertex, then we have the following:
\begin{enumerate}
\item $f\ge 24$. Moreover, if $f=24$, then all vertices have degrees $3$ or $4$. 
\item If an angle appears once in the quadrilateral, then it appears at most twice at a degree $4$ vertex.
\end{enumerate}
\end{lemma}

\begin{proof}
The assumption implies the other three angles appear only at vertices of degree $\ge 4$. Then by \eqref{quadvcountf}, we get
\[
\ssum_{k\ge 4}kv_k
\ge 3f=18+\ssum_{k\ge 4}3(k-3)v_k.
\]
This is the same as
\[
v_4\ge 18+\ssum_{k\ge 5}(2k-9)v_k\ge 18.
\]
Then we get
\begin{align*}
f &=6+\ssum_{k\ge 4}(k-3)v_k
\ge 24+\ssum_{k\ge 5}3(k-4)v_k
\ge 24, \\
f-3v_4
&=6+\ssum_{k\ge 5}(k-3)v_k-2v_4
\le -30-\ssum_{k\ge 5}3(k-5)v_k
< 0.
\end{align*}
The first inequality implies the first statement. The second inequality implies the second statement.
\end{proof}

\begin{lemma}\label{count-att}
In a tiling of the sphere by congruent quadrilaterals, if different angles $\alpha,\theta$ appear once in the quadrilateral, and $\alpha\theta^2$ is the only degree $3$ vertex, then we have the following:
\begin{enumerate}
\item $f\ge 16$. Moreover, if $f=16$, then all vertices have degrees $3$ or $4$. 
\item If an angle is not $\alpha,\theta$, then the angle either appears at a degree $4$ vertex, or appears at least three times at a degree $5$ vertex, or appears at least five times at a degree $6$ vertex, or appears seven times at a degree $7$ vertex.
\end{enumerate}
\end{lemma}

\begin{proof}
By counting the total number of $\theta$ and \eqref{quadvcount3}, the assumption implies
\[
f\ge 2v_3
=16+\ssum_{k\ge 5}2(k-4)v_k
\ge 16+2v_5+4v_6+6v_7+\ssum_{k\ge 8}kv_k
\ge 16.
\]
This immediately implies the first statement. Moreover, by the only degree $3$ vertex $\alpha\theta^2$, if an angle is not $\alpha,\theta$, then it appears only at vertices of degree $\ge 4$. On the other hand, the total number of times it appears in the tiling is $\ge f>2v_5+4v_6+6v_7+\ssum_{k\ge 8}kv_k$. This implies the vertex appears at a vertex of degree $\ge 4$ in the way described in the lemma.
\end{proof}

\subsection{Geometry}
\label{geometry}

To describe the geometrical constraint, we need to fix the precise meaning of polygon. The meaning in this paper is somewhat different from the convention, but is necessary for clarifying the statements in this section.

In this paper, a (spherical) {\em polygon} is a sequence of arcs (called {\em edges}) connected together (at {\em vertices}) to form a closed path, together with the choice of a {\em side}. An arc is part of a great circle. The side means the following. In Figure \ref{polygon}, picking one angle $\theta$ at one vertex determines the choice of all the angles $*$ at the other vertices, by the criterion of being ``on the same side'' of $\theta$ along the path. Picking the complement $2\pi-\theta$ of $\theta$ gives another side, and therefore another polygon. If we forget the side, then we get the closed path only, which we call the {\em boundary} of the polygon.

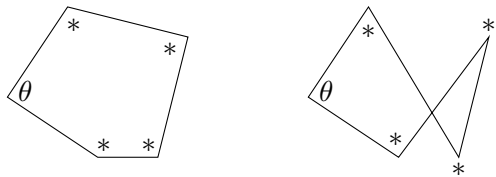
\begin{figure}[htp]
\centering
\begin{tikzpicture}[scale=0.8]

\draw
	(0,0) -- (1,1.5) -- (3,1) -- (2.5,-1) -- (1.5,-1) -- cycle;

\node at (0.3,0.1) {\small $\theta$};
\node at (1.1,1.2) {\small $*$};
\node at (2.7,0.8) {\small $*$};
\node at (2.35,-0.8) {\small $*$};
\node at (1.6,-0.8) {\small $*$};

\begin{scope}[xshift=5cm]

\draw
	(0,0) -- (1,1.5) -- (2.5,-1) -- (3,1) -- (1.5,-1) -- cycle;

\node at (0.3,0.1) {\small $\theta$};
\node at (1,1.1) {\small $*$};
\node at (3,1.2) {\small $*$};
\node at (2.5,-1.2) {\small $*$};
\node at (1.45,-0.7) {\small $*$};
	
\end{scope}

\end{tikzpicture}
\caption{Polygon determined by closed path and one angle.}
\label{polygon}
\end{figure}

A polygon is {\em simple} if the boundary is simple, i.e., not crossing itself. The first of Figure \ref{polygon} is simple, and the second is not simple. A simple polygon is equivalent to a ($2$-dimensional) subset of the sphere, such that the (topological) boundary is a piecewise straight simple closed path. By \cite[Lemma 1]{gsy}, in a tiling of the sphere by congruent polygons, the polygon must be simple. 

The simple property implies that, if two edges are connected at a vertex, then one edge has length $<\pi$. For example, we have $a<\pi$ in all the quadrilaterals in Figure \ref{quad}.

A polygon is {\em convex} if it is simple and all angles are $\le \pi$. This is equivalent to that the polygon is the intersection of finitely many hemispheres. A convex polygon has the property that any two vertices can be connected by an arc inside the polygon.

In the next result, the tiling is not assumed to be edge-to-edge.

\begin{lemma}\label{geometry10}
In a spherical tiling by $f$ congruent convex polygons, all angles $\ge \frac{2}{f}\pi$. If an angle is $\frac{2}{f}\pi$, then the tiling is the hosohedron (earth map tiling by congruent $2$-gons), or its flip modification. 
\end{lemma}

\begin{proof}
Let $\theta$ be an angle of the convex polygon. The convexity implies that the polygon is contained in the $2$-gon with angle $\theta$. Since $f$ congruent polygons tile the sphere, the polygon has area $\frac{4}{f}\pi$. Since the $2$-gon has area $2\theta$, we get $\frac{4}{f}\pi\le 2\theta$.

If $\theta=\frac{2}{f}\pi$, then after ignoring some vertices, the polygon is the whole $2$-gon $G$ with angle $\theta$. We get a (not necessarily edge-to-edge) tiling by congruent copies of $G$. 

The angle combination at a vertex $\bullet$ is $\theta^k$ or $\theta^k\pi$, where $\pi$ appears when $\bullet$ lies in the interior of an edge of another tile. If $\bullet$ is $\theta^k$, then the $k$ copies of $G$ at $\bullet$ form an earth map tiling of the sphere. See the first of Figure \ref{loontiling}.

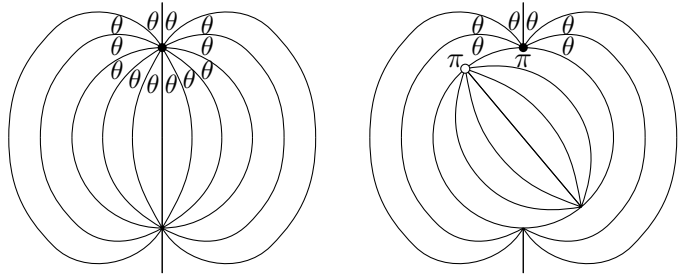
\begin{figure}[htp]
\centering
\begin{tikzpicture}[>=latex, scale=1.2]

\foreach \a in {1,-1}
\foreach \b in {1,-1}
\foreach \c in {0,1}
\draw[xshift=4*\c cm, rotate=40*\c, xscale=\a, yscale=\b]
	(0,0) -- (0,1) arc (90:0:1)
	(0.33,0) to[out=90,in=-60] (0,1)
	(0.66,0) to[out=90,in=-30] (0,1);

\foreach \a in {1,-1}
\foreach \b in {1,-1}
\foreach \c in {0,1}
\draw[xshift=4*\c cm, xscale=\a, yscale=\b]
	(0,1) -- (0,1.5)
	(1.35,0) to[out=90,in=-45] (45:1.35) to[out=135,in=30] (0,1)
	(1.7,0) to[out=90,in=-45] (45:1.7) to[out=135,in=60] (0,1);
	
\fill
	(0,1) circle (0.05)
	(4,1) circle (0.05);

\filldraw[xshift=4cm, fill=white]
	(130:1) circle (0.05);
		
\foreach \a in {1,-1}
{
\begin{scope}[xscale=\a]
	
\node at (0.1,0.6) {\footnotesize $\theta$};
\node at (0.3,0.65) {\footnotesize $\theta$};
\node at (0.5,0.75) {\footnotesize $\theta$};

\node at (0.1,1.3) {\footnotesize $\theta$};
\node at (0.5,1.25) {\footnotesize $\theta$};
\node at (0.5,1.02) {\footnotesize $\theta$};

\end{scope}
}

\foreach \a in {1,-1}
{
\begin{scope}[xshift=4cm, xscale=\a]

\node at (0.1,1.3) {\footnotesize $\theta$};
\node at (0.5,1.25) {\footnotesize $\theta$};
\node at (0.5,1.02) {\footnotesize $\theta$};

\end{scope}
}

\node at (4,0.85) {\small $\pi$};
\node at (3.25,0.85) {\small $\pi$};

\end{tikzpicture}
\caption{Tilings by $2$-gons.}
\label{loontiling}
\end{figure}

Next, we may assume all vertices are $\theta^k\pi$. Then the $k$ copies of $G$ at the $\theta^k$ part of the $\bullet$-vertex $\theta^k\pi$ form a half earth map tiling of a hemisphere. See the second of Figure \ref{loontiling}. Then any tile in the complementary hemisphere has a $\circ$-vertex $\theta^k\pi$, and $k$ copies of $G$ at the $\theta^k$ part of $\circ$ form a half earth map tiling of the complementary hemisphere. The second of Figure \ref{loontiling} is obtained from the first by flipping the complementary hemisphere. 
\end{proof}

\begin{lemma}\label{geometry4}
If a simple almost equilateral quadrilateral satisfies $\gamma<\pi$, then $\alpha+2\beta>\pi$ and $\alpha+2\delta>\pi$. 
\end{lemma}

\begin{lemma}\label{geometry3}
If a simple almost equilateral quadrilateral satisfies $\alpha,\beta,\gamma<\pi$, then $\beta+\gamma<\delta+\pi$ and $\gamma+\delta<\beta+\pi$. 
\end{lemma}

\begin{lemma}\label{geometry5}
If a simple almost equilateral quadrilateral satisfies $\gamma,\delta<\pi$, then $\alpha>\gamma$ if and only if $\beta>\delta$.
\end{lemma} 

By the exchange of $(\alpha,\delta)$ with $(\beta,\gamma)$, Lemma \ref{geometry4} says that, if $\delta<\pi$, then we get $2\alpha+\beta>\pi$ and $\beta+2\gamma>\pi$. Moreover, Lemma \ref{geometry3} says that, if $\alpha,\beta,\delta<\pi$, then we get $\alpha+\delta<\gamma+\pi$ and $\gamma+\delta<\alpha+\pi$. 

Akama and van Cleemput \cite[Lemma 2.1]{ac} proved Lemma \ref{geometry3} for convex almost equilateral quadrilateral.

\begin{proof}
We use planer pictures to argue about spherical quadrilaterals. In order for the pictures to be authentic, we draw stereographic projections, which is conformal. Figure \ref{sphere-geometry} describes the stereographic projection of a simple quadrilateral $\square ABCD$ satisfying $\gamma=\angle BCD<\pi$ and $AB,BC,AD<\pi$. The three pictures correspond to the cases $CD<\pi$, $CD>\pi$ and $CD=\pi$. 

By $\gamma<\pi$, we know $B,D$ are not antipodal points. By simple quadrilateral, we also know $B\ne D$. Therefore there is a unique great circle $\bigcirc BD$ passing through $B$ and $D$. Using $\bigcirc BD$ as the {\em equator} (indicated by the dotted circle), we get the stereographic projection. Under the projection, the arcs on the sphere intersect the equator at antipodal points. We always denote the antipode of $X$ by $X^*$. In the third picture, we have $D=C^*$ and $C=D^*$.

We have $CD<\pi$ in the first picture. This means $BC$ and $CD$ lie in the same hemisphere bounded by $\bigcirc BD$. We have $CD>\pi$ in the second picture. This means $BC$ and $CD^*$ lie in the same hemisphere bounded by $\bigcirc BD$, and $DD^*$ lies in the other hemisphere. We have $CD=\pi$ in the third picture. This means $C$ lies on the equator $\bigcirc BD$.

We indicate the vertices $B,C,D$ by $\bullet$, and indicate the vertex $A$ by $\circ$. Multiple $\circ$ means all the possible locations of $A$, such that the quadrilateral is simple. For example, if $A$ is in $\triangle BCD^*$ in the first picture, or $A$ is in $\triangle B^*CD^*$ in the second picture, or $A$ is in $\triangle B^*CD$ in the third picture, then the quadrilateral is not simple.  

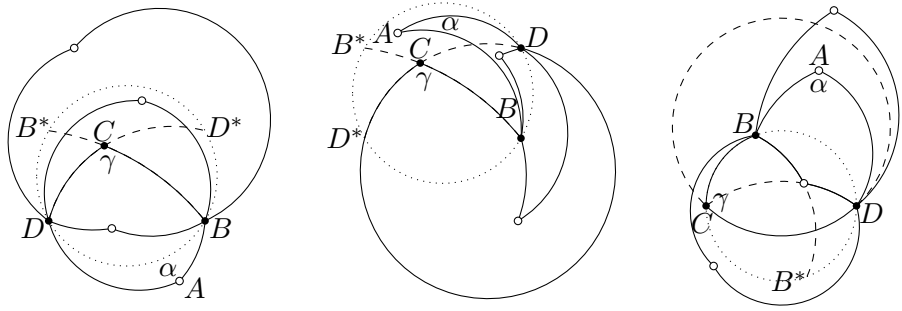
\begin{figure}[htp]
\centering
\begin{tikzpicture}


\begin{scope}[xscale=-1]

\draw[dotted]
	(0,0) circle (1.2);

\coordinate (B1) at (30:1.2);
\coordinate (B2) at (30:-1.2);
\coordinate (C1) at (-30:1.2);
\coordinate (C2) at (-30:-1.2);
\coordinate (D) at (0.3,0.4);
\coordinate (A1) at (-0.7,-1.4);
\coordinate (A2) at (0.7,1.7);
\coordinate (A3) at (-0.2,1);
\coordinate (A4) at (0.2,-0.7);

\arcThroughThreePoints[dashed]{C1}{D}{C2};
\arcThroughThreePoints[dashed]{B1}{D}{B2};

\arcThroughThreePoints{C1}{C2}{D};
\arcThroughThreePoints{D}{B1}{B2};
\arcThroughThreePoints{B2}{B1}{A1};
\arcThroughThreePoints{A1}{C2}{C1};
\arcThroughThreePoints{A2}{B1}{B2};
\arcThroughThreePoints{C1}{C2}{A2};
\arcThroughThreePoints{A3}{B1}{B2};
\arcThroughThreePoints{C1}{C2}{A3};
\arcThroughThreePoints{B2}{B1}{A4};
\arcThroughThreePoints{A4}{C2}{C1};

\fill
	(0.3,0.4) circle (0.05)
	(-30:1.2) circle (0.05)
	(30:-1.2) circle (0.05)
	;

\filldraw[fill=white]
	(-0.7,-1.4) circle (0.05)
	(0.7,1.7) circle (0.05)
	(-0.2,1) circle (0.05)
	(0.2,-0.7) circle (0.05)
	;
		
\node at (0.3,0.6) {\small $C$};
\node at (0.25,0.15) {\small $\gamma$};

\node at (-0.9,-1.5) {\small $A$};
\node at (-0.55,-1.25) {\small $\alpha$};

\node at (-1.25,-0.7) {\small $B$};
\node at (1.25,-0.7) {\small $D$};

\node at (-1.3,0.65) {\small $D^*$};
\node at (1.25,0.65) {\small $B^*$};

\end{scope}


\begin{scope}[shift={(4.2cm, 1.1cm)}, xscale=-1]

\draw[dotted]
	(0,0) circle (1.2);

\coordinate (B1) at (30:1.2);
\coordinate (B2) at (30:-1.2);
\coordinate (C1) at (-30:1.2);
\coordinate (C2) at (-30:-1.2);
\coordinate (D) at (0.3,0.4);
\coordinate (A1) at (-1,-1.7);
\coordinate (A2) at (-0.75,0.5);
\coordinate (A3) at (0.6,0.8);

\arcThroughThreePoints[dashed]{C1}{D}{C2};
\arcThroughThreePoints[dashed]{B1}{D}{B2};

\arcThroughThreePoints{C2}{C1}{D};
\arcThroughThreePoints{D}{B1}{B2};
\arcThroughThreePoints{B2}{B1}{A1};
\arcThroughThreePoints{C2}{C1}{A1};
\arcThroughThreePoints{A2}{B1}{B2};
\arcThroughThreePoints{A2}{C1}{C2};
\arcThroughThreePoints{A3}{B1}{B2};
\arcThroughThreePoints{A3}{C1}{C2};

\fill
	(0.3,0.4) circle (0.05)
	(-30:-1.2) circle (0.05)
	(30:-1.2) circle (0.05)
	;

\filldraw[fill=white]
	(-1,-1.7) circle (0.05)
	(-0.75,0.5) circle (0.05)
	(0.6,0.8) circle (0.05);
	
\node at (0.3,0.6) {\small $C$};
\node at (0.25,0.15) {\small $\gamma$};


\node at (0.8,0.8) {\small $A$};
\node at (-0.1,0.92) {\small $\alpha$};

\node at (-0.85,-0.2) {\small $B$};
\node at (1.3,-0.6) {\small $D^*$};

\node at (-1.25,0.75) {\small $D$};
\node at (1.25,0.65) {\small $B^*$};

\end{scope}


\begin{scope}[shift={(8.7cm, -0.4cm)}]

\draw[dotted]
	(0,0) circle (1);

\draw (110:1) arc (110:180:1);
\coordinate (B) at (110:1);
\coordinate (BB) at (110:-1);

\coordinate (C) at (-1,0);
\coordinate (D) at (1,0);
\coordinate (X) at (0,-0.4);

\coordinate (A1) at (0.3,0.3);
\coordinate (A2) at (0.5,1.8);
\coordinate (A3) at (0.7,2.6);
\coordinate (A4) at (-0.9,-0.8);

\arcThroughThreePoints[dashed]{A1}{D}{C};
\arcThroughThreePoints[dashed]{D}{X}{C};
\arcThroughThreePoints[dashed]{BB}{B}{A1};

\arcThroughThreePoints{C}{X}{D};

\arcThroughThreePoints{A1}{BB}{B};
\arcThroughThreePoints{D}{C}{A1};

\arcThroughThreePoints{A1}{BB}{B};
\arcThroughThreePoints{D}{C}{A1};

\arcThroughThreePoints{A2}{BB}{B};
\arcThroughThreePoints{D}{C}{A2};

\arcThroughThreePoints{A3}{BB}{B};
\arcThroughThreePoints{D}{C}{A3};

\arcThroughThreePoints{B}{BB}{A4};
\arcThroughThreePoints{A4}{C}{D};

\fill
	(B) circle (0.05)
	(C) circle (0.05)
	(D) circle (0.05);

\foreach \a in {1,...,4}
\filldraw[fill=white]
	(A\a) circle (0.05);

\node at (0.5,2) {\small $A$};
\node at (0.5,1.6) {\small $\alpha$};
\node at (-0.5,1.1) {\small $B$};
\node at (0.1,-1) {\small $B^*$};
\node at (-1.05,-0.23) {\small $C$};
\node at (1.2,-0.1) {\small $D$};
\node at (-0.8,0.05) {\small $\gamma$};

\end{scope}

\end{tikzpicture}
\caption{Simple spherical quadrilateral with $\gamma<\pi$ and $AB,BC,AD<\pi$.}
\label{sphere-geometry}
\end{figure}

For simple almost equilateral quadrilateral, we have $AB=BC=AD=a<\pi$. Moreover, all three lemmas assume $\gamma<\pi$.  Therefore the quadrilateral is described by  Figure \ref{sphere-geometry}.

The angle $\gamma<\pi$ determines (according to Figure \ref{polygon}) a triangle $\triangle BCD$, where $BD$ is the minor arc connecting $B,D$. If we also have $\alpha<\pi$, then $\alpha$ also determines an isosceles triangle $\triangle ABD$ with the same minor arc $BD$. Moreover, we find $\triangle ABD,\triangle BCD$ lie on different sides of $BD$. This means the schematic picture in Figure \ref{scheme} is correct. Let $\theta$ be the base angle of the isosceles triangle $\triangle ABD$. Let $\beta',\delta'$ be the angles of $\triangle BCD$ at $B,D$. Then $\beta=\theta+\beta'$ and $\delta=\theta+\delta'$.

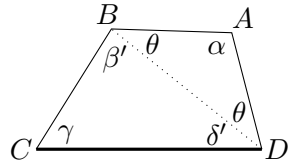
\begin{figure}[htp]
\centering
\begin{tikzpicture}[scale=1]

\draw[dotted]
	(0,0) -- (-2,1.6);

\draw
	(0,0) -- (-0.4,1.55) -- (-2,1.6) -- (-3,0);

\draw[line width=1.2]
	(0,0) -- (-3,0);
	
\node at (-0.25,1.75) {\small $A$};
\node at (-2.05,1.8) {\small $B$};
\node at (-3.2,0) {\small $C$};
\node at (--0.2,0) {\small $D$};

\node at (-0.6,1.35) {\small $\alpha$};
\node at (-1.45,1.4) {\small $\theta$};
\node at (-0.3,0.5) {\small $\theta$};

\node at (-1.95,1.2) {\small $\beta'$};
\node at (-0.6,0.2) {\small $\delta'$};
\node at (-2.6,0.2) {\small $\gamma$};

\end{tikzpicture}
\caption{Schematic picture for proving Lemmas \ref{geometry4}, \ref{geometry3}, \ref{geometry5}.}
\label{scheme}
\end{figure}

For Lemma \ref{geometry4}, the conclusion holds when $\alpha\ge \pi$. Therefore we may assume $\alpha<\pi$. This means we are in the geometric situation described above. The area of $\triangle ABD$ implies $\alpha+2\theta>\pi$. Then by $\beta=\theta+\beta'>\theta$ and $\delta=\theta+\delta'>\theta$, we get $\alpha+2\beta>\pi$ and $\alpha+2\delta>\pi$. This completes the proof of Lemma \ref{geometry4}. 

For Lemma \ref{geometry3}, we are in the same geometric situation. Moreover, in the second and third of Figure \ref{sphere-geometry}, we always have $\alpha\ge\pi$ or $\beta\ge\pi$. Therefore under the assumption $\alpha,\beta<\pi$, we are in the first of Figure \ref{sphere-geometry}. Then all three edges of $\triangle BCD$ are $<\pi$, and the triangle lies in the two $2$-gons with angles $\beta'$ and $\delta'$. The $2$-gons have respective areas $2\beta'$ and $2\delta'$. Therefore we have $\beta'+\gamma+\delta'-\pi<2\beta'$ and $\beta'+\gamma+\delta'-\pi<2\delta'$. These are the same as $\gamma+\delta'<\beta'+\pi$ and $\beta'+\gamma<\delta'+\pi$. Adding $\theta$ to the inequalities, we get $\gamma+\delta<\beta+\pi$ and $\beta+\gamma<\delta+\pi$. This completes the proof of Lemma \ref{geometry3}.

For Lemma \ref{geometry5}, we will actually prove that, if $\gamma<\pi$, then $\beta>\delta$ is equivalent to $b>a$. By the symmetry of exchanging $(\alpha,\delta)$ with $(\beta,\gamma)$, if $\gamma,\delta<\pi$, then we have
\[
\beta>\delta
\iff b>a
\iff \alpha>\gamma.
\]
This proves Lemma \ref{geometry5}.

The isosceles triangle $\triangle ABD$ implies $\beta-\beta'=\delta-\delta'$. This holds even in the case $\alpha=\pi$, and $\triangle ABD$ is reduced to one arc. Therefore $\beta>\delta$ if and only if $\beta'>\delta'$. 

In the first of Figure \ref{sphere-geometry}, we have $b=CD<\pi$. Then $\beta'=\angle CBD>\delta'=\angle CDB$ is equivalent to $b=CD>a=BC$. In the second picture, we have $b>\pi>a$. We also have $\beta'>\pi>\delta'$. In the third picture, we have $b=\pi>a$. We also have $\beta'=\pi>\delta'$. Combining the three cases, we get $\beta'>\delta'$ if and only if $b>a$. 
\end{proof}

The following is \cite[Lemma 3]{wy2}. 

\begin{lemma}\label{geometry1}
For a simple almost equilateral quadrilateral, $\alpha<\beta$ is equivalent to $\gamma<\delta$.
\end{lemma}

The lemma implies $\alpha=\beta$ is equivalent to $\gamma=\delta$. In this case, the quadrilateral is {\em symmetric}. In other words, for non-symmetric almost equilateral quadrilateral, we either have $\alpha<\beta$ and $\gamma<\delta$, or have $\alpha>\beta$ and $\gamma>\delta$.

By the parity lemma for a tiling by congruent almost equilateral quadrilaterals, degree $3$ vertices with $\gamma,\delta$ are $\alpha\gamma^2,\alpha\delta^2,\beta\gamma^2,\beta\delta^2,\alpha\gamma\delta,\beta\gamma\delta$. If three such vertices appear, then comparing the angle sums and using Lemma \ref{geometry1}, we know the quadrilateral is symmetric. On the other hand, if $\alpha\gamma\delta$ and $\beta\delta^2$ are vertices, then comparing the angle sums gives $\alpha+\gamma=\beta+\delta$, contradicting Lemma \ref{geometry1}. By similar reason, we may exclude some other combinations of two degree $3$ vertices with $\gamma,\delta$. Then we get the following. The pentagon version of the result is \cite[Lemma 13]{cly}.

\begin{lemma}\label{geometry2}
If an almost equilateral quadrilateral is simple and non-symmetric, then the following are the only possible combinations of two degree $3$ vertices with $\gamma,\delta$:
\begin{enumerate}
\item $\alpha\delta^2,\beta\gamma^2$.
\item $\alpha\gamma\delta,\beta\gamma^2$.
\item $\beta\gamma\delta,\alpha\delta^2$.
\end{enumerate}
\end{lemma}

The strongest geometrical constraint comes from spherical trigonometric identities. In fact, Akama, Wang and Yan \cite{awy} classified edge-to-edge tilings of the sphere by congruent equilateral pentagons by first using these identities to calculate all the pentagons that might be suitable for such tilings. On the other hand, whenever a tiling is constructed in \cite{awy,wy1,wy2} (often by mostly combinatorial means), the existence of the pentagon in the tiling needs to be justified by spherical trigonometric calculations. 

Almost equilateral quadrilaterals allow three free parameters. Therefore the four angles are related by one equation, and we expect any four angles satisfying the equation determine the quadrilateral. The equation was proposed by Ueno and Agaoka \cite{ua2}, and a sharper version was proved for convex almost equilateral quadrilaterals by Coolsaet \cite{coolsaet}. In the following lemma, we give the most comprehensive result, in the sense that we do not require the usual restrictions (which are implicitly assumed in all the other lemmas) that angles and edge lengths must be in $(0,2\pi)$. We only assume $a,b>0$ in the lemma. We also remark that the quadrilateral in the lemma may not be simple.

\begin{lemma}\label{geometry6}
If the angles $\alpha,\beta,\gamma,\delta$ and the edge length $a$ of an almost equilateral quadrilateral satisfy $a\not\in {\bb Z}\pi$, then
\begin{align}
\sin\tfrac{1}{2}\alpha
\sin(\delta-\tfrac{1}{2}\beta)
&=\sin\tfrac{1}{2}\beta
\sin(\gamma-\tfrac{1}{2}\alpha), \label{coolsaet_eq1} \\
(1-\cos\alpha)\sin\delta\cos a
&=\sin\alpha\cos\delta+\sin\gamma, \label{coolsaet_eq2}  \\
(1-\cos\beta)\sin\gamma\cos a
&=\sin\beta\cos\gamma+\sin\delta. \label{coolsaet_eq3} 
\end{align}
Conversely, if $\alpha,\beta,\gamma,\delta,a$ satisfy the three equalities, and $\alpha,\beta\not\in 2{\bb Z}\pi$, then there is an almost equilateral quadrilateral with the given $\alpha,\beta,\gamma,\delta,a$ .
\end{lemma}
 
The pentagon version of the result is \cite[Lemma 18]{cly}, and is much more complicated.

We remark that, under the first equality \eqref{coolsaet_eq1}, the second the third equalities \eqref{coolsaet_eq2} and \eqref{coolsaet_eq3} are almost equivalent. More specifically, if we multiply $(1-\cos\beta)\sin\gamma$ to \eqref{coolsaet_eq2}, and multiply $(1-\cos\alpha)\sin\delta$ to \eqref{coolsaet_eq3}, and subtracting the two, then we get
\[
(1-\cos\beta)\sin\gamma(\sin\alpha\cos\delta+\sin\gamma)
=(1-\cos\alpha)\sin\delta(\sin\beta\cos\gamma+\sin\delta).
\]
The left side subtracting the right side is $-2(\sin\frac{1}{2}\alpha\sin(\delta+\frac{1}{2}\beta)+\sin\frac{1}{2}\beta\sin(\gamma+\frac{1}{2}\alpha))$ multiplied to $\sin\frac{1}{2}\alpha\sin(\delta-\frac{1}{2}\beta)-\sin\frac{1}{2}\beta\sin(\gamma-\frac{1}{2}\alpha)$. If \eqref{coolsaet_eq1} holds, then the second factor is $0$. Therefore multiplying $(1-\cos\beta)\sin\gamma$ to \eqref{coolsaet_eq2} is equivalent to multiplying $(1-\cos\alpha)\sin\delta$ to \eqref{coolsaet_eq3}. In particular, if $\alpha\not\in 2{\bb Z}\pi$ and $\delta\not\in {\bb Z}\pi$, then \eqref{coolsaet_eq2} implies \eqref{coolsaet_eq3}. Similarly, if $\beta\not\in 2{\bb Z}\pi$ and $\gamma\not\in {\bb Z}\pi$, then \eqref{coolsaet_eq3} implies \eqref{coolsaet_eq2}.

A consequence of the equality \eqref{coolsaet_eq1} is the following useful constraint. Like all the lemmas except Lemma \ref{geometry6}, the following lemma implicitly assumes all angles are in $(0,2\pi)$.

\begin{lemma}\label{geometry7}
If an almost equilateral quadrilateral satisfies $\gamma,\delta\le \pi$, then $\alpha<2\gamma$ if and only if $\beta<2\delta$, and $\alpha=2\gamma$ if and only if $\beta=2\delta$.
\end{lemma}

To prove Lemma \ref{geometry6}, we regard spherical trigonometry as relations in the orthogonal group. Specifically, suppose we have a spherical $n$-gon $P$ with successive edge lengths $a_1,a_2,\dots,a_n$ and successive angles $\alpha_1,\alpha_2,\dots,\alpha_n$ ($\alpha_i$ is between $a_{i-1}$ and $a_i$). Let $\alpha_1$ be at the $z$-axis, and let $a_n$ be in the positive $x$-direction. See the top of Figure \ref{polygon1}. 

\begin{figure}[htp]
\centering
\begin{tikzpicture}[>=latex]

\begin{scope}[shift={(4.5cm,3.7cm)}]

\begin{scope}[gray!50]

\draw
	(0,0) circle (1.5);

\draw[->]
	(0,0) -- (0,1.8);

\draw[->]
	(0,0) -- (-1.5,-0.99);

\draw[->]
	(0,0) -- (1.8,-0.4);

\node at (-1.6,-1) {\small $x$};
\node at (1.9,-0.4) {\small $y$};
\node at (0.15,1.8) {\small $z$};

\draw
	(1.5,0) arc (0:-180:1.5 and 0.6);
\draw[dashed]
	(1.5,0) arc (0:180:1.5 and 0.6);
		
\draw[rotate=60]
	(1.5,0) arc (0:180:1.5 and 0.5);
\draw[rotate=60, dashed]
	(1.5,0) arc (0:-180:1.5 and 0.5);
	
\draw[rotate=-30]
	(1.5,0) arc (0:180:1.5 and 0.78);
\draw[rotate=-30, dashed]
	(1.5,0) arc (0:-180:1.5 and 0.78);
			
\end{scope}

\draw
	(1,0.8) to[out=170, in=0]
	(0,0.86) to[out=229, in=62] 
	(-0.57,0) to[out=-25, in=175] 
	(1.1,-0.5) to[out=90, in=-80]
	(1,0.8);

\node at (-0.5,0.55) {\small $a_n$};
\node at (0.15,0.65) {\small $\alpha_1$};
\node at (0.5,1) {\small $a_1$};
\node at (0.8,0.65) {\small $\alpha_2$};
\node at (1.3,0.15) {\small $a_2$};

\node at (0.4,0.2) {\small $P$};

\end{scope}


\begin{scope}[gray!50]

\draw[->]
	(-1,0) -- (1.3,0);
\draw[->]
	(0,-0.5) -- (0,1.8);	
\node at (1.45,0) {\small $x$};
\node at (0.2,1.8) {\small $y$};

\end{scope}

\draw
	(0,0) -- ++(130:1) -- ++(30:1.2) -- (0:0.9) -- cycle;

\node at (0.15,0.15) {\small $\alpha_1$};
\node at (-0.4,0.25) {\small $a_1$};
\node at (-0.3,0.7) {\small $\alpha_2$};
\node at (-0.2,1.25) {\small $a_2$};
\node at (0.5,-0.2) {\small $a_n$};

\draw[->, very thick]
	(1.15,1.3) -- ++(1,0);
\node at (1.65,1.6) {\small $Z(\pi-\alpha_1)$};


\begin{scope}[xshift=3.3cm]

\begin{scope}[gray!50]

\draw[->]
	(-1.3,0) -- (0.8,0);
\draw[->]
	(0,-0.5) -- (0,1.8);	
\node at (0.95,0) {\small $x$};
\node at (0.2,1.8) {\small $y$};

\end{scope}

\draw[rotate=50]
	(0,0) -- ++(130:1) -- ++(30:1.2) -- (0:0.9) -- cycle;

\node at (-0.15,0.15) {\small $\alpha_1$};
\node at (-0.4,-0.2) {\small $a_1$};
\node at (-0.75,0.15) {\small $\alpha_2$};
\node at (-1.1,0.6) {\small $a_2$};
\node at (0.5,0.25) {\small $a_n$};

\draw[->, very thick]
	(0.5,1.3) -- ++(1,0);
\node at (1,1.6) {\small $Y(a_1)$};
	
\end{scope}


\begin{scope}[xshift=5.3cm]

\begin{scope}[gray!50]

\draw[->]
	(-0.5,0) -- (1.5,0);
\draw[->]
	(0,-0.5) -- (0,1.8);	
\node at (1.65,0) {\small $x$};
\node at (0.2,1.8) {\small $y$};

\end{scope}

\draw[xshift=1 cm, rotate=50]
	(0,0) -- ++(130:1) -- ++(30:1.2) -- (0:0.9) -- cycle;

\node at (0.85,0.15) {\small $\alpha_1$};
\node at (0.6,-0.2) {\small $a_1$};
\node at (0.25,0.15) {\small $\alpha_2$};
\node at (-0.1,0.6) {\small $a_2$};
\node at (1.5,0.25) {\small $a_n$};

\draw[->, very thick]
	(1.35,1.3) -- ++(1,0);
\node at (1.85,1.6) {\small $Z(\pi-\alpha_2)$};
		
\end{scope}


\begin{scope}[xshift=9cm]

\begin{scope}[gray!50]

\draw[->]
	(-1.5,0) -- (0.5,0);
\draw[->]
	(0,-0.5) -- (0,1.8);	
\node at (0.65,0) {\small $x$};
\node at (0.2,1.8) {\small $y$};

\end{scope}

\draw[rotate=100, xshift=1 cm, rotate=50]
	(0,0) -- ++(130:1) -- ++(30:1.2) -- (0:0.9) -- cycle;

\node at (-0.35,0.9) {\small $\alpha_1$};
\node at (0.15,0.5) {\small $a_1$};
\node at (-0.25,0.15) {\small $\alpha_2$};
\node at (-0.5,-0.2) {\small $a_2$};
\node at (-0.4,1.3) {\small $a_n$};
	
\end{scope}

\end{tikzpicture}
\caption{A polygon is a sequence of rotations.}
\label{polygon1}
\end{figure}
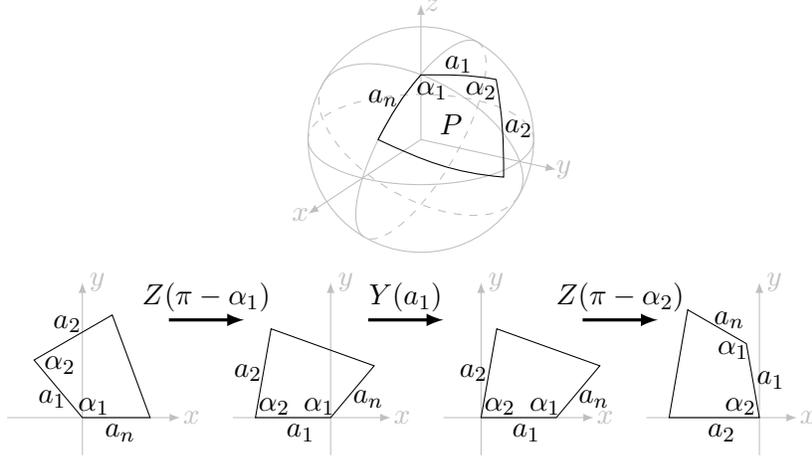

The second row of Figure \ref{polygon1} describes a sequence of movements of the polygon as viewed from above the $z$-axis. We first rotate $P$ such that $a_1$ lies in the $x$-axis. The movement is the rotation $Z(\pi-\alpha_1)$ around the $z$-axis by angle $\pi-\alpha_1$. Next we slide $a_1$ from the negative $x$-direction to the positive $x$-direction, so that $\alpha_2$ is now at the $z$-axis. The movement is the rotation $Y(a_1)$ around the $y$-axis by angle $a_1$. Then we rotate $P$ around the $z$-axis by $\pi-\alpha_2$, etc. By going around $P$, we get a sequence of rotations. At the end, the polygon $P$ is back to the initial position. This means the composition of rotations fixes $P$. Therefore the composition is the identity
\begin{equation}\label{coolsaet_eq0}
Y(a_n)Z(\pi-\alpha_n)\cdots
Y(a_2)Z(\pi-\alpha_2)
Y(a_1)Z(\pi-\alpha_1)=I.
\end{equation}
The equality is equivalent to the existence of the polygon in the sense of Figure \ref{polygon}.

\begin{proof}[Proof of Lemma \ref{geometry6}]
For the quadrilateral in the lemma, the general equality \eqref{coolsaet_eq0} becomes
\[
Y(a)Z(\pi-\alpha)Y(a)Z(\pi-\beta)Y(a)Z(\pi-\gamma)Y(b)Z(\pi-\delta)=I.
\]
Then the existence of the quadrilateral means 
\begin{equation}\label{coolsaet_eq4}
K=Z(\pi-\delta)Y(a)Z(\pi-\alpha)Y(a)Z(\pi-\beta)Y(a)Z(\pi-\gamma)
\end{equation}
is $Y(b)^T$ for some choice of $b$. 

Let $\vec{e}_2=(0,1,0)^T$ be the second standard basis vector of ${\bb R}^3$. Then $K=Y(b)^T$ implies $K\vec{e}_2=\vec{e}_2$. Conversely, if $K\vec{e}_2=\vec{e}_2$, then the $(2,2)$-entry of $K$ is $1$. Since $K$ is a matrix in $SO(3)$, this implies $K=Y(b)^T$ for some $b$. Therefore the existence of the quadrilateral is equivalent to $K\vec{e}_2=\vec{e}_2$. 

Let
\[
L=Y(a)Z(\pi-\beta)Y(a)Z(\pi-\gamma)-Z(\pi-\alpha)^TY(a)^TZ(\pi-\delta)^T.
\]
Then $K\vec{e}_2=\vec{e}_2$ is equivalent to $L\vec{e}_2=\vec{0}$. 

The three coordinates of $L\vec{e}_2$ are
\begin{align*}
l_1
&=
(\cos\beta-1)\sin\gamma\cos^2 a \\
&\quad
+(\cos\alpha\sin\delta+\sin\beta\cos\gamma)\cos a 
+\sin\alpha\cos\delta+\sin\gamma \\
&=
((\cos\beta-1)\sin\gamma\cos a+\sin\beta\cos\gamma+\sin\delta)\cos a \\
&\quad
+(\cos\alpha-1)\sin\delta\cos a+\sin\alpha\cos\delta+\sin\gamma. \\
l_2
&=(\sin\alpha\sin\delta-\sin\beta\sin\gamma)\cos a
-(\cos\alpha\cos\delta-\cos\beta\cos\gamma). \\
l_3
&=((1-\cos\beta)\sin\gamma\cos a-\sin\beta\cos\gamma-\sin\delta)\sin a.
\end{align*}
By 
\[
\cos\tfrac{1}{2}\theta(1-\cos\theta)
=\sin\tfrac{1}{2}\theta\sin\theta, \quad 
\cos\tfrac{1}{2}\theta\sin\theta
=\sin\tfrac{1}{2}\theta(1+\cos\theta),
\]
we get
\begin{align*}
l_2 \sin\tfrac{1}{2}\alpha\sin\tfrac{1}{2}\beta 
&=\sin\tfrac{1}{2}\alpha\sin\tfrac{1}{2}\beta(\sin\alpha\sin\delta\cos a-(1+\cos\alpha)\cos\delta+\cos\delta) \\
&\quad 
-\sin\tfrac{1}{2}\alpha\sin\tfrac{1}{2}\beta(\sin\beta\sin\gamma\cos a-(1+\cos\beta)\cos\gamma+\cos\gamma) \\
&=
\cos\tfrac{1}{2}\alpha\sin\tfrac{1}{2}\beta((1-\cos\alpha)\sin\delta\cos a-\sin\alpha\cos\delta) \\
&\quad 
-\cos\tfrac{1}{2}\beta\sin\tfrac{1}{2}\alpha((1-\cos\beta)\sin\gamma\cos a-\sin\beta\cos\gamma) \\
&\quad
+\sin\tfrac{1}{2}\alpha\sin\tfrac{1}{2}\beta\cos\delta
-\sin\tfrac{1}{2}\alpha\sin\tfrac{1}{2}\beta\cos\gamma \\
&=\cos\tfrac{1}{2}\alpha\sin\tfrac{1}{2}\beta((1-\cos\alpha)\sin\delta\cos a-\sin\gamma-\sin\alpha\cos\delta) \\
&\quad 
-\cos\tfrac{1}{2}\beta\sin\tfrac{1}{2}\alpha((1-\cos\beta)\sin\gamma\cos a-\sin\delta-\sin\beta\cos\gamma) \\
&\quad 
+\sin\tfrac{1}{2}\beta\sin(\gamma-\tfrac{1}{2}\alpha)
-\sin\tfrac{1}{2}\alpha\sin(\delta-\tfrac{1}{2}\beta).
\end{align*}

Suppose $a\not\in{\bb Z}\pi$. Then $\sin a\ne 0$, and $l_3=0$ implies \eqref{coolsaet_eq3}. Then \eqref{coolsaet_eq3} and $l_1=0$ imply \eqref{coolsaet_eq2}. Then by the calculation of $l_2\sin\tfrac{1}{2}\alpha\sin\tfrac{1}{2}\beta$, we see \eqref{coolsaet_eq2}, \eqref{coolsaet_eq3} and $l_2=0$ imply \eqref{coolsaet_eq1}.

Conversely, suppose $\alpha,\beta\not\in 2{\bb Z}\pi$. Then \eqref{coolsaet_eq2} and \eqref{coolsaet_eq3} imply $l_1=0$ and $l_3=0$. Moreover, \eqref{coolsaet_eq1}, \eqref{coolsaet_eq2} and \eqref{coolsaet_eq3} imply $l_2\sin\tfrac{1}{2}\alpha\sin\tfrac{1}{2}\beta=0$. By $\alpha,\beta\not\in 2{\bb Z}\pi$, this further implies $l_2=0$. 
\end{proof}

Now we discuss what happens to Lemma \ref{geometry6} when the assumptions are not satisfied.

Suppose $a\in(2{\bb Z}+1)\pi$. Then $A=C$ and $B=D$, and the two points are antipodal points. Therefore $b\in(2{\bb Z}+1)\pi$. See the first of Figure \ref{coolsaet_f1}, where $A=C$ is the right vertex, and $B=D$ is the left vertex. The only relation between the four angles is $\alpha+\gamma=\beta+\delta$ mod $2\pi$. Therefore \eqref{coolsaet_eq1} does not always hold. 

Suppose $a\in 2{\bb Z}\pi$. Then $A=B=C=D$ is a single point, and $b\in 2{\bb Z}\pi$. See the second of Figure \ref{coolsaet_f1}, where an edge $XY$ leaves $X$ through $X_o$ (for \underline{o}ut of $X$) and arrives at $Y$ through $Y_i$ (for \underline{i}nto $Y$). If we also denote by $X_i,X_o$ the angles of the directions of $X_i,X_o$, then we have $Y_i=X_o+\pi$ mod $2\pi$. Moreover, we have $\alpha=A_i-A_o=D_o-A_o+\pi$ mod $2\pi$, and similar formula for the other angles (all hold mod $2\pi$)
\[
\alpha=D_o-A_o+\pi,\;
\beta=A_o-B_o+\pi,\;
\gamma=B_o-C_o+\pi,\;
\delta=C_o-D_o+\pi.
\]
Therefore the only relation between the four angles is $\alpha+\beta+\gamma+\delta=0$ mod $2\pi$, and \eqref{coolsaet_eq1} does not always hold.

\begin{figure}[htp] 
\centering
\begin{tikzpicture}[>=latex,scale=1.2]


\coordinate (AA) at (-1,0);
\coordinate (A) at (1,0);

\coordinate (X) at (0,0.8);
\coordinate (Y) at (0,0.4);
\coordinate (Z) at (0,-0.5);
\coordinate (W) at (0,-0.2);

\arcThroughThreePoints{A}{X}{AA};
\arcThroughThreePoints[line width=1.2]{A}{Y}{AA};
\arcThroughThreePoints{AA}{Z}{A};
\arcThroughThreePoints{AA}{W}{A};

\draw[->] (0,0.8) -- ++(-0.02,0);
\draw[->] (0,0.4) -- ++(-0.02,0);
\draw[->] (0,-0.2) -- ++(0.02,0);
\draw[->] (0,-0.5) -- ++(0.02,0);
	
\node at (0,1) {\small $BA$};
\node at (0,0) {\small $BC$};
\node at (0,-0.7) {\small $DA$};
\node at (0,0.57) {\small $DC$};

\node at (0.65,0.05) {\small $\alpha$};
\node at (-0.7,0.35) {\small $\beta$};
\node at (0.45,0.05) {\small $\gamma$};
\node at (-0.5,0) {\small $\delta$};

\draw[xshift=1cm]
	(111:0.25) arc (111:227:0.25);
\draw[xshift=1cm]
	(198:0.45) arc (198:144:0.45);
\draw[xshift=-1cm]
	(-18:0.3) arc (-18:70:0.3);
\draw[xshift=-1cm]
	(-45:0.4) arc (-45:36:0.4);

\node at (0,-1.2) {\small $a\in (2{\bb Z}+1)\pi$};


\begin{scope}[shift={(2.7cm,0.2cm)}]

\draw
	(0:1) -- (0:-1)
	(40:1) -- (40:-1)
	(100:0.9) -- (100:-0.9);

\draw[line width=1.2]
	(60:1) -- (60:-1);

\draw[->] (0:0.7) -- ++(0:-0.02);
\draw[->] (100:-0.7) -- ++(100:-0.02);
\draw[->] (100:0.7) -- ++(100:-0.02);
\draw[->] (40:0.7) -- ++(40:0.02);
\draw[->] (40:-0.7) -- ++(40:0.02);
\draw[->] (60:-0.7) -- ++(60:-0.02);
\draw[->] (60:0.7) -- ++(60:-0.02);
\draw[->] (0:-0.7) -- ++(0:-0.02);

\node at (0:1.15) {\small $A_i$};	
\node at (110:-1) {\small $A_o$};	
\node at (110:1) {\small $B_i$};		
\node at (35:1.2) {\small $B_o$};		
\node at (35:-1.2) {\small $C_i$};	
\node at (57:-1.1) {\small $C_o$};
\node at (54:1.15) {\small $D_i$};
\node at (0:-1.15) {\small $D_o$};

\draw
	(0:0.4) arc (0:-80:0.4)
	(40:0.4) arc (40:100:0.4)
	(220:0.22) arc (220:-120:0.22)
	(60:0.31) arc (60:-180:0.31);
	
\node at (-40:0.53) {\small $\alpha$};
\node at (80:0.55) {\small $\beta$};

\node at (0,-1.4) {\small $a\in 2{\bb Z}\pi$};

\end{scope}


\begin{scope}[shift={(5.5cm,0.4cm)}]

\draw
	(0,-0.1) -- (1,-0.1) arc (-90:90:0.05) -- (0,0) -- (220:1)
	(0.8,0) arc (140:220:0.08)
	(220:0.8) arc (40:-140:0.2)
	(220:0.5) arc (220:360:0.5);

\draw[line width=1.2]
	(220:1) -- (220:1.8);

\draw[->]	
	(220:1.4) -- ++(220:0.02);

\begin{scope}[yshift=-0.1cm]

\draw	
	(0.4,0) arc (0:40:0.4);
	
\draw[line width=1.2]	
	(40:0) -- (40:1.2);

\draw[->]	
	(40:0.6) -- ++(40:-0.02);

\end{scope}

\node at (1.2,-0.05) {\small $A$};
\node at (-0.1,0.1) {\small $B$};
\node at (-0.8,-0.5) {\small $C$};
\node at (0,-0.25) {\small $D$};

\node at (0.8,-0.2) {\small $\alpha$};
\node at (0.25,-0.6) {\small $\beta$};
\node at (-0.55,-0.9) {\small $\gamma$};
\node at (0.5,0.15) {\small $\delta$};

\node at (0,-1.4) {\small $\alpha\in 2{\bb Z}\pi$};
\node at (0,-1.8) {\small $\gamma \in (2{\bb Z}+1)\pi$};
	
\end{scope}


\begin{scope}[shift={(8.1cm,0.4cm)}]

\draw
	(0,0) -- (1.2,0) arc (-90:90:0.05) -- (-0.04,0.1) -- 
	++(220:1.2) arc (130:310:0.05)
	(1,0.1) arc (140:220:0.08)
	(0.45,0.1) arc (0:-140:0.5)
	(0.4,0) arc (0:220:0.4);

\draw[rotate=-140, yshift=-0.1 cm]
	(1,0.1) arc (140:220:0.08);
	
\draw[line width=1.2]
	(220:1.18) -- (0,0);

\node at (1.4,0.05) {\small $A$};
\node at (-0.1,0.23) {\small $B$};
\node at (-1.1,-0.9) {\small $C$};
\node at (0.1,-0.15) {\small $D$};

\node at (1,-0.1) {\small $\alpha$};
\node at (0.25,-0.45) {\small $\beta$};
\node at (-0.92,-0.5) {\small $\gamma$};
\node at (-0.2,0.5) {\small $\delta$};

\node at (0,-1.4) {\small $\alpha\in 2{\bb Z}\pi$};
\node at (0,-1.8) {\small $\gamma\in 2{\bb Z}\pi$};
	
\end{scope}

\end{tikzpicture}
\caption{Lemma \ref{geometry6}: Case $a\in {\bb Z}\pi$ and case $\alpha\in 2{\bb Z}\pi$.}
\label{coolsaet_f1}
\end{figure}
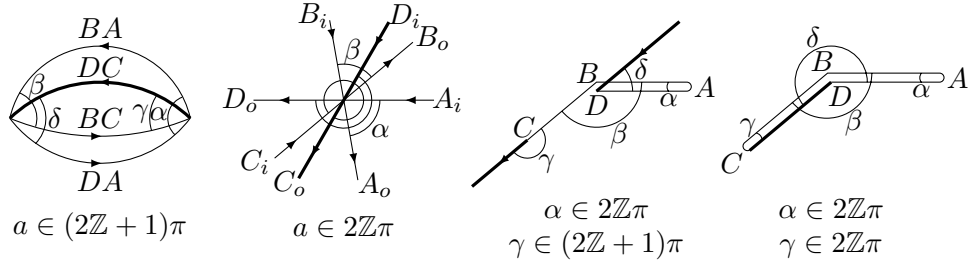

For the converse, suppose $\alpha\in 2{\bb Z}\pi$. Then the three equalities are equivalent to either $\gamma\in (2{\bb Z}+1)\pi$ and $\sin\beta-\sin\delta=0$, or $\gamma\in 2{\bb Z}\pi$ and $\sin\beta+\sin\delta=0$. 

If $\gamma\in (2{\bb Z}+1)\pi$, then the third of Figure \ref{coolsaet_f1} shows $B=D$ and $a+b=0$ mod $2\pi$. The implies the geometrical condition (for the existence of the quadrilateral) $\beta+\delta=\pi$ mod $2\pi$. However, the algebraic condition $\sin\beta-\sin\delta=0$ means $\beta=\delta$ mod $2\pi$, or $\beta+\delta=\pi$ mod $2\pi$. The algebraic condition does not imply the geometrical condition. 

If $\gamma\in 2{\bb Z}\pi$, then the fourth of Figure \ref{coolsaet_f1} shows $B=D$ and $a=b$ mod $2\pi$. See the fourth of Figure \ref{coolsaet_f1}. The implies the geometrical condition $\beta+\delta=0$ mod $2\pi$. However, the algebraic condition $\sin\beta+\sin\delta=0$ means $\beta+\delta=0$ mod $2\pi$, or $\beta=\delta+\pi$ mod $2\pi$. Again algebraic condition does not imply the geometrical condition. 

Finally, we discuss the existence of polygon in a tiling. More specifically, we need to show that a simple spherical polygon with given angles and edge combination actually exists. The issue already appeared in the first paper \cite{gsy} on pentagonal tiling, which give five combinatorially possible tilings. Then further geometrical argument \cite{ay1} showed that four of the five are just the regular dodecahedron. In the subsequent works on pentagonal tilings \cite{awy,wy1,wy2}, after tilings were obtained by combinatorial argument, further geometrical argument was made to show the existence of the pentagon.

The existence issue were less apparent for triangular tilings due to the following simple criterion: There is a simple spherical triangle with given angles $\alpha,\beta,\gamma<\pi$ if and only if the following four equalities hold
\[
\alpha+\beta+\gamma>\pi,\;
\alpha+\pi>\beta+\gamma,\;
\beta+\pi>\alpha+\gamma,\;
\gamma+\pi>\alpha+\beta.
\]
Although the tetrahedron tiling allows simple triangle with an angle $>\pi$, the triangles are convex in all other tilings.

The criteria for the existence of kite and rhombus are also simple because they can be divided into two congruent triangles. Since we obtain not so many tilings by general quadrilaterals, we will give specialised argument for the existence for individual tilings.

We obtain quite a number of tilings by almost equilateral quadrilaterals. The recipe for arguing the existence is the following. First, we make sure the equality \eqref{coolsaet_eq1} holds. Then we use \eqref{coolsaet_eq2} or \eqref{coolsaet_eq3} to uniquely determine $a<\pi$ by
\[
\cos a
=\frac{\sin\alpha\cos\delta+\sin\gamma}{(1-\cos\alpha)\sin\delta}
=\frac{\sin\beta\cos\gamma+\sin\delta}{(1-\cos\beta)\sin\gamma}.
\]
Here the second equality is actually a consequence of \eqref{coolsaet_eq1}. Of course getting $a$ this way assumes the absolute value of the right side is $<1$. This turns out to be true in all the cases. At this stage, by the converse of Lemma \ref{geometry6}, we already know a quadrilateral exists. The problem is whether the quadrilateral is simple. For this purpose, we substitute the four angles and $a$ into $K$ in \eqref{coolsaet_eq4}. Then we use $K=Y(b)^T$ to uniquely determine $0< b\le 2\pi$. We actually find $\sin b>0$ in all cases. Therefore we always get  $0<b<\pi$. Then we may use the following sufficient condition for the quadrilateral to be simple. 

\begin{lemma}\label{geometry8}
If four edges of a spherical quadrilateral are $<\pi$, and three angles are $<\pi$, then the quadrilateral is simple.
\end{lemma}

\begin{proof}
Figure \ref{simpleproperty} is the stereographic projection of a quadrilateral $\square ABCD$ with all edges $<\pi$ and angles at $A,B,D$ being $<\pi$. We draw two circles $\bigcirc ABA^*$ and $\bigcirc ADA^*$. By $\angle BAD<\pi$, the $2$-gon $G$ formed by $\angle BAD$ is the intersection of two hemispheres bounded by $\bigcirc ABA^*$ and $\bigcirc ADA^*$. By $AB<\pi$ and $AD<\pi$, we know $B$ and $D$ are on the two respective boundary edges of $G$. 

\begin{figure}[htp] 
\centering
\begin{tikzpicture}[>=latex,scale=0.6]

\tikzmath{\a1 =210; \a2 = 20; }
	
\draw[gray!50]
	(-1,0) circle (2)
	(1,0) circle (2);

\draw
	(1,0) + (120 : 2) arc  (120:195.5:2);

\draw
	(-1,0) + (10 : 2) arc  (10:60:2);

\pgfmathsetmacro{\rA}{(cos(\a1)+sqrt(3+cos(\a1)*cos(\a1))};
\pgfmathsetmacro{\rAA}{(cos(\a1)-sqrt(3+cos(\a1)*cos(\a1))};

\pgfmathsetmacro{\rC}{(-cos(\a2)+sqrt(3+cos(\a2)*cos(\a2))};
\pgfmathsetmacro{\rCC}{(-cos(\a2)-sqrt(3+cos(\a2)*cos(\a2))};

\coordinate (C) at (0.3,0.8);

\coordinate (A) at (0,{sqrt(3)});
\coordinate (AA) at (0,{-sqrt(3)});

\coordinate (B) at (\a1:\rA);
\coordinate (BB) at (\a1:\rAA);

\coordinate (D) at (\a2:\rC);
\coordinate (DD) at (\a2:\rCC);

\arcThroughThreePoints[gray!50]{AA}{C}{A};

\arcThroughThreePoints{C}{BB}{B};
\arcThroughThreePoints{D}{DD}{C};

\arcThroughThreePoints[dashed]{BB}{B}{C};
\arcThroughThreePoints[dashed]{C}{D}{DD};

\node at (0,2) {\small $A$};
\node at (0,-2) {\small $A^*$};
\node at (-1.3,-0.6) {\small $B$};
\node at (2.9,1.5) {\small $B^*$};
\node at (0.25,0.4) {\small $C$};
\node at (1.2,0.3) {\small $D$};
\node at (-2.9,-1.3) {\small $D^*$};

\end{tikzpicture}
\caption{Sufficient condition for quadrilateral to be simple.}
\label{simpleproperty}
\end{figure}
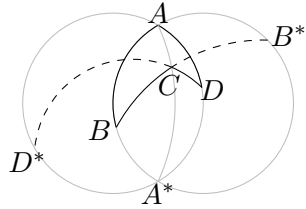

By $BC<\pi$ and $\angle ABC<\pi$, we know $C$ is on the arc $BCB^*$ of length $\pi$ inside the hemisphere bounded by $\bigcirc ABA^*$. By the same reason, we know $C$ is also inside the hemisphere bounded by $\bigcirc ADA^*$. Therefore $C$ lies in the intersection $G$ of the two hemispheres. Then $\triangle ABC$ lies in the $2$-gon $ABA^*C$ and $\triangle ACD$ lies in the $2$-gon $ADA^*C$. Since the two $2$-gons do not overlap, we know the quadrilateral, as the union of two triangles, is simple.
\end{proof}

\subsection{Adjacent Angle Deduction}

The notation for vertices, such as $\alpha^2\beta^2$, does not show the arrangement of angles at the vertex. For example, for almost equilateral quadrilateral, Figure \ref{aad} shows two possible arrangements $\thin\alpha\thin\beta\thin\alpha\thin\beta\thin$ and $\thin\alpha\thin\alpha\thin\beta\thin\beta\thin$ of $\alpha^2\beta^2$. We indicate the arrangements by inserting the edges between the angles. For more example, we have vertices $\thin\alpha\thin\alpha\thin\alpha\thin,\thin\beta\thick\beta\thin\delta\dash\delta\thin,\thick\gamma\dash\gamma\thick\gamma\dash\gamma\thick$ in Figure \ref{aaa-abbB}. 

We may further indicate the arrangements of tiles around the vertex by using the {\em adjacent angle deduction} (abbreviated AAD) notation. We denote the first of Figure \ref{aad} by the AAD $\thin^{\beta}\alpha^{\delta}\thin^{\gamma}\beta^{\alpha}\thin^{\beta}\alpha^{\delta}\thin^{\gamma}\beta^{\alpha}\thin$. Here we use ${}^{\lambda}\theta^{\mu}$ to indicate the angles $\lambda,\mu$ adjacent to $\theta$.  The second $\thin^{\beta}\alpha^{\delta}\thin^{\delta}\alpha^{\beta}\thin^{\gamma}\beta^{\alpha}\thin^{\gamma}\beta^{\alpha}\thin$, the third $\thin^{\beta}\alpha^{\delta}\thin^{\delta}\alpha^{\beta}\thin^{\alpha}\beta^{\gamma}\thin^{\gamma}\beta^{\alpha}\thin$, and the fourth $\thin^{\beta}\alpha^{\delta}\thin^{\delta}\alpha^{\beta}\thin^{\gamma}\beta^{\alpha}\thin^{\alpha}\beta^{\gamma}\thin$ are three different possible AADs of $\thin\alpha\thin\alpha\thin\beta\thin\beta\thin$. The AADs of the vertices in Figure \ref{aaa-abbB} are $\thin^{\beta}\alpha^{\delta}\thin^{\beta}\alpha^{\delta}\thin^{\beta}\alpha^{\delta}\thin,\thin^{\alpha}\beta^{\gamma}\thick^{\gamma}\beta^{\alpha}\thin^{\alpha}\delta^{\gamma}\dash^{\gamma}\delta^{\alpha}\thin,\thick^{\beta}\gamma^{\delta}\dash^{\delta}\gamma^{\beta}\thick^{\beta}\gamma^{\delta}\dash^{\delta}\gamma^{\beta}\thick$.

\begin{figure}[htp]
\centering
\begin{tikzpicture}[>=latex,scale=1]


\foreach \b in {-1,1}
{
\begin{scope}[xshift=-3cm, scale=\b]

\draw	
	(0,1) -- (0,-1) -- (1,-1) -- (1,1) -- (0,1)
	(0,0) -- (1,0);
	
\draw[line width=1.2]
	(-1,1) -- (1,1);

\node at (0.2,0.2) {\small $\alpha$};
\node at (0.8,0.2) {\small $\beta$};
\node at (0.8,0.8) {\small $\gamma$};
\node at (0.2,0.8) {\small $\delta$};
	
\node at (0.8,-0.2) {\small $\alpha$};
\node at (0.2,-0.2) {\small $\beta$};
\node at (0.2,-0.8) {\small $\gamma$};
\node at (0.8,-0.8) {\small $\delta$};

\end{scope}
}

\foreach \a in {0,1,2}
\foreach \b in {1,-1}
{
\begin{scope}[xshift=3*\a cm, xscale=\b]

\draw	
	(0,1) -- (0,-1) -- (1,-1) -- (1,1) -- (0,1)
	(0,0) -- (1,0);

\node at (0.2,-0.2) {\small $\beta$};

\node at (0.2,0.2) {\small $\alpha$};
\node at (0.8,0.2) {\small $\beta$};
\node at (0.8,0.8) {\small $\gamma$};
\node at (0.2,0.8) {\small $\delta$};

\draw[line width=1.2]
	(0,1) -- (1,1);

\end{scope}
}


\draw[line width=1.2]
	(0,-1) -- (1,-1)
	(-1,0) -- (-1,-1);
	
\node at (0.8,-0.2) {\small $\alpha$};
\node at (0.2,-0.8) {\small $\gamma$};
\node at (0.8,-0.8) {\small $\delta$};

\node at (-0.2,-0.8) {\small $\alpha$};
\node at (-0.8,-0.2) {\small $\gamma$};
\node at (-0.8,-0.8) {\small $\delta$};


\foreach \b in {-1,1}
{
\begin{scope}[xshift=3cm, xscale=\b]

\draw[line width=1.2]
	(0,-1) -- (1,-1);
	
\node at (0.8,-0.2) {\small $\alpha$};
\node at (0.2,-0.8) {\small $\gamma$};
\node at (0.8,-0.8) {\small $\delta$};

\end{scope}
}


\foreach \b in {-1,1}
{
\begin{scope}[xshift=6cm, xscale=\b]

\draw[line width=1.2]
	(-1,0) -- (-1,-1);

\node at (-0.2,-0.8) {\small $\alpha$};
\node at (-0.8,-0.2) {\small $\gamma$};
\node at (-0.8,-0.8) {\small $\delta$};

\end{scope}
}

\end{tikzpicture}
\caption{Adjacent angle deduction.}
\label{aad}
\end{figure}
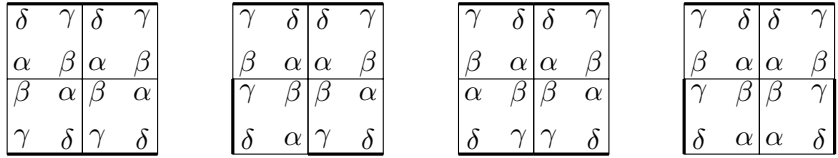

The AAD notation is introduced for pentagon in \cite{cly,wy1}. We will be quite liberal in using the notation. For example, we may also denote the first of Figure \ref{aad} by $\thin\beta\thin\alpha\thin\beta\thin\alpha\thin, \thin\alpha\thin\beta\thin\cdots,\alpha\thin\beta\cdots,\thin^{\gamma}\beta^{\alpha}\thin^{\beta}\alpha^{\delta}\thin^{\gamma}\beta^{\alpha}\thin^{\beta}\alpha^{\delta}\thin, \thin^{\delta}\alpha^{\beta}\thin^{\alpha}\beta^{\gamma}\thin^{\delta}\alpha^{\beta}\thin^{\alpha}\beta^{\gamma}\thin,\thin^{\beta}\alpha^{\delta}\thin^{\gamma}\beta^{\alpha}\thin\cdots$. We also denote  AAD segments such as $\thin\alpha\thin\beta\thin\alpha\thin, \thin^{\beta}\alpha^{\delta}\thin^{\gamma}\beta^{\alpha}\thin$ at the vertex $\alpha^2\beta^2$. 

We remark that ${\theta}^{\lambda}\thin^{\mu}\rho$ in an AAD implies a vertex $\lambda^{\theta}\thin^{\rho}\mu\cdots$. For example, the AAD $\thin^{\beta}\alpha^{\delta}\thin^{\gamma}\beta^{\alpha}\thin^{\beta}\alpha^{\delta}\thin^{\gamma}\beta^{\alpha}\thin$ of the first of Figure \ref{aad} induces vertices $\alpha\thin\beta\cdots=\alpha^{\beta}\thin^{\alpha}\beta\cdots$ and $\gamma\thin\delta\cdots$ on the boundary of the square. This is the {\em reciprocity property}.

Sometimes a vertex may have the {\em unique} AAD. For example, for a general equilateral, the vertex $\gamma^2\delta^2$ can only be arranged as $\dash\gamma\thick\gamma\dash\delta\thin\delta\dash$, and has the unique AAD $\dash^{\delta}\gamma^{\beta}\thick^{\beta}\gamma^{\delta}\dash^{\gamma}\delta^{\alpha}\thin^{\alpha}\delta^{\gamma}\dash$ (see Figure \ref{abb-ccddA}). For another example, for an almost equilateral quadrilateral, we know the possible AADs of consecutive $\alpha\alpha$ are $\thin^{\beta}\alpha^{\delta}\thin^{\beta}\alpha^{\delta}\thin$, $\thin^{\delta}\alpha^{\beta}\thin^{\beta}\alpha^{\delta}\thin$, $\thin^{\beta}\alpha^{\delta}\thin^{\delta}\alpha^{\beta}\thin$. If $\beta\thin\delta\cdots$ and $\delta\thin\delta\cdots$ are not vertices, then $\alpha\alpha$ has the unique AAD $\thin^{\delta}\alpha^{\beta}\thin^{\beta}\alpha^{\delta}\thin$. This further implies no consecutive $\alpha\alpha\alpha$. 

For more detailed discussion about AAD, see \cite[Section 2.5]{wy1}. The AAD is a convenient way of describing a small tiling picture. We may use AAD to make an argument without drawing the corresponding picture. 

Next, we introduce the concept of {\em fans}, which is used only for almost equilateral quadrilaterals. At any vertex, the $b$-edges divide all the angles  into consecutive sequences $\thick\gamma\thin\theta_1\thin\cdots\thin\theta_k\thin\gamma\thick$, $\thick\gamma\thin\theta_1\thin\cdots\thin\theta_k\thin\delta\thick$, $\thick\delta\thin\theta_1\thin\cdots\thin\theta_k\thin\delta\thick$, where $\theta_i$ are $\alpha$ or $\beta$. We call these respectively {\em $\gamma^2$-fan, $\gamma\delta$-fan}, and {\em $\delta^2$-fan}. See Figure \ref{vertex_bedge}. If a vertex has $\gamma,\delta$, then it has fans. Otherwise, it has no fan.

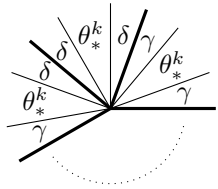
\begin{figure}[htp]
\centering
\begin{tikzpicture}[>=latex,scale=1]


\foreach \a in {20,50,90,120,160,190}
\draw
	(0,0) -- (\a:1.4);

\foreach \b in {0,70,140,210}
\draw[line width=1.2]
	(0,0) -- (\b:1.4);

\foreach \x in {10,60,200}
\node at (\x:1) {\footnotesize $\gamma$};

\foreach \y in {80,130,150}
\node at (\y:1) {\footnotesize $\delta$};

\foreach \z in {35,105,175}
\node at (\z:1) {\footnotesize $\theta_*^k$};

\draw[dotted]
	(220:1) arc (220:350:1);

\end{tikzpicture}
\caption{Fans at a vertex.}
\label{vertex_bedge}
\end{figure}

\begin{lemma}\label{fbalance}
In a tiling of the sphere by congruent non-symmetric almost equilateral quadrilaterals, a vertex has $\gamma^2$-fan if and only if a (possibly different) vertex has $\delta^2$-fan. 
\end{lemma}

\begin{proof}
If no vertex has $\gamma^2$-fan, then every vertex has only $\gamma\delta$-fans and $\delta^2$-fans, or has no fan. This implies the number of $\gamma$ at any vertex is always no more than the number of $\delta$. Then by the counting lemma (Lemma \ref{counting}), the number of $\gamma$ is the same as the number of $\delta$ at all vertices. By no $\gamma^2$-fan, this further implies every vertex has no $\delta^2$-fan. By exchanging $\gamma,\delta$, this also proves that no $\delta^2$-fan implies no $\gamma^2$-fan.
\end{proof}

\begin{lemma}\label{klem3}
In a tiling of the sphere by congruent non-symmetric almost equilateral quadrilaterals, if $\alpha^2\cdots$ is not a vertex, then a $\delta^2$-fan consists of a single $\alpha$, and a vertex has at most one $\delta^2$-fan. If we further know $\gamma>\delta$, then $\thin\alpha\thin^{\gamma}\beta^{\alpha}\thin\cdots=\alpha\beta^k$. 
\end{lemma}

\begin{proof}
By no $\alpha^2\cdots$, we get the unique AAD $\thin^{\alpha}\beta^{\gamma}\thin^{\alpha}\delta^{\gamma}\thick$ of $\thin\beta\thin\delta\thick$. Then by no $\alpha^2\cdots$ again,  we get the unique AAD $\thin^{\alpha}\beta^{\gamma}\thin\cdots\thin^{\alpha}\beta^{\gamma}\thin^{\alpha}\delta^{\gamma}\thick$ of consecutive $\beta\cdots\beta\delta$. Since a $\delta^2$-fan $\thick\delta\thin\beta\thin\cdots\thin\beta\thin\delta\thick$ without $\alpha$ has consecutive $\beta\cdots\beta\delta$ in both directions, the unique AAD implies such a fan does not exist. Therefore a $\delta^2$-fan must have $\alpha$. Then by no $\alpha^2\cdots$, a $\delta^2$-fan has a single $\alpha$, and a vertex has at most one $\delta^2$-fan.

Suppose we further know $\gamma>\delta$. Then by the parity lemma, and the angle sum \eqref{anglesum} for quadrilateral, we know $\alpha\beta\gamma\cdots$ is not a vertex. Then by no $\alpha^2\cdots$ and no $\alpha\beta\gamma\cdots$, we know $\alpha\beta\cdots$ is either $\alpha\beta^k$ or a single $\delta^2$-fan $\alpha\beta^k\delta^2$. By no $\alpha^2\cdots$, the AAD of $\alpha\beta^k\delta^2$ is $\thick^{\gamma}\delta^{\alpha}\thin\cdots\thin^{\gamma}\beta^{\alpha}\thin\alpha\thin^{\alpha}\beta^{\gamma}\thin\cdots\thin^{\alpha}\delta^{\gamma}\thick$. Since $\thin\alpha\thin^{\gamma}\beta^{\alpha}\thin$ is incompatible with this AAD, we conclude $\thin\alpha\thin^{\gamma}\beta^{\alpha}\thin\cdots=\alpha\beta^k$. 
\end{proof}

\section{Tiling by Rhombus or Non-general Triangle}
\label{rhombustiling}

We first classify tilings by congruent rhombi in the fourth of Figure \ref{quad}. Then we derive tilings by congruent equilateral or isosceles triangles in Figure \ref{triangle}.

\begin{proposition}\label{rhombus}
Tilings of the sphere by congruent rhombi, are the quadricentric subdivisions $C_{\square}P_4,C_{\square}P_6,C_{\square}P_{12}$ of Platonic solids, the earth map tiling $E_{\square}^R1$, and the flip modification $FE_{\square}^R1$.
\end{proposition}

We remark that $C_{\square}P_4=P_6$, $C_{\square}P_6=C_{\square}P_8=Q_{\square}P_4$, and $C_{\square}P_{12}=C_{\square}P_{20}$. The earth map tiling $E_{\square}^R1$ is obtained in Figure \ref{rhombus_emt1}, and the flip modification $FE_{\square}^R1$ is obtained in Figure \ref{rhombus_emt2}. 

\begin{proof}
The rhombus has two angles $\alpha,\beta$. By the angle sum \eqref{anglesum} for quadrilateral, we get
\[
\alpha+\beta=(1+\tfrac{2}{f})\pi.
\]

If $\alpha=\beta$, then the quadrilateral has four $\alpha$ satisfying $\alpha=(\frac{1}{2}+\tfrac{1}{f})\pi$. By $4\alpha>2\pi$, the only vertex is $\alpha^3$. Therefore $\alpha=\frac{2}{3}\pi$, and the tiling is the regular cube $P_6$, which is also the quadricentric subdivision $C_{\square}P_4$ of the tetrahedron. 

Next, we assume $\alpha\ne\beta$. Up to the symmetry of exchanging $\alpha,\beta$, we may further assume $\alpha<\beta$. Then we have $\alpha+\beta>\pi$ and $4\beta>2\pi$. This implies $\alpha\beta^2,\beta^3,\alpha^k,\alpha^k\beta$ are all the possible vertices. Since the number of $\alpha$ is strictly less than the number of $\beta$ in $\alpha\beta^2,\beta^3$, and the number of $\alpha$ is strictly more than the number of $\beta$ in $\alpha^k,\alpha^k\beta$, by the counting lemma, one of $\alpha\beta^2,\beta^3$ is a vertex, and one of $\alpha^k,\alpha^k\beta$ is a vertex. 

\subsubsection*{Case. $\beta^3$ is a vertex}

The angle sum of $\beta^3$ and the angle sum for quadrilateral imply
\[
\alpha=(\tfrac{1}{3}+\tfrac{2}{f})\pi,\;
\beta=\tfrac{2}{3}\pi. 
\]
We also know one of $\alpha^k,\alpha^k\beta$ is a vertex. By $\alpha<\beta$ and the angle values above, we get the following possibilities:
\begin{align*}
f=12 &\colon
\alpha=\tfrac{1}{2}\pi,\;
\beta=\tfrac{2}{3}\pi,\;
\text{AVC}
=\{\alpha^4,\beta^3\}. \\
f=18 &\colon
\alpha=\tfrac{4}{9}\pi,\;
\beta=\tfrac{2}{3}\pi,\;
\text{AVC}
=\{\alpha^3\beta,\beta^3\}. \\
f=30 &\colon
\alpha=\tfrac{2}{5}\pi,\;
\beta=\tfrac{2}{3}\pi,\;
\text{AVC}
=\{\alpha^5,\beta^3\}.
\end{align*}
For $f=12$, we get a graph with only degree $3$ and degree $4$ vertices, and any edge connects a degree $3$ vertex to a degree $4$ vertex. For $f=30$, we have only degree $3$ and degree $5$ vertices, and any edge connects a degree $3$ vertex to a degree $5$ vertex. Guided by this, in Figure \ref{rhombus_division}, we start with a degree $3$ vertex the at the center and then construct the tiling outwards. We use $\bullet$ to denote $\beta^3$, and use $\circ$ to denote $\alpha^4,\alpha^5$. These are the quadricentric subdivisions $C_{\square}P_6,C_{\square}P_{12}$ of the cube and the dodecahedron, and are the same as the ones in the last row of Figure \ref{subdivision_platonic}.

\begin{figure}[htp]
\centering
\begin{tikzpicture}[>=latex]
	

\foreach \a in {0,1,2}
{
\begin{scope}[rotate=120*\a]

\draw
	(0,0) -- (0,1.2)
	(0,-0.6) -- (0,-1.8)
	(90:0.6) -- (30:0.6) -- (-30:0.6)
	(90:1.2) -- (30:1.2) -- (-30:1.2)
	;

\fill
	(0,0) circle (0.07)
	(30:0.6) circle (0.07)
	(0,1.2) circle (0.07)
	(0,-1.8) -- ++(30:0.1) arc (30:150:0.1);
	
\end{scope}
}

\foreach \a in {0,1,2}
\filldraw[fill=white, rotate=120*\a]
	(0,0.6) circle (0.07)
	(30:1.2) circle (0.07);
	

\foreach \a in {0,1,2}
{
\begin{scope}[xshift=4.5cm, rotate=120*\a]

\draw
	(0,0) -- (0,0.5) 
	(120:0.866) -- (30:0.5) -- (30:1) -- (90:1) -- (150:1)
	(60:0.866) -- (150:0.5)
	(0:1.1547) -- (60:1.1547) -- (120:1.1547)
	(90:1) -- (90:1.6) -- (30:1.6) -- (-30:1.6)
	(60:1.1547) -- (30:1.6) -- (0:1.1547)
	(30:1.6) -- (30:2)
	;

\fill
	(0,0) circle (0.07)
	(30:0.5) circle (0.07)
	(0:0.866) circle (0.07)
	(60:0.866) circle (0.07)
	(0:1.1547) circle (0.07)
	(60:1.1547) circle (0.07)
	(90:1.6) circle (0.07)	
	(-90:2) -- ++(30:0.1) arc (30:150:0.1);
		
\end{scope}
}

\foreach \a in {0,1,2}
\filldraw[xshift=4.5cm, fill=white, rotate=120*\a]
	(0,0.5) circle (0.07)
	(90:1) circle (0.07)
	(30:1) circle (0.07)
	(30:1.6) circle (0.07) ;

\end{tikzpicture}
\caption{Proposition \ref{rhombus}: Tilings for $\{\alpha^4\circ,\beta^3\bullet\}$ and $\{\alpha^5\circ,\beta^3\bullet\}$.}
\label{rhombus_division}
\end{figure}
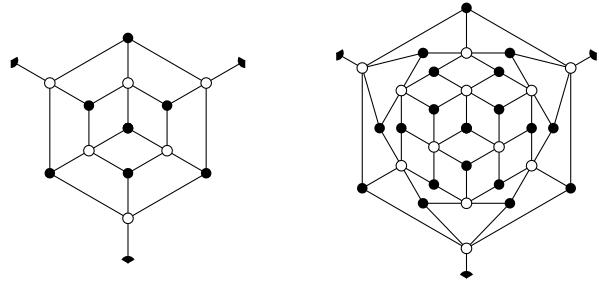

For $f=18$, we argue that the tiling does not exist. In the first of Figure \ref{rhombus_emt1}, we consider four tiles $T_1,T_2,T_3,T_4$ around a vertex $\alpha^3\beta$. Then $T_1,T_2$ share $\beta^2\cdots=\beta^3$ and $T_2,T_3$ share $\beta^2\cdots=\beta^3$. This determines $T_5,T_6$. Then $T_2,T_5,T_6$ share $\alpha^3\cdots=\alpha^3\beta$. This determines $T_7$. Then $T_5,T_7$ share $\alpha\beta\cdots=\alpha^3\beta$. This determines $T_8$. Then $T_1,T_5,T_8$ share $\alpha^2\beta\cdots=\alpha^3\beta$. This determines $T_9$. Then we find $T_1,T_4,T_9$ share $\alpha\beta^2\cdots$, a contradiction. 

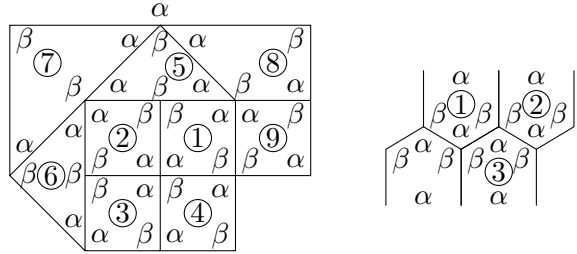
\begin{figure}[htp]
\centering
\begin{tikzpicture}[>=latex]

\draw
	(-1,-1) rectangle (1,1)
	(0,-1) -- (0,1)
	(-1,-1) -- (-2,0) -- (0,2) -- (1,1)
	(-2,0) -- (-2,2) -- (2,2) -- (2,0) -- (-1,0)
	(1,1) -- (2,1);

\node at (0.2,0.8) {\small $\beta$};
\node at (0.8,0.2) {\small $\beta$};
\node at (0.2,0.2) {\small $\alpha$};
\node at (0.8,0.8) {\small $\alpha$};

\node at (-0.2,0.8) {\small $\beta$};
\node at (-0.8,0.2) {\small $\beta$};
\node at (-0.2,0.2) {\small $\alpha$};
\node at (-0.8,0.8) {\small $\alpha$};

\node at (-0.2,-0.8) {\small $\beta$};
\node at (-0.8,-0.2) {\small $\beta$};
\node at (-0.2,-0.2) {\small $\alpha$};
\node at (-0.8,-0.8) {\small $\alpha$};

\node at (0.2,-0.2) {\small $\beta$};
\node at (0.8,-0.8) {\small $\beta$};
\node at (0.2,-0.8) {\small $\alpha$};
\node at (0.8,-0.2) {\small $\alpha$};

\node at (-1.15,0) {\small $\beta$};
\node at (-1.75,0) {\small $\beta$};
\node at (-1.15,0.6) {\small $\alpha$};
\node at (-1.15,-0.6) {\small $\alpha$};

\node at (0,1.15) {\small $\beta$};
\node at (0, 1.75) {\small $\beta$};
\node at (-0.55, 1.2) {\small $\alpha$};
\node at (0.55, 1.2) {\small $\alpha$};

\node at (-1.15,1.15) {\small $\beta$};
\node at (-1.8,1.8) {\small $\beta$};
\node at (-0.4,1.8) {\small $\alpha$};
\node at (-1.8,0.4) {\small $\alpha$};

\node at (1.1,1.2) {\small $\beta$};
\node at (1.8,1.8) {\small $\beta$};
\node at (0.5,1.8) {\small $\alpha$};
\node at (1.8,1.2) {\small $\alpha$};

\node at (0,2.2) {\small $\alpha$};

\node at (1.2,0.2) {\small $\beta$};
\node at (1.8,0.8) {\small $\beta$};
\node at (1.2,0.8) {\small $\alpha$};
\node at (1.8,0.2) {\small $\alpha$};

\node[inner sep=0.5,draw,shape=circle] at (0.5,0.5) {\small $1$};
\node[inner sep=0.5,draw,shape=circle] at (-0.5,0.5) {\small $2$};
\node[inner sep=0.5,draw,shape=circle] at (-0.5,-0.5) {\small $3$};
\node[inner sep=0.5,draw,shape=circle] at (0.5,-0.5) {\small $4$};
\node[inner sep=0.5,draw,shape=circle] at (0.25,1.45) {\small $5$};
\node[inner sep=0.5,draw,shape=circle] at (-1.45,0) {\small $6$};
\node[inner sep=0.5,draw,shape=circle] at (-1.5,1.5) {\small $7$};
\node[inner sep=0.5,draw,shape=circle] at (1.5,1.5) {\small $8$};
\node[inner sep=0.5,draw,shape=circle] at (1.5,0.5) {\small $9$};

\begin{scope}[shift={(3.5cm,0.5cm)}]

\foreach \a in {0,1}
{
\begin{scope}[xshift=\a cm]

\draw
	(0,0.9) -- (0,0.15) -- (-0.5,-0.15) -- (-0.5,-0.9)
	(1,0.9) -- (1,0.15) -- (0.5,-0.15) -- (0.5,-0.9)
	(0,0.15) -- (0.5,-0.15);

\node at (0.5,0.8) {\small $\alpha$};
\node at (0.2,0.3) {\small $\beta$};
\node at (0.5,0.1) {\small $\alpha$};
\node at (0.8,0.3) {\small $\beta$};

\node at (0,-0.8) {\small $\alpha$};
\node at (-0.3,-0.3) {\small $\beta$};
\node at (0,-0.1) {\small $\alpha$};
\node at (0.3,-0.3) {\small $\beta$};

\end{scope}
}

\node[inner sep=0.5,draw,shape=circle] at (0.5,0.45) {\small $1$};
\node[inner sep=0.5,draw,shape=circle] at (1.5,0.45) {\small $2$};
\node[inner sep=0.5,draw,shape=circle] at (1,-0.45) {\small $3$};

\end{scope}

\end{tikzpicture}
\caption{Proposition \ref{rhombus}: No tiling for $\{\alpha^3\beta,\beta^3\}$, and earth map tiling.}
\label{rhombus_emt1}
\end{figure}

The existence of quadricentric subdivision tilings follows from the existence of regular Platonic solids.

\subsubsection*{Case. $\alpha\beta^2$ is a vertex}

The angle sum of $\alpha\beta^2$ and the angle sum for quadrilateral imply
\[
\alpha=\tfrac{4}{f}\pi,\; 
\beta=(1-\tfrac{2}{f})\pi.
\]
Then we get the list of possible vertices
\begin{equation}\label{triavc7}
\text{AVC}
=\{\alpha\beta^2,\alpha^{\frac{f}{2}},\alpha^{\frac{f+2}{4}}\beta\}.
\end{equation}

Given consecutive $\alpha\alpha$ at a vertex, we have $T_1,T_2$ in the second of Figure \ref{rhombus_emt1}. The two tiles share $\beta^2\cdots=\alpha\beta^2$. This determines $T_3$. If $\alpha^{\frac{f}{2}}$ is a vertex, then we may apply the argument to the $\frac{f}{2}$ consecutive $\alpha\alpha$ in $\alpha^{\frac{f}{2}}$, and get the earth map tiling $E_{\square}^R1$.

If $\alpha^{\frac{f}{2}}$ is not a vertex, then by the counting lemma, we know $\alpha^{\frac{f+2}{4}}\beta$ is a vertex. We denote the vertex as $\alpha^{q+1}\beta$, where $f=4q+2$ and $\beta=q\alpha$. Applying the argument in the second of Figure \ref{rhombus_emt1} to the $q$ consecutive $\alpha\alpha$ in the $\alpha^{q+1}$ part of $\alpha^{q+1}\beta$, we get a partial earth map tiling consisting of tiles on the left and right of the shaded edges in Figure \ref{rhombus_emt2}, and containing $T_1,T_4,T_5,T_6$. We also have $T_5$ containing $\beta$ in $\alpha^{q+1}\beta$.

Each tile has two $\alpha$. In the subsequent argument, if we wish to specify which of the two, then we denote one by $\alpha$ and the other by $\dot{\alpha}$. We do the same for $\beta$ and $\dot{\beta}$.

We have $\dot{\alpha}_1\dot{\beta}_5\cdots=\alpha\beta^2,\alpha^{q+1}\beta$. The first of Figure \ref{rhombus_emt2} is the case $\dot{\alpha}_1\dot{\beta}_5\cdots=\alpha\beta^2$. This implies $\dot{\beta}_2\cdots=\alpha^2\beta\cdots=\alpha^{q+1}\beta$. Then the $\alpha^{q+1}$ part of this vertex determines a partial earth map tiling, consisting of $P,T_1,T_5,T_6$. Then $\dot{\beta}_3\cdots=\alpha^q\beta\cdots=\alpha^{q+1}\beta$ determines $T_7$. We note that the part consisting of $P,T_5,T_6,T_7$ is the flip in Figure \ref{flip3}. Therefore the tiling is $FE_{\square}^R1$.

\begin{figure}[htp]
\centering
\begin{tikzpicture}[>=latex]

\foreach \a in {0,6}
\draw[xshift=\a cm, gray, line width=3]
	(0.5,-1.6) -- (0.5,1.6) 
	(2.5,1.6) -- (2.5,0.8) -- (3.5,-0.8) -- (3.5,-1.6); 
	
\foreach \a in {0,3,4,6,10}
\foreach \b in {1,-1}
\draw[xshift=\a cm, scale=\b]
	(0.5,-1.6) -- (0.5,1.6)
	(-0.5,0.8) -- (0.5,-0.8);	

\foreach \a/\b in {3/-1,4/1,4/-1,6/-1,7/-1,9/1,10/1,10/-1}
{
\begin{scope}[xshift=\a cm, scale=\b]
	
\node at (-0.3,0.9) {\small $\beta$};
\node at (0.3,0.8) {\small $\beta$};
\node at (0.3,-0.2) {\small $\alpha$};
\node at (0,1.5) {\small $\alpha$};

\end{scope}
}

\foreach \b in {0,6}
{
\begin{scope}[xshift=\b cm]

\node at (-0.3,0.9) {\small $\beta$};
\node at (0.3,0.8) {\small $\dot{\beta}$};
\node at (0.3,-0.2) {\small $\alpha$};
\node at (0,1.5) {\small $\alpha$};

\end{scope}
}

\node at (0.3,-0.9) {\small $\dot{\beta}$};
\node at (-0.3,-0.8) {\small $\beta$};
\node at (-0.3,0.2) {\small $\alpha$};
\node at (0,-1.5) {\small $\alpha$};

\begin{scope}[xshift=3 cm]

\node at (-0.3,0.9) {\small $\dot{\beta}$};
\node at (0.3,0.8) {\small $\beta$};
\node at (0.3,-0.2) {\small $\alpha$};
\node at (0,1.5) {\small $\alpha$};

\end{scope}


\begin{scope}[xshift=1.5cm]
	
\foreach \b in {0,1}
{
\begin{scope}[xshift=6*\b cm]

\node at (0,1) {\small $\beta$};
\node at (0,1.5) {\small $\beta$};	
\node at (0.8,1) {\small $\alpha$};
\node at (-0.8,1) {\small $\dot{\alpha}$};

\end{scope}
}

\foreach \b in {0,1}
{
\begin{scope}[xshift=7*\b cm]

\node at (0,-1) {\small $\beta$};
\node at (0,-1.5) {\small $\beta$};	
\node at (0.8,-1) {\small $\alpha$};
\node at (-0.8,-1) {\small $\alpha$};

\end{scope}
}

\draw
	(-1,0.8) -- (1,0.8)
	(-1,-0.8) -- (1,-0.8)
	(0,0.8) -- (0,0.5)
	(0,-0.8) -- (0,-0.5);
	
\node at (0.2,0.6) {\small $\beta$};
\node at (-0.2,0.6) {\small $\alpha$};
\node at (-0.8,0.6) {\small $\beta$};

\node at (-0.2,-0.6) {\small $\beta$};
\node at (0.2,-0.6) {\small $\alpha$};
\node at (0.8,-0.6) {\small $\beta$};

\node[rotate=90] at (0.8,0.4) {\small $\alpha^{q-1}$};
\node[rotate=90] at (-0.8,-0.4) {\small $\alpha^{q-1}$};

\node at (0,0) {\small $P$};

\end{scope}

\node[draw,shape=circle, inner sep=0.5] at (0,1.2) {\small $1$};
\node[draw,shape=circle, inner sep=0.5] at (0,-1.2) {\small $2$};
\node[draw,shape=circle, inner sep=0.5] at (3,1.2) {\small $3$};
\node[draw,shape=circle, inner sep=0.5] at (4,-1.2) {\small $4$};
\node[draw,shape=circle, inner sep=0.5] at (1.9,1.2) {\small $5$};
\node[draw,shape=circle, inner sep=0.5] at (1.1,-1.2) {\small $6$};
\node[draw,shape=circle, inner sep=0.5] at (3,-1.2) {\small $7$};


\begin{scope}[xshift=8cm]

\foreach \b in {1,-1}
{
\begin{scope}[scale=\b]

\draw
	(-1.5,0.8) -- (0.5,0.8)
	(0.5,1.6) -- (0.5,0.8) -- (1.5,-0.8)
	(-0.5,0.8) -- ++(0.2,-0.32);

\end{scope}
}

\node at (0.35,0.6) {\small $\beta$};
\node at (-0.2,0.6) {\small $\alpha$};
\node at (-0.6,0.6) {\small $\beta$};

\node at (-0.35,-0.6) {\small $\beta$};
\node at (0.2,-0.6) {\small $\alpha$};
\node at (0.6,-0.6) {\small $\beta$};

\node[rotate=-60] at (-1,0.45) {\small $\alpha^{q-1}$};
\node[rotate=-60] at (1,-0.45) {\small $\alpha^{q-1}$};

\node at (0,0) {\small $P$};

\end{scope}

\begin{scope}[xshift=6cm]
	
\node[draw,shape=circle, inner sep=0.5] at (0,1.2) {\small $1$};
\node[draw,shape=circle, inner sep=0.5] at (0,-1.2) {\small $2$};
\node[draw,shape=circle, inner sep=0.5] at (3,1.2) {\small $3$};
\node[draw,shape=circle, inner sep=0.5] at (4,-1.2) {\small $4$};
\node[draw,shape=circle, inner sep=0.5] at (1.9,1.2) {\small $5$};
\node[draw,shape=circle, inner sep=0.5] at (1,-1.2) {\small $6$};
\node[draw,shape=circle, inner sep=0.5] at (2.1,-1.2) {\small $7$};

\end{scope}

\end{tikzpicture}
\caption{Proposition \ref{rhombus}: Tiling for $\{\alpha\beta^2,\alpha^{q+1}\beta\}$.}
\label{rhombus_emt2}
\end{figure}
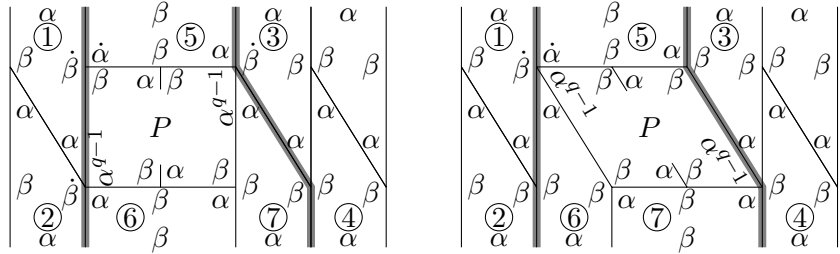

The second of Figure \ref{rhombus_emt2} is the case $\dot{\alpha}_1\dot{\beta}_5\cdots=\cdots=\alpha^{q+1}\beta$. Again the $\alpha^{q+1}$ part of this vertex determines a partial earth map tiling, consisting of $P,T_5,T_6,T_7$. By the horizontal flip, the tiling actually becomes the first one. Therefore we get $FE_{\square}^R1$ again.

Finally, the rhombus in the earth map tiling is the union of four right triangles with $\frac{1}{2}\alpha=\frac{2}{f}\pi$ and $\frac{1}{2}\beta=(\frac{1}{2}-\frac{1}{f})\pi$. By $\frac{1}{2}\alpha,\frac{1}{2}\beta<\frac{1}{2}\pi$ and $\frac{1}{2}\alpha+\frac{1}{2}\beta>\frac{1}{2}\pi$, the right triangle exists. Therefore the rhombus exists. 
\end{proof}

Next we classify tilings by congruent equilateral triangles or isosceles triangles. 

Tilings by congruent equilateral triangles (third of Figure \ref{triangle}) must be Platonic solids. They are the regular tetrahedron $P_4$, the regular octahedron $P_8$, and the regular icosahedron $P_{20}$.

Tilings by congruent isosceles triangles (second of Figure \ref{triangle}) are closely related to the rhombus tilings. Each isosceles triangle has a companion sharing the $b$-edge. The union of the triangle with its companion is a rhombus, as in Figure \ref{isosceles1}. Therefore isosceles triangular tilings give rhombus tilings. If all vertices of the rhombus tilings have degree $\ge 3$, then we may apply Proposition \ref{rhombus}, and conclude the isosceles triangular tilings are the simple triangular subdivisions of the tilings in the proposition. 

\begin{figure}[htp]
\centering
\begin{tikzpicture}[>=latex]

\foreach \a in {0,1}
\foreach \b in {-1,1}
\foreach \c in {0,1}
\draw[xshift=2.6*\c cm, yshift=0.4*\b cm, yscale=\b, rotate=120*\a]
	(-30:0.8) -- (90:0.8);

\draw[line width=1.2]
	(-0.693,0) -- (0.693,0);

\draw[->, very thick]
	(1,0) -- ++(0.6,0);

\node at (-0.5,0.65) {\small $a$};
\node at (0.5,0.65) {\small $a$};
\node at (0,0.2) {\small $b$};
\node at (-0.5,-0.65) {\small $a$};
\node at (0.5,-0.65) {\small $a$};

\node at (0,0.85) {\small $\alpha$};
\node at (-0.35,0.2) {\small $\beta$};
\node at (0.35,0.2) {\small $\beta$};

\node at (0,-0.85) {\small $\alpha$};
\node at (-0.35,-0.2) {\small $\beta$};
\node at (0.35,-0.2) {\small $\beta$};

\begin{scope}[xshift=2.6 cm]

\node at (0,0.85) {\small $\alpha$};
\node at (0,-0.85) {\small $\alpha$};
\node at (-0.35,0) {\small $2\beta$};
\node at (0.35,0) {\small $2\beta$};

\end{scope}

\end{tikzpicture}
\caption{Isosceles triangular tiling to rhombus tiling.}
\label{isosceles1}
\end{figure}
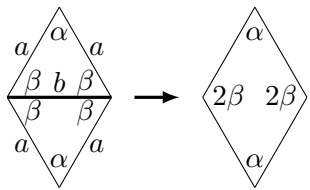

First, we have the simple triangular subdivisions of the quadricentric subdivisions $C_{\square}P_i$ of the regular Platonic solids. We know $C_{\square}P_4$ is the regular cube, and the simple triangular subdivisions of the regular cube are the seven tilings $S_{\triangle}P_6$ in Figure \ref{simple_subdivision_cube}. The simple triangular subdivisions of $C_{\square}P_6=C_{\square}P_8,C_{\square}P_{12}=C_{\square}P_{20}$ are the triangular subdivisions $T_{\triangle}P_6,T_{\triangle}P_8,T_{\triangle}P_{12},T_{\triangle}P_{20}$. 

Second, the simple triangular subdivisions of $E_{\square}^R1$ are $E_{\triangle}^J1$ and $E_{\triangle}4$. The simple triangular subdivisions of the flip modification $FE_{\square}^R1$ are the flip modifications $FE_{\triangle}^J1$ and $FE_{\triangle}4$.

It remains to consider the case the rhombus tiling associated to the isosceles triangular tiling has a degree $2$ vertex. Each angle in the rhombus tiling corresponds to $\alpha$ or $\beta^2=\beta\thick\beta$ in the isosceles triangular tiling. Therefore the degree $2$ vertex is a combination of two from $\alpha,\beta^2$. Moreover, the combination should have degree $\ge 3$ in terms of $\alpha,\beta$. Therefore the degree $2$ vertex in the rhombus tiling corresponds to $\alpha\beta^2$ or $\beta^4$ in the triangular tiling. 

The angle sum of $\alpha\beta^2$ and the angle sum \eqref{3anglesum} for triangle imply $f=4$. Then it is easy to show that the triangular tiling is the tetrahedron tiling $P_4$ in the first of Figure \ref{isosceles2}.

\begin{figure}[htp]
\centering
\begin{tikzpicture}[>=latex]

\begin{scope}[shift={(-2 cm, -0.3cm)}]

\draw
	(-30:0.7) -- (90:0.7) -- (210:0.7)
	(-30:0.7) -- (-30:1.3)
	(210:0.7) -- (210:1.3);

\draw[line width=1.2]
	(-30:0.7) -- (210:0.7)
	(90:0.7) -- (90:1.3);

\end{scope}


\foreach \b in {-1,1}
{
\begin{scope}[xshift=5 cm, scale=\b] 

\draw[gray, line width=3]
	(1,-1) -- (1,1);
	
\draw
	(-1,0) -- (0,0.7) -- (1,0);

\draw[line width=1.2]
	(0,0.4) -- (0,1.2);

\node at (-0.5,1) {\small $\beta$};
\node at (0.5,1) {\small $\beta$};

\node at (-0.2,0.8) {\small $\beta$};
\node at (0.2,0.8) {\small $\beta$};

\node at (0.8,0.35) {\small $\alpha$};
\node at (0.8,-0.35) {\small $\alpha$};

\node at (-0.45,0.07) {\small $\alpha^{q-2}$};

\end{scope}
}

\foreach \a in {0,1,3,6}
\foreach \b in {-1,1}
{
\begin{scope}[xshift=\a cm, yscale=\b]
	
\draw
	(0,0) -- (0,1)
	(1,0) -- (1,1);

\draw[line width=1.2]
	(0,0) -- (1,0);
	
\node at (0.5,0.9) {\small $\alpha$};
\node at (0.2,0.25) {\small $\beta$};
\node at (0.8,0.25) {\small $\beta$};

\end{scope}
}
	
\node[inner sep=0.5,draw,shape=circle] at (3.5,0.5) {\small $1$};
\node[inner sep=0.5,draw,shape=circle] at (4.45,0.65) {\small $2$};
\node[inner sep=0.5,draw,shape=circle] at (5.55,0.65) {\small $3$};
\node[inner sep=0.5,draw,shape=circle] at (6.5,0.5) {\small $4$};
\node[inner sep=0.5,draw,shape=circle] at (3.5,-0.5) {\small $5$};
\node[inner sep=0.5,draw,shape=circle] at (6.5,-0.5) {\small $6$};

\end{tikzpicture}
\caption{Isosceles tetrahedron tiling, and isosceles earth map tiling. }
\label{isosceles2}
\end{figure}
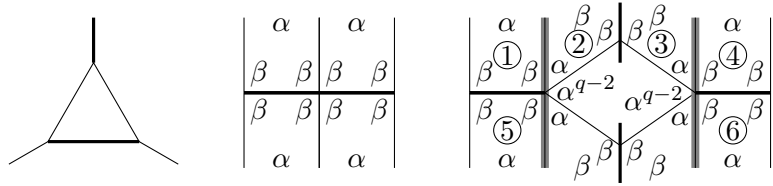

The angle sum of $\beta^4$ and the angle sum for triangle imply 
\[
\alpha=\tfrac{4}{f}\pi,\;
\beta=\tfrac{1}{2}\pi.
\]
By the parity lemma (Lemma \ref{parity}), a vertex has even number of $\beta$. Then we find the list of possible vertices
\[
\text{AVC}=\{\beta^4,\alpha^{\frac{f}{2}},\alpha^{\frac{f}{4}}\beta^2\}.
\]
The triangle exists because it is half of the $2$-gon of angle $\alpha$.

A tile determines a companion tile across the $b$-edge as in Figure \ref{isosceles1}. Then consecutive $\alpha^k=\alpha\cdots\alpha$ determines $k$ companion tiles, which give the other end $\alpha^k$ of a partial earth map tiling. See the second of Figure \ref{isosceles2}. In particular, if $\alpha^{\frac{f}{2}}$ is a vertex, then we get the earth map tiling $E_{\triangle}^I1$.

If $\alpha^{\frac{f}{2}}$ is not a vertex, then $\alpha^q\beta^2=\cdots\thin\alpha\thin\beta\thick\beta\thin\alpha\thin\cdots$, where $f=4q$, is a vertex. We get $T_1,T_2,T_3,T_4$ in the third of Figure \ref{isosceles2}. Moreover, the $\alpha^q$ part of the vertex gives a partial earth map tiling consisting of tiles on the left and right of the shaded edges, and containing $T_1,T_4,T_5,T_6$. Then $\alpha_2\cdots=\alpha_3\cdots=\alpha\beta^2\cdots=\alpha^q\beta^2$. Then the $\alpha^q$ parts of the two vertices give another partial earth map tiling, which fills the middle of the second of Figure \ref{isosceles2}. This middle part is the flip in the second of Figure \ref{flip4} (with $\beta=\gamma$). Therefore the tiling is $FE_{\triangle}^I1$.

\section{Tiling by General Triangle or Kite}
\label{kitetiling}

We first classify tilings by congruent general triangles in the first of Figure \ref{triangle}. Then we derive tilings by congruent kites in the second of Figure \ref{quad}.

By the parity lemma (the first part of Lemma \ref{parity}), the numbers of $\alpha,\beta,\gamma$ at any vertex have the same parity. Depending on whether the parity is even or odd, we call the vertex {\em even} or {\em odd}. It is easy to see that an even (odd) vertex is equivalent to even degree (odd degree) vertex. 

By the parity lemma, a degree $3$ vertex is $\alpha\beta\gamma$. By the angle sum of $\alpha\beta\gamma$ and the angle sum \eqref{3anglesum} for triangle, we get $f=4$. Then by \eqref{trivcountf4}, we know all vertices have degree $3$. Therefore $\alpha\beta\gamma$ is the only vertex. Then it is easy to conclude that the tiling is the deformed tetrahedron $P_4$. 

In the subsequent discussion, therefore, we may assume there is no degree $3$ vertex. Then by \eqref{trivcountf3}, we get $f\ge 8$. Without loss of generality, we may also assume $\alpha<\beta<\gamma$. Then by the angle sum for triangle, we get $R(\gamma^4)<R(\beta^2\gamma^2)<2\alpha$. Then by $\alpha<\beta<\gamma$ and the parity lemma, this implies an even vertex $\gamma\cdots=\beta^2\gamma^2,\gamma^4,\alpha^k\gamma^2$. On the other hand, an odd vertex $\gamma^3\cdots$ is $\alpha\beta\gamma^3\cdots$. By $\alpha<\beta<\gamma$ and the angle sum for triangle, we get $R(\alpha\beta\gamma^3)<\alpha<\beta,\gamma$. Therefore an odd vertex $\gamma^3\cdots=\alpha\beta\gamma^3$. In summary, we get the list of all possible vertices
\begin{equation}\label{triavc1}
\text{AVC}
=\{\alpha\beta\gamma^3,\beta^2\gamma^2,\gamma^4,\alpha^k\gamma^2,\alpha^k\beta^l,\alpha^k\beta^l\gamma\}.
\end{equation}

\begin{proposition}\label{2a2b}
Tilings of the sphere by congruent general triangles, such that $\alpha<\beta<\gamma$ and $\beta^2\gamma^2$ is a vertex, are the earth map tiling $E_{\triangle}1$ and its flip modifications $FE_{\triangle}1,F'E_{\triangle}1,F''E_{\triangle}1$.
\end{proposition}

The earth map tiling $E_{\triangle}1$ has the timezone given by $T_1,T_2,T_3,T_4$ in Figure \ref{3angleB}. The flip modification $FE_{\triangle}1$ is obtained in Figure \ref{3angleD}. The further flip modification $F'E_{\triangle}1$ is obtained in Figure \ref{3angleC}.

\begin{proof}
By $\beta^2\gamma^2$ and $\alpha<\beta<\gamma$, we know $\alpha\beta\gamma^3,\gamma^4,\alpha^2\gamma^2$ are not vertices. Therefore the general AVC \eqref{triavc1} is reduced to
\[
\text{AVC}
=\{\beta^2\gamma^2,\alpha^k\gamma^2(k\ge 4),\alpha^k\beta^l,\alpha^k\beta^l\gamma\}.
\]
This implies $\gamma^2\cdots$ is $\beta^2\gamma^2=\gamma\thick\gamma\cdots$ or $\alpha^k\gamma^2=\gamma\dash\gamma\cdots$. Therefore we get $\gamma\thick\gamma\cdots=\beta^2\gamma^2$, and $\gamma\dash\gamma\cdots=\alpha^k\gamma^2(k\ge 4)$.

A vertex $\thin\beta\dash\beta\thin\alpha\thick\cdots$ determines $T_1,T_2,T_3$ in the first of Figure \ref{3angleA}. Then $\gamma_1\dash\gamma_2\cdots=\alpha^k\gamma^2$ determines $T_4$. Then $\alpha_2\beta_3\gamma_4\cdots=\alpha^k\beta^l\gamma=\dash\gamma_4\thick\alpha_2\thin\beta_3\dash\beta\thin\cdots$ determines $T_5$. Then $\gamma_3\dash\gamma_5\cdots=\alpha^k\gamma^2$ determines $T_6$. This proves $\thin\beta\dash\beta\thin\alpha\thick\cdots=\thin\beta\dash\beta\thin\alpha\thick\gamma\dash\cdots$. Then $\thick\alpha\thin\beta\dash\beta\thin\alpha\thick\cdots=\dash\gamma\thick\alpha\thin\beta\dash\beta\thin\alpha\thick\gamma\dash\cdots=\alpha^2\beta^2\gamma^2\cdots$, contradicting the angle sum for triangle. Therefore $\thick\alpha\thin\beta\dash\beta\thin\alpha\thick\cdots$ is not a vertex. 

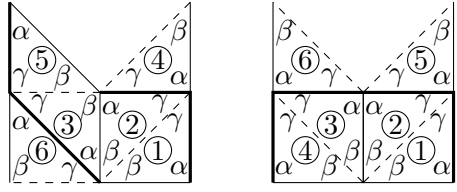
\begin{figure}[htp]
\centering
\begin{tikzpicture}[>=latex]


\foreach \a in {-1,1}
{
\begin{scope}[xscale=\a]

\draw[dashed]
	(0,-1.2) -- (1.2,0)
	(0,0) -- (1.2,1.2);

\draw[line width=1.2]
	(0,0) -- (1.2,0) -- (1.2,-1.2);

\draw
	(0,0) -- (0,-1.2) -- (1.2,-1.2)
	(1.2,0) -- (1.2,1.2);

\node at (0.8,-0.15) {\small $\gamma$};
\node at (0.15,-0.8) {\small $\beta$};
\node at (0.15,-0.2) {\small $\alpha$};

\node at (1,-0.4) {\small $\gamma$};
\node at (0.45,-1) {\small $\beta$};
\node at (1.05,-1) {\small $\alpha$};

\node at (1.05,0.8) {\small $\beta$};
\node at (0.45,0.2) {\small $\gamma$};
\node at (1.05,0.2) {\small $\alpha$};

\end{scope}
}
	
\node[inner sep=0.5,draw,shape=circle] at (0.77,-0.77) {\small $1$};
\node[inner sep=0.5,draw,shape=circle] at (0.43,-0.43) {\small $2$};
\node[inner sep=0.5,draw,shape=circle] at (-0.43,-0.43) {\small $3$};
\node[inner sep=0.5,draw,shape=circle] at (-0.77,-0.77) {\small $4$};
\node[inner sep=0.5,draw,shape=circle] at (0.77,0.43) {\small $5$};
\node[inner sep=0.5,draw,shape=circle] at (-0.77,0.43) {\small $6$};


\begin{scope}[xshift=-3.5cm, xscale=-1]

\draw[dashed]
	(-1.2,1.2) -- (0,0) -- (1.2,0)
	(-1.2,0) -- (0,-1.2) -- (1.2,-1.2);
	
\draw[line width=1.2]
	(-1.2,-1.2) -- (-1.2,0) -- (0,0) 
	(1.2,1.2) -- (1.2,0) -- (0,-1.2);

\draw
	(1.2,1.2) -- (0,0) -- (0,-1.2) -- (-1.2,-1.2)
	(-1.2,0) -- (-1.2,1.2)
	(1.2,0) -- (1.2,-1.2);

\node at (0.8,-0.15) {\small $\gamma$};
\node at (0.15,-0.2) {\small $\beta$};
\node at (0.15,-0.8) {\small $\alpha$};

\node at (-0.8,-0.15) {\small $\gamma$};
\node at (-0.15,-0.8) {\small $\beta$};
\node at (-0.15,-0.2) {\small $\alpha$};

\node at (-1,-0.4) {\small $\gamma$};
\node at (-0.45,-1) {\small $\beta$};
\node at (-1.05,-1) {\small $\alpha$};

\node at (1.05,0.2) {\small $\gamma$};
\node at (0.5,0.2) {\small $\beta$};
\node at (1.05,0.8) {\small $\alpha$};

\node at (0.4,-1.05) {\small $\gamma$};
\node at (1.05,-1) {\small $\beta$};
\node at (1.05,-0.4) {\small $\alpha$};

\node at (-1.05,0.8) {\small $\beta$};
\node at (-0.45,0.2) {\small $\gamma$};
\node at (-1.05,0.2) {\small $\alpha$};

\node[inner sep=0.5,draw,shape=circle] at (-0.77,-0.77) {\small $1$};
\node[inner sep=0.5,draw,shape=circle] at (-0.43,-0.43) {\small $2$};
\node[inner sep=0.5,draw,shape=circle] at (0.43,-0.43) {\small $3$};
\node[inner sep=0.5,draw,shape=circle] at (-0.77,0.43) {\small $4$};
\node[inner sep=0.5,draw,shape=circle] at (0.77,0.43) {\small $5$};
\node[inner sep=0.5,draw,shape=circle] at (0.77,-0.77) {\small $6$};

\end{scope}

\end{tikzpicture}
\caption{Proposition \ref{2a2b}: $\thin\beta\dash\beta\thin\alpha\thick$ and $\thin\beta\dash\beta\thin\beta\dash\beta\thin$.}
\label{3angleA}
\end{figure}

A vertex $\thin\beta\dash\beta\thin\beta\dash\beta\thin\cdots$ determines $T_1,T_2,T_3,T_4$ in the second of Figure \ref{3angleA}. Then $\gamma_1\dash\gamma_2\cdots=\gamma_3\dash\gamma_4\cdots=\alpha^k\gamma^2$ determines $T_5,T_6$. Then $\dash\gamma_5\thick\alpha_2\thin\alpha_3\thick\gamma_6\dash\cdots=\alpha^k\gamma^2$. By $\thick\alpha\thin$, the remainder of $\dash\gamma_5\thick\alpha_2\thin\alpha_3\thick\gamma_6\dash\cdots$ has no $\alpha$, contradicting $k\ge 4$. Therefore $\thin\beta\dash\beta\thin\beta\dash\beta\thin\cdots$ is not a vertex. This implies $\beta^l$ is not a vertex.

By the edge consideration, the number $l$ of $\beta$ in a sequence $\thick\alpha\thin\beta^l\thin\alpha\thick=\thick\alpha\thin\beta\cdots\beta\thin\alpha\thick$ is even. Then by no $\thick\alpha\thin\beta\dash\beta\thin\alpha\thick\cdots,\thin\beta\dash\beta\thin\beta\dash\beta\thin\cdots$, we get $l=0$, and the sequence must be $\thick\alpha\thin\alpha\thick$. 

By $\thick\alpha\thin\beta^l\thin\alpha\thick=\thick\alpha\thin\alpha\thick$, no $\thin\beta\dash\beta\thin\beta\dash\beta\thin\cdots$, and the edge consideration, we get $\alpha^k\beta^l=\alpha^k$, and $\alpha^k\beta^l\gamma$ is a combination of one of $\thick\alpha\thin\beta\dash\gamma\thick,\thick\alpha\thin\beta\dash\beta\thin\beta\dash\gamma\thick$, and copies of $\thick\alpha\thin\alpha\thick$. By $\thin\beta\dash\beta\thin\alpha\thick\cdots=\thin\beta\dash\beta\thin\alpha\thick\gamma\dash\cdots$, we get $\thick\alpha\thin\beta\dash\beta\thin\beta\dash\gamma\thick\cdots=\dash\gamma\thick\alpha\thin\beta\dash\beta\thin\beta\dash\gamma\thick\cdots=\alpha^k\beta^l\gamma$. Since the vertex has only one $\gamma$, the vertex is $\thick\alpha\thin\beta\dash\beta\thin\beta\dash\gamma\thick=\alpha\beta^3\gamma$. Therefore we get the updated list of possible vertices
\[
\text{AVC}
=\{\alpha\beta^3\gamma,\beta^2\gamma^2,\alpha^k,\alpha^k\beta\gamma,\alpha^k\gamma^2(k\ge 4)\}.
\]
This implies $\beta\dash\beta\cdots=\alpha\beta^3\gamma$ and $\gamma\dash\gamma\cdots=\alpha^k\gamma^2$.

Given $\thick\alpha\thin\alpha\thick$ at a vertex, we get two triangles sharing a vertex $\beta\thin\beta\cdots=\beta^2\gamma^2,\alpha\beta^3\gamma$, like $T_1,T_2$ or $T_5,T_6$ in Figure \ref{3angleB}. For $\beta_1\thin\beta_2\cdots\cdots=\beta^2\gamma^2$, we get $T_1,T_2,T_3,T_4$. For $\beta_5\thin\beta_6\cdots=\alpha\beta^3\gamma$, we get $T_5,T_6,T_7,T_8,T_9$, and further get $\gamma_6\dash\gamma_9\cdots=\alpha^k\gamma^2=\thick\alpha\thin\alpha\thick\gamma_6\dash\gamma_9\thick\cdots$, which determines $T_{10},T_{11}$. We conclude that, if $\beta\thin\beta\cdots$ induced from $\thick\alpha\thin\alpha\thick$ is $\alpha\beta^3\gamma$, then $\thick\alpha\thin\alpha\thick$ extends to $\thick\alpha\thin\alpha\thick\gamma\dash$. 

Suppose $\alpha^k$ is a vertex. Since the vertex is even, we have $\alpha^k=\thick\alpha\thin\alpha\thick\cdots\thick\alpha\thin\alpha\thick=(\thick\alpha\thin\alpha\thick)^{\frac{k}{2}}$. Since no $\thick\alpha\thin\alpha\thick$ in $\alpha^k$ is extended to $\thick\alpha\thin\alpha\thick\gamma\dash$, we know every $\thick\alpha\thin\alpha\thick$ induces $\beta\thin\beta\cdots=\beta^2\gamma^2$. Then we get the earth map tiling $E_{\triangle}1$ with $T_1,T_2,T_3,T_4$ as the timezone. In general, if $\thick\alpha\thin\alpha\thick$ at both sides of a sequence $(\thick\alpha\thin\alpha\thick)^l=\thick\alpha\thin\alpha\thick\cdots\thick\alpha\thin\alpha\thick$ induce $\beta\thin\beta\cdots=\beta^2\gamma^2$, then the sequence induces a partial earth map tiling, with $(\thin\alpha\thick\alpha\thin)^l=\thin\alpha\thick\alpha\thin\cdots\thin\alpha\thick\alpha\thin$ at the other end.

\begin{figure}[htp]
\centering
\begin{tikzpicture}[>=latex]

\draw[dashed]
	(-4,0) -- (-1,0) -- (0.5,0.5)
	(0.5,-0.5) -- (0.5,-1)
	(2,0.3) -- (2,1)
	(1.6,-0.6) -- (2,0.3) -- (2.4,-0.6)
	(3.5,-0.5) -- (3.5,-1);
	
\draw[line width=1.2]
	(-4,0) -- (-4,1)
	(-3,0) -- (-3,-1)
	(-2,0) -- (-2,1)
	(-1,0) -- (-1,-1)
	(0.5,-0.5) -- (0.5,1)
	(1.6,-0.6) -- (0.5,0.5)
	(2.4,-0.6) -- (3.5,0.5)
	(3.5,-0.5) -- (3.5,1);
	
\draw
	(-4,0) -- (-4,-1)
	(-3,0) -- (-3,1)
	(-2,0) -- (-2,-1)
	(-1,0) -- (-1,1)
	(-1,0) -- (0.5,-0.5)
	(0.5,0.5) -- (2,0.3) -- (3.5,0.5);

\node at (-3.5,0.9) {\small $\alpha$};
\node at (-2.5,0.9) {\small $\alpha$};
\node at (-1.5,0.9) {\small $\alpha$};
\node at (-0.25,0.9) {\small $\alpha$};
\node at (-3.5,-0.9) {\small $\alpha$};
\node at (-2.5,-0.9) {\small $\alpha$};
\node at (-1.5,-0.9) {\small $\alpha$};

\node at (-3.8,0.2) {\small $\gamma$};
\node at (-3.2,0.2) {\small $\beta$};

\node at (-2.8,0.2) {\small $\beta$};
\node at (-2.8,-0.2) {\small $\gamma$};
\node at (-2.2,0.2) {\small $\gamma$};
\node at (-2.2,-0.2) {\small $\beta$};

\node at (-1.8,0.2) {\small $\gamma$};
\node at (-1.8,-0.2) {\small $\beta$};
\node at (-1.2,-0.2) {\small $\gamma$};
\node at (-1.2,0.2) {\small $\beta$};

\node at (-3.8,-0.2) {\small $\beta$};
\node at (-3.2,-0.2) {\small $\gamma$};

\node at (-0.8,0.25) {\small $\beta$};
\node at (0.3,0.65) {\small $\gamma$};
\node at (0.3,0.25) {\small $\gamma$};
\node at (0.3,-0.25) {\small $\alpha$};
\node at (-0.4,0) {\small $\beta$};

\node at (-0.8,-0.25) {\small $\alpha$};
\node at (0.3,-0.65) {\small $\beta$};
\node at (-0.25,-0.9) {\small $\gamma$};

\node at (1.2,0.9) {\small $\gamma$};
\node at (0.7,0.65) {\small $\alpha$};
\node at (1.85,0.5) {\small $\beta$};

\node at (1.55,-0.3) {\small $\gamma$};
\node at (1.75,0.1) {\small $\beta$};
\node at (1,0.25) {\small $\alpha$};

\node at (2.8,0.9) {\small $\gamma$};
\node at (3.3,0.65) {\small $\alpha$};
\node at (2.15,0.5) {\small $\beta$};

\node at (2.45,-0.3) {\small $\gamma$};
\node at (2.3,0.15) {\small $\beta$};
\node at (3,0.25) {\small $\alpha$};

\node[rotate=-90] at (0.75,-0.15) {\small $\alpha^{k-2}$};

\node[inner sep=0.5,draw,shape=circle] at (-3.5,0.5) {\small $1$};
\node[inner sep=0.5,draw,shape=circle] at (-2.5,0.5) {\small $2$};
\node[inner sep=0.5,draw,shape=circle] at (-3.5,-0.5) {\small $3$};
\node[inner sep=0.5,draw,shape=circle] at (-2.5,-0.5) {\small $4$};
\node[inner sep=0.5,draw,shape=circle] at (-0.4,0.5) {\small $6$};
\node[inner sep=0.5,draw,shape=circle] at (-1.5,0.5) {\small $5$};
\node[inner sep=0.5,draw,shape=circle] at (0,0) {\small $9$};
\node[inner sep=0.5,draw,shape=circle] at (-0.4,-0.5) {\small $8$};
\node[inner sep=0.5,draw,shape=circle] at (-1.5,-0.5) {\small $7$};

\node[inner sep=0,draw,shape=circle] at (1.55,0.8) {\footnotesize $10$};
\node[inner sep=0,draw,shape=circle] at (2.45,0.8) {\footnotesize $12$};
\node[inner sep=0,draw,shape=circle] at (1.35,0.1) {\footnotesize $11$};
\node[inner sep=0,draw,shape=circle] at (2.65,0.1) {\footnotesize $13$};

\end{tikzpicture}
\caption{Proposition \ref{2a2b}: Timezone, and impact of the vertex $\alpha\beta^3\gamma$.}
\label{3angleB}
\end{figure}
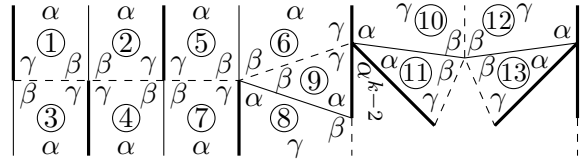

In the subsequent argument, we assume $\alpha^k$ is not a vertex. By applying the counting lemma to $\beta,\gamma$, we know $\alpha\beta^3\gamma$ is a vertex if and only if $\alpha^k\gamma^2$ is a vertex.

\subsubsection*{Case. $\alpha\beta^3\gamma,\alpha^k\gamma^2$ are not vertices}

By no $\alpha^k$, we only need to consider $\text{AVC}
=\{\beta^2\gamma^2,\alpha^k\beta\gamma\}$. 

Since $\alpha^k\beta\gamma$ is an odd vertex, we know $k=2q+1$ is odd. Then $\alpha^{2q+1}\beta\gamma=\thick\alpha\thin\beta\dash\gamma\thick\alpha\thin\alpha\thick\cdots\thick\alpha\thin\alpha\thick=\cdots\thick\alpha\thin\alpha\thick\alpha\thin\beta\dash\gamma\thick\alpha\thin\alpha\thick\cdots$ determines $T_1,T_2,T_3,T_4,T_5,T_6$ in Figure \ref{3angleD}. By $\beta^2\cdots=\gamma^2\cdots=\beta^2\gamma^2$, we know $T_1,T_2$ determine $T_7,T_8$, and $T_5,T_6$ determine $T_9,T_{10}$. The same happens to all the other $\thick\alpha\thin\alpha\thick$ and $\thin\alpha\thick\alpha\thin$ in $\alpha^{2q+1}\beta\gamma$. Then the $\alpha^{2q+1}=\thick\alpha_5\thin\alpha_6\thick\alpha\thin\cdots\thin\alpha_1\thick\alpha_2\thin$ part of $\alpha^{2q+1}\beta\gamma$ induces $\thin\alpha_9\thick\alpha_{10}\thin\alpha\thick\cdots\thin\alpha_7\thin\alpha_8\thick=\alpha^{2q+1}$. Therefore we get a partial earth map tiling with $\alpha^{2q+1}$ at both ends, and consisting of tiles on the left and right of the shaded edges. 

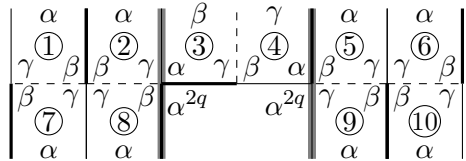
\begin{figure}[htp]
\centering
\begin{tikzpicture}[>=latex]

\draw[gray, line width=3]
	(1,-1) -- (1,1)
	(-1,-1) -- (-1,1);

\foreach \a in {1,-1}	
\foreach \b in {1,-1}
{
\begin{scope}[xshift=-2*\a cm, scale=\a, xscale=\b]

\draw[dashed]
	(0,0) -- (1,0);

\draw
	(1,0) -- (1,1)
	(0,0) -- (0,-1);
	
\draw[line width=1.2]
	(0,0) -- (0,1)
	(1,0) -- (1,-1);

\node at (0.2,0.2) {\small $\beta$};
\node at (0.8,0.2) {\small $\gamma$};	
\node at (0.5,0.9) {\small $\alpha$};

\node at (0.8,-0.2) {\small $\beta$};
\node at (0.2,-0.2) {\small $\gamma$};	
\node at (0.5,-0.9) {\small $\alpha$};

\end{scope}
}	

\draw[dashed]
	(0,0) -- (0,1);

\draw
	(0,0) -- (1,0);

\draw[line width=1.2]
	(0,0) -- (-1,0);

\node at (-0.2,0.2) {\small $\gamma$};
\node at (-0.5,0.9) {\small $\beta$};
\node at (-0.8,0.2) {\small $\alpha$};	

\node at (0.5,0.9) {\small $\gamma$};
\node at (0.2,0.2) {\small $\beta$};
\node at (0.8,0.2) {\small $\alpha$};	

\node at (-0.65,-0.25) {\small $\alpha^{2q}$};
\node at (0.65,-0.25) {\small $\alpha^{2q}$};

\node[inner sep=0.5,draw,shape=circle] at (-2.5,0.5) {\small $1$};
\node[inner sep=0.5,draw,shape=circle] at (-1.5,0.5) {\small $2$};
\node[inner sep=0.5,draw,shape=circle] at (-0.5,0.5) {\small $3$};
\node[inner sep=0.5,draw,shape=circle] at (0.5,0.5) {\small $4$};
\node[inner sep=0.5,draw,shape=circle] at (1.5,0.5) {\small $5$};
\node[inner sep=0.5,draw,shape=circle] at (2.5,0.5) {\small $6$};
\node[inner sep=0.5,draw,shape=circle] at (-2.5,-0.5) {\small $7$};
\node[inner sep=0.5,draw,shape=circle] at (-1.5,-0.5) {\small $8$};
\node[inner sep=0.5,draw,shape=circle] at (1.5,-0.5) {\small $9$};
\node[inner sep=0,draw,shape=circle] at (2.5,-0.5) {\footnotesize $10$};

\end{tikzpicture}
\caption{Proposition \ref{2a2b}: Tiling when $\alpha\beta^3\gamma,\alpha^k\gamma^2$ are not vertices.}
\label{3angleD}
\end{figure}

We have $\alpha_3\beta_8\gamma_2\cdots=\alpha_4\beta_5\gamma_9\cdots=\alpha^{2q+1}\beta\gamma$. By the argument above, the two $\alpha^{2q+1}$ are the two ends of another partial earth map tiling that fills the middle of the picture. This partial earth map tiling is obtained by the flip in the second of Figure \ref{flip4}. Therefore we get the flip modification tiling $FE_{\triangle}1$.

\subsubsection*{Case. $\alpha\beta^3\gamma,\alpha^k\gamma^2$ are vertices}

The angle sums of $\alpha\beta^3\gamma,\beta^2\gamma^2$ and the angle sum for triangle imply
\[
\alpha=\tfrac{4}{f}\pi,\;
\beta=(\tfrac{1}{2}-\tfrac{2}{f})\pi,\;
\gamma=(\tfrac{1}{2}+\tfrac{2}{f})\pi.
\]
Then we get the updated list of vertices
\[
\text{AVC}
=\{\alpha\beta^3\gamma,\beta^2\gamma^2,\alpha^{2q+1}\beta\gamma,\alpha^{2q}\gamma^2\},\quad
f=8q+4.
\]

We know $\alpha^{2q}\gamma^2=\dash\gamma\thick\alpha\thin\alpha\thick\cdots\thick\alpha\thin\alpha\thick\gamma\dash=\dash\gamma(\thick\alpha\thin\alpha\thick)^q\gamma\dash$. If $\thick\alpha\thin\alpha\thick$ at both sides of the sequence $(\thick\alpha\thin\alpha\thick)^q$ induce $\beta\thin\beta\cdots=\alpha\beta^3\gamma$, then we also get $T_{12},T_{13}$ in Figure \ref{3angleB}. Then $\beta_{10}\beta_{11}\beta_{12}\beta_{13}\cdots=\beta^4\cdots$ contradicts the AVC. Therefore at least one side of $(\thick\alpha\thin\alpha\thick)^q$ induces $\beta\thin\beta\cdots=\beta^2\gamma^2$.

Suppose one side of $(\thick\alpha\thin\alpha\thick)^q$ induces $\beta\thin\beta\cdots=\alpha\beta^3\gamma$, and the other side induces $\beta\thin\beta\cdots=\beta^2\gamma^2$. Then we express the vertex as $\alpha^{2q}\gamma^2=\cdots\thick\alpha\thin\alpha\thick\gamma\dash\gamma\thick\alpha\thin\alpha\thick\cdots$, corresponding to $T_1,T_2,T_3,T_4,T_5,T_6$ in the first of Figure \ref{3angleC}. Moreover, we have $\beta_1\thin\beta_2\cdots=\alpha\beta^3\gamma=\thick\alpha\thin\beta\dash\beta_2\thin\beta_1\dash\gamma\thick$ at one side and $\beta_5\thin\beta_6\cdots=\beta^2\gamma^2=\thick\gamma\dash\beta_5\thin\beta_6\dash\gamma\thick$ at the other side. This determines $T_7,T_8,T_9,T_{10},T_{11}$. Then $\alpha_3\gamma_2\gamma_9\cdots=\alpha^{2q}\gamma^2$ determines $T_{12},T_{13}$, with $\alpha^{2q-3}$ between the two tiles. Then $\beta_3\beta_4\beta_{12}\cdots=\alpha\beta^3\gamma$ determines $T_{14},T_{15}$. Then $\alpha_4\beta_{10}\beta_{15}\gamma_5\cdots=\alpha\beta^3\gamma$ determines $T_{16}$. Then $\alpha_{14}\gamma_{15}\gamma_{16}\cdots=\alpha^{2q}\gamma^2$ determines $T_{17},T_{18}$, with $\alpha^{2q-3}$ between $T_{14},T_{18}$. 

Combining $\alpha^{2q-3}$ at $\gamma_{16}\cdots$ with $\alpha_{14}$, we get a sequence $(\thick\alpha\thin\alpha\thick)^{q-1}$. By $\alpha^{2q-3}$ at $\gamma_2\cdots$ and $\gamma_{16}\cdots$, we get $\beta$ just outside $\beta_{14}$ and $\gamma$ just outside $\gamma_{12}$. Therefore $\beta_{14}\gamma_{12}\cdots=\beta^2\gamma^2\cdots=\beta^2\gamma^2$. This shows the $\thick\alpha_{14}\thin\alpha\thick$ side of the sequence $(\thick\alpha\thin\alpha\thick)^{q-1}$ induces $\beta\thin\beta\cdots=\beta^2\gamma^2$. Since the other side of the sequence $(\thick\alpha\thin\alpha\thick)^{q-1}$ is further extended by $\alpha_{18}$ (instead of $\gamma$), it also induces $\beta\thin\beta\cdots=\beta^2\gamma^2$. Therefore $(\thick\alpha\thin\alpha\thick)^{q-1}$ induces a partial earth map tiling. The other end of this partial earth map tiling is $(\thin\alpha\thick\alpha\thin)^{q-1}$ obtained by combining  $\alpha^{2q-3}$ at $\gamma_2\cdots$ with $\alpha_{12}$. This partial earth map tiling fills the rest of the first of Figure \ref{3angleC}. 

\begin{figure}[htp]
\centering
\begin{tikzpicture}[>=latex]


\foreach \a in {-1,1}
{
\begin{scope}[scale=\a]

\draw[dashed]
	(1.2,1.2) -- (1.2,2)
	(1.2,1.2) -- (0.4,0.4)
	(1.2,-1.2) -- (2,0) -- (2.8,0);

\draw[line width=1.2]
	(1.2,1.2) -- (1.2,-2)
	(-1.2,1.2) -- (0.4,0.4)
	(2,0) -- (2,2)
	(2.8,0) -- (2.8,-2);
	
\draw
	(-1.2,1.2) -- (1.2,1.2)
	(1.2,-1.2) -- (0.4,0.4)
	(1.2,1.2) -- (2,0) -- (2,-2)
	(2.8,0) -- (2.8,2);

\node at (0,1.9) {\small $\gamma$};
\node at (1,1.4) {\small $\beta$};
\node at (-1,1.4) {\small $\alpha$};

\node at (0.4,0.6) {\small $\gamma$};
\node at (0.8,1) {\small $\beta$};
\node at (-0.5,1.05) {\small $\alpha$};

\node at (1,0.8) {\small $\gamma$};
\node at (0.65,0.45) {\small $\beta$};
\node at (1.05,-0.5) {\small $\alpha$};

\node at (1.35,-0.7) {\small $\gamma$};
\node at (1.78,-0.1) {\small $\beta$};
\node at (1.35,0.7) {\small $\alpha$};

\node at (1.4,-1.25) {\small $\gamma$};
\node at (1.8,-0.5) {\small $\beta$};
\node at (1.6,-1.9) {\small $\alpha$};

\node at (1.6,1.9) {\small $\gamma$};
\node at (1.4,1.3) {\small $\beta$};
\node at (1.8,0.6) {\small $\alpha$};

\node at (2.2,0.2) {\small $\gamma$};
\node at (2.6,0.2) {\small $\beta$};
\node at (2.4,1.9) {\small $\alpha$};

\node at (2.6,-0.2) {\small $\gamma$};
\node at (2.2,-0.2) {\small $\beta$};
\node at (2.4,-1.9) {\small $\alpha$};

\node[rotate=-45] at (-0.45,0.5) {\small $\alpha^{2q-3}$};

\end{scope}
}

\node at (0.15,0.3) {\small $\gamma$};
\node at (0.35,0.05) {\small $\beta$};

\draw[dashed]
	(2.8,0) -- (3.6,0);

\draw[line width=1.2]
	(3.6,0) -- (3.6,2);

\draw
	(3.6,0) -- (3.6,-2);

\node at (3.4,0.2) {\small $\gamma$};
\node at (3,0.2) {\small $\beta$};
\node at (3.2,1.9) {\small $\alpha$};

\node at (3,-0.2) {\small $\gamma$};
\node at (3.4,-0.2) {\small $\beta$};
\node at (3.2,-1.9) {\small $\alpha$};

\node[inner sep=0.5,draw,shape=circle] at (-2.4,1) {\small $1$};
\node[inner sep=0.5,draw,shape=circle] at (-1.7,1.5) {\small $2$};
\node[inner sep=0.5,draw,shape=circle] at (0,1.5) {\small $3$};
\node[inner sep=0.5,draw,shape=circle] at (1.7,1.5) {\small $4$};
\node[inner sep=0.5,draw,shape=circle] at (2.4,1) {\small $5$};
\node[inner sep=0.5,draw,shape=circle] at (3.2,1) {\small $6$};
\node[inner sep=0.5,draw,shape=circle] at (-2.4,-1) {\small $7$};
\node[inner sep=0.5,draw,shape=circle] at (-1.7,-1.5) {\small $8$};
\node[inner sep=0.5,draw,shape=circle] at (-1.45,0) {\small $9$};
\node[inner sep=0,draw,shape=circle] at (2.4,-1) {\footnotesize $10$};
\node[inner sep=0,draw,shape=circle] at (3.2,-1) {\footnotesize $11$};
\node[inner sep=0,draw,shape=circle] at (0,0.95) {\footnotesize $12$};
\node[inner sep=0,draw,shape=circle] at (-0.9,0) {\footnotesize $13$};
\node[inner sep=0,draw,shape=circle] at (0.9,0) {\footnotesize $14$};
\node[inner sep=0,draw,shape=circle] at (1.5,0) {\footnotesize $15$};
\node[inner sep=0,draw,shape=circle] at (1.7,-1.5) {\footnotesize $16$};
\node[inner sep=0,draw,shape=circle] at (0,-1.5) {\footnotesize $17$};
\node[inner sep=0,draw,shape=circle] at (0,-0.95) {\footnotesize $18$};


\begin{scope}[xshift=8cm]

\draw[dashed]
	(-1.2,1.2) -- (-1.2,2)
	(0.2,-1.2) -- (0.2,-2)
	(-1.2,1.2) -- (-0.4,0)
	(-1.2,-1.2) -- (-2,0) -- (-2.8,0)
	(1,0) -- (1.8,0);

\draw[line width=1.2]
	(-1.2,1.2) -- (-1.2,-2)
	(-2,0) -- (-2,2)
	(-2.8,0) -- (-2.8,-2)
	(0.2,-1.2) -- (1,0)
	(1,0) -- (1,2)
	(1.8,0) -- (1.8,-2)
	(-0.4,0) -- (1,0);
	
\draw
	(-1.2,1.2) to[out=0, in=135] (1,0)
	(-1.2,-1.2) -- (0.2,-1.2)
	(-1.2,1.2) -- (-2,0) -- (-2,-2)
	(-2.8,0) -- (-2.8,2)
	(1,0) -- (1,-2)
	(1.8,0) -- (1.8,2)
	(-0.4,0) -- (-1.2,-1.2);

\node at (-1.35,-0.7) {\small $\gamma$};
\node at (-1.78,0.05) {\small $\beta$};
\node at (-1.35,0.7) {\small $\alpha$};

\node at (-1.4,-1.2) {\small $\gamma$};
\node at (-1.85,-0.6) {\small $\beta$};
\node at (-1.6,-1.9) {\small $\alpha$};

\node at (-1.6,1.9) {\small $\gamma$};
\node at (-1.4,1.2) {\small $\beta$};
\node at (-1.8,0.6) {\small $\alpha$};

\node at (-2.6,-0.2) {\small $\gamma$};
\node at (-2.2,-0.2) {\small $\beta$};
\node at (-2.4,-1.9) {\small $\alpha$};

\node at (-2.2,0.2) {\small $\gamma$};
\node at (-2.6,0.2) {\small $\beta$};
\node at (-2.4,1.9) {\small $\alpha$};

\node at (1.6,-0.2) {\small $\gamma$};
\node at (1.2,-0.2) {\small $\beta$};
\node at (1.4,-1.9) {\small $\alpha$};

\node at (1.2,0.2) {\small $\gamma$};
\node at (1.6,0.2) {\small $\beta$};
\node at (1.4,1.9) {\small $\alpha$};

\node at (1.2,0.2) {\small $\gamma$};
\node at (1.6,0.2) {\small $\beta$};
\node at (1.4,1.9) {\small $\alpha$};

\node at (-0.5,-1.9) {\small $\gamma$};
\node at (0,-1.4) {\small $\beta$};
\node at (-1,-1.4) {\small $\alpha$};

\node at (0.4,-1.25) {\small $\gamma$};
\node at (0.4,-1.9) {\small $\beta$};
\node at (0.85,-0.6) {\small $\alpha$};

\node at (-0.1,1.9) {\small $\gamma$};
\node at (-1,1.4) {\small $\beta$};
\node at (0.8,0.45) {\small $\alpha$};

\node at (-1.05,0.7) {\small $\gamma$};
\node at (-1.05,-0.7) {\small $\alpha$};
\node at (-0.6,-0.05) {\small $\beta$};

\node at (-0.3,0.2) {\small $\gamma$};
\node at (-0.7,0.9) {\small $\beta$};
\node at (0.6,0.2) {\small $\alpha$};

\node[rotate=30] at (-0.55,-0.8) {\small $\alpha^{2q-2}$};
\node[rotate=30] at (0.4,-0.45) {\small $\alpha^{2q-2}$};

\node[inner sep=0.5,draw,shape=circle] at (-2.4,1) {\small $1$};
\node[inner sep=0.5,draw,shape=circle] at (-1.7,1.5) {\small $2$};
\node[inner sep=0.5,draw,shape=circle] at (-0.1,1.4) {\small $3$};
\node[inner sep=0.5,draw,shape=circle] at (1.4,1) {\small $4$};
\node[inner sep=0.5,draw,shape=circle] at (-2.4,-1) {\small $5$};
\node[inner sep=0.5,draw,shape=circle] at (1.4,-1) {\small $6$};
\node[inner sep=0.5,draw,shape=circle] at (-1.45,0) {\small $7$};
\node[inner sep=0.5,draw,shape=circle] at (-0.95,0) {\small $8$};
\node[inner sep=0.5,draw,shape=circle] at (-0.1,0.5) {\small $9$};
\node[inner sep=0,draw,shape=circle] at (-1.7,-1.5) {\footnotesize $10$};
\node[inner sep=0,draw,shape=circle] at (0.7,-1.5) {\footnotesize $11$};
\node[inner sep=0,draw,shape=circle] at (-0.5,-1.5) {\footnotesize $12$};

\end{scope}

\end{tikzpicture}
\caption{Proposition \ref{2a2b}: Tilings when $\alpha\beta^3\gamma,\alpha^k\gamma^2$ are vertices.}
\label{3angleC}
\end{figure}
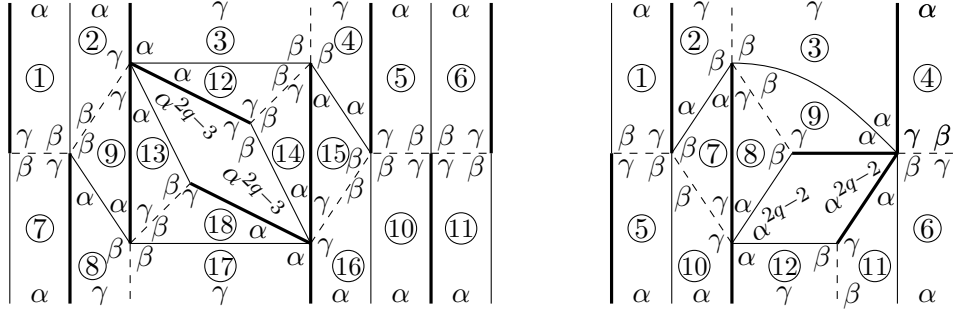

Suppose both sides of $(\thick\alpha\thin\alpha\thick)^q$ induce $\beta\thin\beta\cdots=\beta^2\gamma^2$. Then we get a partial earth map tiling with $(\thin\alpha\thick\alpha\thin)^q$ at the other end. We express the vertex as $\alpha^{2q}\gamma^2=\cdots\thin\alpha\thick\gamma\dash\gamma\thick\alpha\thin\cdots$, corresponding to $T_1,T_2,T_3,T_4$ in the second of Figure \ref{3angleC}. We also have $T_5,T_6$ as part of the partial earth map tiling. We note that, although $T_2,T_3$ are drawn in different sizes, the six tiles are symmetric with respect to the horizontal flip. Therefore by $\beta\dash\beta\cdots=\alpha\beta^3\gamma$, we may assume $T_7,T_8,T_9$ as in the picture. Then $\alpha_2\beta_5\beta_7\gamma_1\cdots=\alpha\beta^3\gamma$ determines $T_{10}$, and $\alpha_3\beta_6\gamma_4\alpha_9\cdots=\alpha^{2q+1}\beta\gamma$ determines $T_{11}$ and gives $\alpha^{2q-2}$ between $T_9,T_{11}$. Then $\alpha_8\gamma_7\gamma_{10}\cdots=\alpha^{2q}\gamma^2$ determines $T_{12}$ and gives $\alpha^{2q-2}$ between $T_8,T_{12}$. Both sides of the sequence $\alpha^{2q-2}=(\thick\alpha\thin\alpha\thick)^{q-1}$ at $\gamma_4\cdots$ are extended by $\alpha_9,\alpha_{11}$ (instead of $\gamma$), and therefore both induce $\beta\thin\beta\cdots=\beta^2\gamma^2$. This implies the sequence $(\thick\alpha\thin\alpha\thick)^{q-1}$ induces a partial earth map tiling, with the other end being $\alpha^{2q-2}=(\thin\alpha\thick\alpha\thin)^{q-1}$ at $\gamma_7\cdots$. This partial earth map tiling fills the rest of the second of Figure \ref{3angleC}. 

We note that the angles in Figure \ref{3angleC} satisfy $\alpha+\beta=\gamma$ as in the third and fourth of Figure \ref{flip4}. Therefore two triangles in the first of Figure \ref{3angleI} form a rectangle. This happens to $T_4,T_{15}$ in the first of Figure \ref{3angleC}, and $T_8,T_9$ in the first of Figure \ref{3angleC}, and $T_2,T_7$ in the second of Figure \ref{3angleC}. Due to the rectangle, we may apply the flip in the second of Figure \ref{flip5} and still get a tiling. The flip changes two $\alpha\beta^3\gamma$ around the rectangles to $\beta^2\gamma^2$. After the change, the tilings have $\text{AVC}=\{\beta^2\gamma^2,\alpha^k\beta\gamma\}$, and is therefore the flip modification $FE_{\triangle}1$. Therefore the tilings in Figure \ref{3angleC} are $F'E_{\triangle}1$ obtained by applying the flip in the fourth of Figure \ref{flip4} to $FE_{\triangle}1$.

\medskip

\noindent{\em Geometry of Triangle}

\medskip

The angles in the earth map tiling satisfy 
\[
\alpha=\tfrac{4}{f}\pi,\;
\beta+\gamma=\pi,\;
f=0\text{ mod }4.
\]
This means the triangle is half of the $2$-gon with angle $\alpha$. See Figure \ref{3angleI}. The tiling has one free parameter, which can be $\gamma\in (\frac{1}{2}\alpha,\pi-\frac{1}{2}\alpha)=(\frac{2}{f}\pi,(1-\frac{2}{f})\pi)$, or can be $b\in (0,\pi)$. We note that $\alpha>\beta$ when $\gamma$ is close to $(1-\frac{2}{f})\pi$. In this case, the triangle still exists, and the tiling also exists and is given by Proposition \ref{2a2c}.

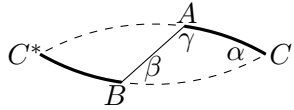
\begin{figure}[htp]
\centering
\begin{tikzpicture}[>=latex]


\foreach \a in {1,-1}
{
\begin{scope}[scale=\a]

\draw[dashed]
	(1.5,0) arc (60:120:3);

\draw
	(1.5,0) arc (60:82:3) -- (0,0);
	
\draw[line width=1.2]
	(1.5,0) arc (60:82:3) ;

\end{scope}
}

\node at (0.45,0.55) {\small $A$};
\node at (-0.5,-0.55) {\small $B$};
\node at (-1.7,0) {\small $C^*$};
\node at (1.7,0) {\small $C$};

\node at (0.45,0.15) {\small $\gamma$};
\node at (0,-0.2) {\small $\beta$};
\node at (1.1,0) {\small $\alpha$};

\end{tikzpicture}
\caption{Proposition \ref{2a2b}: Geometry of triangle.}
\label{3angleI}
\end{figure}

To get flip modified tiling in Figure \ref{3angleD}, we just require $f=4$ mod $8$. For the further rectangular flip, we need $f=4$ mod $8$ and $\gamma=(\tfrac{1}{2}+\tfrac{2}{f})\pi$. The existence of the triangle means $\gamma$ is in the range above, which means $f>8$. By $f=4$ mod $8$, the condition is the same as $f\ne 4$. By the degree $4$ vertex $\beta^2\gamma^2$ and \eqref{trivcountf4}, we indeed have $f\ne 4$.
\end{proof}

\begin{proposition}\label{2a2c}
Tilings of the sphere by congruent general triangles, such that $\alpha<\beta<\gamma$ and $\alpha^2\gamma^2$ is a vertex, are the earth map tiling $E_{\triangle}1$ and its flip modification $FE_{\triangle}1$.
\end{proposition}

We get the two tilings after the exchange of $(\alpha,b)$ with $(\beta,c)$. Moreover, we remark that the earth map tiling includes the octahedron $P_8$ as a special case.

\begin{proof}
By $\alpha^2\gamma^2$ and $\alpha<\beta<\gamma$, we know $\alpha\beta\gamma^3,\beta^2\gamma^2,\gamma^4$ are not vertices. Then the AVC \eqref{triavc1} is reduced to $\{\alpha^2\gamma^2,\alpha^k\beta^l,\alpha^k\beta^l\gamma\}$. By $k\ge 1$ in $\alpha^k\beta^l\gamma$, and applying the counting lemma to $\alpha,\gamma$, we get $k=1$ in $\alpha^k\beta^l\gamma$, and $l=0$ in $\alpha^k\beta^l$. Then we get $\text{AVC}
=\{\alpha^2\gamma^2,\beta^l,\alpha\beta^l\gamma\}$. After exchanging $\alpha,\beta$, this is included in the AVC in Proposition \ref{2a2b}. By the earlier argument, we get the earth map tiling (when $\beta^l$ is a vertex) or the flip modification in Figure \ref{3angleD} (when $\beta^l$ is not a vertex, which corresponds to the case $\alpha\beta^3\gamma,\alpha^k\gamma^2$ are not vertices, in Proposition \ref{2a2b}). 
\end{proof}

\begin{proposition}\label{3abc}
There is no tiling of the sphere by congruent general triangles, such that $\alpha<\beta<\gamma$ and $\alpha\beta\gamma^3$ is a vertex.
\end{proposition}

\begin{proof}
The angle sum of $\alpha\beta\gamma^3$ and the angle sum for triangle imply
\[
\alpha+\beta=(\tfrac{1}{2}+\tfrac{6}{f})\pi,\;
\gamma=(\tfrac{1}{2}-\tfrac{2}{f})\pi.
\]
By $\alpha<\beta<\gamma<\tfrac{1}{2}\pi$, there are no vertices of degree $3$ or $4$. Therefore the general AVC \eqref{triavc1} is reduced to
\[
\text{AVC}
=\{\alpha\beta\gamma^3,\alpha^k\beta^l(k+l\ge 6),
\alpha^k\beta^l\gamma(k+l\ge 4),\alpha^k\gamma^2(k\ge 4)\}.
\]

The vertex $\alpha^k\beta^l\gamma$ is odd. By $\alpha<\beta$ and $3\alpha+3\beta+\gamma=(2+\tfrac{16}{f})\pi>2\pi$, we get $\alpha^k\beta^l\gamma=\alpha^k\beta\gamma,\alpha\beta^3\gamma$. By $\alpha\beta\gamma^3$ and $\beta<\gamma$, we know $\alpha\beta^3\gamma$ is not a vertex. By $\alpha\beta\gamma^3$ and $\alpha<\gamma$, we also get $k\ge 5$ in $\alpha^k\beta\gamma$. Therefore $\alpha^k\beta^l\gamma=\alpha^k\beta\gamma(k\ge 5)$. Then by the AVC, we get $\beta^2\cdots=\alpha^k\beta^l$. 

A vertex $\thick\alpha\thin\alpha\thick\alpha\thin\alpha\thick\cdots$ determines $T_1,T_2,T_3,T_4$ in Figure \ref{3angleF}. By $\thick\alpha\thin$, we get $\beta_1\thin\beta_2\cdots=\alpha^k\beta^l=\dash\beta_1\thin\beta_2\dash\beta\thin\cdots$. This determines $T_5$. By the same reason, we determine $T_6$. Then we get $\gamma_2\gamma_3\gamma_5\gamma_6\cdots=\gamma^4\cdots$, contradicting the AVC. Therefore $\thick\alpha\thin\alpha\thick\alpha\thin\alpha\thick\cdots$ is not a vertex. This implies $\alpha^k\beta\gamma(k\ge 5),\alpha^k\gamma^2(k\ge 4)$ are not vertices. 

\begin{figure}[htp]
\centering
\begin{tikzpicture}[>=latex]

\foreach \a in {1,-1}
{
\begin{scope}[xscale=\a]

\draw[dashed]
	(0,0) -- (2,0);
	
\draw[line width=1.2]
	(0,0) -- (0,1)
	(2,0) -- (2,1)
	(0,0) -- (0.2,-1);

\draw
	(1,-1) -- (1,1);

\node at (0.2,0.2) {\small $\gamma$};
\node at (0.8,0.2) {\small $\beta$};
\node at (0.5,0.9) {\small $\alpha$};

\node at (1.8,0.2) {\small $\gamma$};
\node at (1.2,0.2) {\small $\beta$};
\node at (1.5,0.9) {\small $\alpha$};

\node at (0.2,-0.15) {\small $\gamma$};
\node at (0.8,-0.2) {\small $\beta$};
\node at (0.5,-0.9) {\small $\alpha$};

\end{scope}
}

\node[inner sep=0.5,draw,shape=circle] at (1.5,0.5) {\small $1$};
\node[inner sep=0.5,draw,shape=circle] at (0.5,0.5) {\small $2$};
\node[inner sep=0.5,draw,shape=circle] at (-0.5,0.5) {\small $3$};
\node[inner sep=0.5,draw,shape=circle] at (-1.5,0.5) {\small $4$};

\node[inner sep=0.5,draw,shape=circle] at (0.5,-0.5) {\small $5$};
\node[inner sep=0.5,draw,shape=circle] at (-0.5,-0.5) {\small $6$};

\end{tikzpicture}
\caption{Proposition \ref{3abc}: Vertex $\thick\alpha\thin\alpha\thick\alpha\thin\alpha\thick$.}
\label{3angleF}
\end{figure}

The vertex $\alpha^k\beta^l$ is even. By $\alpha<\beta$, and $\alpha+\beta>\frac{1}{2}\pi$, and no degree $3$ or $4$ vertices, we get $\alpha^k\beta^l=\alpha^2\beta^4,\beta^6,\alpha^k(k\ge 6),\alpha^k\beta^2(k\ge 4)$. By no $\thick\alpha\thin\alpha\thick\alpha\thin\alpha\thick\cdots$, we know $k=4$ in $\alpha^k\beta^2$, and $\alpha^k$ is not a vertex. By $\alpha<\beta$, no two of $\alpha^4\beta^2,\alpha^2\beta^4,\beta^6$ can be vertices at the same time. Therefore the vertices are $\alpha\beta\gamma^3$ and one of $\alpha^4\beta^2,\alpha^2\beta^4,\beta^6$. Applying the counting lemma to $\alpha,\beta$, we get a contradiction.
\end{proof}

\begin{proposition}\label{4a}
Tilings of the sphere by congruent general triangles, such that $\alpha<\beta<\gamma$ and $\gamma^4$ is a vertex, are the barycentric subdivisions $B_{\triangle}P_8,B_{\triangle}P_{20}$ of the octahedron and the icosahedron, the flip modification $FB_{\triangle}P_8$ of the barycentric subdivision of the regular octahedron, and the earth map tiling $E_{\triangle}3$ and its flip modification $FE_{\triangle}3$.
\end{proposition}

\begin{proof} 
By $\gamma^4$ and $\beta<\gamma$, we know $\beta^2\gamma^2$ is not a vertex. The angle sum of $\gamma^4$ and the angle sum for triangle imply
\[
\alpha+\beta=(\tfrac{1}{2}+\tfrac{4}{f})\pi,\;\gamma=\tfrac{1}{2}\pi.
\]
By $\alpha<\beta<\gamma=\tfrac{1}{2}\pi$, the only vertex of degree $3$ or $4$ is $\gamma^4$. Moreover, by $\alpha+\beta+3\gamma=(2+\frac{4}{f})\pi>2\pi$, we know $\alpha\beta\gamma^3$ is not a vertex. Therefore the general AVC \eqref{triavc1} is reduced to
\[
\text{AVC}
=\{\gamma^4,\alpha^k\beta^l,\alpha^k\beta^l\gamma,\alpha^k\gamma^2\}.
\]
The AVC implies $\gamma\thick\gamma\cdots=\gamma^4$. We denote the union of four tiles around $\gamma^4$ by $n(\gamma^4)$. See the first of Figure \ref{3angleG}.

The vertex $\alpha^k\beta^l\gamma$ is odd. By $\alpha<\beta$ and $3\alpha+3\beta+\gamma=(2+\frac{12}{f})\pi>2\pi$, we get $\alpha^k\beta^l\gamma=\alpha\beta^3\gamma,\alpha^k\beta\gamma(k\ge 3)$. 

The vertex $\alpha^k\beta^l$ is even. Since $\gamma^4$ is the only degree $4$ vertex, we know $k+l\ge 6$. Then by $\alpha<\beta$ and $\alpha+\beta>\frac{1}{2}\pi$, we get $\alpha^k\beta^l=\alpha^2\beta^4,\beta^6,\alpha^k(k\ge 6),\alpha^k\beta^2(k\ge 4)$. 

We see that the numbers of $\alpha$ and $\beta$ are never the same at any vertex. By applying the counting lemma to $\alpha,\beta$, this implies there is a vertex with strictly fewer $\alpha$ than $\beta$. The vertex is $\alpha^2\beta^4,\alpha\beta^3\gamma,\beta^6$.

\subsubsection*{Case. $\beta^6$ is a vertex}

The angle sum of $\beta^6$ further implies 
\[
\alpha=(\tfrac{1}{6}+\tfrac{4}{f})\pi,\;
\beta=\tfrac{1}{3}\pi,\;
\gamma=\tfrac{1}{2}\pi.
\]
By $\tfrac{1}{6}\pi<\alpha<\beta=\tfrac{1}{3}\pi$ and the parity lemma, besides $\beta^6,\gamma^4$, all the possible vertices are $\alpha^4\gamma^2,\alpha^5\beta\gamma,\alpha^6\beta^2,\alpha^8,\alpha^{10}$. We may also use the angle sum for triangle to calculate the corresponding $f$. Then we get updated list of vertices:
\begin{align*}
f=48 &\colon 
	\alpha=\tfrac{1}{4}\pi, \;
	\beta=\tfrac{1}{3}\pi,\;
	\gamma=\tfrac{1}{2}\pi,\;
	\text{AVC}=\{\beta^6,\gamma^4,\alpha^4\gamma^2,\alpha^8\}. \\
f=60 &\colon
	\alpha=\tfrac{7}{30}\pi, \;
	\beta=\tfrac{1}{3}\pi,\;
	\gamma=\tfrac{1}{2}\pi,\;
	\text{AVC}=\{\beta^6,\gamma^4,\alpha^5\beta\gamma\}. \\
f=72 &\colon
	\alpha=\tfrac{2}{9}\pi, \;
	\beta=\tfrac{1}{3}\pi,\;
	\gamma=\tfrac{1}{2}\pi,\;
	\text{AVC}=\{\beta^6,\gamma^4,\alpha^6\beta^2\}. \\
f=120 &\colon
	\alpha=\tfrac{1}{5}\pi, \;
	\beta=\tfrac{1}{3}\pi,\;
	\gamma=\tfrac{1}{2}\pi,\;
	\text{AVC}=\{\beta^6,\gamma^4,\alpha^{10}\}.
\end{align*}

For $f=48,120$, we have $\beta\cdots=\beta^6$. This implies each tile belongs to one neighborhood $n(\beta^6)$ of $\beta^6$, given by the second of Figure \ref{3angleG}. Then the whole tiling is a tiling of copies of the hexagon $n(\beta^6)$. In fact, by $2\gamma=\pi$ and $2\alpha=\frac{1}{2}\pi(f=48),\frac{2}{5}\pi(f=120)$, the hexagon $n(\beta^6)$ is the regular triangular face $F$ of the octahedron (for $f=48$) or the icosahedron (for $f=120$). Therefore the problem becomes the tiling of the sphere by $F$. Moreover, we get $n(\beta^6)$ by applying the barycentric subdivision to each regular face. 

The midpoint of an edge of $F$ is $\gamma^2\cdots=\alpha^4\gamma^2,\gamma^4$. If the midpoint is always $\gamma^4$, then the tiling by $F$ is edge-to-edge. Then the tiling by $F$ is the regular octahedron $P_8$ or the regular icosahedron $P_{20}$, and the triangular tiling is the barycentric subdivision $B_{\triangle}P_8$ or $B_{\triangle}P_{20}$.

If the midpoint of one edge of one $F$ is $\alpha^4\gamma^2$, then $f=48$ and the edges match only in half (therefore the tiling by $F$ is not edge-to-edge). Then it is easy to see that the tiling by $F$ is obtained from the regular octahedron by rotating half of the tiling by $\frac{1}{4}\pi$. This is equivalent to the flip of the half tiling in Figure \ref{flip1}. Then we get the flip modification $FB_{\triangle}P_8$. 

\begin{figure}[htp]
\centering
\begin{tikzpicture}[>=latex]


\foreach \a in {-1,1}
{
\begin{scope}[xshift=-3cm, scale=\a]

\draw
	(0.8,-0.8) -- (0.8,0.8) -- (-0.8,0.8);
	
\draw[dashed]
	(0,0) -- (0.8,0.8);

\draw[line width=1.2]
	(0,0) -- (0.8,-0.8);

\node at (0.25,0) {\small $\gamma$};
\node at (0,0.25) {\small $\gamma$};

\node at (0.6,0.4) {\small $\beta$};
\node at (0.4,0.6) {\small $\beta$};

\node at (0.6,-0.4) {\small $\alpha$};
\node at (0.4,-0.6) {\small $\alpha$};

\end{scope}
}

\node at (-3,-1.15) {$n(\gamma^4)$};


\foreach \a in {0,1,2}
{
\begin{scope}[rotate=120*\a]

\draw[dashed]
	(0,0) -- (30:0.8);

\draw
	(0,0) -- (-30:1.6);

\draw[line width=1.2]
	(-30:1.6) -- (90:1.6);

\node at (17:0.65) {\small $\gamma$};
\node at (43:0.65) {\small $\gamma$};

\node at (0:0.35) {\small $\beta$};	
\node at (60:0.35) {\small $\beta$};	

\node at (80:0.95) {\small $\alpha$};
\node at (100:0.95) {\small $\alpha$};
	
\end{scope}
}

\node at (0,-1.15) {$n(\beta^6)$};


\begin{scope}[xshift=3.5cm]

\draw[dashed]
	(-30:1.6) -- (-90:0.8)
	(30:0.8) -- (150:0.8) -- (210:1.6);

\draw[line width=1.2]
	(210:1.6) -- (-90:0.8) -- (30:0.8) -- (90:1.6);
	
\draw
	(-30:1.6) -- (30:0.8)
	(90:1.6) -- (150:0.8) -- (-90:0.8);

\node at (30:0.45) {\small $\gamma$};
\node at (150:0.4) {\small $\beta$};
\node at (-90:0.45) {\small $\alpha$};

\node at (150:1) {\small $\beta^3$};

\node at (0.4,0.55) {\small $\gamma$};
\node at (-0.4,0.6) {\small $\beta$};
\node at (0,1.2) {\small $\alpha$};

\node at (0.3,-0.65) {\small $\gamma$};
\node at (-30:1.25) {\small $\beta$};
\node at (0.7,0) {\small $\alpha$};

\node at (210:1.25) {\small $\gamma$};
\node at (-0.7,0) {\small $\beta$};
\node at (-0.35,-0.6) {\small $\alpha$};

\node[inner sep=0.5,draw,shape=circle] at (-30:0.8) {\small $1$};
\node[inner sep=0.5,draw,shape=circle] at (0,0) {\small $2$};
\node[inner sep=0.5,draw,shape=circle] at (210:0.8) {\small $3$};
\node[inner sep=0.5,draw,shape=circle] at (90:0.8) {\small $4$};

\end{scope}

\end{tikzpicture}
\caption{Proposition \ref{4a}: Neighborhoods of $\gamma^4$ and $\beta^6$, and no $\alpha^5\beta\gamma$.}
\label{3angleG}
\end{figure}
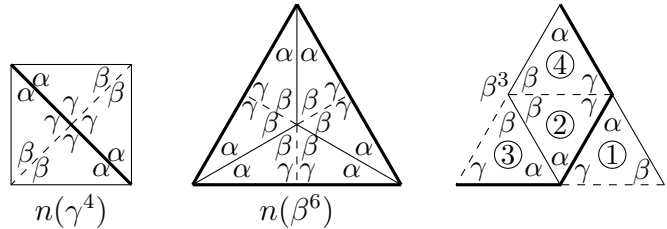

For $f=60$, we have $\thick\alpha\thin\alpha\thick\gamma\dash$ in $\alpha^5\beta\gamma$. This determines $T_1,T_2,T_3$ in the third of Figure \ref{3angleG}. Then $\beta_2\beta_3\cdots=\beta^6$ determines $T_4$. Then we get $\alpha_1\gamma_2\gamma_4\cdots=\alpha\gamma^2\cdots$, contradicting the AVC.

For $f=72$, by $\gamma\cdots=\gamma^4$, the tiling is a tiling of the rhombus $n(\gamma^4)$. The AVC $\{\beta^6,\gamma^4,\alpha^6\beta^2\}$ for the triangular tiling becomes the AVC $\{\bar{\beta}^3,\bar{\alpha}^3\bar{\beta}\}$ for the rhombus tiling, where $\bar{\alpha}=2\alpha$ and $\bar{\beta}=2\beta$. Since there is no degree $2$ vertex, the rhombus tiling is edge-to-edge. By the argument in Figure \ref{rhombus_emt1}, there is no such tiling.

\subsubsection*{Case. $\alpha^2\beta^4$ or $\alpha\beta^3\gamma$ is a vertex}

By the earlier argument, we know the list of vertices 
\[
\text{AVC}=
\{
\alpha^2\beta^4,
\alpha\beta^3\gamma,
\gamma^4,
\alpha^k(k\ge 6),
\alpha^k\beta^2(k\ge 4),
\alpha^k\beta\gamma,
\alpha^k\gamma^2
\}.
\]
By the edge consideration, we have the unique arrangements $\thick\alpha\thin\beta\dash\beta\thin\beta\dash\beta\thin\alpha\thick$, $\thick\alpha\thin\beta\dash\beta\thin\beta\dash\gamma\thick$, $\thick\alpha\thin\beta\dash\beta\thin\alpha\thick\alpha\thin\alpha\thick\cdots\thick\alpha\thin\alpha\thick$ of $\alpha^2\beta^4,\alpha\beta^3\gamma,\alpha^k\beta^2(k\ge 4)$.

In Figure \ref{3angleH}, we consider four tiles $T_1,T_2,T_3,T_4$ determined by $\thick\alpha\thin\alpha\thick\alpha\thin\alpha\thick$. Then $\gamma_2\thick\gamma_3\cdots=\gamma^4$ determines one $n(\gamma^4)$, including $T_5,T_6$. Then $\dash\beta_1\thin\beta_2\dash\beta_5\thin\cdots=\alpha^2\beta^4,\alpha\beta^3\gamma$. By the unique arrangements of $\alpha^2\beta^4,\alpha\beta^3\gamma$, we get $\dash\beta_1\thin\beta_2\dash\beta_5\thin\cdots=\dash\beta_1\thin\beta_2\dash\beta_5\thin\alpha\thick\cdots$. This determines $T_7$. Similarly, we determine $T_8$. Then we get $\dash\beta_7\thin\alpha_5\thick\alpha_6\thin\beta_8\dash\cdots=\alpha^2\beta^4,\alpha^k\beta^2(k\ge 4)$. Since $\dash\beta_7\thin\alpha_5\thick\alpha_6\thin\beta_8\dash$ is incompatible with the unique arrangement of $\alpha^k\beta^2$, we get $\dash\beta_7\thin\alpha_5\thick\alpha_6\thin\beta_8\dash\cdots=\alpha^2\beta^4$. This determines $T_9,T_{10}$. Moreover, by $\alpha^2\beta^4$ and $\alpha+\beta>\gamma$, we know $\alpha\beta^3\gamma$ is not a vertex. This implies $\alpha_8\beta_3\beta_4\beta_6\cdots=\alpha^2\beta^4$, which determines $T_{11},T_{12}$. Then $\gamma_8\gamma_{10}\gamma_{12}\cdots=\gamma^4$ determines one $n(\gamma^4)$, including $T_{13}$.

\begin{figure}[htp]
\centering
\begin{tikzpicture}[>=latex]


\foreach \a in {1,-1}
{
\begin{scope}[xscale=\a]

\draw[dashed]
	(0,0.6) -- (2.4,0.6)
	(0,-0.6) -- (2.4,-0.6);

\draw[line width=1.2]
	(0,-0.6) -- (0,1.6)
	(1.2,0.6) -- (1.2,-1.6)
	(2.4,-0.6) -- (2.4,1.6);
	
\draw
	(0,-1.6) -- (0,-0.6) -- (1.2,0.6) -- (1.2,1.6)
	(1.2,0.6) -- (2.4,-0.6) -- (2.4,-1.6);

\node at (0.2,0.45) {\small $\gamma$};
\node at (0.8,0.4) {\small $\beta$};
\node at (0.2,-0.2) {\small $\alpha$};

\node at (1,-0.45) {\small $\gamma$};
\node at (0.45,-0.4) {\small $\beta$};
\node at (1,0.2) {\small $\alpha$};

\node at (1.4,-0.45) {\small $\gamma$};
\node at (1.95,-0.4) {\small $\beta$};
\node at (1.4,0.2) {\small $\alpha$};

\node at (2.2,0.45) {\small $\gamma$};
\node at (1.65,0.4) {\small $\beta$};
\node at (2.2,-0.2) {\small $\alpha$};

\node at (0.2,0.8) {\small $\gamma$};
\node at (1,0.8) {\small $\beta$};
\node at (0.6,1.5) {\small $\alpha$};

\node at (2.2,0.8) {\small $\gamma$};
\node at (1.4,0.8) {\small $\beta$};
\node at (1.8,1.5) {\small $\alpha$};

\node at (1,-0.8) {\small $\gamma$};
\node at (0.2,-0.8) {\small $\beta$};
\node at (0.6,-1.5) {\small $\alpha$};

\node at (1.4,-0.8) {\small $\gamma$};
\node at (2.2,-0.8) {\small $\beta$};
\node at (1.8,-1.5) {\small $\alpha$};

\end{scope}
}

\node[inner sep=0.5,draw,shape=circle] at (-1.8,1.1) {\small $1$};
\node[inner sep=0.5,draw,shape=circle] at (-0.6,1.1) {\small $2$};
\node[inner sep=0.5,draw,shape=circle] at (0.6,1.1) {\small $3$};
\node[inner sep=0.5,draw,shape=circle] at (1.8,1.1) {\small $4$};

\node[inner sep=0.5,draw,shape=circle] at (-0.43,0.17) {\small $5$};
\node[inner sep=0.5,draw,shape=circle] at (0.43,0.17) {\small $6$};
\node[inner sep=0.5,draw,shape=circle] at (-0.77,-0.17) {\small $7$};
\node[inner sep=0.5,draw,shape=circle] at (0.77,-0.17) {\small $8$};

\node[inner sep=0.5,draw,shape=circle] at (-0.6,-1.1) {\small $9$};
\node[inner sep=0,draw,shape=circle] at (0.6,-1.1) {\footnotesize $10$};

\node[inner sep=0,draw,shape=circle] at (1.97,0.17) {\footnotesize $11$};
\node[inner sep=0,draw,shape=circle] at (1.63,-0.17) {\footnotesize $12$};
\node[inner sep=0,draw,shape=circle] at (1.8,-1.1) {\footnotesize $13$};
\node[inner sep=0,draw,shape=circle] at (3,1.1) {\footnotesize $14$};
\node[inner sep=0,draw,shape=circle] at (2.83,-0.17) {\footnotesize $15$};

\draw[dashed]
	(3.6,0.6) -- (3.6,1.6)
	(3.6,-0.6) -- (2.4,-0.6);

\draw
	(3.6,0.6) -- (2.4,0.6) -- (3.6,-0.6);

\node at (3,1.5) {\small $\gamma$};
\node at (3.4,0.8) {\small $\beta$};
\node at (2.6,0.8) {\small $\alpha$};

\node at (2.6,-0.45) {\small $\gamma$};
\node at (3.1,-0.4) {\small $\beta$};
\node at (2.6,0.2) {\small $\alpha$};

\node at (3.2,0.37) {\small $\alpha^{k-2}$};
	
\end{tikzpicture}
\caption{Proposition \ref{4a}: $\thick\alpha\thin\alpha\thick\alpha\thin\alpha\thick$ and $\thick\alpha\thin\alpha\thick\alpha\thin\alpha\thick\gamma\dash$.}
\label{3angleH}
\end{figure}

If $\thick\alpha\thin\alpha\thick\alpha\thin\alpha\thick$ is part of $\thick\alpha\thin\alpha\thick\alpha\thin\alpha\thick\gamma\dash$, then we further determine $T_{14}$. Then $\alpha_{14}\gamma_4\gamma_{11}\cdots=\alpha^k\gamma^2$ determines $T_{15}$. Then we get $\alpha_{11}\beta_{12}\beta_{13}\gamma_{15}\cdots=\alpha\beta^3\gamma$, contradicting no $\alpha\beta^3\gamma$. Therefore there is no $\thick\alpha\thin\alpha\thick\alpha\thin\alpha\thick\gamma\dash$. This implies $\alpha^k\gamma^2$ is not a vertex, and $k=3$ in $\alpha^k\beta\gamma$. However, the angle sums of $\alpha^2\beta^4,\alpha^3\beta\gamma,\gamma^4$ and the angle sum for triangle imply $\alpha=\frac{2}{5}\pi<\beta=\frac{3}{10}\pi$, contradicting $\alpha<\beta$. Therefore $\alpha^k\beta\gamma$ is not a vertex, and we get the updated list of vertices 
\[
\text{AVC} =
\{
\alpha^2\beta^4,
\gamma^4,
\alpha^{\frac{f}{4}},
\alpha^{\frac{f}{8}+1}\beta^2
\}.
\]
This implies $\gamma\cdots=\gamma^4$. Therefore the whole tiling is actually a tiling of the rhombus $n(\gamma^4)$, with 
\[
\text{AVC} =
\{\bar{\alpha}\bar{\beta}^2,
\bar{\alpha}^{\frac{\bar{f}}{2}},
\bar{\alpha}^{\frac{\bar{f}+2}{4}}\bar{\beta}
\},\quad
\bar{\alpha}=2\alpha,\;
\bar{\beta}=2\beta,\;
\bar{f}=\tfrac{f}{4}.
\]

This is the same as the AVC \eqref{triavc7}. Since there is no degree $2$ vertex, the rhombus tiling is edge-to-edge. By the earlier argument, the rhombus tiling is $E_{\square}^R1$, or its flip modification $FE_{\square}^R1$. Then the tilings for the proposition are the triangular subdivisions $E_{\triangle}3$ and $FE_{\triangle}3$ of the two rhombus tilings.
\end{proof}

After Propositions \ref{2a2b}, \ref{3abc}, \ref{4a}, we may assume $\alpha\beta\gamma^3,\beta^2\gamma^2,\gamma^4$ are not vertices. Then the general AVC \eqref{triavc1} is reduced to $\text{AVC}=\{\alpha^k\gamma^2,\alpha^k\beta^l,\alpha^k\beta^l\gamma\}$. By the parity lemma, we have $k\ge 2$ in $\alpha^k\gamma^2$ and $k\ge 1$ in $\alpha^k\beta^l\gamma$. Then by applying the counting lemma to $\alpha,\gamma$, we get $\text{AVC}=\{\alpha^2\gamma^2,\beta^l,\alpha\beta^l\gamma\}$. By applying the counting lemma again to $\beta,\gamma$, we know either $\alpha^2\gamma^2$ is a vertex, or $\alpha\beta\gamma$ is the only vertex. The case $\alpha^2\gamma^2$ is a vertex is handled by Proposition \ref{2a2c}. As explained at the beginning of the section, the vertex $\alpha\beta\gamma$ implies the tiling is the tetrahedron $P_4$. This completes the classification of tilings of the sphere by congruent general triangles.

Next we classify tilings by congruent kites, given by the second of Figure \ref{quad} and the left of Figure \ref{kite1}. We may divide a kite into two congruent triangles. The triangle may be general, or isosceles. Therefore a kite tiling induces a general or isosceles triangular tiling. 

In the general triangular tiling induced from a kite tiling, any vertex has even number of $\bar{\beta}$. Therefore all vertices are even. In the isosceles triangular tiling induced from a kite tiling, any vertex has even number of $\bar{\alpha}$. Then by the parity lemma (the last part of Lemma \ref{parity}), we also know all vertices have even degree.

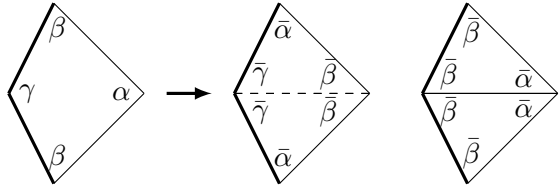
\begin{figure}[htp]
\centering
\begin{tikzpicture}[>=latex]

\foreach \a in {0,3,5.5}
\foreach \b in {-1,1}
{
\begin{scope}[xshift=\a cm, yscale=\b]

\draw
	(0,1.2) -- (1.2,0);

\draw[line width=1.2]
	(-0.6,0) -- (0,1.2);

\end{scope}
}

\node at (0.9,0) {\small $\alpha$};
\node at (0.05,0.85) {\small $\beta$};
\node at (0.05,-0.85) {\small $\beta$};
\node at (-0.35,0) {\small $\gamma$};

\draw[->, very thick]
	(1.5,0) -- ++(0.6,0);

\foreach \b in {-1,1}
{
\begin{scope}[xshift=3 cm, yscale=\b]

\draw[dashed]
	(-0.6,0) -- (1.2,0);

\node at (0.65,0.25) {\small $\bar{\beta}$};
\node at (0.05,0.85) {\small $\bar{\alpha}$};
\node at (-0.25,0.25) {\small $\bar{\gamma}$};

\end{scope}
}

\foreach \b in {-1,1}
{
\begin{scope}[xshift=5.5 cm, yscale=\b]

\draw
	(-0.6,0) -- (1.2,0);

\node at (0.75,0.2) {\small $\bar{\alpha}$};
\node at (0.05,0.8) {\small $\bar{\beta}$};
\node at (-0.25,0.25) {\small $\bar{\beta}$};

\end{scope}
}

\end{tikzpicture}
\caption{Kite tiling to general or isosceles triangular tiling.}
\label{kite1}
\end{figure}

Conversely, we start with a general triangular tiling and try to derive a kite tiling. It is necessary that all vertices have even degree. Moreover, we need to remove all $a$-edges, or remove all $b$-edges, or remove all $c$-edges, and then examine whether we actually get a kite tiling after the removal. 

For the barycentric subdivisions $B_{\triangle}P_8,B_{\triangle}P_{20}$ of the octahedron and icosahedron, and the flip modification $FB_{\triangle}P_8$, we may remove all $a$-edges to get the quadrilateral subdivisions $Q_{\square}P_8,Q_{\square}P_{20}$, and the flip modification $FQ_{\square}P_8$. We cannot remove all $b$-edges or all $c$-edges from $B_{\triangle}P_8,B_{\triangle}P_{20},FB_{\triangle}P_8$ because this produces degree $2$ vertices. 

For the earth map tiling $E_{\triangle}1$, we may remove all $a$-edges or all $b$-edges, and get the kite earth map tiling $E_{\square}^K1$. We cannot remove all $c$-edges because this produces degree $2$ vertices. The flip modification $FE_{\triangle}1$ has odd degree vertex $\bar{\alpha}^{2q+1}\bar{\beta}\bar{\gamma}$, and the further flip modification $F'E_{\triangle}1$ has odd degree vertex $\bar{\alpha}\bar{\beta}^3\bar{\gamma}$. Therefore these modifications cannot give kite tilings.

For the earth map tiling $E_{\triangle}3$, removing all $a$-edges produces quadrilaterals not suitable for tiling (third of Figure \ref{edges2}). Moreover, removing all $b$-edges or all $c$-edges produces degree $2$ vertices. The same reason applies to the flip modification $FE_{\triangle}3$. Therefore $E_{\triangle}3$ and $FE_{\triangle}3$ cannot give kite tilings.

Now we start with an isosceles triangular tiling and try to derive a kite tiling. In the last part of Section \ref{rhombustiling}, we found isosceles triangular tilings are the following: The tetrahedron tiling $P_4$, the triangular subdivisions $T_{\triangle}P_i$, the simple triangular subdivisions $S_{\triangle}P_6$, the earth map tilings $E_{\triangle}^I1,E_{\triangle}^J1,E_{\triangle}4$, and the flip modifications $FE_{\triangle}^I1,FE_{\triangle}^J1,FE_{\triangle}4$. By the requirement that all vertices have even degree, we may dismiss $P_4$, and all except $E_{12\triangle}2$ among $S_{\triangle}P_6$, and $T_{\triangle}P_4,T_{\triangle}P_8,T_{\triangle}P_{12},T_{\triangle}P_{20}$, and $E_{\triangle}4,FE_{\triangle}^J1, FE_{\triangle}4$. For the remaining $T_{\triangle}P_6,E_{\triangle}^I1,E_{\triangle}^J1,FE_{\triangle}^I1$, we try to reduce all consecutive $\bar{\alpha}^{2k}$ to $\alpha^{k}$ for $\alpha=2\bar{\alpha}$, and determine whether we get a kite tiling. Applying the process to $T_{\triangle}P_6$, we get degree $2$ vertex. Applying the process to $E_{\triangle}^I1,E_{\triangle}^J1$, we get the earth map tiling $E_{\square}^K1$. Note that $E_{\triangle}^I1$ has two $\bar{\alpha}^{2q}$. For each $\bar{\alpha}^{2q}$, there are two choices in combining consecutive $\bar{\alpha}\bar{\alpha}$ into $\alpha$. The choices for the two $\bar{\alpha}^{2q}$ need to be related in alternating way in order to get a kite tiling. Due to this requirement, we cannot get a kite tiling from the flip modification $FE_{\triangle}^I1$.

\section{Tiling by General Quadrilateral}
\label{generalquad}

The general quadrilateral is the first of Figure \ref{quad}. Any $b$-edge is shared by exactly two tiles. We call one tile the {\em $b$-companion} of the other. The two tiles may be matched or twisted, according to the locations of $\beta,\gamma$. See Figure \ref{tiling_aabc}. One consequence of the companion pair is that $f$ is twice of the number of $b$-edges. In particular, we know $f$ is even.

\begin{figure}[htp]
\centering
\begin{tikzpicture}


\draw	
	(-1,0) -- (-1,1) -- (1,1) -- (1,0);

\draw[line width=1.2]
	(0,1) -- (0,0);

\draw[dashed]
	(-1,0) -- (1,0);
	
\node at (0.8,0.8) {\small $\alpha$};
\node at (0.2,0.8) {\small $\beta$};
\node at (0.8,0.2) {\small $\delta$};
\node at (0.2,0.2) {\small $\gamma$};

\node at (-0.8,0.8) {\small $\alpha$};
\node at (-0.2,0.8) {\small $\beta$};
\node at (-0.8,0.2) {\small $\delta$};
\node at (-0.2,0.2) {\small $\gamma$};

\node at (0,-0.3) {matched};	


\begin{scope}[xshift=3cm]

\draw	
	(-1,0) -- (-1,1) -- (0,1)
	(0,0) -- (1,0) -- (1,1);

\draw[line width=1.2]
	(0,1) -- (0,0);

\draw[dashed]
	(-1,0) -- (0,0)
	(1,1) -- (0,1);

\node at (0.8,0.2) {\small $\alpha$};
\node at (0.2,0.2) {\small $\beta$};
\node at (0.8,0.8) {\small $\delta$};
\node at (0.2,0.8) {\small $\gamma$};

\node at (-0.8,0.8) {\small $\alpha$};
\node at (-0.2,0.8) {\small $\beta$};
\node at (-0.8,0.2) {\small $\delta$};
\node at (-0.2,0.2) {\small $\gamma$};
	
\node at (0,-0.3) {twisted};	

\end{scope}

\end{tikzpicture}
\caption{$b$-companion.}
\label{tiling_aabc}
\end{figure}
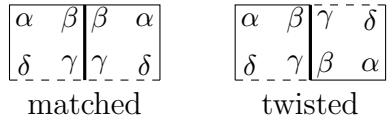

By the angle sum \eqref{anglesum} for quadrilateral, we know $\alpha\beta\gamma\delta\cdots$ is not a vertex. 

\begin{proposition}\label{bnotd}
Tilings of the sphere by congruent general quadrilaterals, such that $\beta=\delta$, is the earth map tiling $E_{\square}1$. 
\end{proposition}

\begin{proof}
In the first of Figure \ref{dispi}, we already have the isosceles triangle $\triangle ABD$. Then $\beta=\delta$ implies either $b=c$ (i.e., $\triangle BCD$ is isosceles), or $\gamma=\pi$ (i.e., $BCD$ is one arc). Since $b\ne c$ for the general quadrilateral, we get $\gamma=\pi$.

By $\gamma=\pi$, we know $\gamma^2\cdots$ is not a vertex. Therefore two $b$-companion tiles must be twisted as $T_1,T_2$ in Figure \ref{dispi}. The same happens to the similar $c$-companion pairs, such as $T_2,T_3$. Then we get a sequence of tiles along the two sides of a great circle indicated by shaded horizontal line in the second of  Figure \ref{dispi}.

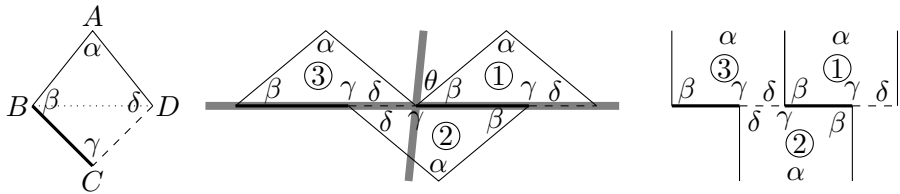
\begin{figure}[htp]
\centering
\begin{tikzpicture}


\begin{scope}[xshift=-3.1cm]

\draw
	(-0.8,0) -- (0,1) -- (0.8,0);
	
\draw[dashed]
	(0.8,0) -- (0,-0.8);
	
\draw[line width=1.2]
	(-0.8,0) -- (0,-0.8);

\draw[dotted]
	(-0.8,0) -- (0.8,0);
	
\node at (0,0.75) {\small $\alpha$};
\node at (-0.55,0.05) {\small $\beta$};
\node at (0,-0.55) {\small $\gamma$};
\node at (0.55,0.05) {\small $\delta$};

\node at (0,1.2) {\small $A$};
\node at (-1,0) {\small $B$};
\node at (0,-1) {\small $C$};
\node at (1,0) {\small $D$};

\end{scope}


\draw[gray,line width=3]
	(-1.6,0) -- (3.9,0)
	(1.1,-1) -- (1.3,1);

\foreach \a in {0,1}
{
\begin{scope}[xshift=2.4*\a cm]

\draw
	(-1.2,0) -- (0,1) -- (1.2,0);

\draw[line width=1.2]
	(-1.2,0) -- (0.3,0);

\draw[dashed]
	(0.3,0) -- (1.2,0);
	
\node at (0,0.8) {\small $\alpha$};
\node at (-0.7,0.2) {\small $\beta$};
\node at (0.3,0.2) {\small $\gamma$};
\node at (0.65,0.2) {\small $\delta$};

\end{scope}
}

\draw
	(0.3,0) -- (1.5,-1) -- (2.7,0);

\node at (1.5,-0.8) {\small $\alpha$};
\node at (2.25,-0.2) {\small $\beta$};
\node at (1.2,-0.2) {\small $\gamma$};
\node at (0.8,-0.2) {\small $\delta$};

\node at (1.4,0.35) {\small $\theta$};

\node[inner sep=0.5,draw,shape=circle] at (-0.1,0.4) {\small $3$};
\node[inner sep=0.5,draw,shape=circle] at (1.6,-0.4) {\small $2$};
\node[inner sep=0.5,draw,shape=circle] at (2.3,0.4) {\small $1$};

\begin{scope}[xshift=4.6cm]

	
\foreach \a in {0,1}
{
\begin{scope}[xshift=1.5*\a cm]

\draw
	(0,0) -- (0,1) 
	(1.5,0) -- (1.5,1)
	(0.9,0) -- (0.9,-1);

\draw[line width=1.2]
	(0,0) -- (0.9,0);

\draw[dashed]
	(0.9,0) -- (1.5,0);

\node at (0.75,0.9) {\small $\alpha$};
\node at (0.2,0.2) {\small $\beta$};
\node at (0.9,0.2) {\small $\gamma$};
\node at (1.3,0.2) {\small $\delta$};

\end{scope}
}

\node at (1.65,-0.9) {\small $\alpha$};
\node at (2.2,-0.25) {\small $\beta$};
\node at (1.5,-0.2) {\small $\gamma$};
\node at (1.1,-0.2) {\small $\delta$};

\node[inner sep=0.5,draw,shape=circle] at (2.2,0.5) {\small $1$};
\node[inner sep=0.5,draw,shape=circle] at (0.7,0.5) {\small $3$};
\node[inner sep=0.5,draw,shape=circle] at (1.7,-0.5) {\small $2$};

\end{scope}

\end{tikzpicture}
\caption{Proposition \ref{bnotd}: $\beta\gamma\delta$ is a vertex, and gives earth map tiling.}
\label{dispi}
\end{figure}

If $\beta\gamma\delta$ is not a vertex,  as in the second of Figure \ref{dispi}, then we have an angle $\theta$ next to $\beta_1$. By no $\alpha\beta\gamma\delta\cdots,\gamma^2\cdots$, we know $\theta=\beta,\delta$. Then we get another sequence of tiles on the two sides of another great circle indicated by another shaded line passing through the vertex $\beta_1\gamma_2\delta_3\cdots$. Since the two sequences of tiles overlap, we get a contradiction. Therefore $\beta\gamma\delta$ is a vertex, as in the third picture. Then we get the earth map tiling $E_{\square}1$.
\end{proof}

After Proposition \ref{bnotd}, we will always implicitly assume $\beta\ne\delta$.

By the parity lemma (the second part of Lemma \ref{parity}), the numbers of $\beta,\gamma,\delta$ at any vertex have the same parity. Depending on whether the parity is even or odd, we call the vertex {\em even} or {\em odd}. An odd vertex is of the form $\beta\gamma\delta\cdots$. Then by the angle sum \eqref{anglesum} for quadrilateral, an odd vertex has no $\alpha$. Therefore $\alpha\cdots$ is an even vertex.

By \eqref{quadvcount3}, there are degree $3$ vertices. By the parity lemma and Lemma \ref{nod}, we know $\alpha^3,\alpha\beta^2,\alpha\delta^2,\beta\gamma\delta$ are all the degree $3$ vertices. 

\begin{proposition}\label{bcd}
Tilings of the sphere by congruent general quadrilaterals, such that $\beta\gamma\delta$ is a vertex, is the earth map tiling $E_{\square}1$.
\end{proposition}

In the geometrical discussion in the proof, we show that, for each $f$, the tiling allows two free parameters that can be parameterised by the points in the region in Figure \ref{bcdC}.

\begin{proof}
The vertex $\beta\gamma\delta$ implies $\beta\gamma\delta$ is the only odd vertex. 

Suppose $\gamma^2\cdots$ is a vertex. By the balance lemma, we know $\beta^2\cdots,\delta^2\cdots$ are also vertices. Then $\beta,\gamma,\delta<\pi$. By $\beta\gamma\delta$, this implies $\beta+\delta,\beta+\gamma,\gamma+\delta>\pi$. Therefore $\beta^2\gamma^2\cdots,\beta^2\delta^2\cdots,\gamma^2\delta^2\cdots$ are not vertices. By no $\beta^2\gamma^2\cdots,\gamma^2\delta^2\cdots$, an even vertex $\gamma\cdots$ has no $\beta,\delta$. Since $\beta\gamma\delta$ is the only odd vertex, this implies $\beta\gamma\cdots=\gamma\delta\cdots=\beta\gamma\delta$.

The vertex $\beta\gamma\delta$ determines $T_1,T_2,T_3$ in Figure \ref{bcdA}. Then $
\gamma_1\delta_3\cdots=\beta\gamma\delta$ determines $T_4$. Then the argument can be repeated for the new $\beta\gamma\delta=\beta_4\gamma_1\delta_3$. More repetitions give the earth map tiling $E_{\square}1$.

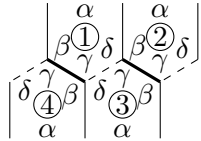
\begin{figure}[htp]
\centering
\begin{tikzpicture}


\foreach \a in {0,1,2}
\draw[xshift=\a cm]
	(0,0.9) -- (0,0.15)
	(-0.5,-0.9) -- (-0.5,-0.15);

\foreach \a in {0,1,2}
\draw[dashed, xshift=\a cm]
	(0,0.15) -- (-0.5,-0.15);	

\foreach \a in {0,1}
{
\begin{scope}[xshift=\a cm]

\draw[line width=1.2]
	(0,0.15) -- (0.5,-0.15);	
	
\node at (0.5,0.8) {\small $\alpha$};
\node at (0.2,0.3) {\small $\beta$};
\node at (0.5,0.1) {\small $\gamma$};
\node at (0.8,0.3) {\small $\delta$};

\node at (0,-0.8) {\small $\alpha$};
\node at (-0.3,-0.25) {\small $\delta$};
\node at (0,-0.1) {\small $\gamma$};
\node at (0.3,-0.35) {\small $\beta$};

\end{scope}
}

\node[inner sep=0.5,draw,shape=circle] at (0.5,0.45) {\small $1$};
\node[inner sep=0.5,draw,shape=circle] at (1.5,0.45) {\small $2$};
\node[inner sep=0.5,draw,shape=circle] at (1,-0.45) {\small $3$};
\node[inner sep=0.5,draw,shape=circle] at (0,-0.45) {\small $4$};
	
\end{tikzpicture}
\caption{Proposition \ref{bcd}: Earth map tiling.}
\label{bcdA}
\end{figure}

\medskip

\noindent{\em Geometry of Quadrilateral}

\medskip

The angle sum of $\beta\gamma\delta$ and the angle sum for quadrilateral imply
\[
\alpha=\tfrac{4}{f}\pi,\;
\beta+\gamma+\delta=2\pi.
\]
A simple general quadrilateral is suitable for the earth map tiling, if and only if it satisfies equalities above. Geometrically, this means $\alpha=\frac{4}{f}\pi$ and the area of the quadrilateral is $\alpha$.

In the first and second of Figure \ref{bcdB}, we first fix $AB=AD=a<\pi$, and $\angle BAD=\alpha$. We have the $2$-gon $G$ with angle $\alpha$ and vertices $A$ and its antipode $A^*$. In the first picture, we connect $A$ and $A^*$ by any half circle inside $G$. Let $C$ be the point on the half circle satisfying $A^*C=a$. Then $\triangle ABC$ and $\triangle A^*BC$ are congruent, and $\triangle ACD$ and $\triangle A^*CD$ are also congruent. Therefore $\text{Area}(\square ABCD)=\text{Area}(\triangle ABC)+\text{Area}(\triangle ACD)=\frac{1}{2}\text{Area}(G)=\alpha$. In other words, the quadrilateral can tile the earth map tiling in the second of Figure \ref{bcdA}.

We note that $A^*C>a$ implies $\text{Area}(\square ABCD)<\alpha$, and $A^*C<a$ implies $\text{Area}(\square ABCD)>\alpha$. Therefore for any given $a$ and any half circle in $G$ connecting $A$ and $A^*$, there is a unique $C$ on the half circle, such that $\square ABCD$ is geometrically suitable for the earth map tiling. If we further vary $0<a<\pi$, then we find $C$ can be any point in the interior of $G$. This corresponds to the first grey region Figure \ref{bcdC}. The first of Figure \ref{bcdC} is the normal earth view, where the equator is the horizontal middle line. The second of Figure \ref{bcdC} is the polar view, where the equator is the inner circle of radius $1$. From the viewpoint of $A$, the range of the first region is $(-\frac{1}{2}\alpha,\frac{1}{2}\alpha)$.

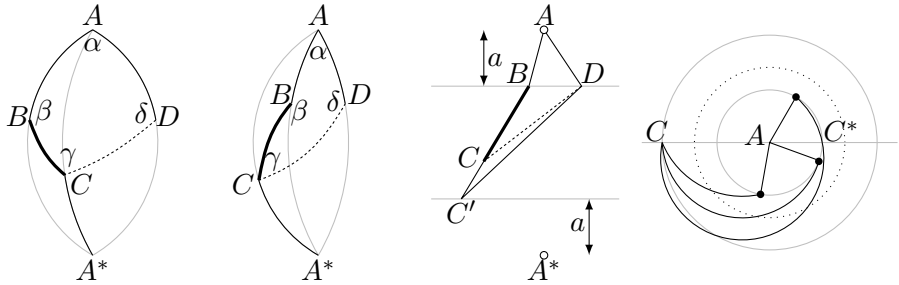
\begin{figure}[htp]
\centering
\begin{tikzpicture}[>=latex]


\draw[gray!50]
	(0,1.5) arc (60:-60:1.732)
	(0,1.5) arc (120:240:1.732)
	(0,1.5) arc (150:210:3);
	
\draw
	(0,1.5) arc (60:10:1.732) 
	(0,1.5) arc (120:170:1.732)
	(0,-1.5) arc (210:188:3);

\draw[line width=1.2]
	(-0.37,-0.43) to[out=140, in=-70] (-0.84,0.3);

\draw[dash pattern=on 1pt off 1pt]	
	(0.84,0.3) to[out=220, in=20] (-0.37,-0.43);

\node at (0,1.7) {\small $A$};
\node at (0,-1.7) {\small $A^*$};
\node at (-1,0.35) {\small $B$};
\node at (-0.15,-0.55) {\small $C$};
\node at (1,0.35) {\small $D$};

\node at (0,1.3) {\small $\alpha$};
\node at (-0.65,0.4) {\small $\beta$};
\node at (-0.32,-0.2) {\small $\gamma$};
\node at (0.65,0.4) {\small $\delta$};


\begin{scope}[xshift=3cm]

\draw[gray!50]
	(0,1.5) arc (150:210:3)
	(0,1.5) arc (30:-30:3)
	(0,1.5) arc (120:240:1.732);

\draw
	(0,1.5) arc (150:170:3)
	(0,1.5) arc (30:10:3)
	(0,-1.5) arc (240:196:1.732);

\draw[line width=1.2]
	(-0.36,0.52) to[out=230,in=80] (-0.79,-0.5);

\draw[dash pattern=on 1pt off 1pt]	
	(0.36,0.52) to[out=240, in=20] (-0.79,-0.5);
		
\node at (0,1.7) {\small $A$};
\node at (0,-1.7) {\small $A^*$};
\node at (-0.5,0.65) {\small $B$};
\node at (-1,-0.55) {\small $C$};
\node at (0.55,0.65) {\small $D$};

\node at (0,1.2) {\small $\alpha$};
\node at (-0.25,0.4) {\small $\beta$};
\node at (-0.6,-0.25) {\small $\gamma$};
\node at (0.2,0.55) {\small $\delta$};

\end{scope}


\begin{scope}[xshift=6cm]

\draw[gray!50]
	(-1.5,0.75) -- (1,0.75)
	(-1.5,-0.75) -- (1,-0.75);
	
\draw
	(0,1.5) -- (-0.2,0.75) -- (-1.1,-0.75)
	(0,1.5) -- (0.5,0.75) -- (-1.1,-0.75);

\draw[line width=1.2]
	(-0.2,0.75) -- (-0.8,-0.25);

\draw[dash pattern=on 1pt off 1pt]
	(0.5,0.75) -- (-0.8,-0.25);

\filldraw[fill=white]
	(0,1.5) circle (0.05)
	(0,-1.5) circle (0.05);	
	
\draw[<->]
	(-0.8,1.5) -- (-0.8,0.75);
\node at (-0.65,1.1) {\small $a$};

\draw[<->]
	(0.6,-1.5) -- (0.6,-0.75);
\node at (0.45,-1.1) {\small $a$};

\node at (0,1.7) {\small $A$};
\node at (0,-1.7) {\small $A^*$};
\node at (-0.35,0.92) {\small $B$};
\node at (0.65,0.92) {\small $D$};
\node at (-1,-0.2) {\small $C$};
\node at (-1.1,-0.92) {\small $C'$};

\end{scope}


\begin{scope}[xshift=9cm]

\draw[dotted]
	(0,0) circle (1);
	
\draw[gray!50]
	(0,0) circle ({1/0.7})
	(0,0) circle (0.7)
	(-1.7,0) -- (1.7,0);
	
\coordinate (CX) at (180:{1/0.7});
\coordinate (CCX) at (0:0.7);

\coordinate (DX) at (-100:0.7);
\coordinate (EX) at (60:0.7);
\coordinate (LX) at (-20:0.7);

\draw
	(0,0) -- (DX)
	(0,0) -- (EX)
	(0,0) -- (LX);
	
\arcThroughThreePoints{CX}{CCX}{DX};
\arcThroughThreePoints{CX}{CCX}{EX};
\arcThroughThreePoints{CX}{CCX}{LX};

\node at (-0.2,0.1) {\small $A$};
\node at (-1.5,0.15) {\small $C$};
\node at (0.95,0.15) {\small $C^*$};

\fill
	(-100:0.7) circle (0.05)
	(-21:0.7) circle (0.05)
	(60:0.7) circle (0.05);
	
\end{scope}
	
\end{tikzpicture}
\caption{Proposition \ref{bcd}: Geometry of quadrilateral.}
\label{bcdB}
\end{figure}

The discussion above is the case $C$ lies inside $G$. The case $C\not\in G$ means $\beta\ge \pi$ or $\delta\ge \pi$. Up to the symmetry of exchanging $\beta,\delta$, we consider the second of Figure \ref{bcdB}, corresponding to $\beta\ge\pi$. We have $\triangle ABD\sub \square ABCD$. Therefore $\text{Area}(\triangle ABD)<\alpha=\frac{1}{2}\text{Area}(G)$. This implies $a<\frac{1}{2}\pi$. Again, if $A^*C=a$, then $\triangle ABC$ and $\triangle ACD$ are congruent respectively to $\triangle A^*BC$ and $\triangle A^*CD$. This implies $\text{Area}(\square ABCD)=\alpha$, and the quadrilateral is geometrically suitable for the earth map tiling. 

To show the necessity of $A^*C=a$, we consider the third (schematic) picture of Figure \ref{bcdB}. The grey lines are the points of distance $a$ from $A$ and $A^*$. If $A^*C>a$, then we may extend $BC$ to intersect the lower grey line at $C'$. Then $A^*C'=a$ and $\triangle BCD\sub \triangle BC'D$. This implies $\text{Area}(\square ABCD)<\text{Area}(\square ABC'D)=\alpha$. Similarly, if $A^*C<a$, then we get $\text{Area}(\square ABCD)>\alpha$. This proves $\text{Area}(\square ABCD)=\alpha$ if and only if $A^*C=a$.

Like the case $C$ lying in $G$, the equality $A^*C=a$ implies the quadrilateral in the earth map tiling is uniquely parameterized by $C$. The fourth of Figure \ref{bcdB} is the stereographic projection from $A^*$ (therefore authentic picture) of various possibilities of the quadrilateral with $\beta\ge \pi$. Here $C$ is fixed, and $C^*$ is the antipode of $C$, and $\bullet$ are possible locations of $B$ and $D$. The dotted circle is the equator, the inner grey circle has distance $a$ to $A$, and the outer grey circle has distance $a$ to $A^*$. We find the condition for the quadrilateral to be simple is exactly $\angle CAB<\pi$. Therefore the possible location of $C$ is the second grey region in Figure \ref{bcdC}. From the viewpoint of $A$, the range of the second region is $(-\frac{1}{2}\alpha-\pi,-\frac{1}{2}\alpha]$.

\begin{figure}[htp]
\centering
\begin{tikzpicture}[>=latex]


\begin{scope}[xshift=0cm]
	
\fill[gray!50]
	(0,1) arc (30:-30:{1/sin(30)}) arc (210:150:{1/sin(30)});

\fill[gray!50]
	(0,-1) arc (-140:0:{1/cos(50)}) -- (0,0)
	(0,-1) arc (-40:-180:{1/cos(50)}) -- (0,0);

\draw[gray]
	(0,-1) arc (-30:0:{1/sin(30)})
	(0,-1) arc (210:180:{1/sin(30)});

\draw[dotted]
	(0,1) arc (140:0:{1/cos(50)})
	(0,1) arc (40:180:{1/cos(50)});
		
\filldraw[fill=white]
	(0,1) circle (0.05)
	(0,-1) circle (0.05);

\node at (0,1.23) {\small $A$};
\node at (0,-1.25) {\small $A^*$};

\node at (-0.6,0.5) {\scriptsize $-\frac{1}{2}\alpha$};
\node at (0.45,0.5) {\scriptsize $\frac{1}{2}\alpha$};
\node at (-1.3,1.3) {\scriptsize $-\frac{1}{2}\alpha-\pi$};
\node at (1.3,1.3) {\scriptsize $\frac{1}{2}\alpha+\pi$};

\node[inner sep=0.5,draw,shape=circle] at (0,0) {\small $1$};
\node[inner sep=0.5,draw,shape=circle] at (-1.3,-0.7) {\small $2$};
\node[inner sep=0.5,draw,shape=circle] at (1.3,-0.7) {\small $3$};

\end{scope}


\begin{scope}[xshift=5.5cm]

\fill[gray!50]
	(0,0) circle (1.4);

\fill[white]
	(-120:0.8) -- (0,0) -- (-60:0.8) arc (-60:240:0.8);

\draw[dotted]
	(0,0) -- (60:0.8)
	(0,0) -- (120:0.8);

\draw[gray]
	(-60:0.8) -- (-60:1.6)
	(60:0.8) -- (60:1.6)
	(-60:-0.8) -- (-60:-1.6)
	(60:-0.8) -- (60:-1.6)
	;

\filldraw[fill=white]
	(0,0) circle (0.05);
	
\draw[<->]
	(-120:1.1) arc (-120:-60:1.1);
\node[inner sep=0.5, draw, fill=gray!50, shape=circle] at (-90:1.1) {\small $1$};

\draw[<->]
	(240:1.1) arc (240:60:1.12);
\node[inner sep=0.5,draw, fill=gray!50, shape=circle] at (150:1.1) {\small $2$};

\draw[<->]
	(-60:1.1) arc (-60:120:1.08);
\node[inner sep=0.5,draw, fill=gray!50, shape=circle] at (30:1.1) {\small $3$};

\node at (0,0.23) {\small $A$};

\node at (0,-0.3) {\small $\alpha$};
\node at (1,-1.5) {\scriptsize $\frac{1}{2}\alpha$};
\node at (-0.95,-1.5) {\scriptsize $-\frac{1}{2}\alpha$};
\node at (-1.3,1.45) {\scriptsize $\frac{1}{2}\alpha+\pi$};
\node at (1.3,1.45) {\scriptsize $-\frac{1}{2}\alpha-\pi$};

\end{scope}
	
\end{tikzpicture}
\caption{Proposition \ref{bcd}: Moduli of earth map tilings.}
\label{bcdC}
\end{figure}
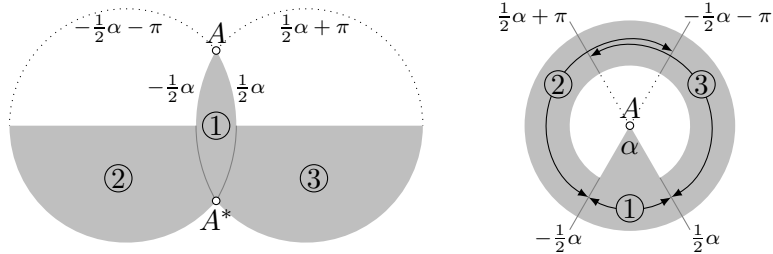

If $\delta\ge \pi$, then we may similarly find the quadrilateral is uniquely parameterized by $C$, which lies in the third grey region in Figure \ref{bcdC}. From the viewpoint of $A$, the range of the third region is $[\frac{1}{2}\alpha,\frac{1}{2}\alpha+\pi)$. We remark that the second and third regions actually overlap. This means the first of Figure \ref{bcdC} is not quite authentic, and the second picture is authentic.

Finally, we remark that the symmetry of exchanging $(\beta,b)$ with $(\delta,c)$ is the horizontal flip of Figure \ref{bcdC}. Moreover, the quadrilateral is reduced to a kite, i.e., $b=c$, if and only if the line $AC$ is at angle $0$. 
\end{proof}

After Proposition \ref{bcd}, we may assume $\beta\gamma\delta$ is not a vertex. For the remaining degree $3$ vertices $\alpha^3,\alpha\beta^2,\alpha\delta^2$, we divide the classification into the following cases:
\begin{enumerate}
\item $\alpha^3,\alpha\beta^2$ are the only degree $3$ vertices.
\item $\alpha^3$ is the only degree $3$ vertex.
\item $\alpha\beta^2$ is the only degree $3$ vertex.
\end{enumerate}
We also recall that $\beta\ne\delta$, after Proposition \ref{bnotd}.

By no $\beta\gamma\delta$ and the balance lemma (Lemma \ref{balance_aabc}), we know $\beta^2\cdots,\gamma^2\cdots,\delta^2\cdots$ are vertices. This implies $\beta,\gamma,\delta<\pi$. Moreover, since $\gamma,\delta$ do not appear at degree $3$ vertices in the three cases above, by Lemma \ref{deg3miss} and the parity lemma, we know $\gamma^4,\delta^4,\gamma^2\delta^2$ is a vertex. Then by \eqref{quadvcountf} and $f$ even, we get $f\ge 8$.

\begin{proposition}\label{abb-ccdd}
There is no tiling of the sphere by congruent general quadrilaterals, such that $\alpha\beta^2,\gamma^2\delta^2$ are vertices.
\end{proposition}

\begin{proof}
The angle sums of $\alpha\beta^2,\gamma^2\delta^2$ and the angle sum for quadrilateral imply
\[
\alpha=\tfrac{8}{f}\pi,\;
\beta=(1-\tfrac{4}{f})\pi,\;
\gamma+\delta=\pi.
\]
By $f\ge 8$, we get $\beta\ge\frac{1}{2}\pi$. By Lemma \ref{geometry10}, we get $R(\beta\gamma\delta)=\frac{4}{f}\pi<2\beta,2\gamma,2\delta$. This implies all vertices are even. 

The unique AAD $\thick^{\gamma}\beta^{\alpha}\thin^{\beta}\alpha^{\delta}\thin^{\alpha}\beta^{\gamma}\thick$ of $\alpha\beta^2$ implies $\alpha\delta\cdots$ is a vertex. If $\alpha\delta\cdots$ has $\beta,\gamma$, then (the even vertex) $\alpha\delta\cdots=\alpha\beta^2\delta^2\cdots,\alpha\gamma^2\delta^2\cdots$, contradicting $\alpha\beta^2,\gamma^2\delta^2$. Therefore $\alpha\delta\cdots=\alpha^k\delta^l$($k\ge 1$, $l\ge 2$). By $\alpha\beta^2$ and $\beta\ne\delta$, this implies $\beta>\delta$. 

We know $\gamma\cdots$ is an even vertex. By $\gamma^2\delta^2$ and $\beta>\delta$, we know $\beta^2\gamma^2\cdots$ is not a vertex. Therefore $\beta\gamma\cdots$ is not a vertex. Then by $\alpha\delta\cdots=\alpha^k\delta^l$ and Lemma \ref{nod}, we know $\alpha\gamma\cdots$ is not a vertex. Therefore $\gamma\cdots=\gamma^k\delta^l$. Then by $\gamma^2\delta^2$, we get $\gamma\cdots=\gamma^2\delta^2,\gamma^k(k\ge 4)$. Since $\alpha^k\delta^l$($k\ge 1$, $l\ge 2$) is a vertex, by applying the counting lemma to $\gamma,\delta$, this implies $\gamma^k$ is a vertex. By $k\ge 4$ in $\gamma^k$ and $\gamma+\delta=\pi$, we get $\gamma\le \frac{1}{2}\pi\le \delta<\beta$. 

By $\gamma^2\delta^2$ and $\gamma,\delta<\beta$, we get $R(\beta^2)<2\beta,2\gamma,2\delta$. Then by $\alpha\beta^2$ and all vertices being even, this implies $\beta\cdots=\beta^2\cdots=\alpha\beta^2$. 

\begin{figure}[htp]
\centering
\begin{tikzpicture}

\foreach \a in {1,-1}
{
\begin{scope}[xscale=\a]

\draw
	(0,0) -- (0,-1)
	(0,1) -- (2,1) 
	(0,-1) -- (1,-1)
	(1,0) -- (1,1);

\draw[line width=1.2]
	(0,0) -- (0,1)
	(2,0) -- (2,1)
	(1,0) -- (1,-1);

\draw[dashed]
	(0,0) -- (2,0);

\node at (0.8,0.8) {\small $\alpha$};
\node at (0.2,0.8) {\small $\beta$};
\node at (0.8,0.2) {\small $\delta$};
\node at (0.2,0.2) {\small $\gamma$};

\node at (0.2,-0.8) {\small $\alpha$};
\node at (0.8,-0.8) {\small $\beta$};
\node at (0.2,-0.2) {\small $\delta$};
\node at (0.8,-0.2) {\small $\gamma$};

\node at (1.2,0.8) {\small $\alpha$};
\node at (1.8,0.8) {\small $\beta$};
\node at (1.2,0.2) {\small $\delta$};
\node at (1.8,0.2) {\small $\gamma$};

\node at (1.2,-0.2) {\small $\gamma$};

\node at (0.8,1.2) {\small $\beta$};

\end{scope}
}

\node at (0,1.15) {\small $\alpha$};

\node[inner sep=0.5,draw,shape=circle] at (-0.5,0.5) {\small $1$};
\node[inner sep=0.5,draw,shape=circle] at (0.5,0.5) {\small $2$};
\node[inner sep=0.5,draw,shape=circle] at (0.5,-0.5) {\small $3$};
\node[inner sep=0.5,draw,shape=circle] at (-0.5,-0.5) {\small $4$};
\node[inner sep=0.5,draw,shape=circle] at (-1.5,0.5) {\small $5$};
\node[inner sep=0.5,draw,shape=circle] at (1.5,0.5) {\small $6$};

\end{tikzpicture}
\caption{Proposition \ref{abb-ccdd}: Vertex $\dash\gamma\thick\gamma\dash\delta\thick\delta\dash$.}
\label{abb-ccddA}
\end{figure}

The unique AAD $\dash^{\delta}\gamma^{\beta}\thick^{\beta}\gamma^{\delta}\dash^{\gamma}\delta^{\alpha}\thin^{\alpha}\delta^{\gamma}\dash$ of $\gamma^2\delta^2$ determines $T_1,T_2,T_3,T_4$ in Figure \ref{abb-ccddA}. Then $\gamma_4\delta_1\cdots=\gamma_3\delta_2\cdots=\gamma^2\delta^2$ determines $T_5,T_6$. We also have $\beta_1\beta_2\cdots=\alpha\beta^2$. Then the $\beta$ adjacent to $\alpha$ at $\beta_1\beta_2\cdots$ gives either $\alpha_1\alpha_5\beta\cdots$ or $\alpha_2\alpha_6\beta\cdots$, contradicting $\beta\cdots=\alpha\beta^2$.
\end{proof}

\begin{proposition}\label{aaa-abb}
Tiling of the sphere by congruent general quadrilaterals, such that $\alpha^3,\alpha\beta^2$ are the only degree $3$ vertices, is the flip modification $FQ_{\square}P_6$ of the quadrilateral subdivision of the cube.
\end{proposition}

\begin{proof}
We know one of $\gamma^4,\delta^4,\gamma^2\delta^2$ is a vertex. By Proposition \ref{abb-ccdd}, we may assume one of $\gamma^4,\delta^4$ is a vertex. The angle sum of one of $\gamma^4,\delta^4$, and the angle sums of $\alpha^3,\alpha\beta^2$, and the angle sum for quadrilateral, imply
\begin{align*}
\gamma^4 &\colon
\alpha=\beta=\tfrac{2}{3}\pi,\;
\gamma=\tfrac{1}{2}\pi,\;
\delta=(\tfrac{1}{6}+\tfrac{4}{f})\pi. \\
\delta^4 &\colon
\alpha=\beta=\tfrac{2}{3}\pi,\;
\gamma=(\tfrac{1}{6}+\tfrac{4}{f})\pi,\;
\delta=\tfrac{1}{2}\pi.
\end{align*}

The unique AAD $\thick^{\gamma}\beta^{\alpha}\thin^{\beta}\alpha^{\delta}\thin^{\alpha}\beta^{\gamma}\thick$ of $\alpha\beta^2$ implies $\alpha\delta\cdots$ is a vertex. Since an odd vertex has no $\alpha$, we know $\alpha\delta\cdots=\alpha\delta^2\cdots$ is an even vertex.

For the case $\delta^4$ is a vertex, we have $R(\alpha\delta^2)=\frac{1}{3}\pi<\alpha,\beta,2\gamma,\delta$. This implies the even vertex $\alpha\delta^2\cdots=\alpha\delta^2$, contradicting $\alpha\beta^2$ and $\beta\ne\delta$.

For the case $\gamma^4$ is a vertex, we consider the vertex $\beta\delta\cdots$ induced by the possible AADs $\thin^{\beta}\alpha^{\delta}\thin^{\beta}\alpha^{\delta}\thin^{\beta}\alpha^{\delta}\thin$ and $\thin^{\beta}\alpha^{\delta}\thin^{\beta}\alpha^{\delta}\thin^{\delta}\alpha^{\beta}\thin$ of $\alpha^3$. By the parity lemma, we know $\beta\delta\cdots=\beta\gamma\delta\cdots,\beta^2\delta^2\cdots$. We have $0<R(\beta\gamma\delta)=(\tfrac{2}{3}-\tfrac{4}{f})\pi<\alpha,\beta,2\gamma,4\delta$ and $R(\beta^2\delta^2)=(\tfrac{1}{3}-\tfrac{8}{f})\pi<\alpha,\beta,\gamma,2\delta$. Then by the parity lemma, we get $\beta\delta\cdots=\beta\gamma\delta^3,\beta^2\delta^2$. The angle sum of $\beta\gamma\delta^3$ or $\beta^2\delta^2$ further implies
\begin{align*}
\gamma^4,\beta\gamma\delta^3 &\colon 
\alpha=\beta=\tfrac{2}{3}\pi,\;
\gamma=\tfrac{1}{2}\pi,\;
\delta=\tfrac{5}{18}\pi,\;
f=36. \\
\gamma^4,\beta^2\delta^2 &\colon
\alpha=\beta=\tfrac{2}{3}\pi,\;
\gamma=\tfrac{1}{2}\pi,\;
\delta=\tfrac{1}{3}\pi,\;
f=24. 
\end{align*}

For $f=36$, the angle values imply there is no even vertex $\alpha\delta^2\cdots$. Therefore $f=24$, and we derive all vertices satisfying the parity lemma
\begin{equation}\label{aaa-abbeq1}
\text{AVC}
=\{\alpha^3,\alpha\beta^2,\alpha^2\delta^2,\beta^2\delta^2,\gamma^4,\alpha\delta^4,\delta^6\}.
\end{equation}

The unique AAD of $\gamma^4$ determines the neighborhood $n(\gamma^4)$ in the first of Figure \ref{aaa-abbB}. By $\gamma\cdots=\gamma^4$, the whole tiling is a tiling of six copies of $n(\gamma^4)$. 

The unique arrangements $\thin\alpha\thin\beta\thick\beta\thin$ and $\thin\beta\thick\beta\thin\delta\dash\delta\thin$ of $\alpha\beta^2$ and $\beta^2\delta^2$ imply $\beta\thin\beta\cdots$ is not a vertex. This further implies the unique AAD $\thin^{\beta}\alpha^{\delta}\thin^{\beta}\alpha^{\delta}\thin^{\beta}\alpha^{\delta}\thin$ of $\alpha^3$, and determines three $n(\gamma^4)$ around $\alpha^3$ in the second of Figure \ref{aaa-abbB}. The union of the three $n(\gamma^4)$ is the unique tiling $N(\alpha^3)$ of the extended neighborhood of $\alpha^3$. 

\begin{figure}[htp]
\centering
\begin{tikzpicture}[>=latex]


\draw
	(-1,-1) rectangle (1,1);

\draw[dashed]
	(0,-1) -- (0,1);

\draw[line width=1.2]
	(1,0) -- (-1,0);
	
\foreach \a in {-1,1}
\foreach \b in {-1,1}
{
\begin{scope}[xscale=\a, yscale=\b]

\node at (0.8,0.8) {\small $\alpha$};
\node at (0.8,0.2) {\small $\beta$};
\node at (0.2,0.2) {\small $\gamma$};
\node at (0.2,0.8) {\small $\delta$};
	
\end{scope}
}

\node at (0,-1.3) {$n(\gamma^4)$};


\begin{scope}[xshift=3.5cm]

\foreach \a in {0,1,2}
{
\begin{scope}[rotate=120*\a]

\draw
	(0,0) -- (90:2) -- (30:2) -- (-30:2);

\draw[line width=1.2]
	(90:1) -- (-1.732,0);
	
\draw[dashed]
	(90:1) -- (1.732,0);

\node at (30:0.2) {\small $\alpha$};
\node at (30:0.8) {\small $\gamma$};
\node at (30:1.2) {\small $\gamma$};
\node at (30:1.8) {\small $\alpha$};

\node at (0.2,0.65) {\small $\delta$};
\node at (0.2,1.1) {\small $\delta$};
\node at (0.2,1.7) {\small $\alpha$};

\node at (-0.2,0.65) {\small $\beta$};
\node at (-0.2,1.1) {\small $\beta$};
\node at (-0.2,1.7) {\small $\alpha$};

\node at (50:1.05) {\small $\gamma$};
\node at (130:1.05) {\small $\gamma$};

\node at (63:1.5) {\small $\beta$};
\node at (47:1.55) {\small $\beta$};
\node at (115:1.55) {\small $\delta$};
\node at (133:1.6) {\small $\delta$};

\end{scope}
}

\node at (0,-2.3) {$N(\alpha^3)$};

\end{scope}


\foreach \a in {1,-1}
{
\begin{scope}[xshift=6.5cm, yshift=-0.2 cm+ 1.2*\a cm]

\draw
	(-0.5,0) -- (5.5,0)
	(2,0) -- ++(0,-0.5); 

\draw[line width=1.2]
	(3,0) -- ++(0,-0.5);

\draw[dashed]
	(1,0) -- ++(0,-0.5)
	(5,0) -- ++(0,-0.5);
	
\fill
	(2,0) circle (0.07);
	
\node at (0,-0.2) {\small $\alpha$};
\node at (1.2,-0.2) {\small $\delta$};
\node at (0.8,-0.2) {\small $\delta$};
\node at (1.8,-0.2) {\small $\alpha$};
\node at (2.2,-0.2) {\small $\alpha$};
\node at (2.8,-0.2) {\small $\beta$};
\node at (3.2,-0.2) {\small $\beta$};
\node at (4,-0.2) {\small $\alpha$};
\node at (4.8,-0.2) {\small $\delta$};
\node at (5.2,-0.2) {\small $\delta$};

\node at (2.5,-0.8) {\small $N(\alpha^3)$};

\end{scope}
}

\begin{scope}[shift={(6.5cm, 1cm)}]

\draw[<->]
	(0,0.7) -- node[inner sep=0.5, fill=white] {\small $n(\gamma^4)$} (4,0.7);
	
\draw
	(0,0) -- ++(0,0.5)
	(4,0) -- ++(0,0.5); 

\draw[line width=1.2]
	(1,0) -- ++(0,0.5);

\draw[dashed]
	(3,0) -- ++(0,0.5);
	
\node at (0.2,0.2) {\small $\alpha$};
\node at (0.8,0.2) {\small $\beta$};
\node at (1.2,0.2) {\small $\beta$};
\node at (2,0.2) {\small $\alpha$};
\node at (3.2,0.2) {\small $\delta$};
\node at (2.8,0.2) {\small $\delta$};
\node at (3.8,0.2) {\small $\alpha$};

\end{scope}

\begin{scope}[shift={(7.5cm, -1.4cm)}]

\draw[<->]
	(0,0.7) -- node[inner sep=0.5, fill=white] {\small $n(\gamma^4)$} (4,0.7);
	
\draw
	(0,0) -- ++(0,0.5)
	(4,0) -- ++(0,0.5); 

\draw[line width=1.2]
	(3,0) -- ++(0,0.5);

\draw[dashed]
	(1,0) -- ++(0,0.5);
	
\node at (0.2,0.2) {\small $\alpha$};
\node at (0.8,0.2) {\small $\delta$};
\node at (1.2,0.2) {\small $\delta$};
\node at (2,0.2) {\small $\alpha$};
\node at (3.2,0.2) {\small $\beta$};
\node at (2.8,0.2) {\small $\beta$};
\node at (3.8,0.2) {\small $\alpha$};

\end{scope}

\end{tikzpicture}
\caption{Proposition \ref{aaa-abb}: $n(\gamma^4)$, $N(\alpha^3)$, and how to glue $n(\gamma^4)$ to $N(\alpha^3)$.}
\label{aaa-abbB}
\end{figure}
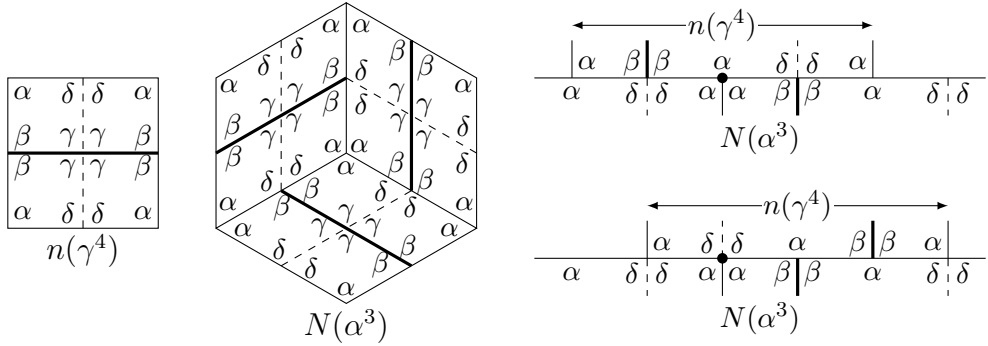

To find how the other three $n(\gamma^4)$ are glued to $N(\alpha^3)$, we draw the angles along the boundary of $N(\alpha^3)$ as the lower parts of the two lines in the third of Figure \ref{aaa-abbB}. Starting with $\alpha\thin\alpha\cdots$ at $\bullet$, we fill the angles above the lines. 

The top line assumes the vertex $\bullet$ is $\alpha\thin\alpha\cdots=\alpha^3$. By no $\beta^3\cdots$, we know $\beta\thick\beta\cdots$ next to $\alpha\thin\alpha\cdots$ is $\beta^2\delta^2$. This further implies $\delta\dash\delta\cdots$ next to $\alpha\thin\alpha\cdots$ is also $\beta^2\delta^2$. Then this determines how one $n(\gamma^4)$ above the line is attached. 

The bottom line assumes the vertex $\bullet$ is  $\alpha\thin\alpha\cdots=\alpha^2\delta^2$. This implies $\beta\thick\beta\cdots$ next to $\alpha\thin\alpha\cdots$ is $\alpha\beta^2$. This further implies $\alpha\cdots$ next to $\beta\thick\beta\cdots$ is $\alpha\beta\cdots=\alpha\beta^2$. Then this determines how one $n(\gamma^4)$ above the line is attached. 

There are three $\alpha\thin\alpha\cdots$ along the boundary of $N(\alpha^3)$. Each $\alpha\thin\alpha\cdots$ leads to one $n(\gamma^4)$ attached to $N(\alpha^3)$ either like the top line or the bottom line. Then it is easy to verify that all three $n(\gamma^4)$ must be attached to $N(\alpha^3)$ in the same way. Moreover, these three $n(\gamma^4)$ together form another copy of the extended neighborhood tiling $N(\alpha^3)$. Therefore the tiling is obtained by glueing two $N(\alpha^3)$ together according to the upper or lower line in Figure \ref{aaa-abbB}. If we glue according to the upper line, then we get the first of Figure \ref{flip8}. If we glue according to the lower line, then we get the fourth of Figure \ref{flip8}. Therefore we obtain the quadrilateral subdivision $Q_{\square}P_6$ of the cube, and its flip modification $FQ_{\square}P_6$.

We note that $Q_{\square}P_6$ has only $\alpha^3,\beta^2\delta^2,\gamma^4$, and fails the assumption of the proposition. In fact, the tiling will appear in Proposition \ref{aaa_general}. The flip modification $FQ_{\square}P_6$ has $\alpha^3,\alpha\beta^2,\alpha^2\delta^2,\beta^2\delta^2,\gamma^4$, and satisfies the assumption of the proposition.

\medskip

\noindent{\em Geometry of Quadrilateral}

\medskip

The three tiles around $\alpha^3$ form a neighborhood tiling $n(\alpha^3)$ in the first of Figure \ref{tiling4}. By $\beta+\delta=\pi$, $n(\alpha^3)$ is an equilateral triangle with angle $\gamma=\frac{1}{2}\pi$. The tiling $Q_{\square}P_6$ has eight $\alpha^3$. Therefore it is the tiling of eight copies of $n(\alpha^3)$, and is actually the regular octahedron with $n(\alpha^3)$ as faces. We also draw the gray $n(\alpha^3)$ to show how two adjacent $n(\alpha^3)$ are glued together. Then we explicitly see that the tiling is the quadrilateral subdivision $Q_{\square}P_8$ ($=Q_{\square}P_6$) of the regular octahedron. Moreover, the tiling can be parameterised by the length $b\in (0,\frac{1}{2}\pi)$, with $b$ and $\frac{1}{2}\pi-b$ corresponding to equivalent tilings.  

\begin{figure}[htp]
\centering
\begin{tikzpicture}

	
\begin{scope}[shift={(-4.3cm,-0.4cm)}]

\foreach \a in {0,1,2}
{
\begin{scope}[rotate=120*\a]

\draw
	(0,0) -- (0.3,-0.75);

\draw[line width=1.2]
	(90:1.5) -- ++(-60:1);

\draw[dashed]
	(90:1.5) -- ++(-120:1.6);

\node at (-0.08,0.18) {\small $\alpha$};
\node at (0.45,-0.55) {\small $\beta$};
\node at (0,-0.55) {\small $\delta$};
\node at (0,1.15) {\small $\gamma$};
	
\end{scope}
}

\node at (-1.1,0.7) {\small $n(\alpha^3)$};

\draw[gray, dashed]
	(-30:1.5) -- (2.6,1.5) -- (90:1.5);

\draw[gray]
	(30:1.5) -- (0.5,0.64)
	(30:1.5) -- (1.8,0.1)
	(30:1.5) -- (1.6,1.5);

\draw[gray, line width=1.2]
	(-30:1.5) -- ++(60:1)
	(2.6,1.5) -- (1.6,1.5);

\draw[dashed]
	(90:1.5) -- ++(-120:1.6);

\end{scope}


\draw[dotted]
	(-0.8,-0.8) rectangle (0.8,0.8);

\draw
	(0.8,0.8) -- (1.1,0) -- (0.8,-0.8) -- (0,-0.5) -- (-0.8,-0.8) -- (-1.1,0) -- (-0.8,0.8) -- (0,0.5) -- cycle;

\draw[line width=1.2]
	(0,0.5) -- (0,-0.5);

\draw[dashed]
	(-1.1,0) -- (1.1,0);

\draw[gray]
	(0.8,0.8) -- (1.6,1.1) -- (2.4,0.8) -- (2.1,0) -- (2.4,-0.8) -- (1.6,-1.1) -- (0.8,-0.8);

\draw[gray,line width=1.2]
	(1.1,0) -- (2.1,0);

\draw[gray,dashed]
	(1.6,1.1) -- (1.6,-1.1);

\node at (0.7,0.6) {\small $\alpha$};
\node at (0.2,0.2) {\small $\gamma$};

\node at (-0.3,0.72) {\tiny $\theta$};
\node at (-0.88,0.35) {\tiny $\theta$};

\node at (0,-1) {\small $n(\gamma^4)$};


\foreach \a in {1,-1}
{
\begin{scope}[xshift=4cm, xscale=\a]

\draw
	(0,1) -- (0.8,0.1) -- (0.5,-1);

\draw[line width=1.2]
	(-0.5,-1) -- (0.5,-1);

\draw[gray, dashed]
	(0,1) -- (0,-1);

\node at (0,0.7) {\small $\bar{\alpha}$};
\node at (0.55,0.05) {\small $\bar{\beta}$};
\node at (0.35,-0.75) {\small $\bar{\delta}$};

\end{scope}
}
	
\end{tikzpicture}
\caption{Proposition \ref{aaa-abb}: Three viewpoints of the tiling.}
\label{tiling4}
\end{figure}
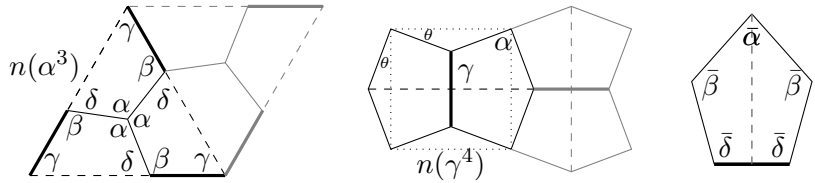

We may also view the tiling by considering the regular cube $P_6$, which is obtained by connecting the centers of eight $n(\alpha^3)$ together. The dotted square in the second of Figure \ref{tiling4} is one face of the regular cube, with angle $\alpha=\frac{2}{3}\pi$. We rotate half edges of the square by the same angle $\theta$, according to the directions indicated by the picture. We get a four way symmetric and equilateral octagon that still maintain angle $\alpha$ at the four square vertices. Then the octagon is divided into four congruent quadrilaterals. This changes the square face into an $n(\gamma^4)$. We carry out the changes for all six faces in the compatible way, with the compatibility indicated by the gray $n(\gamma^4)$. Then we explicitly see that the tiling is the quadrilateral subdivision $Q_{\square}P_6$ of the regular cube. Moreover, the tiling can be parameterised by $\theta\in (-\frac{1}{3}\pi,\frac{1}{3}\pi)$, with $\theta$ and $-\theta$ corresponding to equivalent tilings. 

We have yet the third viewpoint. By $2\gamma=\pi$, we may remove all the $c$-edges and get an edge-to-edge tiling of the sphere by $12$ congruent symmetric pentagons. The third of Figure \ref{tiling4} gives the symmetric pentagon, with angles $\bar{\alpha}=2\delta,\bar{\beta}=\alpha,\bar{\delta}=\beta$ and edges $\bar{a}=a,\bar{b}=2b$. The tiling by $12$ congruent symmetric pentagons is the deformed dodecahedron, and $Q_{\square}P_6$ is the simple quadrilateral subdivision of the dodecahedron. See the first of Figure \ref{simple_subdivision_dodecahedron}. 

For the flip modification $FQ_{\square}P_6$, we have $\bar{\alpha}=\bar{\beta}=\bar{\delta}=\tfrac{2}{3}\pi$. This means the dodecahedron is regular. Then we may use the regular dodecahedron to get $a=\arccos\frac{\sqrt{5}}{3}$, $b=\arccos \frac{\sqrt{5}+1}{2\sqrt{3}}$, $c=\arccos\frac{\sqrt{5}-1}{2\sqrt{3}}$ for the special case.  
\end{proof}

\begin{proposition}\label{aaa_general}
Tiling of the sphere by congruent general quadrilaterals, such that $\alpha^3$ is the only degree $3$ vertex, is the quadrilateral subdivision $Q_{\square}P_6$ of the cube. 
\end{proposition}

The geometrical discussion $Q_{\square}P_6$ is in the proof of Proposition \ref{aaa-abb}. See Figure \ref{tiling4}. The tiling has one free parameter.
 
\begin{proof}
By Lemma \ref{count-aaa}, we know $f\ge 24$. By applying Lemma \ref{deg3miss} to any two of $\beta,\gamma,\delta$, and using the parity lemma, we get the following: 
\begin{itemize}
\item One of $\beta^4,\gamma^4,\beta^2\gamma^2$ is a vertex.
\item One of $\beta^4,\delta^4,\beta^2\delta^2$ is a vertex.
\item One of $\gamma^4,\delta^4,\gamma^2\delta^2$ is a vertex.
\end{itemize}
This is equivalent to that one of the following combinations are vertices:
\[
\{\beta^4,\gamma^4\},\;
\{\beta^4,\delta^4\},\;
\{\gamma^4,\delta^4\},
\]
\[
\{\beta^4,\gamma^2\delta^2\},\;
\{\beta^2\delta^2,\gamma^4\},\;
\{\beta^2\gamma^2,\delta^4\},\;
\{\beta^2\gamma^2,\beta^2\delta^2,\gamma^2\delta^2\}.
\]
Up to the symmetry of exchanging $(\beta,b)$ with $(\delta,c)$, and using $\beta\ne\delta$, we only need to consider $\{\beta^4,\gamma^4\}, \{\beta^4,\gamma^2\delta^2\},\{\beta^2\delta^2,\gamma^4\}$. 

\subsubsection*{Case. $\beta^4,\gamma^4$ are vertices}

The angle sums of $\alpha^3,\beta^4,\gamma^4$ and the angle sum for quadrilateral imply
\[
\alpha=\tfrac{2}{3}\pi,\;
\beta=\gamma=\tfrac{1}{2}\pi,\;
\delta=(\tfrac{1}{3}+\tfrac{4}{f})\pi.
\]
By $\beta\ne\delta$, we get $f\ne 24$. Therefore $f>24$.

By the balance lemma, we know $\delta^2\cdots$ is a vertex. By $\beta+\gamma+3\delta=(2+\tfrac{12}{f})\pi>2\pi$, we know $\delta^2\cdots$ is an even vertex. Then by $6\delta>2\pi$, we get $\delta^2\cdots=\delta^2\cdots,\delta^4\cdots$, with no $\delta$ in the remainders. By the parity lemma and the angle values, this implies either $R(\delta^2)=(\tfrac{4}{3}-\tfrac{8}{f})\pi$ or $R(\delta^4)=(\tfrac{2}{3}-\tfrac{16}{f})\pi$ is of the form $(\tfrac{2}{3}k+l)\pi$, where $k,l$ are non-negative integers. By $f>24$, we find this to be impossible.

\subsubsection*{Case. $\beta^4,\gamma^2\delta^2$ are vertices}

The angle sums of $\alpha^3,\beta^4,\gamma^2\delta^2$ and the angle sum for quadrilateral imply
\[
\alpha=\tfrac{2}{3}\pi,\;
\beta=\tfrac{1}{2}\pi,\;
\gamma+\delta=\pi,\;
f=24.
\]
The possible AADs $\thin^{\beta}\alpha^{\delta}\thin^{\beta}\alpha^{\delta}\thin^{\beta}\alpha^{\delta}\thin$ and $\thin^{\beta}\alpha^{\delta}\thin^{\beta}\alpha^{\delta}\thin^{\delta}\alpha^{\beta}\thin$ of $\alpha^3$ imply $\beta\delta\cdots$ is a vertex. By $f=24$ and the first part of Lemma \ref{count-aaa}, all vertices have degrees $3$ or $4$. Then by $\beta+\gamma+\delta=\tfrac{3}{2}\pi\ne 2\pi$ and the parity lemma, we get $\beta\delta\cdots=\beta^2\delta^2$, contradicting $\gamma^2\delta^2$ and $\beta\ne\delta$. 

\subsubsection*{Case. $\beta^2\delta^2,\gamma^4$ are vertices}

The angle sums of $\alpha^3,\beta^2\delta^2,\gamma^4$ and the angle sum for quadrilateral imply
\[
\alpha=\tfrac{2}{3}\pi,\;
\beta+\delta=\pi,\;
\gamma=\tfrac{1}{2}\pi,\;
f=24.
\]
By $f=24$ and the first part of Lemma \ref{count-aaa}, all vertices have degrees $3$ or $4$. Then by $\beta^2\delta^2,\gamma^4$, and $\beta\ne\delta$, we know a degree $4$ vertex other than $\beta^2\delta^2,\gamma^4$ must have $\alpha$. By Lemma \ref{nod} (or the angle values), we know $\alpha^2\gamma^2$ is not a vertex. Therefore the only vertices besides $\alpha^3,\beta^2\delta^2,\gamma^4$ are $\alpha^2\beta^2,\alpha^2\delta^2$. Up to the symmetry of exchanging $(\beta,b)$ with $(\delta,c)$, and using $\beta\ne\delta$, we may assume $\text{AVC}=\{\alpha^3,\beta^2\delta^2,\gamma^4,\alpha^2\beta^2\}$. 

By applying the counting lemma to $\beta,\delta$, we know $\alpha^2\beta^2$ is not a vertex. Then we get $\text{AVC}
=\{\alpha^3,\beta^2\delta^2,\gamma^4\}$. The AVC is contained in \eqref{aaa-abbeq1}, and we know $\alpha^3$ is a vertex. By the earlier argument, we find the tiling is the quadrilateral subdivision $Q_{\square}P_6$ of the cube. 
\end{proof}

\begin{proposition}\label{abb_general}
Tilings of the sphere by congruent general quadrilaterals, such that $\alpha\beta^2$ is the only degree $3$ vertex, are the earth map tiling $E_{\square}2$ and its flip modifications $F_1E_{\square}2,F_2E_{\square}2$.
\end{proposition}

\begin{proof}
By Lemma \ref{count-att}, we have $f\ge 16$. 

We know $\alpha\cdots$ is an even vertex. Then a vertex $\alpha\gamma\delta\cdots$ implies $\alpha+2\gamma+2\delta\le 2\pi$. Combined with the angle sum of $\alpha\beta^2$, we find the angle sum for quadrilateral is $\le 2\pi$, a contradiction. Therefore $\alpha\gamma\delta\cdots$ is not a vertex. 

By $\alpha\beta^2$, if (the even vertex) $\alpha\cdots$ has $\beta$, then it is $\alpha\beta^2$. This implies that, if $\alpha\cdots$ has $\gamma,\delta$, then it has no $\beta$. By no $\alpha\gamma\delta\cdots$, the vertex is $\alpha^k\gamma^l,\alpha^k\delta^l$. Then by Lemma \ref{nod}, we get $\alpha\cdots=\alpha\beta^2,\alpha^k\delta^l$. Moreover, the unique AAD $\thick^{\gamma}\beta^{\alpha}\thin^{\beta}\alpha^{\delta}\thin^{\alpha}\beta^{\gamma}\thick$ of $\alpha\beta^2$ implies $\alpha\delta\cdots=\alpha^k\delta^l$ ($k\ge 1$, $l\ge 2$) is a vertex.  

By Proposition \ref{abb-ccdd}, we know $\gamma^2\delta^2$ is not a vertex. By applying Lemma \ref{deg3miss} to $\gamma,\delta$, and no $\gamma^2\delta^2$, and the parity lemma, we know one of $\gamma^4,\delta^4$ is a vertex. 

If $\gamma^4$ is a vertex, then we apply the second part of Lemma \ref{count-att} to $\delta$, and use the parity lemma to conclude that one of the following is a vertex:
\begin{itemize}
\item Degree $4$ vertex $\delta\cdots$: By no $\alpha\beta\gamma\delta$, this is an even vertex. Then by no $\gamma^2\delta^2$, this is $\alpha^2\delta^2,\beta^2\delta^2,\delta^4$. 
\item Degree $5$ vertex $\delta^3\cdots$: The even vertex is $\alpha\delta^4$. The odd vertex is $\beta\gamma\delta^3$.
\item Degree $6$ vertex $\delta^5\cdots$: The vertex is even, and must be $\delta^6$.
\item Degree $7$ vertex $\delta^7$: This is not a vertex.
\end{itemize}

Similarly, if $\delta^4$ is a vertex, then we know one of $\beta^2\gamma^2,\gamma^4,\beta\gamma^3\delta,\gamma^6$ is a vertex (no $\alpha^2\gamma^2,\alpha\gamma^4$ by Lemma \ref{nod}). In summary, we know that, in addition to $\alpha\beta^2$, one of the following pairs of vertices appear:
\[
\{\alpha^2\delta^2,\gamma^4\},\;
\{\alpha\delta^4,\gamma^4\},\;
\{\beta^2\gamma^2,\delta^4\},\;
\{\beta^2\delta^2,\gamma^4\},\;
\]
\[
\{\beta\gamma^3\delta,\delta^4\},\;
\{\beta\gamma\delta^3,\gamma^4\},\;
\{\gamma^4,\delta^4\},\;
\{\gamma^4,\delta^6\},\;
\{\gamma^6,\delta^4\}.
\]

\subsubsection*{Case. $\alpha^2\delta^2,\gamma^4$ are vertices}

The angle sums of $\alpha\beta^2,\alpha^2\delta^2,\gamma^4$ and the angle sum for quadrilateral imply
\[
\alpha=(1-\tfrac{8}{f})\pi,\;
\beta=(\tfrac{1}{2}+\tfrac{4}{f})\pi,\;
\gamma=\tfrac{1}{2}\pi,\;
\delta=\tfrac{8}{f}\pi.
\]
By $f\ge 16$, we get $\alpha\ge\frac{1}{2}\pi$.

We know $\alpha\cdots=\alpha\beta^2,\alpha^k\delta^l$, where $l$ is even. By $\alpha+\delta=\pi$, we get $\alpha^k\delta^l=\alpha^k,\alpha^2\delta^2,\alpha\delta^l,\delta^l$. By $\alpha\ge\frac{1}{2}\pi$ and the only degree $3$ vertex $\alpha\beta^2$, we get $k=4$ in $\alpha^k$. Therefore $\alpha\cdots=\alpha\beta^2,\alpha^4,\alpha^2\delta^2,\alpha\delta^l$.

It remains to consider vertices without $\alpha$. By $\beta>\gamma=\frac{1}{2}\pi$, the only vertex with at least four from $\beta,\gamma$ is $\gamma^4$. Then we are left with $\beta^2\delta^l,\gamma^2\delta^l,\beta\gamma\delta^l$. We conclude all the possible vertices
\begin{equation}\label{aaa-abbeq4}
\text{AVC}
=\{\alpha\beta^2,\alpha^4,\alpha^2\delta^2,\gamma^4,\alpha\delta^l,\beta^2\delta^l,\gamma^2\delta^l,\beta\gamma\delta^l,\delta^l\}.
\end{equation}

First we assume $f>16$. By $\alpha=(1-\tfrac{8}{f})\pi$, this implies $\alpha^4$ is not a vertex. 

The AAD $\dash^{\delta}\gamma^{\beta}\thick^{\gamma}\beta^{\alpha}\thin^{\alpha}\delta^{\gamma}\dash$ at $\beta\gamma\delta^l$ determines $T_1,T_2,T_3$ in the first of Figure \ref{abbA}. Then $\alpha_1\alpha_3\cdots=\alpha^2\delta^2=\thin\alpha_1\thin\alpha_3\thin\delta\dash\delta\thin$ determines $T_4$. Then $\alpha_4\thin\delta_1\cdots=\alpha_4\thin\delta_1\dash\delta\cdots$ (this can be $\alpha^2\delta^2$ or $\alpha\delta^l$) gives $\delta$ next to $\delta_1$. On the other hand, $\beta_2\gamma_1\cdots=\beta\gamma\delta^l$ gives $\delta$ next to $\gamma_1$. We get two $\delta$ in a tile, a contradiction. Therefore $\beta\gamma\delta^l$ is not a vertex.

\begin{figure}[htp]
\centering
\begin{tikzpicture}


\draw
	(1,0) -- (0,0)
	(0,-0.3) -- (0,2) 
	(-1,-0.3) -- (-1,0)
	(-1,1) -- (0,1)
	(-1,0) -- (-2,0) -- (-2,1);

\draw[line width=1.2]
	(-1,1) -- (-1,0)
	(0,2) -- (-1,2)
	(1,0) -- (1,1);

\draw[dashed]
	(-1,0) -- (0,0)
	(0,1) -- (1,1)
	(-1,2) -- (-1,1) -- (-2,1);
	
\node at (-0.2,0.8) {\small $\alpha$};
\node at (-0.8,0.8) {\small $\beta$};
\node at (-0.2,0.2) {\small $\delta$};
\node at (-0.8,0.2) {\small $\gamma$};

\node at (-0.8,-0.2) {\small $\delta$};
\node at (-0.2,-0.2) {\small $\delta$};

\node at (0.2,0.2) {\small $\alpha$};
\node at (0.8,0.2) {\small $\beta$};
\node at (0.2,0.8) {\small $\delta$};
\node at (0.8,0.8) {\small $\gamma$};

\node at (0.2,1.2) {\small $\delta$};

\node at (-1.8,0.2) {\small $\alpha$};
\node at (-1.2,0.2) {\small $\beta$};
\node at (-1.8,0.8) {\small $\delta$};
\node at (-1.2,0.8) {\small $\gamma$};

\node at (-0.2,1.2) {\small $\alpha$};
\node at (-0.2,1.75) {\small $\beta$};
\node at (-0.8,1.2) {\small $\delta$};
\node at (-0.8,1.8) {\small $\gamma$};

\node[inner sep=0.5,draw,shape=circle] at (-0.5,0.5) {\small $1$};
\node[inner sep=0.5,draw,shape=circle] at (-1.5,0.5) {\small $2$};
\node[inner sep=0.5,draw,shape=circle] at (-0.5,1.5) {\small $3$};
\node[inner sep=0.5,draw,shape=circle] at (0.5,0.5) {\small $4$};


\begin{scope}[xshift=4cm]

\foreach \a in {1,-1}
{
\begin{scope}[xscale=\a]

\draw
	(0,0) -- (2,0)
	(1,0) -- (1,1)
	(0,1) -- (0.2,2) -- (1.2,2);

\draw[line width=1.2]
	(0,1) -- (0,0)
	(1.2,2) -- (1,1)
	(2,0) -- (2,1);

\draw[dashed]
	(0,1) -- (2,1)
	(1,-0.3) -- (1,0);

\node at (0.8,0.2) {\small $\alpha$};
\node at (0.2,0.2) {\small $\beta$};
\node at (0.8,0.8) {\small $\delta$};
\node at (0.2,0.8) {\small $\gamma$};

\node at (0.8,-0.2) {\small $\delta$};
\node at (1.2,-0.2) {\small $\delta$};

\node at (1.2,0.2) {\small $\alpha$};
\node at (1.8,0.2) {\small $\beta$};
\node at (1.2,0.8) {\small $\delta$};
\node at (1.8,0.8) {\small $\gamma$};

\node at (0.35,1.85) {\small $\alpha$};
\node at (0.95,1.8) {\small $\beta$};
\node at (0.25,1.2) {\small $\delta$};
\node at (0.85,1.2) {\small $\gamma$};

\end{scope}
}

\node[inner sep=0.5,draw,shape=circle] at (-0.5,0.5) {\small $1$};
\node[inner sep=0.5,draw,shape=circle] at (0.5,0.5) {\small $2$};
\node[inner sep=0.5,draw,shape=circle] at (-0.6,1.5) {\small $3$};
\node[inner sep=0.5,draw,shape=circle] at (0.6,1.5) {\small $4$};
\node[inner sep=0.5,draw,shape=circle] at (-1.5,0.5) {\small $5$};
\node[inner sep=0.5,draw,shape=circle] at (1.5,0.5) {\small $6$};

\end{scope}
	
\end{tikzpicture}
\caption{Proposition \ref{abb_general}: $\beta\gamma\delta^l$ and $\gamma^2\delta^l$.}
\label{abbA}
\end{figure}
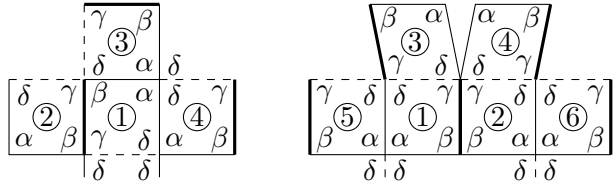

The AAD $\thin^{\alpha}\delta^{\gamma}\dash^{\delta}\gamma^{\beta}\thick^{\beta}\gamma^{\alpha}\dash^{\gamma}\delta^{\alpha}\thin$ at $\gamma^2\delta^l$ determines $T_1,T_2,T_3,T_4$ in the second of Figure \ref{abbA}. By no $\beta\gamma\delta^l$, we get $\gamma_3\delta_1\cdots=\gamma^2\delta^l=\thick\gamma_3\dash\delta_1\thin\delta\dash\cdots$. This determines $T_5$. Then $\alpha_1\thin\alpha_5\cdots=\alpha^2\delta^2=\thin\alpha_1\thin\alpha_5\thin\delta\dash\delta\thin$ gives $\delta$ next to $\alpha_1$. By the same argument, we also get $\delta$ next to $\alpha_2$. On the other hand, we have $\beta_1\beta_2\cdots=\alpha\beta^2,\beta^2\delta^l$. If $\beta_1\beta_2\cdots=\alpha\beta^2$, then we have $\alpha$ adjacent to two $\delta$ in a tile. If $\beta_1\beta_2\cdots=\beta^2\delta^l$, then we have two $\delta$ in a tile. Both are contradictions. Therefore $\gamma^2\delta^l$ is not a vertex.

We conclude the updated list of possible vertices for $f>16$:
\begin{equation}\label{aaa-abbeq2}
\text{AVC}
=\{\alpha\beta^2,\alpha^2\delta^2,\gamma^4,\alpha\delta^l,\beta^2\delta^l,\delta^l\}. 
\end{equation}

For $f=16$, we get
\begin{equation}\label{aaa-abbeq3}
\alpha=\gamma=\delta=\tfrac{1}{2}\pi,\;
\beta=\tfrac{3}{4}\pi.
\end{equation}
By $R(\beta^2)=\alpha\le \beta,\gamma,\delta$ and the parity lemma, we get $\beta^2\cdots=\alpha\beta^2$. Then $\beta^2\delta^l$ is not a vertex, and the unique arrangement $\thick\beta\thin\alpha\thin\beta\thick$ of $\alpha\beta^2$ implies $\beta\thin\beta\cdots$ is not a vertex. By $R(\beta\gamma\delta)=\tfrac{1}{4}\pi<\delta$, we also know $\beta\gamma\delta^l$ is not a vertex. Then by \eqref{aaa-abbeq4}, we know $\beta\delta\cdots$ is not a vertex. By no $\beta\thin\beta\cdots,\beta\delta\cdots$, the AAD implies $\alpha^4$ is not a vertex. Therefore all possible vertices still belong to \eqref{aaa-abbeq2}.

It remains to find the tiling for the AVC \eqref{aaa-abbeq2}. By $\gamma\cdots=\gamma^4$, we may regard the tiling as a tiling of $n(\gamma^4)$ in Figure \ref{aaa-abbB}. We will actually take the alternative approach of removing all the $c$-edges (and all $\gamma^4$ vertices). Then we get a tiling by congruent symmetric pentagons in the third of Figure \ref{tiling4}, with angles $\bar{\alpha}=2\delta,\bar{\beta}=\alpha,\bar{\delta}=\beta$ and edge combination $\bar{a}^4\bar{b}$ ($\bar{a}=a,\bar{b}=2b$). The AVC \eqref{aaa-abbeq2} becomes the following for the symmetric pentagonal tiling
\[
\text{AVC}
=\{\bar{\beta}\bar{\delta}^2,\bar{\alpha}\bar{\beta}^2,\bar{\alpha}^{\bar{l}}\bar{\beta},\bar{\alpha}^{\bar{l}}\bar{\delta}^2,\bar{\alpha}^{\bar{l}}\}, \quad
\bar{f}=\tfrac{f}{2},\;
\bar{l}=\tfrac{l}{2}. 
\]
Since the AVC has no degree $2$ vertex, the pentagonal tiling is edge-to-edge. By \cite[Proposition 1]{cly} (specifically, the proof for the case $\bar{\alpha}\bar{\beta}^2$ is a vertex), we know edge-to-edge tilings of the sphere by congruent symmetric pentagons with edge combination $\bar{a}^4\bar{b}$ are the earth map tiling $E_{\pentagon}1$ and its flip modifications $F_1E_{\pentagon}1,F_2E_{\pentagon}1$. Then our quadrilateral tilings are obtained by adding $c$ back, i.e., simple quadrilateral subdivisions of these tilings. Figure \ref{pemt} shows the simple quadrilateral subdivision of $E_{\pentagon}1$ is $E_{\square}2$. Then the simple quadrilateral subdivisions of $F_1E_{\pentagon}1,F_2E_{\pentagon}1$ are the flip modifications $F_1E_{\square}2,F_2E_{\square}2$ in Figure \ref{flip10}.

We remark that the geometrical existence of the symmetric pentagon for $E_{\pentagon}1$ (no additional condition is required for $F_1E_{\pentagon}1,F_2E_{\pentagon}1$) is discussed in \cite{cly}. In fact, the discussion is exactly the existence of the quadrilateral for $E_{\square}2$.

\subsubsection*{Case. $\alpha\delta^4,\gamma^4$ are vertices}

The angle sums of $\alpha\beta^2,\alpha\delta^4,\gamma^4$ and the angle sum for quadrilateral imply
\[
\alpha=\tfrac{16}{f}\pi,\;
\beta=(1-\tfrac{8}{f})\pi,\;
\gamma=\tfrac{1}{2}\pi,\;
\delta=(\tfrac{1}{2}-\tfrac{4}{f})\pi.
\]
The AAD $\dash^{\gamma}\delta^{\alpha}\thin^{\alpha}\delta^{\gamma}\dash$ at $\alpha\delta^4$ implies $\alpha^2\cdots$ is a vertex. Therefore $\alpha<\pi$, which means $f>16$. Then $\beta>\gamma=\frac{1}{2}\pi>\delta>\frac{1}{4}\pi$. This implies $\beta^2\gamma^2\cdots$ is not a vertex.

By $R(\beta\gamma\delta)=\frac{12}{f}\pi<\alpha,2\beta,2\gamma,4\delta$, and the parity lemma, we know $\beta\gamma\delta\cdots=\beta\gamma\delta^3$. The angle sum of $\beta\gamma\delta^3$ further implies
\[
\alpha=\tfrac{4}{5}\pi,\;
\beta=\tfrac{3}{5}\pi,\;
\gamma=\tfrac{1}{2}\pi,\;
\delta=\tfrac{3}{10}\pi,\;
f=20.
\]
Then $R(\alpha^2)=\tfrac{2}{5}\pi<\alpha,\beta,\gamma,2\delta$. By the parity lemma, this implies $\alpha^2\cdots$ is not a vertex, a contradiction. Therefore $\beta\gamma\delta\cdots$ is not a vertex. In particular, all vertices are even.

By $f>16$, we get $\beta>\gamma=\frac{1}{2}\pi>\delta>\frac{1}{4}\pi$. This implies $n=4$ in $\gamma^n$, and $n=6$ in $\delta^n$, and $\beta^2\gamma^2\cdots,\beta^n$ are not vertices. Moreover, we have $R(\beta^2\delta^2)< R(\gamma^2\delta^2)=\frac{8}{f}\pi<\alpha,\beta,\gamma,2\delta$. This implies $\beta^2\delta^2\cdots=\beta^2\delta^2$ and $\gamma^2\delta^2\cdots=\gamma^2\delta^2$. Then by $\gamma+\delta<\pi$, we know $\gamma^2\delta^2\cdots$ is not a vertex. Therefore $\beta^2\delta^2,\gamma^4,\delta^6$ are all the vertices without $\alpha$. 

We know $\alpha\cdots=\alpha\beta^2,\alpha^k\delta^l$. By $\alpha\delta^4$, we know $\alpha^k\delta^l=\alpha\delta^4,\alpha^k,\alpha^k\delta^2$. Therefore $\alpha\beta^2,\alpha\delta^4,\beta^2\delta^2,\gamma^4,\delta^6,\alpha^k,\alpha^k\delta^2$ are all the vertices.

The angle sum of any one of $\alpha^2\delta^2,\beta^2\delta^2,\delta^6$ further implies $f=24$. If $f\ne 24$, then $\alpha\beta^2,\alpha\delta^4,\gamma^4,\alpha^k,\alpha^k\delta^2(k\ge 3)$ are all the vertices. This implies $\beta\delta\cdots$ and $\beta\thin\beta\cdots$ are not vertices. Then the AAD of $\thin\alpha\thin\alpha\thin$ is $\thin^{\beta}\alpha^{\delta}\thin^{\delta}\alpha^{\beta}\thin$. This further implies no consecutive $\alpha\alpha\alpha$. Therefore $\alpha^k,\alpha^k\delta^2$ are not vertices, and we find $\alpha\beta^2,\alpha\delta^4,\gamma^4$ are all the vertices. Then we get
\[
f=\#\alpha
=\#\alpha\beta^2+\#\alpha\delta^4
=\tfrac{1}{2}\#\beta+\tfrac{1}{4}\#\delta
=\tfrac{3}{4}f,
\]
a contradiction. 

Therefore $f=24$, and we get
\begin{equation}\label{abbeq1}
\alpha=\beta=\tfrac{2}{3}\pi,\;
\gamma=\tfrac{1}{2}\pi,\;
\delta=\tfrac{1}{3}\pi.
\end{equation}
By the only degree $3$ vertex $\alpha\beta^2$, we get all the vertices
\begin{equation}\label{abbeq2}
\text{AVC}
=\{\alpha\beta^2,\alpha\delta^4,\gamma^4,\alpha^2\delta^2,\beta^2\delta^2,\delta^6\}
\end{equation}
This is a special case of the AVC \eqref{aaa-abbeq2}. Then we get the earth map tiling $E_{\square}2$ with three timezones, and the two flip modifications $F_1E_{\square}2,F_2E_{\square}2$.

\subsubsection*{Case. $\beta^2\delta^2,\gamma^4$ are vertices}

The angle sums of $\alpha\beta^2,\beta^2\delta^2,\gamma^4$ and the angle sum for quadrilateral imply
\[
\alpha=(\tfrac{1}{2}+\tfrac{4}{f})\pi,\;
\beta=(\tfrac{3}{4}-\tfrac{2}{f})\pi,\;
\gamma=\tfrac{1}{2}\pi,\;
\delta=(\tfrac{1}{4}+\tfrac{2}{f})\pi.
\]
The AAD $\thin^{\alpha}\beta^{\gamma}\thick^{\gamma}\beta^{\alpha}\thin^{\alpha}\delta^{\gamma}\dash^{\gamma}\delta^{\alpha}\thin$ of $\beta^2\delta^2$ implies $\alpha^2\cdots=\alpha^k\delta^l(k\ge 2)$ is a vertex. The angle values and the only degree $3$ vertex $\alpha\beta^2$ imply the vertex is $\alpha^2\delta^2$. The angle sum of $\alpha^2\delta^2$ further implies \eqref{abbeq1}. Then we get the AVC \eqref{abbeq2}, and the tilings are $E_{\square}2,F_1E_{\square}2,F_2E_{\square}2$.

\subsubsection*{Case. $\beta^2\gamma^2,\delta^4$ are vertices}

The angle sums of $\alpha\beta^2,\beta^2\gamma^2,\delta^4$ and the angle sum for quadrilateral imply
\[
\alpha=(\tfrac{1}{2}+\tfrac{4}{f})\pi,\;
\beta=(\tfrac{3}{4}-\tfrac{2}{f})\pi,\;
\gamma=(\tfrac{1}{4}+\tfrac{2}{f})\pi,\;
\delta=\tfrac{1}{2}\pi.
\]
The AAD $\thick^{\gamma}\beta^{\alpha}\thin^{\alpha}\beta^{\gamma}\thick^{\beta}\gamma^{\delta}\dash^{\delta}\gamma^{\beta}\thick$ of $\beta^2\gamma^2$ implies $\alpha^2\cdots=\alpha^k\delta^l(k\ge 2)$ is a vertex. However, by $\alpha>\delta=\frac{1}{2}\pi$, and the only degree $3$ vertex $\alpha\beta^2$, this is not a vertex.

\subsubsection*{Case. $\gamma^4,\delta^4$ are vertices}

The angle sums of $\alpha\beta^2,\gamma^4,\delta^4$ and the angle sum for quadrilateral imply
\[
\alpha=\tfrac{8}{f}\pi,\;
\beta=(1-\tfrac{4}{f})\pi,\;
\gamma=\delta=\tfrac{1}{2}\pi.
\]
We know $\gamma^2\delta^2$ is not a vertex. By $f\ge 16$, we get $\beta>\gamma=\delta=\tfrac{1}{2}\pi$. Therefore $\gamma^4,\delta^4$ are the only vertices without $\alpha$. We also know $\alpha\cdots=\alpha\beta^2,\alpha^k\delta^l(k\ge 1)$. By $\delta=\tfrac{1}{2}\pi$ and the parity lemma, we get $l=0,2$ in $\alpha^k\delta^l(k\ge 1)$. 

Therefore $\alpha\beta^2,\gamma^4,\delta^4,\alpha^k,\alpha^k\delta^2$ are all the vertices. This implies $\beta\delta\cdots$ and $\beta\thin\beta\cdots$ are not vertices. Then the AAD of $\thin\alpha\thin\alpha\thin$ is $\thin^{\beta}\alpha^{\delta}\thin^{\delta}\alpha^{\beta}\thin$, and further implies no consecutive $\alpha\alpha\alpha$. Therefore $\alpha^k$ is not a vertex, and $k<3$ in $\alpha^k\delta^2$. By the only degree $3$ vertex $\alpha\beta^2$, we get $\alpha^k\delta^2=\alpha^2\delta^2$. The angle sum of $\alpha^2\delta^2$ further implies the angle values in \eqref{aaa-abbeq3}. By the earlier argument, we get the earth map tiling $E_{\square}2$ with $f=16$ (i.e., two timezones).

\subsubsection*{Case. $\gamma^4,\delta^6$ are vertices}

The angle sums of $\alpha\beta^2,\gamma^4,\delta^6$ and the angle sum for quadrilateral imply
\[
\alpha=(\tfrac{1}{3}+\tfrac{8}{f})\pi,\;
\beta=(\tfrac{5}{6}-\tfrac{4}{f})\pi,\;
\gamma=\tfrac{1}{2}\pi,\;
\delta=\tfrac{1}{3}\pi.
\]
Recall $\alpha\delta\cdots=\alpha^k\delta^l$ is a vertex, and $l$ is even. By $f\ge 16$ and the angle values, we get $\alpha\delta\cdots=\alpha\delta^4,\alpha^2\delta^2,\alpha^3\delta^2$. We have discussed the case $\alpha^2\delta^2,\gamma^4$ are vertices, and the case $\alpha\delta^4,\gamma^4$ are vertices. We obtain the tilings $E_{\square}2,F_1E_{\square}2,F_2E_{\square}2$.

The angle sum of $\alpha^3\delta^2$ further implies
\[
\alpha=\tfrac{4}{9}\pi,\;
\beta=\tfrac{7}{9}\pi,\;
\gamma=\tfrac{1}{2}\pi,\;
\delta=\tfrac{1}{3}\pi,\;
f=72.
\]
By $R(\beta\gamma\delta)=\frac{7}{18}\pi<\beta,\gamma,2\delta$, there is no odd vertex. Then by $R(\beta^2)=\alpha<\beta,\gamma,2\delta$, we get $\beta\cdots=\alpha\beta^2$. This implies $\beta\delta\cdots$ and $\beta\thin\beta\cdots$ are not vertices, and further implies no consecutive $\alpha\alpha\alpha$. Therefore $\alpha^3\delta^2$ is not a vertex.

\subsubsection*{Case. $\gamma^6,\delta^4$ are vertices}

The angle sums of $\alpha\beta^2,\gamma^6,\delta^4$ and the angle sum for quadrilateral imply
\[
\alpha=(\tfrac{1}{3}+\tfrac{8}{f})\pi,\;
\beta=(\tfrac{5}{6}-\tfrac{4}{f})\pi,\;
\gamma=\tfrac{1}{3}\pi,\;
\delta=\tfrac{1}{2}\pi.
\]
Recall $\alpha\delta\cdots=\alpha^k\delta^l$ is a vertex, and $l$ is even. By $f\ge 16$ and the angle values, we get $\alpha\delta\cdots=\alpha^2\delta^2$. 

The angle sum of $\alpha^2\delta^2$ further implies
\[
\alpha=\delta=\tfrac{1}{2}\pi,\;
\beta=\tfrac{3}{4}\pi,\;
\gamma=\tfrac{1}{3}\pi,\;
f=48.
\]
By the angle values and the parity lemma, we find $\alpha\beta^2,\alpha^4,\alpha^2\delta^2,\gamma^6,\delta^4$ are all the vertices. This implies $\beta\delta\cdots$ and $\beta\thin\beta\cdots$ are not vertices, and further implies no consecutive $\alpha\alpha\alpha$. Therefore $\alpha^4$ is not a vertex, and we get
\[
\text{AVC}=\{\alpha\beta^2,\alpha^2\delta^2,\gamma^6,\delta^4\}.
\]
The AVC implies all vertices have unique AADs.

The vertex $\delta^4$ determines four tiles around it in Figure \ref{abbF}, including $T_1,T_2$. Then $\gamma_1\cdots=\gamma^6$ determines $T_3,T_4,T_5$. Then $\alpha_1\alpha_2\cdots=\alpha^2\delta^2$ and $\beta_1\beta_3\cdots=\alpha\beta^2$ determine $T_6$. Then $\alpha_3\beta_6\cdots=\alpha\beta^2$ determines $T_7$. Then $\alpha_7\delta_3\delta_4\cdots=\alpha^2\delta^2$ and no $\beta\delta\cdots$ determine $T_8$. By the symmetry, we also determine $T_9$. Then $\beta_4\beta_5\cdots=\alpha\beta^2$ and $\alpha_4\beta_8\cdots=\alpha_5\beta_9\cdots=\alpha\beta^2$ imply that $\alpha$ is adjacent to two $\beta$ in a tile, a contradiction. 

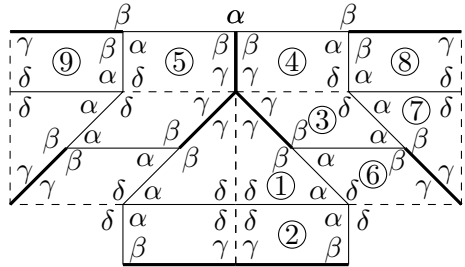
\begin{figure}[htp]
\centering
\begin{tikzpicture}

\foreach \a in {1,-1}
{
\begin{scope}[xscale=\a]

\draw
	(0,0) -- (1.5,0) 
	(1.5,-0.8) -- (1.5,0) -- (0.75,0.75) -- (2.25,0.75) -- (1.5,1.5) -- (1.5,2.3) -- (0,2.3)
	(1.5,1.5) -- (3,1.5);

\draw[line width=1.2]
	(1.5,-0.8) -- (0,-0.8)
	(0.75,0.75) -- (0,1.5)
	(2.25,0.75) -- (3,0)
	(0,1.5) -- (0,2.3)
	(1.5,2.3) -- (3,2.3);

\draw[dashed]
	(0,-0.8) -- (0,1.5)
	(1.5,0) -- (3,0) -- (3,2.3)
	(0,1.5) -- (1.5,1.5);

\node at (1.3,-0.2) {\small $\alpha$};
\node at (1.3,-0.6) {\small $\beta$};
\node at (0.2,-0.2) {\small $\delta$};
\node at (0.2,-0.6) {\small $\gamma$};

\node at (1.1,0.15) {\small $\alpha$};
\node at (0.6,0.6) {\small $\beta$};
\node at (0.2,0.2) {\small $\delta$};
\node at (0.2,1) {\small $\gamma$};

\node at (1.7,-0.2) {\small $\delta$};

\node at (1.2,0.6) {\small $\alpha$};
\node at (2.15,0.55) {\small $\beta$};
\node at (1.55,0.2) {\small $\delta$};
\node at (2.5,0.2) {\small $\gamma$};

\node at (1.9,0.9) {\small $\alpha$};
\node at (0.85,0.95) {\small $\beta$};
\node at (1.45,1.3) {\small $\delta$};
\node at (0.45,1.3) {\small $\gamma$};

\node at (1.3,2.1) {\small $\alpha$};
\node at (0.2,2.1) {\small $\beta$};
\node at (1.3,1.7) {\small $\delta$};
\node at (0.2,1.7) {\small $\gamma$};

\node at (0,2.5) {\small $\alpha$};

\node at (1.95,1.3) {\small $\alpha$};
\node at (2.45,0.85) {\small $\beta$};
\node at (2.8,1.3) {\small $\delta$};
\node at (2.8,0.4) {\small $\gamma$};

\node at (1.7,1.7) {\small $\alpha$};
\node at (1.7,2.05) {\small $\beta$};
\node at (2.8,1.7) {\small $\delta$};
\node at (2.8,2.1) {\small $\gamma$};

\node at (1.5,2.5) {\small $\beta$};

\end{scope}
}

\node[inner sep=0.5,draw,shape=circle] at (0.6,0.25) {\small $1$};
\node[inner sep=0.5,draw,shape=circle] at (0.75,-0.4) {\small $2$};
\node[inner sep=0.5,draw,shape=circle] at (1.15,1.125) {\small $3$};
\node[inner sep=0.5,draw,shape=circle] at (0.75,1.9) {\small $4$};
\node[inner sep=0.5,draw,shape=circle] at (-0.75,1.9) {\small $5$};
\node[inner sep=0.5,draw,shape=circle] at (1.825,0.4) {\small $6$};
\node[inner sep=0.5,draw,shape=circle] at (2.4,1.25) {\small $7$};
\node[inner sep=0.5,draw,shape=circle] at (2.25,1.9) {\small $8$};
\node[inner sep=0.5,draw,shape=circle] at (-2.25,1.9) {\small $9$};

\end{tikzpicture}
\caption{Proposition \ref{abb_general}: No tiling for $\{\alpha\beta^2,\alpha^2\delta^2,\gamma^6,\delta^4\}$.}
\label{abbF}
\end{figure}

\subsubsection*{Case. $\beta\gamma^3\delta,\delta^4$ are vertices}

The angle sums of $\alpha\beta^2,\beta\gamma^3\delta,\delta^4$ and the angle sum for quadrilateral imply
\[
\alpha=(\tfrac{1}{2}+\tfrac{6}{f})\pi,\;
\beta=(\tfrac{3}{4}-\tfrac{3}{f})\pi,\;
\gamma=(\tfrac{1}{4}+\tfrac{1}{f})\pi,\;
\delta=\tfrac{1}{2}\pi.
\]
By $\alpha>\delta=\frac{1}{2}\pi$, and the only degree $3$ vertex $\alpha\beta^2$, we know $\alpha\delta\cdots=\alpha^k\delta^l$ is not a vertex. This is a contradiction.

\subsubsection*{Case. $\beta\gamma\delta^3,\gamma^4$ are vertices}

The angle sums of $\alpha\beta^2,\beta\gamma\delta^3,\gamma^4$ and the angle sum for quadrilateral imply
\[
\alpha=(\tfrac{1}{2}+\tfrac{6}{f})\pi,\;
\beta=(\tfrac{3}{4}-\tfrac{3}{f})\pi,\;
\gamma=\tfrac{1}{2}\pi,\;
\delta=(\tfrac{1}{4}+\tfrac{1}{f})\pi.
\]
Recall $\alpha\delta\cdots=\alpha^k\delta^l$ is a vertex, and $l$ is even. By $f\ge 16$ and the angle values, we get $\alpha\delta\cdots=\alpha^2\delta^2,\alpha\delta^4$. We already discussed the case $\alpha^2\delta^2,\gamma^4$ are vertices, and also the case $\alpha\delta^4,\gamma^4$ are vertices. In both cases, $\beta\gamma\delta^3$ is not a vertex.
\end{proof}

\section{Tiling by Almost Equilateral Quadrilateral}
\label{almostquad}

The almost equilateral quadrilateral is the third of Figure \ref{quad}. Since every tile has a $b$-companion similar to the general quadrilateral in Figure \ref{tiling_aabc}, we know $f$ is even. We call a vertex {\em $b$-vertex} if it has $\gamma,\delta$. Otherwise the vertex is a {\em $\hat{b}$-vertex}.

We first deal with the possibility that some angles may have the same value.

By Lemma \ref{geometry1}, the quadrilateral is symmetric if and only if $\alpha=\beta$ and $\gamma=\delta$. The following shows that there is no tiling for this kind of quadrilaterals. Akama and van Cleemput \cite[Theorem 3.3]{ac} proved the proposition for convex almost equilateral quadrilaterals.

\begin{proposition}\label{symmetric}
There is no tiling of the sphere by congruent almost equilateral quadrilaterals, such that $\alpha=\beta$ and $\gamma=\delta$.
\end{proposition}

\begin{proof}
By the angle sum for quadrilateral, we get
\[
\alpha+\gamma=(1+\tfrac{2}{f})\pi.
\]
If $\alpha=\gamma$, then $\alpha=\gamma=(\frac{1}{2}+\tfrac{1}{f})\pi<\pi$. This implies the almost equilateral quadrilateral is actually a rhombus, contradicting $a\ne b$. Therefore we have $\alpha\ne\gamma$.

Suppose $\alpha<\gamma$. Then $\alpha<(\frac{1}{2}+\tfrac{1}{f})\pi<\gamma$. By $\alpha+\gamma>\pi$ and the parity lemma (third part of Lemma \ref{parity}), this implies $\gamma\cdots=\alpha\gamma^2$. The angle sum of $\alpha\gamma^2$ further implies
\[
\alpha=\tfrac{4}{f}\pi,\;
\gamma=(1-\tfrac{2}{f})\pi.
\]
By $\alpha,\gamma<\pi$ and Lemma \ref{geometry3}, we get $\alpha+\pi>2\gamma$. This means $f<8$. Therefore $f=6$, and $\alpha=\gamma$, a contradiction. 

Suppose $\alpha>\gamma$. Then $\alpha>(\frac{1}{2}+\tfrac{1}{f})\pi>\gamma$. By $\alpha+\gamma>\pi$ and the parity lemma, this implies $\alpha^3,\alpha\gamma^k,\gamma^k$ are all the vertices. Then by the counting lemma, we know $\alpha^3$ is a vertex. The angle sum of $\alpha^3$ further implies
\[
\alpha=\tfrac{2}{3}\pi,\;
\gamma=(\tfrac{1}{3}+\tfrac{2}{f})\pi.
\]
The possible AADs $\thin^{\beta}\alpha^{\delta}\thin^{\beta}\alpha^{\delta}\thin^{\beta}\alpha^{\delta}\thin$ and $\thin^{\beta}\alpha^{\delta}\thin^{\beta}\alpha^{\delta}\thin^{\delta}\alpha^{\beta}\thin$ of $\alpha^3$ imply $\alpha\gamma\cdots=\alpha\gamma^k$ is a vertex. By $\alpha^3$, and $\alpha\ne\gamma$, and the parity lemma, we know $k\ge 4$ in $\alpha\gamma^k$. By $\alpha+4\gamma=(2+\tfrac{8}{f})\pi>2\pi$, we get a contradiction. 
\end{proof}

\begin{proposition}\label{isosceles}
Tilings of the sphere by congruent almost equilateral quadrilaterals, such that $\beta=\delta$, are the earth map tiling $E_{\square}^A1$ and the flip modification $FE_{\square}^A1$.
\end{proposition}

The proposition also applied to the case $\alpha=\gamma$. The tilings are special cases of $E_{\square}^A1$ and $FE_{\square}^A1$. The general case appears in Proposition \ref{acd}. 

\begin{proof}
The equality $\beta=\delta$ implies $a=b$ or $\gamma=\pi$. Since the quadrilateral is not a rhombus, we have $a\ne b$. Therefore $\gamma=\pi$, and the quadrilateral is actually an isosceles triangle. See Figure \ref{isoscelesquad}. Since $\alpha,\beta$ face edges $a+b,a$ in the isosceles triangle, we get $\alpha>\beta$.

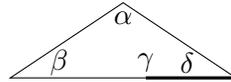
\begin{figure}[htp]
\centering
\begin{tikzpicture}

\draw
	(0.3,-0.5) -- (-1.5,-0.5) -- (0,0.5) -- (1.5,-0.5);

\draw[line width=1.2]
	(0.3,-0.5) -- (1.5,-0.5);

\node at (0,0.3) {\small $\alpha$};
\node at (-0.85,-0.3) {\small $\beta$};
\node at (0.3,-0.3) {\small $\gamma$};
\node at (0.85,-0.3) {\small $\delta$};
	
\end{tikzpicture}
\caption{Proposition \ref{isosceles}: $\gamma=\pi$ implies isosceles triangle.}
\label{isoscelesquad}
\end{figure}

By $\gamma=\pi$, we know $\gamma^2\cdots$ is not a vertex. By the balance lemma and the parity lemma, this implies a $b$-vertex is $\alpha^l\beta^k\gamma\delta$. By $\alpha>\beta$ and the angle sum for quadrilateral, the vertices are $\alpha\gamma\delta,\beta^k\gamma\delta$. If $\alpha\gamma\delta$ is not a vertex, then $\gamma\cdots=\beta^k\gamma\delta$. By applying the counting lemma to $\beta,\gamma$, we get $k=1$ in $\beta^k\gamma\delta$. Therefore either $\alpha\gamma\delta$ or $\beta\gamma\delta$ is a vertex. We also know $\gamma^2\cdots$ and $\delta^2\cdots$ are not vertices. Therefore, up to the exchange of $(\alpha,\delta)$ with $(\beta,\gamma)$, we are in the first case in the proof of Proposition \ref{acd}. In fact, by $\gamma=\pi$, the AVC is \eqref{acd_avc1}. We also note that the argument in the later proof allows $\beta=\delta$. Then we get the earth map tiling $E_{\square}^A1$ and the flip modification $FE_{\square}^A1$.
\end{proof}

By Propositions \ref{symmetric} and \ref{isosceles}, and Lemma \ref{geometry1}, in all the subsequent propositions, we may implicitly assume the quadrilateral satisfy 
\[
\alpha\ne\beta,\;
\gamma\ne\delta,\;
\alpha\ne\gamma,\;
\beta\ne\delta.
\]
By Lemma \ref{geometry1}, this implies either $\alpha<\beta$ and $\gamma<\delta$, or $\alpha>\beta$ and $\gamma>\delta$. 

Further classification is divided into two cases: at least two degree $3$ vertices, and only one degree $3$ vertex.

\subsection{Two Degree $3$ Vertices}

In this section, we assume the tiling has at least two degree $3$ vertices. Degree $3$ $\hat{b}$-vertices are $\alpha^3,\alpha^2\beta,\alpha\beta^2,\beta^3$. By $\alpha\ne\beta$, we cannot have two degree $3$ $\hat{b}$-vertices.

Lemma \ref{geometry2} gives all pairs of degree $3$ $b$-vertices. Up to the exchange of $(\alpha,\delta)$ with $(\beta,\gamma)$, we get two possible combinations
\[
\{\alpha\delta^2,\beta\gamma^2\},\;
\{\alpha\gamma\delta,\beta\gamma^2\}.
\]
The two cases are discussed in Propositions \ref{add-bcc} and \ref{acd-bcc}. Next, we consider one degree $3$ $b$-vertex, and another degree $3$ $\hat{b}$-vertex. Up to the exchange of $(\alpha,\delta)$ with $(\beta,\gamma)$, one vertex is $\alpha\gamma^2,\alpha\delta^2,\alpha\gamma\delta$, and the other vertex is $\alpha^3,\alpha^2\beta,\alpha\beta^2,\beta^3$. For $\alpha\gamma^2$, by $\alpha\ne\gamma$, we only need to consider the following combinations
\[
\{\alpha^2\beta,\alpha\gamma^2\},\;
\{\alpha\beta^2,\alpha\gamma^2\},\;
\{\alpha\gamma^2,\beta^3\}.
\]
These are discussed in Propositions \ref{aab-acc}, \ref{abb-acc}, \ref{acc-bbb}. For $\alpha\delta^2$, by $\beta\ne\delta$, we only need to consider the following combinations
\[
\{\alpha^3,\alpha\delta^2\},\;
\{\alpha^2\beta,\alpha\delta^2\},\;
\{\alpha\delta^2,\beta^3\}.
\]
These are discussed in Propositions \ref{aaa-add}, \ref{aab-add}, \ref{add-bbb}. Then we have four combinations for $\alpha\gamma\delta$. We discuss the four combinations together in Proposition \ref{acd}.

In all the propositions in this section, either $\gamma^2\cdots$ or $\delta^2\cdots$ are vertices. By the balance lemma, we know both $\gamma^2\cdots,\delta^2\cdots$ are vertices. Then they imply $\gamma,\delta<\pi$.

\begin{proposition}\label{add-bcc}
There is no tiling of the sphere by congruent almost equilateral quadrilaterals, such that $\alpha\delta^2,\beta\gamma^2$ are vertices.
\end{proposition}

\begin{proof}
The angle sums of $\alpha\delta^2,\beta\gamma^2$ and the angle sum for quadrilateral imply
\[
\alpha+\beta=\tfrac{8}{f}\pi,\;
\gamma+\delta=(2-\tfrac{4}{f})\pi.
\]

If $f\ge 8$, then $\alpha,\beta<\pi$. We also know $\gamma,\delta<\pi$. By Lemma \ref{geometry3}, we get $\alpha+\pi,\beta+\pi>\gamma+\delta=(2-\tfrac{4}{f})\pi$. This implies $\alpha+\beta>(2-\tfrac{8}{f})\pi\ge\pi$, contradicting $\alpha+\beta=\tfrac{8}{f}\pi\le\pi$. 

Therefore we have $f=6$. By \eqref{quadvcountf}, this implies all vertices have degree $3$. The AAD $\thick^{\gamma}\delta^{\alpha}\thin^{\beta}\alpha^{\delta}\thin^{\alpha}\delta^{\gamma}\thick$ of $\alpha\delta^2$ implies $\alpha\beta\cdots$ is a vertex. We also know this is a degree $3$ vertex. By the parity lemma, we get $\alpha\beta\cdots=\alpha^2\beta,\alpha\beta^2$. Combined with $\alpha+\beta=\tfrac{8}{6}\pi$, we get $\alpha=\beta=\tfrac{2}{3}\pi$, a contradiction. 
\end{proof}

\begin{proposition}\label{acd-bcc}
There is no tiling of the sphere by congruent almost equilateral quadrilaterals, such that 
$\alpha\gamma\delta,\beta\gamma^2$ are vertices. 
\end{proposition}

\begin{proof}
The angle sums of $\alpha\gamma\delta,\beta\gamma^2$ and the angle sum for quadrilateral imply
\[
\alpha+\delta=(1+\tfrac{2}{f})\pi,\;
\beta=\tfrac{4}{f}\pi,\;
\gamma=(1-\tfrac{2}{f})\pi.
\]

If $\gamma<\delta$, then by $\alpha+\delta>\pi$ and $\beta\gamma^2$, we get $R(\delta^2)<2\alpha,\beta$. Then by $\beta=\tfrac{4}{f}\pi<2\gamma<2\delta$, and the parity lemma, this implies $\delta^2\cdots=\alpha\delta^2$, contradicting $\alpha\gamma\delta$ and $\gamma\ne\delta$. Therefore $\gamma>\delta$. By Lemma \ref{geometry1}, this implies $\alpha>\beta$.

By $\alpha\gamma\delta$ and $\delta<\pi$, we know $\alpha+\gamma>\pi$. Therefore $\alpha^2\gamma^2\cdots$ is not a vertex. By the same reason, we also know $\alpha^2\delta^2\cdots$ is not a vertex. Then by $\alpha\gamma\delta$ and the parity lemma, we conclude $\alpha^2\cdots=\alpha^k\beta^l$. 

If $\alpha^2\cdots$ is a vertex, then $\beta<\alpha<\pi$. Then by Lemma \ref{geometry3}, we get $\alpha+\beta+\pi>\alpha+\gamma+\delta=2\pi$. By $\delta<\pi$, this implies $\alpha+\beta>\pi$. Then by $\alpha>\beta$, this implies $\alpha^2\cdots=\alpha^k\beta^l=\alpha^3,\alpha^2\beta$. By $\beta\gamma^2$ and $\alpha\ne\gamma$, we know $\alpha^2\beta$ is not a vertex. Therefore $\alpha^2\cdots=\alpha^3$. The angle sum of $\alpha^3$ further implies
\[
\alpha=\tfrac{2}{3}\pi,\;
\beta=\tfrac{4}{f}\pi,\;
\gamma=(1-\tfrac{2}{f})\pi,\;
\delta=(\tfrac{1}{3}+\tfrac{2}{f})\pi.
\]
By $\alpha>(1-\tfrac{4}{f})\pi$, we get $f<12$. By $\alpha>\beta$, we get $f>6$. Then $f=8,10$, and we get the corresponding angle values: 
\begin{align*}
f=8 &\colon 
\alpha=\tfrac{2}{3}\pi,\;
\beta=\tfrac{1}{2}\pi,\;
\gamma=\tfrac{3}{4}\pi,\;
\delta=\tfrac{7}{12}\pi. \\
f=10 &\colon 
\alpha=\tfrac{2}{3}\pi,\;
\beta=\tfrac{2}{5}\pi,\;
\gamma=\tfrac{4}{5}\pi,\;
\delta=\tfrac{8}{15}\pi.
\end{align*}
Since both fail \eqref{coolsaet_eq1}, we conclude $\alpha^2\cdots$ is not a vertex. 

By $\alpha\gamma\delta$, and $\gamma>\delta$, and the parity lemma, we get $\alpha\gamma\cdots=\alpha\gamma\delta$. This implies $\beta\gamma\cdots$ has no $\alpha$. Then by $\beta\gamma^2$, we get $\beta\gamma\cdots=\beta\gamma^2,\beta^k\gamma\delta^l$.

The AAD $\thick^{\delta}\gamma^{\beta}\thin^{\alpha}\beta^{\gamma}\thin^{\beta}\gamma^{\delta}\thick$ of $\beta\gamma^2$ determines $T_1,T_2,T_3$ in Figure \ref{acd_bccA}. Then $\thin^{\alpha}\beta_1^{\gamma}\thin^{\beta}\gamma_2^{\delta}\thick\cdots=\beta\gamma^2,\beta^k\gamma\delta^l$. In $\beta^k\gamma\delta^l$, $\thin\beta\thin\gamma\thick$ must be part of a $\gamma\delta$-fan without $\alpha$. By no $\alpha^2\cdots$, the AAD of the fan is $\thick^{\gamma}\delta^{\alpha}\thin^{\gamma}\beta^{\alpha}\thin\cdots\thin^{\gamma}\beta^{\alpha}\thin^{\beta}\gamma^{\delta}\thick$. Since this is inconsistent with $\thin^{\alpha}\beta_1^{\gamma}\thin^{\beta}\gamma_2^{\delta}\thick$, we conclude $\thin^{\alpha}\beta_1^{\gamma}\thin^{\beta}\gamma_2^{\delta}\thick\cdots=\beta\gamma^2$. This determines $T_4$. Then by no $\alpha^2\cdots$, and $\gamma>\delta$, and Lemma \ref{klem3}, we get $\thin\alpha_1\thin^{\gamma}\beta_4^{\alpha}\thin\cdots=\thin\alpha_2\thin^{\gamma}\beta_3^{\alpha}\thin\cdots=\alpha\beta^k$. Then by no $\alpha^2\cdots$, the AAD of the two vertices is $\thin\alpha\thin^{\gamma}\beta^{\alpha}\thin\cdots\thin^{\gamma}\beta^{\alpha}\thin$. This determines $T_5,T_6$. Then $\alpha_3\gamma_6\cdots=\alpha\gamma\delta$ determines $T_7$. By no $\alpha^2\cdots$, we get $\alpha_5=\alpha_7$. However, this implies a vertex $\alpha\delta^2$, contradicting $\alpha\gamma\delta$ and $\gamma\ne\delta$.
\end{proof}

\begin{figure}[htp]
\centering
\begin{tikzpicture}

\draw
	(-2.5,-0.7) -- (-2.5,0.7) -- (1,0.7) -- (0.4,0) -- (-0.4,0) -- (-1,-0.7) -- (-2.5,-0.7)
	(0.4,1.4) -- (1.6,0) -- (1,-0.7) -- (-1,-0.7) -- (-1.6,0) -- (-0.4,1.4);

\draw[line width=1.2]
	(-0.4,0) -- (-1,0.7)
	(0.4,0) -- (1,-0.7)
	(0.4,1.4) -- (-0.4,1.4)
	(-1.6,0) -- (-2.5,0);

\node at (0.65,0.55) {\small $\alpha$}; 
\node at (0.35,0.2) {\small $\beta$};
\node at (-0.6,0.53) {\small $\delta$};
\node at (-0.3,0.2) {\small $\gamma$};

\node at (-0.7,-0.55) {\small $\alpha$}; 
\node at (-0.35,-0.2) {\small $\beta$};
\node at (0.6,-0.52) {\small $\delta$};
\node at (0.3,-0.2) {\small $\gamma$};

\node at (-1.4,0) {\small $\alpha$}; 
\node at (-1,-0.4) {\small $\beta$};
\node at (-1.05,0.4) {\small $\delta$};
\node at (-0.65,0) {\small $\gamma$};

\node at (1.4,0) {\small $\alpha$}; 
\node at (0.95,0.4) {\small $\beta$};
\node at (1,-0.4) {\small $\delta$};
\node at (0.65,0) {\small $\gamma$};

\node at (0.55,0.9) {\small $\beta$}; 
\node at (0.3,1.2) {\small $\gamma$};
\node at (-0.7,0.85) {\small $\alpha$};
\node at (-0.35,1.2) {\small $\delta$};

\node at (-2.3,-0.55) {\small $\alpha$}; 
\node at (-1.4,-0.5) {\small $\beta$}; 
\node at (-2.3,-0.2) {\small $\delta$};
\node at (-1.7,-0.2) {\small $\gamma$};

\node at (-1.35,0.55) {\small $\alpha$}; 
\node at (-2.3,0.5) {\small $\beta$}; 
\node at (-1.7,0.2) {\small $\delta$};
\node at (-2.3,0.2) {\small $\gamma$};

\node[draw,shape=circle, inner sep=0.5] at (0,0.35) {\small $1$};
\node[draw,shape=circle, inner sep=0.5] at (0,-0.35) {\small $2$};
\node[draw,shape=circle, inner sep=0.5] at (-1,0) {\small $3$};
\node[draw,shape=circle, inner sep=0.5] at (1,0) {\small $4$};
\node[draw,shape=circle, inner sep=0.5] at (0,1.05) {\small $5$};
\node[draw,shape=circle, inner sep=0.5] at (-2,-0.35) {\small $6$};
\node[draw,shape=circle, inner sep=0.5] at (-2,0.35) {\small $7$};

\end{tikzpicture}
\caption{Proposition \ref{acd-bcc}: no tiling.}
\label{acd_bccA}
\end{figure}
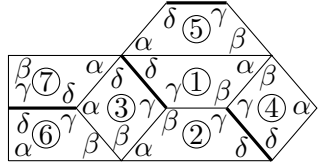

\begin{proposition}\label{aab-acc}
There is no tiling of the sphere by congruent almost equilateral quadrilaterals, such that 
$\alpha^2\beta,\alpha\gamma^2$ are vertices. 
\end{proposition}

\begin{proof}
By $\alpha^2\beta$ and the angle sum for quadrilateral, we get $\beta+2\gamma+2\delta=(2+\frac{8}{f})\pi$. By $\alpha^2\beta,\alpha\gamma^2$, we get $\alpha+\beta=2\gamma$ and $\beta+2\pi=4\gamma$. Then $3\gamma+\delta=\alpha+\beta+\gamma+\delta=(2+\frac{4}{f})\pi$. Combining the two equalities, we get $\beta+\delta=\gamma+\frac{4}{f}\pi>\gamma$.

Suppose $\alpha<\beta$ and $\gamma<\delta$. By $\alpha+\beta=2\gamma$, we get $\alpha<\gamma<\beta$. Then $4\gamma=\beta+2\pi>\gamma+2\pi$, and $4\gamma<3\gamma+\delta=(2+\frac{4}{f})\pi$. This implies $\frac{2}{3}\pi<\gamma<(\frac{1}{2}+\frac{1}{f})\pi$, or $f<6$, a contradiction. By Lemma \ref{geometry1}, therefore, we know $\alpha>\beta$ and $\gamma>\delta$. Then by $\alpha+\beta=2\gamma$, we get $\alpha>\gamma>\beta$.

By $\alpha^2\beta$, we get $\delta<\gamma<\alpha<\pi$. By Lemma \ref{geometry5}, this implies $\beta>\delta$. Therefore $\alpha>\gamma>\beta>\delta$. By $\alpha<\alpha+\beta=2\gamma$ and Lemma \ref{geometry7}, we get $\beta<2\delta$. 

By the parity lemma, we know $\beta\gamma\cdots=\beta\gamma^2\cdots,\beta\gamma\delta\cdots$. Then by $\gamma>\delta$, and the angle sum for quadrilateral, we know $\beta\gamma\cdots$ has no $\alpha$. By $\alpha\gamma^2$ and $\alpha>\gamma>\beta>\delta$, we know $\beta\gamma^2,\beta\gamma\delta$ are not vertices. By $\beta+2\gamma+2\delta=(2+\frac{8}{f})\pi>2\pi$, we get $R(\beta\gamma^2)<2\delta<2\beta,2\gamma$. By no $\beta\gamma^2$ and the parity lemma, this implies $\beta\gamma^2\cdots=\beta^2\gamma^2$. Then it remains to consider $\beta\gamma\delta\cdots$ with no $\alpha,\gamma$ in the remainder. By $\beta+2\gamma+2\delta>2\pi$, we get $R(\beta\gamma\delta)<\gamma+\delta<\beta+2\delta$. By $\delta<\beta<2\delta$, and no $\beta\gamma\delta$, and the parity lemma, this implies the vertex $\beta\gamma\delta\cdots=\beta^2\gamma\delta,\beta^3\gamma\delta,\beta\gamma\delta^3$.

The angle sum of one of $\beta^2\gamma^2,\beta^2\gamma\delta,\beta^3\gamma\delta,\beta\gamma\delta^3$, and the angle sums of $\alpha^2\beta,\alpha\gamma^2$, and the angle sum for quadrilateral imply the following angle values. We also include the corresponding equation \eqref{coolsaet_eq1}.
\begin{itemize}
\item $\beta^2\gamma^2$:
	$\alpha=\tfrac{4}{5}\pi,\;
	\beta=\tfrac{2}{5}\pi,\;
	\gamma=\tfrac{3}{5}\pi,\;
	\delta=(\tfrac{1}{5}+\tfrac{4}{f})\pi$.
	\newline
	$\sin\tfrac{2}{5}\pi
	\sin\tfrac{4}{f}\pi
	=\sin\tfrac{1}{5}\pi
	\sin\tfrac{1}{5}\pi$.
\item $\beta^2\gamma\delta$:
	$\alpha=(\tfrac{2}{3}+\tfrac{4}{3f})\pi,\;
	\beta=(\tfrac{2}{3}-\tfrac{8}{3f})\pi,\;
	\gamma=(\tfrac{2}{3}-\tfrac{2}{3f})\pi,\;
	\delta=\tfrac{6}{f}\pi$.
	\newline
	$\sin(\tfrac{1}{3}+\tfrac{2}{3f})\pi
	\sin(\tfrac{22}{3f}-\tfrac{1}{3})\pi
	=\sin(\tfrac{1}{3}-\tfrac{4}{3f})\pi
	\sin(\tfrac{1}{3}-\tfrac{4}{3f})\pi$.
\item $\beta^3\gamma\delta$:
	$\alpha=(\tfrac{4}{5}+\tfrac{4}{5f})\pi,\;
	\beta=(\tfrac{2}{5}-\tfrac{8}{5f})\pi,\;
	\gamma=(\tfrac{3}{5}-\tfrac{2}{5f})\pi,\;
	\delta=(\tfrac{1}{5}+\tfrac{26}{5f})\pi$.
	\newline
	$\sin(\tfrac{2}{5}+\tfrac{2}{5f})\pi
	\sin\tfrac{6}{f}\pi
	=\sin(\tfrac{1}{5}-\tfrac{4}{5f})\pi
	\sin(\tfrac{1}{5}-\tfrac{4}{5f})\pi$.
\item $\beta\gamma\delta^3$:
	$\alpha=(1-\tfrac{6}{f})\pi,\;
	\beta=\tfrac{12}{f}\pi,\;
	\gamma=(\tfrac{1}{2}+\tfrac{3}{f})\pi,\;
	\delta=(\tfrac{1}{2}-\tfrac{5}{f})\pi$.
	\newline
	$\sin(\tfrac{1}{2}-\tfrac{3}{f})\pi
	\sin(\tfrac{1}{2}-\tfrac{11}{f})\pi
	=\sin\tfrac{6}{f}\pi
	\sin\tfrac{6}{f}\pi$.
\end{itemize}
By \eqref{quadvcountf}, the vertices of degree $\ge 4$ imply $f>6$. Since all equations have no solution for even $f>6$, we conclude $\beta\gamma\cdots$ is not a vertex. 

By $\alpha^2\beta$, we get $R(\alpha^2)=\beta<\alpha,2\delta$. By $\gamma>\delta$ and the parity lemma, this implies $\alpha^2\cdots=\alpha^2\beta$. By no $\beta\gamma\cdots$, we get $\alpha^2\cdots=\thin\alpha\thin^{\beta}\alpha^{\delta}\thin^{\gamma}\beta^{\alpha}\thin$. This implies the AAD of $\thin\alpha\thin\alpha\thin$ is $\thin\alpha\thin^{\beta}\alpha^{\delta}\thin$, and further implies $\delta\thin\delta\cdots=\thick^{\gamma}\delta^{\alpha}\thin^{\alpha}\delta^{\gamma}\thick\cdots$ is not a vertex.

By $\alpha^2\beta=\thin\alpha\thin^{\beta}\alpha^{\delta}\thin^{\gamma}\beta^{\alpha}\thin$, we know $\gamma\thin\delta\cdots$ is a vertex. By no $\beta\gamma\cdots$, we know $\gamma\thin\delta\cdots$ has no $\beta$. By the edge consideration and $3\gamma+\delta>2\pi$, we get $\gamma\thin\delta\cdots=\gamma\delta^3\cdots,\gamma^2\delta^2\cdots$. By $\alpha+2\gamma+2\delta>\alpha+\gamma+3\delta>\alpha+\beta+\gamma+\delta>2\pi$, we know $\gamma\thin\delta\cdots$ has no $\alpha$. Therefore $\gamma\thin\delta\cdots=\gamma^k\delta^l$. By no $\delta\thin\delta\cdots$ and the edge consideration, we know $k\ge l$. Then by $3\gamma+\delta>2\pi$ and the parity lemma, we conclude $\gamma\thin\delta\cdots=\gamma^2\delta^2$. The angle sum of $\gamma^2\delta^2$ further implies the following angle values. We also include \eqref{coolsaet_eq1}. 
\begin{align*}
\gamma^2\delta^2 &\colon \alpha=(1-\tfrac{4}{f})\pi,\;
\beta=\tfrac{8}{f}\pi,\;
\gamma=(\tfrac{1}{2}+\tfrac{2}{f})\pi,\;
\delta=(\tfrac{1}{2}-\tfrac{2}{f})\pi. \\
&\sin(\tfrac{1}{2}-\tfrac{2}{f})\pi
\sin(\tfrac{1}{2}-\tfrac{6}{f})\pi
=\sin\tfrac{4}{f}\pi
\sin\tfrac{4}{f}\pi.
\end{align*}
The equation has no solution for even $f\ge 6$.  
\end{proof}

\begin{proposition}\label{abb-acc}
There is no tiling of the sphere by congruent almost equilateral quadrilaterals, such that 
$\alpha\beta^2,\alpha\gamma^2$ are vertices. 
\end{proposition}

\begin{proof}
By $\alpha\beta^2$ and the angle sum for quadrilateral, we get $\alpha+2\gamma+2\delta=(2+\frac{8}{f})\pi>2\pi$. The angle sums of $\alpha\beta^2,\alpha\gamma^2$ and the angle sum for quadrilateral imply
\[
\beta=\gamma,\;
\delta=\tfrac{4}{f}\pi.
\]
By $f\ge 6$, we get $\delta\le\frac{2}{3}\pi$. Then $\alpha<\beta=\gamma<\delta$ implies $\alpha+2\beta<2\pi$, contradicting $\alpha\beta^2$. By Lemma \ref{geometry1}, therefore, we get $\alpha>\beta=\gamma>\delta$. 

If $\alpha<\pi$, then $\beta=\gamma<\alpha<\pi$. By Lemma \ref{geometry3}, we get $\alpha+\delta>\alpha+\beta+\gamma-\pi=\pi$. If $\alpha\ge\pi$, we also get $\alpha+\delta>\pi$. Therefore we always have $\alpha+\delta>\pi$. Combined with $\alpha\beta^2,\alpha\gamma^2$, and $\alpha>\beta=\gamma$, we get $R(\alpha^2)<\alpha,\beta,\gamma,2\delta$. By the parity lemma, this implies $\alpha^2\cdots$ is not a vertex.

By $\alpha\gamma^2$ and Lemma \ref{fbalance}, there is a vertex with $\delta^2$-fan. By no $\alpha^2\cdots$ and Lemma \ref{klem3}, a $\delta^2$-fan has a single $\alpha$, and a vertex has at most one $\delta^2$-fan. By $\alpha\beta^2$, the $\delta^2$-fan is $\thick\delta\thin\alpha\thin\delta\thick,\thick\delta\thin\alpha\thin\beta\thin\delta\thick$. The vertex is then a single $\delta^2$-fan combined with $\gamma\delta$-fans and $\gamma^2$-fans without $\alpha$. By $\alpha\gamma^2$ and $\gamma>\delta$, we know $\alpha\delta^2$ is not a vertex. Then by $\gamma>\delta$, and $\alpha+2\gamma+2\delta>2\pi$, and the angle sum for quadrilateral, a vertex with $\thick\delta\thin\alpha\thin\delta\thick$ is the fan combined with a single $\thick\gamma\thin\delta\thick$. By the angle sum for quadrilateral, a vertex with $\thick\delta\thin\alpha\thin\beta\thin\delta\thick$ is the fan itself. Therefore the vertex with $\delta^2$-fan is $\alpha\beta\delta^2,\alpha\gamma\delta^3$. 

The angle sum of one of $\alpha\beta\delta^2,\alpha\gamma\delta^3$ further implies the following angle values. We also include \eqref{coolsaet_eq1}.
\begin{itemize}
\item $\alpha\beta\delta^2$:
	$\alpha=(2-\tfrac{16}{f})\pi,\;
	\beta=\gamma=\tfrac{8}{f}\pi,\;
	\delta=\tfrac{4}{f}\pi$.
	\newline
	$0=\sin\tfrac{4}{f}\pi
	\sin(1-\tfrac{16}{f})\pi$. 
\item $\alpha\gamma\delta^3$:
	$\alpha=(2-\tfrac{24}{f})\pi,\;
	\beta=\gamma=\tfrac{12}{f}\pi,\;
	\delta=\tfrac{4}{f}\pi$.
	\newline
	$\sin(1-\tfrac{12}{f})\pi
	\sin\tfrac{2}{f}\pi
	=\sin\tfrac{6}{f}\pi
	\sin(1-\tfrac{24}{f})\pi$. 
\end{itemize}
For $\alpha\gamma\delta^3$, by $\alpha>0$, we have $f>12$. The equation has no solution for even $f>12$. For $\alpha\beta\delta^2$, by $\alpha>0$, we have $f>8$. The only solution of the equation for even $f>8$ is $f=16$. Then we further get
\begin{equation}\label{abb-accangle}
\alpha=\pi,\;
\beta=\gamma=\tfrac{1}{2}\pi,\;
\delta=\tfrac{1}{4}\pi.
\end{equation}
Then by the parity lemma, and the fact that $\alpha\beta\delta^2$ is the only $\delta^2$-fan, we get the list of vertices
\[
\text{AVC}
=\{\alpha\beta^2,\alpha\gamma^2,\alpha\beta\delta^2,\beta^4,\beta^2\gamma^2,\beta\gamma^2\delta^2,\gamma^4\}.
\]
Moreover, we know $\beta\gamma^2\delta^2$ is the combination of two $\gamma\delta$-fans $\thick\gamma\thin\beta\thin\delta\thick$ and $\thick\gamma\thin\delta\thick$. 

By no $\alpha^2\cdots$, the AAD of $\alpha\beta\delta^2$ is $\thick^{\gamma}\delta^{\alpha}\thin\alpha\thin^{\alpha}\beta^{\gamma}\thin^{\alpha}\delta^{\gamma}\thick$. Since $\alpha\beta\cdots=\alpha\beta^2,\alpha\beta\delta^2$ and $\thin\alpha\thin^{\gamma}\beta^{\alpha}\thin$ is not compatible with the AAD of $\alpha\beta\delta^2$, we get $\thin\alpha\thin^{\gamma}\beta^{\alpha}\thin\cdots=\alpha\beta^2$. Then by no $\alpha^2\cdots$, we get $\thin\alpha\thin^{\gamma}\beta^{\alpha}\thin\cdots=\thin\alpha\thin^{\gamma}\beta^{\alpha}\thin^{\gamma}\beta^{\alpha}\thin$. 

If $\beta\gamma^2\delta^2$ is a vertex, then by no $\alpha^2\cdots$, we get the AAD $\thick^{\delta}\gamma^{\beta}\thin^{\alpha}\beta^{\gamma}\thin^{\alpha}\delta^{\gamma}\thick$ at the vertex. This determines $T_1,T_2,T_3$ in the first of Figure \ref{abb-accA}. The fan $\thick\gamma\thin\delta\thick$ at the vertex is yet to be determined. Then $\thin\alpha_2\thin^{\gamma}\beta_1^{\alpha}\thin\cdots=\thin\alpha\thin^{\gamma}\beta^{\alpha}\thin^{\gamma}\beta^{\alpha}\thin$ determines $T_4$, and $\alpha_3\gamma_2\cdots=\alpha\gamma^2$ determines $T_5$. Then $\alpha_1\gamma_4\cdots=\alpha\gamma^2$ determines $T_6$, and $\alpha_4\delta_2\delta_5\cdots=\alpha\beta\delta^2$ and no $\alpha^2\cdots$ determine $T_7$. Then $\alpha_5\gamma_7\cdots=\alpha\gamma^2$ determines $T_8$. Then $\beta_3\beta_5\beta_8\cdots=\beta^4$ and no $\alpha^2\cdots$ determine $T_9$. Then $\alpha_9\gamma_3\cdots=\alpha\gamma^2$ determines $T_{10}$, and further determines the fan $\thick\gamma\thin\delta\thick$ at $\beta\gamma^2\delta^2$, including $T_{11}$. Then $\thin\alpha_{10}\thin^{\gamma}\beta_{11}^{\alpha}\thin\cdots=\thin\alpha\thin^{\gamma}\beta^{\alpha}\thin^{\gamma}\beta^{\alpha}\thin$ determines $T_{12}$. Then $\alpha_{11}\gamma_{12}\cdots=\alpha\gamma^2$ determines $T_{13}$. Then we get a vertex $\beta_6\beta_{13}\delta_1\delta_{11}\cdots$, a contradiction. Therefore $\beta\gamma^2\delta^2$ is not a vertex. This implies $\beta\gamma\cdots=\beta^2\gamma^2$, and $\gamma\delta\cdots$ is not a vertex.

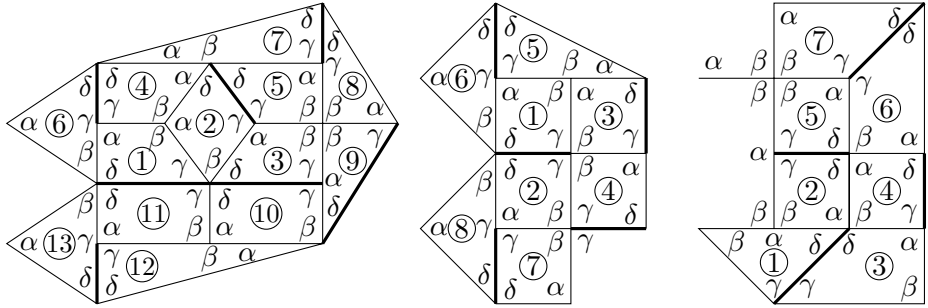
\begin{figure}[htp]
\centering
\begin{tikzpicture}


\draw
	(-1.5,1.6) -- (1.5,1.6) -- (1.5,-0.8) -- (-1.5,-0.8) -- (-1.5,0.8) -- (-0.6,0.8) -- (0,0) -- (0.6,0.8) -- (2.5,0.8) -- (1.5,2.4) -- (-1.5,1.6) -- (-2.7,0.8) -- (-1.5,0) -- (-2.7,-0.8) -- (-1.5,-1.6) -- (1.5,-0.8)
	(0,0) -- (0,-0.8)
	(-0.6,0.8) -- (0,1.6);

\draw[line width=1.2]
	(-1.5,0) -- (1.5,0)
	(-1.5,1.6) -- (-1.5,0.8)
	(0.6,0.8) -- (0,1.6)
	(1.5,2.4) -- (1.5,1.6)
	(-1.5,-1.6) -- (-1.5,-0.8)
	(1.5,-0.8) -- (2.5,0.8);

\node at (0.65,0.6) {\small $\alpha$};
\node at (1.3,0.6) {\small $\beta$};
\node at (1.3,0.2) {\small $\gamma$};
\node at (0.35,0.2) {\small $\delta$};

\node at (-1.3,0.6) {\small $\alpha$};
\node at (-0.7,0.6) {\small $\beta$};
\node at (-0.4,0.2) {\small $\gamma$};
\node at (-1.3,0.2) {\small $\delta$};

\node at (1.3,1.4) {\small $\alpha$};
\node at (1.3,1) {\small $\beta$};
\node at (0.65,1) {\small $\gamma$};
\node at (0.4,1.4) {\small $\delta$};

\node at (-0.35,1.4) {\small $\alpha$};
\node at (-0.65,1) {\small $\beta$};
\node at (-1.3,1) {\small $\gamma$};
\node at (-1.3,1.4) {\small $\delta$};

\node at (-0.35,0.8) {\small $\alpha$};
\node at (0.05,0.3) {\small $\beta$};
\node at (0.35,0.8) {\small $\gamma$};
\node at (-0.05,1.3) {\small $\delta$};

\node at (0.2,-0.6) {\small $\alpha$};
\node at (1.3,-0.6) {\small $\beta$};
\node at (1.3,-0.2) {\small $\gamma$};
\node at (0.2,-0.2) {\small $\delta$};

\node at (-1.3,-0.6) {\small $\alpha$};
\node at (-0.2,-0.6) {\small $\beta$};
\node at (-0.2,-0.2) {\small $\gamma$};
\node at (-1.3,-0.2) {\small $\delta$};

\node at (-0.5,1.75) {\small $\alpha$};
\node at (0,1.8) {\small $\beta$};
\node at (1.3,1.8) {\small $\gamma$};
\node at (1.3,2.2) {\small $\delta$};

\node at (0.5,-0.95) {\small $\alpha$};
\node at (0,-1) {\small $\beta$};
\node at (-1.3,-1) {\small $\gamma$};
\node at (-1.3,-1.35) {\small $\delta$};

\node at (2.2,1) {\small $\alpha$};
\node at (1.65,1) {\small $\beta$};
\node at (1.65,1.55) {\small $\gamma$};
\node at (1.65,1.9) {\small $\delta$};

\node at (1.65,0.05) {\small $\alpha$};
\node at (1.65,0.6) {\small $\beta$};
\node at (2.2,0.6) {\small $\gamma$};
\node at (1.65,-0.3) {\small $\delta$};

\node at (-2.4,0.8) {\small $\alpha$};
\node at (-1.65,0.4) {\small $\beta$};
\node at (-1.65,0.8) {\small $\gamma$};
\node at (-1.65,1.3) {\small $\delta$};

\node at (-2.4,-0.8) {\small $\alpha$};
\node at (-1.65,-0.3) {\small $\beta$};
\node at (-1.65,-0.8) {\small $\gamma$};
\node at (-1.65,-1.25) {\small $\delta$};

\node[inner sep=0.5, draw, shape=circle] at (-0.9,0.3) {\small $1$};
\node[inner sep=0.5, draw, shape=circle] at (0,0.8) {\small $2$};
\node[inner sep=0.5, draw, shape=circle] at (0.9,0.3) {\small $3$};
\node[inner sep=0.5, draw, shape=circle] at (0.9,1.3) {\small $5$};
\node[inner sep=0.5, draw, shape=circle] at (-0.9,1.3) {\small $4$};
\node[inner sep=0.5, draw, shape=circle] at (-2,0.8) {\small $6$};
\node[inner sep=0.5, draw, shape=circle] at (0.9,1.9) {\small $7$};
\node[inner sep=0.5, draw, shape=circle] at (1.9,1.3) {\small $8$};
\node[inner sep=0.5, draw, shape=circle] at (1.9,0.3) {\small $9$};
\node[inner sep=0, draw, shape=circle] at (0.75,-0.4) {\footnotesize $10$};
\node[inner sep=0, draw, shape=circle] at (-0.75,-0.4) {\footnotesize $11$};
\node[inner sep=0, draw, shape=circle] at (-0.9,-1.1) {\footnotesize $12$};
\node[inner sep=0, draw, shape=circle] at (-2,-0.8) {\footnotesize $13$};


\begin{scope}[shift={(4.8cm, 0.4cm)}]

\draw
	(0,-1) -- (-1,-1) -- (-1,1) -- (1,1) -- (-1,2) -- (-2,1) -- (-1,0) -- (-2,-1) -- (-1,-2)
	(-1,-2) -- (0,-2) -- (0,1)
	(0,0) -- (1,0) -- (1,-1)
	;

\draw[line width=1.2]
	(0,0) -- (-1,0)
	(1,0) -- (1,1)
	(0,-1) -- (1,-1)
	(-1,1) -- (-1,2) 
	(-1,-1) -- (-1,-2);

\node at (0.2,0.8) {\small $\alpha$};
\node at (0.2,0.2) {\small $\beta$};
\node at (0.8,0.2) {\small $\gamma$};
\node at (0.8,0.8) {\small $\delta$};

\node at (0.8,-0.2) {\small $\alpha$};
\node at (0.2,-0.2) {\small $\beta$};
\node at (0.2,-0.8) {\small $\gamma$};
\node at (0.8,-0.8) {\small $\delta$};

\node at (-0.8,0.8) {\small $\alpha$};
\node at (-0.2,0.8) {\small $\beta$};
\node at (-0.2,0.2) {\small $\gamma$};
\node at (-0.8,0.2) {\small $\delta$};

\node at (-0.8,-0.8) {\small $\alpha$};
\node at (-0.2,-0.8) {\small $\beta$};
\node at (-0.2,-0.2) {\small $\gamma$};
\node at (-0.8,-0.2) {\small $\delta$};

\node at (-0.2,-1.8) {\small $\alpha$};
\node at (-0.2,-1.2) {\small $\beta$};
\node at (-0.8,-1.2) {\small $\gamma$};
\node at (-0.8,-1.8) {\small $\delta$};

\node at (0.2,-1.2) {\small $\gamma$};

\node at (-1.75,-1) {\small $\alpha$};
\node at (-1.15,-0.4) {\small $\beta$};
\node at (-1.15,-1) {\small $\gamma$};
\node at (-1.15,-1.6) {\small $\delta$};

\node at (-1.75,1) {\small $\alpha$};
\node at (-1.15,0.45) {\small $\beta$};
\node at (-1.15,1) {\small $\gamma$};
\node at (-1.15,1.65) {\small $\delta$};

\node at (0.45,1.15) {\small $\alpha$};
\node at (0,1.2) {\small $\beta$};
\node at (-0.8,1.2) {\small $\gamma$};
\node at (-0.8,1.7) {\small $\delta$};

\node[inner sep=0.5, draw, shape=circle] at (-0.5,0.5) {\small $1$};
\node[inner sep=0.5, draw, shape=circle] at (-0.5,-0.5) {\small $2$};
\node[inner sep=0.5, draw, shape=circle] at (0.5,0.5) {\small $3$};
\node[inner sep=0.5, draw, shape=circle] at (0.5,-0.5) {\small $4$};
\node[inner sep=0.5, draw, shape=circle] at (-0.5,1.4) {\small $5$};
\node[inner sep=0.5, draw, shape=circle] at (-1.45,1) {\small $6$};
\node[inner sep=0.5, draw, shape=circle] at (-0.5,-1.5) {\small $7$};
\node[inner sep=0.5, draw, shape=circle] at (-1.45,-1) {\small $8$};

\end{scope}


\begin{scope}[shift={(7.5cm, 0.4cm)}]

\draw
 	(0,-1) -- (0,2) -- (2,2) -- (2,0) -- (1,0) -- (1,-1)
	(-1,1) -- (1,1) -- (1,-1) -- (-1,-1) -- (0,-2) -- (2,-2) -- (2,-1) -- (1,-1);

\draw[line width=1.2]
	(0,0) -- (1,0)
	(2,2) -- (1,1)
	(2,-1) -- (2,0)
	(0,-2) -- (1,-1);

\node at (0.8,0.8) {\small $\alpha$};
\node at (0.2,0.8) {\small $\beta$};
\node at (0.2,0.2) {\small $\gamma$};
\node at (0.8,0.2) {\small $\delta$};

\node at (0.8,-0.8) {\small $\alpha$};
\node at (0.2,-0.8) {\small $\beta$};
\node at (0.2,-0.2) {\small $\gamma$};
\node at (0.8,-0.2) {\small $\delta$};

\node at (1.2,-0.2) {\small $\alpha$};
\node at (1.2,-0.8) {\small $\beta$};
\node at (1.8,-0.8) {\small $\gamma$};
\node at (1.8,-0.2) {\small $\delta$};

\node at (1.8,0.2) {\small $\alpha$};
\node at (1.2,0.2) {\small $\beta$};
\node at (1.2,0.9) {\small $\gamma$};
\node at (1.8,1.6) {\small $\delta$};

\node at (0.2,1.8) {\small $\alpha$};
\node at (0.2,1.2) {\small $\beta$};
\node at (0.9,1.2) {\small $\gamma$};
\node at (1.6,1.8) {\small $\delta$};

\node at (-0.8,1.2) {\small $\alpha$};
\node at (-0.2,0) {\small $\alpha$};
\node at (-0.2,1.2) {\small $\beta$};
\node at (-0.2,0.8) {\small $\beta$};
\node at (-0.2,-0.8) {\small $\beta$};

\node at (0,-1.15) {\small $\alpha$};
\node at (-0.5,-1.2) {\small $\beta$};
\node at (0,-1.8) {\small $\gamma$};
\node at (0.55,-1.2) {\small $\delta$};

\node at (1.8,-1.2) {\small $\alpha$};
\node at (1.8,-1.8) {\small $\beta$};
\node at (0.45,-1.8) {\small $\gamma$};
\node at (1,-1.2) {\small $\delta$};

\node[inner sep=0.5, draw, shape=circle] at (0,-1.45) {\small $1$};
\node[inner sep=0.5, draw, shape=circle] at (0.5,-0.5) {\small $2$};
\node[inner sep=0.5, draw, shape=circle] at (1.4,-1.5) {\small $3$};
\node[inner sep=0.5, draw, shape=circle] at (1.5,-0.5) {\small $4$};
\node[inner sep=0.5, draw, shape=circle] at (0.5,0.5) {\small $5$};
\node[inner sep=0.5, draw, shape=circle] at (1.5,0.6) {\small $6$};
\node[inner sep=0.5, draw, shape=circle] at (0.6,1.5) {\small $7$};

\end{scope}

\end{tikzpicture}
\caption{Proposition \ref{abb-acc}: $\beta\gamma^2\delta^2$, $\beta^2\gamma^2$, $\alpha\beta^2$.}
\label{abb-accA}
\end{figure}

Now consider $\beta^2\gamma^2=\thick^{\delta}\gamma^{\beta}\thin\beta\thin\beta\thin^{\beta}\gamma^{\delta}\thick$. This determines $T_1,T_2$ in the second of Figure \ref{abb-accA}. The arrangements of $T_3,T_4$ indicate the two possible arrangements of the tiles containing $\beta$. On the one hand, $\thin\alpha_3\thin^{\gamma}\beta_1^{\alpha}\thin\cdots=\thin\alpha\thin^{\gamma}\beta^{\alpha}\thin^{\gamma}\beta^{\alpha}\thin$ determines $T_5$. Then $\alpha_1\gamma_5\cdots=\alpha\gamma^2$ determines $T_6$. On the other hand, $\beta_2\gamma_4\cdots=\beta^2\gamma^2$ and no $\alpha^2\cdots$ determine $T_7$. Then $\alpha_2\gamma_7\cdots=\alpha\gamma^2$ determines $T_8$. We find that, no matter how $T_3,T_4$ are arranged, we always get the same $T_6,T_8$. Then we get $\beta_6\beta_8\delta_1\delta_2\cdots$, a contradiction. Therefore $\beta^2\gamma^2$ is not a vertex. This implies $\beta^2\cdots=\alpha\beta^2,\beta^4$.

Finally, the proposition assumes $\alpha\beta^2$ is a vertex. By no $\gamma\delta\cdots$, the vertex has the AAD $\thin^{\beta}\alpha^{\delta}\thin^{\alpha}\beta^{\gamma}\thin\beta\thin$. The $\thin^{\beta}\alpha^{\delta}\thin^{\alpha}\beta^{\gamma}\thin$ part of the vertex determines $T_1,T_2$ in the third of Figure \ref{abb-accA}. Then $\alpha_2\delta_1\cdots=\alpha\beta\delta^2$ and no $\alpha^2\cdots$ determine $T_3,T_4$. Then $\alpha_4\delta_2\cdots=\alpha\beta\delta^2$ and no $\alpha^2\cdots$ determine $T_5,T_6$. Then $\alpha_5\gamma_6\cdots=\alpha\gamma^2$ determines $T_7$. Then $\gamma_2\gamma_5\cdots=\alpha\gamma^2,\gamma^4$ and $\beta_5\beta_7\cdots=\alpha\beta^2,\beta^4$ imply $\beta_5\beta_7\cdots=\beta^4$. Then by no $\alpha^2\cdots$, the AAD of this $\beta^4$ implies $\gamma_2\gamma_5\cdots=\alpha\gamma^2$. Then we have one $\alpha$ and two $\beta$ in a tile, again a contradiction. 
\end{proof}

\begin{proposition}\label{acc-bbb}
There is no tilings of the sphere by congruent almost equilateral quadrilaterals, such that 
$\alpha\gamma^2,\beta^3$ are vertices.
\end{proposition}

\begin{proof}
The angle sums of $\alpha\gamma^2,\beta^3$ and the angle sum for quadrilateral imply
\[
\beta=\tfrac{2}{3}\pi,\;
\gamma=(\tfrac{2}{3}-\tfrac{4}{f})\pi+\delta.
\]
By $f\ge 6$ and $\gamma\ne\delta$, we get $\gamma>\delta$. By Lemma \ref{geometry1}, this implies $\alpha>\beta$.

If $\alpha^2\cdots$ is a vertex, then $\alpha<\pi$. By $\alpha\gamma^2$, we also know $\gamma<\pi$ and $2\gamma>\pi>\alpha$. By $\alpha<2\gamma$ and Lemma \ref{geometry7}, we get $\beta<2\delta<2\gamma$. By $\beta^3$ and $\alpha>\beta$, we get $R(\alpha^2)<\beta<\alpha,2\delta,2\gamma$. This implies $\alpha^2\cdots$ is not a vertex.

By $\alpha\gamma^2$ and Lemma \ref{fbalance}, we know there is $\delta^2$-fan. By no $\alpha^2\cdots$ and Lemma \ref{klem3}, we know a $\delta^2$-fan has a single $\alpha$. If the fan has $\beta$, then the vertex is $\alpha\beta\cdots$. If the fan has no $\beta$, then it is $\thick\delta\thin\alpha\thin\delta\thick=\thick^{\gamma}\delta^{\alpha}\thin^{\beta}\alpha^{\delta}\thin^{\alpha}\delta^{\gamma}\thick$. The AAD again implies $\alpha\beta\cdots$ is a vertex. Therefore $\alpha\beta\cdots$ is always a vertex.

By $\alpha>\beta=\tfrac{2}{3}\pi$ and the angle sum for quadrilateral, we get $R(\alpha\beta)<\alpha,\beta,\gamma+\delta$. By $\gamma>\delta$ and the parity lemma, this implies $\alpha\beta\cdots=\alpha\beta\delta^k$. Since a $\delta^2$-fan has a single $\alpha$, we get $k=2$ in $\alpha\beta\delta^k$. The angle sum of $\alpha\beta\delta^2$ further implies $f=12$ and the following angle values. We also include \eqref{coolsaet_eq1}.
\begin{align*}
\alpha\beta\delta^2 &\colon
\beta=\tfrac{2}{3}\pi,\;
\gamma=\pi-\tfrac{1}{2}\alpha,\;
\delta=\tfrac{2}{3}\pi-\tfrac{1}{2}\alpha,\;
f=12. \\
&	\sin\tfrac{1}{2}\alpha 
	\sin(\tfrac{1}{3}\pi-\tfrac{1}{2}\alpha)
	=\sin\tfrac{1}{3}\pi 
	\sin(\pi-\alpha).
\end{align*}
By $\alpha>\beta$ and $\delta>0$, we get $\frac{2}{3}\pi<\alpha<\frac{4}{3}\pi$. The equation has no solution in this range.
\end{proof}

\begin{proposition}\label{aaa-add}
There is no tiling of the sphere by congruent almost equilateral quadrilaterals, such that 
$\alpha^3,\alpha\delta^2$ are vertices. 
\end{proposition}

\begin{proof}
The angle sums of $\alpha^3,\alpha\delta^2$ and the angle sum for quadrilateral imply
\[
\alpha=\delta=\tfrac{2}{3}\pi,\;
\beta+\gamma=(\tfrac{2}{3}+\tfrac{4}{f})\pi.
\]

If $\beta<\pi$, then all angles $<\pi$. By Lemma \ref{geometry3}, we get $\beta+\pi>\gamma+\delta$ and $\gamma+\pi>\alpha+\delta$. If $\beta<\alpha$, then by Lemma \ref{geometry1}, we get $\gamma>\delta$. This implies $\beta+\gamma>2\gamma+\delta-\pi>3\delta-\pi=\pi$. If $\alpha<\beta$, then $\beta+\gamma>\alpha+\beta+\delta-\pi>2\alpha+\delta-\pi=\pi$. 

Since $\beta\ge \pi$ also implies $\beta+\gamma>\pi$, we always have $\beta+\gamma>\pi$. This means $f<12$. Then $f=6,8,10$, and the equality \eqref{coolsaet_eq1} becomes
\begin{align*}
f=6 &\colon
\sin \tfrac{1}{3}\pi\sin(\tfrac{2}{3}\pi-\tfrac{1}{2}\beta)=\sin\tfrac{1}{2}\beta\sin(\pi-\beta). \\
f=8 &\colon
\sin \tfrac{1}{3}\pi\sin(\tfrac{2}{3}\pi-\tfrac{1}{2}\beta)=\sin\tfrac{1}{2}\beta\sin(\tfrac{5}{6}\pi-\beta). \\
f=10 &\colon
\sin \tfrac{1}{3}\pi\sin(\tfrac{2}{3}\pi-\tfrac{1}{2}\beta)=\sin\tfrac{1}{2}\beta\sin(\tfrac{11}{15}\pi-\beta).
\end{align*}
The only solution satisfying $\beta<(\tfrac{2}{3}+\tfrac{4}{f})\pi$ is $\beta=\frac{2}{3}\pi$ for $f=6$. By $\alpha=\beta$, the solution is dismissed.
\end{proof}

\begin{proposition}\label{aab-add}
There is no tiling of the sphere by congruent almost equilateral quadrilaterals, such that 
$\alpha^2\beta,\alpha\delta^2$ are vertices. 
\end{proposition}

\begin{proof}
By $\alpha^2\beta$, we get $\alpha<\pi$. We also know $\gamma,\delta<\pi$.

By $\alpha^2\beta,\alpha\delta^2$, we get $\alpha+\beta=2\delta$. By $\beta<2\delta$ and Lemma \ref{geometry7}, we get $\alpha<2\gamma$.

If $\alpha>\beta$, then by $\alpha+\beta=2\delta$, we get $\alpha>\delta>\beta$. Then by Lemma \ref{geometry5}, we get $\alpha<\gamma$. By $\alpha^2\beta$, this implies $R(\gamma^2)<\beta<\alpha,\gamma,\delta$. This implies $\gamma^2\cdots$ is not a vertex, a contradiction. 

Therefore we have $\alpha<\beta$. By Lemma \ref{geometry1}, we get $\gamma<\delta$. By $\alpha+\beta=2\delta$, we get $\alpha<\delta<\beta$. Then by $\alpha\delta^2$ and $\beta>\delta$, we get $R(\beta^2)<\alpha<\beta,\delta$. Therefore $\beta^2\cdots=\beta^2\gamma^k$. By Lemma \ref{geometry4}, we get $\beta+2\gamma>\pi$. By the parity lemma, this implies $k=2$ in $\beta^2\gamma^k$. The angle sums of $\alpha^2\beta,\alpha\delta^2,\beta^2\gamma^2$ and the angle sum for quadrilateral imply
\[
\alpha=\tfrac{8}{f}\pi,\;
\beta=(2-\tfrac{16}{f})\pi,\;
\gamma=(\tfrac{16}{f}-1)\pi,\;
\delta=(1-\tfrac{4}{f})\pi.
\]
By $\alpha<\delta$, we get $f>12$. By $\alpha<2\gamma$, we get $f<12$. The contradiction proves that $\beta^2\cdots$ is not a vertex.

By $\alpha^2\beta$ and $\alpha<\delta$, we get $R(\beta\delta)<\alpha<\beta,\delta$. This implies $\beta\delta\cdots=\beta\gamma^k\delta$. Therefore $\alpha\beta\cdots$ has no $\delta$. Then by $\alpha^2\beta$ and no $\beta^2\cdots$, we get $\alpha\beta\cdots=\alpha^2\beta,\alpha\beta\gamma^k$. Combining $\alpha\beta\cdots=\alpha^2\beta,\alpha\beta\gamma^k$, and $\beta\delta\cdots=\beta\gamma^k\delta$, and no $\beta^2\cdots$, we get $\beta\cdots=\alpha^2\beta,\alpha\beta\gamma^k,\beta\gamma^k,\beta\gamma^k\delta$. 

By $\alpha^2\beta$, and $\alpha<2\gamma$, and the parity lemma, we know $\alpha\beta\gamma^k$ is not a vertex, and $k=2$ in $\beta\gamma^k$. However, by $\alpha^2\beta$ and $\alpha\ne\gamma$, we know $\beta\gamma^2$ is not a vertex. By $\alpha^2\beta$ and $\alpha<\beta,2\gamma,\delta$, we know $k=1$ in $\beta\gamma^k\delta$. Therefore $\beta\cdots=\alpha^2\beta,\beta\gamma\delta$. By applying the counting lemma to $\alpha,\beta$, this implies $\beta\gamma\delta$ is a vertex. However, by Proposition \ref{acd-bcc} (after exchanging $(\alpha,\delta)$ with $(\beta,\gamma)$), there is no tiling with vertices $\alpha^2\beta,\alpha\delta^2,\beta\gamma\delta$.  
\end{proof}

\begin{proposition}\label{add-bbb}
Tiling of the sphere by congruent almost equilateral quadrilaterals, such that 
$\alpha\delta^2,\beta^3$ are vertices, is $S_{12\square}1$. 
\end{proposition}

The tiling is obtained in Figure \ref{add-bbbA}, and is $S_{12\square}1$ in Figure \ref{sporatic_tiling} after exchanging $(\alpha,\delta)$ with $(\beta,\gamma)$.

\begin{proof}
The angle sums of $\alpha\delta^2,\beta^3$ and the angle sum for quadrilateral imply
\[
\alpha+2\gamma=(\tfrac{2}{3}+\tfrac{8}{f})\pi,\;
\beta=\tfrac{2}{3}\pi,\;
\delta=(\tfrac{2}{3}-\tfrac{4}{f})\pi+\gamma.
\]
We have $2\alpha+3\gamma+\delta=2(\alpha+2\delta)+3(\gamma-\delta)=(2+\frac{12}{f})\pi>2\pi$. 

By $f\ge 6$ and $\gamma\ne\delta$, we get $\gamma<\delta$. By Lemma \ref{geometry1}, this implies $\alpha<\beta=\tfrac{2}{3}\pi$. Then by $\alpha\delta^2$, we get $\alpha<\beta<\delta$. By $\beta<\delta$ and Lemma \ref{geometry5}, we get $\alpha<\gamma$. 

By $\alpha\delta^2$, we get $R(\delta^2)=\alpha<\beta,\gamma,\delta$. This implies $\delta^2\cdots=\alpha\delta^2=\delta\thick\delta\cdots$. Therefore $\delta\thin\delta\cdots$ is not a vertex.

By $\delta^2\cdots=\alpha\delta^2$, and $\alpha<\beta$, and the parity lemma, we get $\beta\delta\cdots=\beta\gamma\delta\cdots$, with no $\delta$ in the remainder. By the angle sum for quadrilateral, we get $R(\beta\gamma\delta)<\alpha<\beta,\gamma$. Therefore $\beta\gamma\delta\cdots=\beta\gamma\delta$. By Proposition \ref{acd-bcc} (after exchanging $(\alpha,\delta)$ with $(\beta,\gamma)$), there is no tiling with vertices $\alpha\delta^2,\beta^3,\beta\gamma\delta$. Therefore $\beta\delta\cdots$ is not a vertex.

By no $\beta\delta\cdots,\delta\thin\delta\cdots$, we get the unique AAD $\thin^{\delta}\alpha^{\beta}\thin^{\beta}\alpha^{\delta}\thin$ of $\thin\alpha\thin\alpha\thin$ and the unique AAD $\thin^{\delta}\alpha^{\beta}\thin^{\beta}\gamma^{\delta}\thick$ of $\thin\alpha\thin\gamma\thick$. This implies no consecutive $\alpha\alpha\alpha,\alpha\alpha\gamma,\gamma\alpha\gamma$, and further implies $\alpha^k\gamma^l(k,l\ge 1)$ is not a vertex. 

The possible AADs $\thin^{\alpha}\beta^{\gamma}\thin^{\alpha}\beta^{\gamma}\thin^{\alpha}\beta^{\gamma}\thin$ and $\thin^{\alpha}\beta^{\gamma}\thin^{\alpha}\beta^{\gamma}\thin^{\gamma}\beta^{\alpha}\thin$ of $\beta^3$ imply $\alpha\gamma\cdots$ is a vertex. By $\delta^2\cdots=\alpha\delta^2$ and the parity lemma, we get $\alpha\gamma\cdots=\alpha\gamma^l\cdots(l\ge 2),\alpha\gamma^l\delta\cdots(l\ge 1)$, with no $\gamma,\delta$ in the remainders. 

By no $\alpha^k\gamma^l$($k,l\ge 1$), we get $\alpha\gamma^l\cdots=\alpha\beta\gamma^l\cdots$. Then by $\alpha<\gamma$, and $\alpha+2\gamma>\frac{2}{3}\pi=\beta$, and the parity lemma, we get $\alpha\gamma^l\cdots=\alpha^k\beta\gamma^2,\alpha\beta\gamma^4$. By no consecutive $\alpha\alpha\alpha,\alpha\alpha\gamma,\gamma\alpha\gamma$, and the edge consideration, we get $k=1,2$ in $\alpha^k\beta\gamma^2$.

By the angle sum for quadrilateral, we know $\alpha\gamma^l\delta\cdots$ has no $\beta$. Then by $\alpha+4\gamma+\delta>2\alpha+3\gamma+\delta>2\pi$, and the parity lemma, we get $\alpha\gamma^l\delta\cdots=\alpha^k\gamma\delta,\alpha\gamma^3\delta$. Then by no consecutive $\alpha\alpha\alpha,\alpha\alpha\gamma,\gamma\alpha\gamma$, we get $k=1$ in $\alpha^k\gamma\delta$. However, by $\alpha\delta^2$, we know $\alpha\gamma\delta$ is not a vertex.

We conclude $\alpha\gamma\cdots=\alpha\beta\gamma^2,\alpha\beta\gamma^4,\alpha^2\beta\gamma^2,\alpha\gamma^3\delta$ is a vertex. The angle sum of one of the vertices, and the angle sums of $\alpha\delta^2,\beta^3$, and the angle sum for quadrilateral, imply the following angles. We also include \eqref{coolsaet_eq1}.
\begin{itemize}
\item $\alpha\beta\gamma^2$:
	$\alpha=\tfrac{4}{3}\pi-2\gamma,\;
	\beta=\tfrac{2}{3}\pi,\;
	\delta=\tfrac{1}{3}\pi+\gamma,\;
	f=12$.
	\newline 
	$\sin(\tfrac{2}{3}\pi-\gamma)
	\sin\gamma
	=\sin\tfrac{1}{3}\pi
	\sin(2\gamma-\tfrac{2}{3}\pi)$. 
\item $\alpha\beta\gamma^4$:
	$\alpha=\tfrac{16}{f}\pi,\;
	\beta=\tfrac{2}{3}\pi,\;
	\gamma=(\tfrac{1}{3}-\tfrac{4}{f})\pi,\;
	\delta=(1-\tfrac{8}{f})\pi$.
	\newline 
	$\sin\tfrac{8}{f}\pi
	\sin(\tfrac{2}{3}-\tfrac{8}{f})\pi
	=\sin\tfrac{1}{3}\pi
	\sin(\tfrac{1}{3}-\tfrac{12}{f})\pi$. 
\item $\alpha^2\beta\gamma^2$:
	$\alpha=(\tfrac{2}{3}-\tfrac{8}{f})\pi,\;
	\beta=\tfrac{2}{3}\pi,\;
	\gamma=\tfrac{8}{f}\pi,\;
	\delta=(\tfrac{2}{3}+\tfrac{4}{f})\pi$.
	\newline 
	$\sin(\tfrac{1}{3}-\tfrac{4}{f})\pi
	\sin(\tfrac{1}{3}+\tfrac{4}{f})\pi
	=\sin\tfrac{1}{3}\pi
	\sin(\tfrac{12}{f}-\tfrac{1}{3})\pi$.
\item $\alpha\gamma^3\delta$:
	$\alpha=\tfrac{12}{f}\pi,\;
	\beta=\tfrac{2}{3}\pi,\;
	\gamma=(\tfrac{1}{3}-\tfrac{2}{f})\pi,\;
	\delta=(1-\tfrac{6}{f})\pi$.
	\newline 
	$\sin\tfrac{6}{f}\pi
	\sin(\tfrac{2}{3}-\tfrac{6}{f})\pi
	=\sin\tfrac{1}{3}\pi
	\sin(\tfrac{1}{3}-\tfrac{8}{f})\pi$. 
\end{itemize}
By \eqref{quadvcountf}, the vertices of degree $>3$ imply $f>6$. The last three equations have no solution for even $f>6$. For the first equation, by $\alpha>0$, we get $\gamma<\frac{2}{3}\pi$. The solution satisfying $0<\gamma<\frac{2}{3}\pi$ is $\gamma=0.4568\pi$. Then we get approximate values of $\alpha,\beta,\gamma,\delta$, and the approximate values imply 
\begin{equation}\label{add-bbbAVC}
\text{AVC}=\{\alpha\delta^2,\beta^3,\alpha\beta\gamma^2\}.
\end{equation}
For details on how to calculate the AVC from the approximate values, see \cite[Section 5.1]{awy} and \cite{cl}.

Now we find the tiling for the AVC \eqref{add-bbbAVC}. 
By no $\alpha^2\cdots$, we get the unique AAD $\thin^{\alpha}\beta^{\gamma}\thin^{\alpha}\beta^{\gamma}\thin^{\alpha}\beta^{\gamma}\thin$ of $\beta^3$. This determines the top three tiles in Figure \ref{add-bbbA}. Then by $\delta\cdots=\alpha\delta^2$ and $\alpha\gamma\cdots=\alpha\beta\gamma^2$, we determine the six tiles in the middle row. Then by $\alpha\delta\cdots=\alpha\delta^2$ and $\beta\gamma\cdots=\alpha\beta\gamma^2$, we determine the three bottom tiles. After exchanging $(\alpha,\delta)$ with $(\beta,\gamma)$, the tiling is $S_{12\square}1$.

\begin{figure}[htp]
\centering
\begin{tikzpicture}


\foreach \a in {0,1,2}
\foreach \b in {1,-1}
{
\begin{scope}[xshift=2*\a cm, scale=\b]

\draw
	(1,1.3) -- (1,-1.3)
	(1,0.5) -- (0,0.5) -- (0,0);

\draw[line width=1.2]
	(-1,0.5) -- (0,0.5);

\node at (0.8,0.7) {\small $\alpha$};
\node at (0,1.2) {\small $\beta$};
\node at (0,0.7) {\small $\delta$};
\node at (-0.8,0.7) {\small $\gamma$};

\node at (0.2,0.3) {\small $\alpha$};
\node at (0.8,0.3) {\small $\beta$};
\node at (0.2,-0.3) {\small $\delta$};
\node at (0.8,-0.3) {\small $\gamma$};
	
\end{scope}
}

\end{tikzpicture}
\caption{Proposition \ref{add-bbb}: Tiling for $\{\alpha\delta^2,\beta^3,\alpha\beta\gamma^2\}$. }
\label{add-bbbA}
\end{figure}

\medskip

\noindent{\em Geometry of Quadrilateral}

\medskip

For the AVC $\{\alpha\delta^2,\alpha\beta\gamma^2,\beta^k\}$, we get the similar earth map tiling with $k$ timezones. However, in the subsequent propositions, we only get four timezones. Here we justify the existence of the quadrilateral for three and four timezones, corresponding to the following:
\begin{align*}
f=12 &\colon \alpha=2\pi-2\delta,\;
	\beta=\tfrac{2}{3}\pi,\;
	\gamma=\delta-\tfrac{1}{3}\pi. \\
f=16 &\colon \alpha=2\pi-2\delta,\;
	\beta=\tfrac{1}{2}\pi,\;
	\gamma=\delta-\tfrac{1}{4}\pi.	
\end{align*}
We have $\gamma<\delta$. By Lemma \ref{geometry1}, this implies $\alpha<\beta$. This means $\delta>\frac{2}{3}\pi$ for $f=12$ and $\delta>\frac{3}{4}\pi$ for $f=16$.

Substituting the angle formulae into \eqref{coolsaet_eq1}, we get 
\begin{align*}
f=12 &\colon
8\cos^2\delta-5=0. \\
f=16 &\colon
(1+\sqrt{2})\tan^2\delta+(2-\sqrt{2})\tan\delta-1=0.
\end{align*}
This means $\delta$ is given by 
\[
\tan\delta=\begin{cases}
-\frac{\sqrt{3}}{\sqrt{5}}, & f=12, \\
2-\sqrt{5}-\sqrt{7-3\sqrt{5}}, & f=16,
\end{cases}
\quad \tfrac{1}{2}\pi<\delta<\pi.
\]
This gives precise values of the other angles. Then we substitute the precise values into \eqref{coolsaet_eq2}, and find $a$ is determined by
\[
\cos a=\begin{cases}
\frac{2}{3}\sqrt{5}-1, & f=12, \\
\frac{1}{2}(-3-\sqrt{2}+\sqrt{5}+\sqrt{10}), & f=16,
\end{cases}
\quad a<\pi.
\]
Then we substitute into \eqref{coolsaet_eq4} and use $K=Y(b)^T$ to get $\cos b$ and $\sin b$. We find $\sin b=0.7060$ for $f=12$ and $\sin b=0.3246$ for $f=16$. By $\sin b>0$, we get $b<\pi$, and find $b$ is determined by
\[
\cos b=\begin{cases}
3\sqrt{5}-6, & f=12, \\
-9-6\sqrt{2}+4\sqrt{5}+3\sqrt{10}, & f=16,
\end{cases}
\quad b<\pi.
\]

In summary, we get the following approximate values
\begin{align*}
f=12 &\colon 
\alpha=0.4195\pi,\;
\beta=\tfrac{2}{3}\pi,\;
\gamma=0.4568\pi,\;
\delta=0.7902\pi, \\
&\quad
a=0.3367\pi,\;
b=0.2495\pi.
\\
f=16 &\colon 
\alpha=0.4202\pi,\;
\beta=\tfrac{1}{2}\pi,\;
\gamma=0.5398\pi,\;
\delta=0.7898\pi, \\
&\quad
a=0.3362\pi,\;
b=0.1052\pi.
\end{align*}
The equality $K=Y(b)^T$ means there are quadrilaterals with the given data. By Lemma \ref{geometry8}, we also know the quadrilaterals are simple. Therefore the quadrilaterals are geometrically suitable for tilings.
\end{proof}

\subsection{Single Degree $3$ Vertex}

In this section, we assume the tiling has only one degree $3$ vertex. Up to the exchange of $(\alpha,\delta)$ with $(\beta,\gamma)$, the only degree $3$ vertex is $\alpha^3,\alpha\beta^2,\alpha\gamma^2,\alpha\delta^2,\alpha\gamma\delta$. 

All propositions still implicitly assume $\alpha\ne\beta,\gamma\ne\delta,\alpha\ne\gamma,\beta\ne\delta$. In the proof, we will omit mentioning the parity lemma, which says the total number of $\gamma,\delta$ at any vertex is even. We will also omit mentioning the angle sum \eqref{anglesum} for quadrilateral, and its consequence $\alpha+\beta+\gamma+\delta>2\pi$. 

\begin{proposition}\label{aaa}
Tiling of the sphere by congruent almost equilateral quadrilaterals, such that $\alpha^3$ is the only degree $3$ vertex, is the $a=b$ reduction of $Q_{\square}P_6$.
\end{proposition}

\begin{proof}
By Lemma \ref{count-aaa}, we know $f\ge 24$.  

By the only degree $3$ vertex $\alpha^3$, we know $\gamma,\delta$ do not appear at degree $3$ vertices. By Lemma \ref{deg3miss} (and the parity lemma), we know one of $\gamma^4,\gamma^3\delta,\gamma^2\delta^2,\gamma\delta^3,\delta^4$ is a vertex. By the balance lemma, we know $\gamma^2\cdots,\delta^2\cdots$ are always vertices. Therefore $\gamma,\delta<\pi$.

\subsubsection*{Case. $\alpha^3,\gamma^4$ are vertices}

The angle sums of $\alpha^3,\gamma^4$ (and the angle sum for quadrilateral) imply
\[
\alpha=\tfrac{2}{3}\pi,\;
\beta+\delta=(\tfrac{5}{6}+\tfrac{4}{f})\pi,\;
\gamma=\tfrac{1}{2}\pi.
\]
By $\alpha>\gamma$ and Lemma \ref{geometry5}, we get $\beta>\delta$. This implies $3\beta+2\delta>\frac{5}{2}(\beta+\delta)>2\pi$. 

By the only degree $3$ vertex $\alpha^3$, we know $\beta,\delta$ do not appear at degree $3$ vertices. Then by Lemma \ref{deg3miss}, we know one of $\alpha\beta^3,\beta^4,\alpha\beta\delta^2,\beta^2\gamma\delta,\beta^2\delta^2,\gamma\delta^3,\delta^4,\beta^5,\beta^3\delta^2,\beta\delta^4$ is a vertex. By $\gamma^4$ and $\gamma\ne\delta$, we know $\gamma\delta^3,\delta^4$ are not vertices. By $\beta>\delta$ and $3\beta+2\delta>2\pi$, we know $\beta^5,\beta^3\delta^2$ are not vertices. The angle sum of one of $\alpha\beta^3,\beta^4,\alpha\beta\delta^2,\beta^2\gamma\delta,\beta^2\delta^2,\beta\delta^4$, and the angle sums of $\alpha^3,\gamma^4$, imply the following, including \eqref{coolsaet_eq1}.
\begin{itemize}
\item $\alpha\beta^3$:
	$\alpha=\tfrac{2}{3}\pi,\;
	\beta=\tfrac{4}{9}\pi,\;
	\gamma=\tfrac{1}{2}\pi,\;
	\delta=(\tfrac{7}{18}+\tfrac{4}{f})\pi$.
	\newline 
	$\sin\tfrac{1}{3}\pi
	\sin(\tfrac{1}{6}+\tfrac{4}{f})\pi
	=\sin\tfrac{2}{9}\pi
	\sin\tfrac{1}{6}\pi$. 
\item $\beta^4$:
	$\alpha=\tfrac{2}{3}\pi,\;
	\beta=\tfrac{1}{2}\pi,\;
	\gamma=\tfrac{1}{2}\pi,\;
	\delta=(\tfrac{1}{3}+\tfrac{4}{f})\pi$.
	\newline 
	$\sin\tfrac{1}{3}\pi
	\sin(\tfrac{1}{12}+\tfrac{4}{f})\pi
	=\sin\tfrac{1}{4}\pi
	\sin\tfrac{1}{6}\pi$. 
\item $\alpha\beta\delta^2$:
	$\alpha=\tfrac{2}{3}\pi,\;
	\beta=(\tfrac{1}{3}+\tfrac{8}{f})\pi,\;
	\gamma=\tfrac{1}{2}\pi,\;
	\delta=(\tfrac{1}{2}-\tfrac{4}{f})\pi$.
	\newline 
	$\sin\tfrac{1}{3}\pi
	\sin(\tfrac{1}{3}-\tfrac{8}{f})\pi
	=\sin(\tfrac{1}{6}+\tfrac{4}{f})\pi
	\sin\tfrac{1}{6}\pi$. 
\item $\beta^2\gamma\delta$:
	$\alpha=\tfrac{2}{3}\pi,\;
	\beta=(\tfrac{2}{3}-\tfrac{4}{f})\pi,\;
	\gamma=\tfrac{1}{2}\pi,\;
	\delta=(\tfrac{1}{6}+\tfrac{8}{f})\pi$.
	\newline 
	$\sin\tfrac{1}{3}\pi
	\sin(\tfrac{10}{f}-\tfrac{1}{6})\pi
	=\sin(\tfrac{1}{3}-\tfrac{2}{f})\pi
	\sin\tfrac{1}{6}\pi$. 
\item $\beta^2\delta^2$:
	$\alpha=\tfrac{2}{3}\pi,\;
	\beta+\delta=\pi,\;
	\gamma=\tfrac{1}{2}\pi,\;
	f=24$.
	\newline 
	$\sin\tfrac{1}{3}\pi
	\sin(\pi-\tfrac{3}{2}\beta)
	=\sin\tfrac{1}{2}\beta
	\sin\tfrac{1}{6}\pi$. 
\item $\beta\delta^4$:
	$\alpha=\tfrac{2}{3}\pi,\;
	\beta=(\tfrac{4}{9}+\tfrac{16}{3f})\pi,\;
	\gamma=\tfrac{1}{2}\pi,\;
	\delta=(\tfrac{7}{18}-\tfrac{4}{3f})\pi$.
	\newline 
	$\sin\tfrac{1}{3}\pi
	\sin(\tfrac{1}{6}-\tfrac{4}{f})\pi
	=\sin(\tfrac{2}{9}+\tfrac{8}{3f})\pi
	\sin\tfrac{1}{6}\pi$.
\end{itemize}
All equations except the one for $\beta^2\delta^2$ have no solution for even $f\ge 24$. The solution for $\beta^2\delta^2$ is $\beta=0.5677\pi$. Then we use approximate angle values to get all the vertices
\[
\text{AVC}=\{\alpha^3,\beta^2\delta^2,\gamma^4\}.
\]

The AAD $\thin^{\beta}\gamma^{\delta}\thick^{\delta}\gamma^{\beta}\thin^{\beta}\gamma^{\delta}\thick^{\delta}\gamma^{\beta}\thin$ of $\gamma^4$ determines the center four tiles in Figure \ref{aaaA} that we denote by $n(\gamma^4)$. Then $\beta^2\cdots=\beta^2\delta^2$, $\delta^2\cdots=\beta^2\delta^2$, and $\alpha\cdots=\alpha^3$ determine the first layer of eight tiles around the center tiles. Then $\gamma^2\cdots=\gamma^4$ determines the second layer of another eight tiles. Then $\beta^2\cdots=\beta^2\delta^2$, $\delta^2\cdots=\beta^2\delta^2$, and $\alpha^2\cdots=\alpha^3$ determine the last four tiles. The four neighborhood tilings $n(\gamma^4)$ form six faces of a cube, and the tiling is (the $a=b$ reduction of) the quadrilateral subdivision $Q_{\square}P_6$ of this cube. 

\begin{figure}[htp]
\centering
\begin{tikzpicture}

\foreach \a in {-1,1}
\foreach \b in {-1,1}
{
\begin{scope}[xscale=\a, yscale=\b]

\draw
	(0,0) -- (0,1) -- (1,1) -- (1,0) -- (2.5,0)
	(1,1) -- (2.5,2.5)
	(0,2.5) -- (2.5,2.5) -- (2.5,0)
	(0,1.75) -- (1.75,1.75)
	(0,2.5) -- (0,3.2)
	;

\draw[line width=1.2]
	(0,0) -- (1,0)
	(0,1) -- (0,2.5)
	(1.75,1.75) -- (1.75,-1.75)
	(2.5,0) -- (3.2,0);
	
\node at (0.8,0.8) {\small $\alpha$};
\node at (0.8,0.23) {\small $\delta$};
\node at (0.2,0.8) {\small $\beta$};
\node at (0.2,0.2) {\small $\gamma$};

\node at (0.95,1.2) {\small $\alpha$};
\node at (0.2,1.2) {\small $\delta$};
\node at (1.3,1.55) {\small $\beta$};
\node at (0.2,1.55) {\small $\gamma$};

\node at (1.2,0.95) {\small $\alpha$};
\node at (1.57,1.3) {\small $\delta$};
\node at (1.2,0.2) {\small $\beta$};
\node at (1.57,0.2) {\small $\gamma$};

\node at (2.1,2.35) {\small $\alpha$};
\node at (0.2,2.3) {\small $\delta$};
\node at (1.75,1.95) {\small $\beta$};
\node at (0.2,1.95) {\small $\gamma$};

\node at (2.32,2.1) {\small $\alpha$};
\node at (1.92,1.65) {\small $\delta$};
\node at (2.32,0.2) {\small $\beta$};
\node at (1.92,0.2) {\small $\gamma$};

\node at (2.6,2.6) {\small $\alpha$};
\node at (2.7,0.23) {\small $\delta$};
\node at (0.2,2.7) {\small $\beta$};
\node at (2.9,2.9) {\small $\gamma$};

\end{scope}
}
	
\end{tikzpicture}
\caption{Proposition \ref{aaa}: Tiling for $\{\alpha^3,\beta^2\delta^2,\gamma^4\}$.}
\label{aaaA}
\end{figure}

\medskip

\noindent{\em Geometry of Quadrilateral}

\medskip

The tiling in Figure \ref{aaaA} is derived from the equation $\sin\tfrac{1}{3}\pi\sin(\pi-\tfrac{3}{2}\beta)=\sin\tfrac{1}{2}\beta\sin\tfrac{1}{6}\pi$. By $\beta<\pi$, the equation implies
\[
\cos\beta=\tfrac{1}{2\sqrt{3}}(1-\sqrt{3}),\quad
\beta<\pi.
\]
Then we substitute the precise values into \eqref{coolsaet_eq2}, and find
\[
\cos a=
\tfrac{1}{\sqrt{5-2\sqrt{3}}}, 
\quad a<\pi.
\]
Then we substitute into \eqref{coolsaet_eq4} and use $K=Y(b)^T$ to determine $b$. The calculation shows $a+b=\frac{1}{2}\pi$, which can also obtained by a simple geometric argument. Then by Lemma \ref{geometry8}, the quadrilateral is geometrically suitable for tiling.

\subsubsection*{Case. $\alpha^3,\gamma^3\delta$ are vertices}

The angle sums of $\alpha^3,\gamma^3\delta$ imply
\[
\alpha=\tfrac{2}{3}\pi,\;
\gamma=(\tfrac{1}{3}-\tfrac{2}{f})\pi+\tfrac{1}{2}\beta,\;
\delta=(1+\tfrac{6}{f})\pi-\tfrac{3}{2}\beta.
\]
We have $3\beta+2\delta=(2+\frac{12}{f})\pi>2\pi$.

If $\alpha<\beta$, then $\gamma-\delta=2\beta-(\frac{2}{3}+\tfrac{8}{f})\pi>2\alpha-(\frac{2}{3}+\tfrac{8}{f})\pi=(\frac{2}{3}-\tfrac{8}{f})\pi>0$. By Lemma \ref{geometry1}, this is a contradiction, and we get $\alpha>\beta$ and $\gamma>\delta$. Then by $\gamma<(\tfrac{1}{3}-\tfrac{2}{f})\pi+\tfrac{1}{2}\alpha<\alpha$ and Lemma \ref{geometry5}, we get $\beta>\delta$. Then by $3\beta+2\delta>2\pi$, we get $\beta>\frac{2}{5}\pi$. Then $\alpha+\beta>\pi$ and $2\gamma=(\tfrac{2}{3}-\tfrac{4}{f})\pi+\beta>\alpha$. By Lemma \ref{geometry7}, this implies $\beta<2\delta$. 

The AAD $\thick^{\delta}\gamma^{\beta}\thin^{\alpha}\delta^{\gamma}\thick$ at $\gamma^3\delta$ implies $\alpha\beta\cdots$ is a vertex. By the only degree $3$ vertex $\alpha^3$, the vertex has degree $\ge 4$. If $\alpha\beta\cdots$ is a $b$-vertex, then by $\gamma>\delta$ (and $\alpha+\beta+\gamma+\delta>2\pi$), the vertex has no $\gamma$. Then by $\alpha+\beta+4\delta>\alpha+2\beta+2\delta>3\beta+2\delta>2\pi$, we get $\alpha\beta\cdots=\alpha\beta\delta^2$. If $\alpha\beta\cdots$ is a $\hat{b}$-vertex, then by $\alpha>\beta$, and $\alpha+\beta>\pi$, and $\alpha+4\beta>3\beta+2\delta>2\pi$, we get $\alpha\beta\cdots=\alpha^2\beta,\alpha\beta^2,\alpha\beta^3$. By $\alpha^3$ and $\alpha>\beta$, we know $\alpha^2\beta,\alpha\beta^2$ are not vertices. The angle sum of one of $\alpha\beta^3,\alpha\beta\delta^2$ further implies the following, including \eqref{coolsaet_eq1}.
\begin{itemize}
\item $\alpha\beta^3$:
	$\alpha=\tfrac{2}{3}\pi,\;
	\beta=\tfrac{4}{9}\pi,\;
	\gamma=(\tfrac{5}{9}-\tfrac{2}{f})\pi,\;
	\delta=(\tfrac{1}{3}+\tfrac{6}{f})\pi$.
	\newline 
	$\sin\tfrac{1}{3}\pi
	\sin(\tfrac{1}{9}+\tfrac{6}{f})\pi
	=\sin\tfrac{2}{9}\pi
	\sin(\tfrac{2}{9}-\tfrac{2}{f})\pi$. 
\item $\alpha\beta\delta^2$:
	$\alpha=\tfrac{2}{3}\pi,\;
	\beta=(\tfrac{1}{3}+\tfrac{6}{f})\pi,\;
	\gamma=(\tfrac{1}{2}+\tfrac{1}{f})\pi,\;
	\delta=(\tfrac{1}{2}-\tfrac{3}{f})\pi$.
	\newline 
	$\sin\tfrac{1}{3}\pi
	\sin(\tfrac{1}{3}-\tfrac{6}{f})\pi
	=\sin(\tfrac{1}{6}+\tfrac{3}{f})\pi
	\sin(\tfrac{1}{6}+\tfrac{1}{f})\pi$. 
\end{itemize}
The equations have no solution for even $f\ge 24$.

\subsubsection*{Case. $\alpha^3,\gamma^2\delta^2$ are vertices}

The angle sums of $\alpha^3,\gamma^2\delta^2$ imply
\[
\alpha=\tfrac{2}{3}\pi,\;
\beta=(\tfrac{1}{3}+\tfrac{4}{f})\pi,\;
\gamma+\delta=\pi.
\]
We have $\alpha>\beta$. By Lemma \ref{geometry1}, this implies $\gamma>\delta$. 

By $\gamma>\delta$, we get $R(\gamma^2)<R(\gamma\delta)<\alpha+\beta$. By $\beta<\alpha<2\beta$ and the only degree $3$ vertex $\alpha^3$, this implies that, if the remainders have no $\gamma,\delta$, then we get $\gamma^2\cdots=\beta^2\gamma^2$ and $\gamma\delta\cdots=\beta^2\gamma\delta$. By $\gamma^2\delta^2$ and $\beta\ne\delta$, we know $\beta^2\gamma^2$ is not a vertex. Therefore $R(\gamma^2)$ has $\gamma,\delta$. Then by $\gamma^2\delta^2$, and $\gamma>\delta$, we get $\gamma^2\cdots=\gamma^2\delta^2$. This implies the number of $\gamma$ is always no more than the number of $\delta$ at all the vertices. Then by the counting lemma, we know every vertex has the same number of $\gamma$ and $\delta$. Therefore  $\beta^2\gamma\delta,\gamma^2\delta^2$ are all the $b$-vertices. 

The possible AADs $\thin^{\beta}\alpha^{\delta}\thin^{\beta}\alpha^{\delta}\thin^{\beta}\alpha^{\delta}\thin$ and $\thin^{\beta}\alpha^{\delta}\thin^{\beta}\alpha^{\delta}\thin^{\delta}\alpha^{\beta}\thin$ of $\alpha^3$ imply $\beta\delta\cdots=\beta^2\gamma\delta$ is a vertex. The angle sum of $\beta^2\gamma\delta$ further implies 
\[
\alpha=\tfrac{2}{3}\pi,\;
	\beta=\tfrac{1}{2}\pi,\;
	\gamma+\delta=\pi,\;
	f=24.
\]
Then we use the only $b$-vertices $\gamma^2\delta^2,\beta^2\gamma\delta$ to get all the vertices
\[
\text{AVC}=\{\alpha^3,\beta^4,\beta^2\gamma\delta,\gamma^2\delta^2\}. 
\]
By $\alpha\cdots=\alpha^3$, the AAD of $\beta^2\gamma\delta$ is $\thick^{\gamma}\delta^{\alpha}\thin^{\alpha}\beta^{\gamma}\thin\beta\thin^{\beta}\gamma^{\delta}\thick$. This implies $\alpha\beta\cdots$ or $\alpha\gamma\cdots$ is a vertex, a contradiction.

\subsubsection*{Case. $\alpha^3,\gamma\delta^3$ are vertices}

The angle sums of $\alpha^3,\gamma\delta^3$ imply
\[
\alpha=\tfrac{2}{3}\pi,\;
\gamma=(1+\tfrac{6}{f})\pi-\tfrac{3}{2}\beta,\;
\delta=(\tfrac{1}{3}-\tfrac{2}{f})\pi+\tfrac{1}{2}\beta.
\]
By $\gamma\delta^3$, we get $\delta<\frac{2}{3}\pi$. 
By $\gamma>0$, we get $\beta<(\frac{2}{3}+\frac{4}{f})\pi< \pi$. 

By Lemma \ref{geometry3}, we get $\gamma+\pi>\alpha+\delta$. This implies $\beta<(\frac{1}{2}+\tfrac{4}{f})\pi\le\alpha$. Then by Lemma \ref{geometry1}, we get $\gamma>\delta$. This implies $\beta<(\frac{1}{3}+\frac{4}{f})\pi$, and $\delta-\beta=(\tfrac{1}{3}-\tfrac{2}{f})\pi-\tfrac{1}{2}\beta>(\frac{1}{6}-\frac{4}{f})\pi\ge 0$. Then by Lemma \ref{geometry5}, we get $\alpha<\gamma$. 

We know $\gamma^2\cdots$ is a vertex. By $\gamma\delta^3$ and $\gamma>\delta$, we know $R(\gamma^2)$ has no $\gamma,\delta$. By $\alpha^3$ and $\alpha<\gamma$, we know $R(\gamma^2)$ has no $\alpha$. Then by $3\beta+2\gamma=(2+\tfrac{12}{f})\pi>2\pi$, and the only degree $3$ vertex $\alpha^3$, we get $\gamma^2\cdots=\beta^2\gamma^2$. The angle sum of $\beta^2\gamma^2$ further implies the following, including \eqref{coolsaet_eq1}.
\begin{align*}
\beta^2\gamma^2 &\colon
	\alpha=\tfrac{2}{3}\pi,\;
	\beta=\tfrac{12}{f}\pi,\;
	\gamma=(1-\tfrac{12}{f})\pi,\;
	\delta=(\tfrac{1}{3}+\tfrac{4}{f})\pi. \\
&	\sin\tfrac{1}{3}\pi
	\sin(\tfrac{1}{3}-\tfrac{2}{f})\pi
	=\sin\tfrac{6}{f}\pi
	\sin(\tfrac{2}{3}-\tfrac{12}{f})\pi.
\end{align*}
The equation has no solution for even $f\ge 24$. 

\subsubsection*{Case. $\alpha^3,\delta^4$ are vertices}

The angle sums of $\alpha^3,\delta^4$ imply
\[
\alpha=\tfrac{2}{3}\pi,\;
\beta+\gamma=(\tfrac{5}{6}+\tfrac{4}{f})\pi,\;
\delta=\tfrac{1}{2}\pi.
\]
By $\beta<\pi=2\delta$ and Lemma \ref{geometry7}, we get $\alpha<2\gamma$. This implies $\gamma>\frac{1}{3}\pi$. Then by $\beta+\gamma=(\tfrac{5}{6}+\tfrac{4}{f})\pi\le\pi$, we get $\beta\le \pi-\gamma<\alpha$. By Lemma \ref{geometry1}, this implies $\gamma>\delta$. Then by $\beta+\gamma\le\pi=2\delta$, we get $\beta<\delta$. By Lemma \ref{geometry5}, this implies $\alpha<\gamma$. 

By the only degree $3$ vertex $\alpha^3$, we know $\gamma\cdots=\gamma^2\cdots,\gamma\delta\cdots$ has degree $\ge 4$. By $\delta^4$ and $\gamma>\delta$, we get $R(\gamma^2)<R(\gamma\delta)<\alpha+\beta$, and the remainders have no $\gamma,\delta$. Then by $\alpha>\beta$, we get $\gamma\cdots=\beta^k\gamma^2,\beta^k\gamma\delta$, with $k\ge 2$. By applying the counting lemma to $\beta,\gamma$, this implies $\beta\cdots=\gamma\cdots=\beta^2\gamma^2$. Then $\beta\delta\cdots$ is not a vertex, contradicting the AAD of $\alpha^3$.
\end{proof}

\begin{proposition}\label{abb}
Tilings of the sphere by congruent almost equilateral quadrilaterals, such that $\alpha\beta^2$ is the only degree $3$ vertex, are $S_{16\square}4$ and $S_{36\square}5$.
\end{proposition}

\begin{proof}
By Lemma \ref{count-att}, we know $f\ge 16$. By $\alpha\beta^2$, we know $\beta<\pi$, and $\alpha+2\gamma+2\delta=2(\alpha+\beta+\gamma+\delta)-(\alpha+2\beta)=(2+\frac{8}{f})\pi>2\pi$. 

By the only degree $3$ vertex $\alpha\beta^2$, we know $\gamma,\delta$ do not appear at degree $3$ vertices. By Lemma \ref{deg3miss}, we know one of $\gamma^4,\gamma^3\delta,\gamma^2\delta^2,\gamma\delta^3,\delta^4$ is a vertex. By the balance lemma, we know $\gamma^2\cdots,\delta^2\cdots$ are always vertices. Therefore $\gamma,\delta<\pi$.

\subsubsection*{Case. $\alpha\beta^2,\gamma^4$ are vertices}

By $\gamma^4$ and $\alpha+2\gamma+2\delta=(2+\frac{8}{f})\pi$, we get $\alpha+2\delta=(1+\frac{8}{f})\pi>\pi$. 

Suppose $\alpha<\beta$ and $\gamma<\delta$. By the only degree $3$ vertex $\alpha^3$, we know $\delta\cdots=\delta^2\cdots,\gamma\delta\cdots$ has degree $\ge 4$. By $\gamma^4$ and $\gamma<\delta$, we get $R(\delta^2)<R(\gamma\delta)<\alpha+\beta$, and the remainders have no $\gamma,\delta$. Then by $\alpha<\beta$, we get $\delta\cdots=\alpha^k\gamma\delta,\alpha^k\delta^2$, with $k\ge 2$. By $\alpha\beta^2$, we get a contradiction by applying the counting lemma to $\alpha,\delta$. 

By Lemma \ref{geometry1}, therefore, we have $\alpha>\beta$ and $\gamma>\delta$. By $\alpha\beta^2$ and $\alpha>\beta$, we know $R(\alpha^2)$ has no $\alpha,\beta$. Then by $\gamma>\delta$, and $2\alpha+\gamma+\delta>\alpha+\beta+\gamma+\delta>2\pi$, and $\alpha+2\delta>\pi$, we get $\alpha^2\cdots=\alpha^2\delta^2$. The angle sum of $\alpha^2\delta^2$ further implies the following, including \eqref{coolsaet_eq1}.
\begin{align*}
\alpha^2\delta^2 &\colon
	\alpha=(1-\tfrac{8}{f})\pi,\;
	\beta=(\tfrac{1}{2}+\tfrac{4}{f})\pi,\;
	\gamma=\tfrac{1}{2}\pi,\;
	\delta=\tfrac{8}{f}\pi. \\
& 	\sin(\tfrac{1}{2}-\tfrac{4}{f})\pi
	\sin(\tfrac{6}{f}-\tfrac{1}{4})\pi
	=\sin(\tfrac{1}{4}+\tfrac{2}{f})\pi 
	\sin\tfrac{4}{f}\pi.
\end{align*}
Since the equation has no solution for even $f\ge 16$, we know $\alpha^2\cdots$ is not a vertex. 

By $\gamma^4$ and Lemma \ref{fbalance}, there is a vertex with $\delta^2$-fan. By no $\alpha^2\cdots$ and Lemma \ref{klem3}, a $\delta^2$-fan has a single $\alpha$, and a vertex has at most one $\delta^2$-fan. Then by the only degree $3$ vertex $\alpha\beta^2$, the $\delta^2$-fan is $\thick\delta\thin\alpha\thin\delta\thick,\thick\delta\thin\alpha\thin\beta\thin\delta\thick$, and the vertex is a single $\delta^2$-fan combined with $\gamma\delta$-fans and $\gamma^2$-fans without $\alpha$, and has degree $\ge 4$. By $\alpha+2\gamma+2\delta>2\pi$,  a vertex with $\thick\delta\thin\alpha\thin\delta\thick$ is the fan combined with a single $\thick\gamma\thin\delta\thick$, which is $\alpha\gamma\delta^3$. Moreover, a vertex with $\thick\delta\thin\alpha\thin\beta\thin\delta\thick$ is the fan itself, which is $\alpha\beta\delta^2$. The angle sum of one of $\alpha\beta\delta^2,\alpha\gamma\delta^3$, and the angle sums of $\alpha\beta^2,\gamma^4$, imply the following, including \eqref{coolsaet_eq1}.
\begin{itemize}
\item $\alpha\beta\delta^2$:
	$\alpha=\tfrac{16}{f}\pi,\;
	\beta=(1-\tfrac{8}{f})\pi,\;
	\gamma=\tfrac{1}{2}\pi,\;
	\delta=(\tfrac{1}{2}-\tfrac{4}{f})\pi$.
	\newline
	$0=\sin(\tfrac{1}{2}-\tfrac{4}{f})\pi 
	\sin(\tfrac{1}{2}-\tfrac{8}{f})\pi$.
\item $\alpha\gamma\delta^3 $:
	$\alpha=\tfrac{24}{f}\pi,\;
	\beta=(1-\tfrac{12}{f})\pi,\;
	\gamma=\tfrac{1}{2}\pi,\;
	\delta=(\tfrac{1}{2}-\tfrac{8}{f})\pi$.
	\newline
	$-\sin\tfrac{12}{f}\pi
	\sin\tfrac{2}{f}\pi
	=\sin(\tfrac{1}{2}-\tfrac{6}{f})\pi 
	\sin(\tfrac{1}{2}-\tfrac{12}{f})\pi$.
\end{itemize}
The solution for even $f\ge 16$ is $f=16$ for both cases. For $\alpha\beta\delta^2$, $f=16$ implies the angle values in \eqref{abb-accangle}. The earlier proof shows there is no tiling. For $\alpha\gamma\delta^3$, $f=16$ implies $\delta=0$, a contradiction.

\subsubsection*{Case. $\alpha\beta^2,\delta^4$ are vertices}

By $\delta^4$ and $\alpha+2\gamma+2\delta=(2+\frac{8}{f})\pi$, we get $\delta=\frac{1}{2}\pi$ and $\alpha+2\gamma=(1+\frac{8}{f})\pi>\pi$. 

If $\alpha>\beta$ and $\gamma>\delta$, then $\beta<\alpha<(1+\frac{8}{f})\pi-2\delta=\frac{8}{f}\pi\le \frac{1}{2}\pi$. This implies $\alpha\beta^2$ is not a vertex, a contradiction. By Lemma \ref{geometry1}, therefore, we have $\alpha<\beta$ and $\gamma<\delta$. 

By $\alpha\beta^2$ and $\alpha<\beta$, we get $\frac{2}{3}\pi<\beta<\pi$. By $\delta=\frac{1}{2}\pi$, this implies $\delta<\beta<2\delta$. By Lemmas \ref{geometry5} and \ref{geometry7}, this implies $\gamma<\alpha<2\gamma$. Then $4\gamma>\alpha+2\gamma>\pi$.

By the only degree $3$ vertex $\alpha\beta^2$, we know $\delta\cdots=\gamma\delta\cdots,\delta^2\cdots$ has degree $\ge 4$. By $\gamma<\delta$, we get $R(\beta\delta^2)<R(\beta\gamma\delta)<\alpha<\beta,2\gamma,2\delta$. This implies $\beta\delta\cdots$ is not a vertex.

We have the AAD $\thin^{\beta}\alpha^{\delta}\thin\beta\thin=\thin^{\beta}\alpha^{\delta}\thin^{\alpha}\beta^{\gamma}\thin,\thin^{\beta}\alpha^{\delta}\thin^{\gamma}\beta^{\alpha}\thin$ at the only degree $3$ vertex $\alpha\beta^2$. This implies one of $\alpha\delta\cdots,\gamma\delta\cdots$ is a vertex of degree $\ge 4$. By no $\beta\delta\cdots$, the vertex is $\alpha\delta^2\cdots,\alpha\gamma\delta\cdots,\gamma^k\delta^l(k,l\ge 1)$, with no $\beta$ in $R(\alpha\delta^2)$, and no $\beta,\delta$ in $R(\alpha\gamma\delta)$. By $\alpha+2\gamma+2\delta>2\pi$, we get $R(\alpha\delta^2)<2\gamma<2\alpha,2\delta$. This implies $\alpha\delta^2\cdots=\alpha^2\delta^2$. By $R(\alpha\gamma\delta)<\beta<\pi<\alpha+2\gamma<3\alpha,4\gamma$, we get $\alpha\gamma\delta\cdots=\alpha^2\gamma\delta,\alpha^3\gamma\delta,\alpha\gamma^3\delta$. By $\frac{1}{4}\pi<\gamma<\delta=\frac{1}{2}\pi$, for $k,l\ge 1$, we get $\gamma^k\delta^l=\gamma^5\delta$.

The angle sum of one of $\alpha^2\delta^2,\alpha^2\gamma\delta,\alpha^3\gamma\delta,\alpha\gamma^3\delta,\gamma^5\delta$, and the angle sums of $\alpha\beta^2,\delta^4$, imply the following, including \eqref{coolsaet_eq1}.
\begin{itemize}
\item $\alpha^2\delta^2$:
	$\alpha=\tfrac{1}{2}\pi,\;
	\beta=\tfrac{3}{4}\pi,\;
	\gamma=(\tfrac{1}{4}+\tfrac{4}{f})\pi,\;
	\delta=\tfrac{1}{2}\pi$.
	\newline 
	$\sin\tfrac{1}{4}\pi
	\sin\tfrac{1}{8}\pi
	=\sin\tfrac{3}{8}\pi
	\sin\tfrac{4}{f}\pi$. 
\item $\alpha^2\gamma\delta$:
	$\alpha=(\tfrac{2}{3}-\tfrac{8}{3f})\pi,\;
	\beta=(\tfrac{2}{3}+\tfrac{4}{3f})\pi,\;
	\gamma=(\tfrac{1}{6}+\tfrac{16}{3f})\pi,\;
	\delta=\tfrac{1}{2}\pi$.
	\newline 
	$\sin(\tfrac{1}{3}-\tfrac{4}{3f})\pi
	\sin(\tfrac{1}{6}-\tfrac{2}{3f})\pi
	=\sin(\tfrac{1}{3}+\tfrac{2}{3f})\pi
	\sin(\tfrac{20}{3f}-\tfrac{1}{6})\pi$. 
\item $\alpha^3\gamma\delta$:
	$\alpha=(\tfrac{2}{5}-\tfrac{8}{5f})\pi,\;
	\beta=(\tfrac{4}{5}+\tfrac{4}{5f})\pi,\;
	\gamma=(\tfrac{3}{10}+\tfrac{24}{5f})\pi,\;
	\delta=\tfrac{1}{2}\pi$.
	\newline 
	$\sin(\tfrac{1}{5}-\tfrac{4}{5f})\pi
	\sin(\tfrac{1}{10}-\tfrac{2}{5f})\pi
	=\sin(\tfrac{2}{5}+\tfrac{2}{5f})\pi
	\sin(\tfrac{1}{10}+\tfrac{28}{5f})\pi$. 
\item $\alpha\gamma^3\delta$:
	$\alpha=\tfrac{24}{f}\pi,\;
	\beta=(1-\tfrac{12}{f})\pi,\;
	\gamma=(\tfrac{1}{2}-\tfrac{8}{f})\pi,\;
	\delta=\tfrac{1}{2}\pi$.
	\newline 
	$\sin\tfrac{12}{f}\pi
	\sin\tfrac{6}{f}\pi
	=\sin(\tfrac{1}{2}-\tfrac{6}{f})\pi
	\sin(\tfrac{1}{2}-\tfrac{20}{f})\pi$. 
\item $\gamma^5\delta$:
	$\alpha=(\tfrac{2}{5}+\tfrac{8}{f})\pi,\;
	\beta=(\tfrac{4}{5}-\tfrac{4}{f})\pi,\;
	\gamma=\tfrac{3}{10}\pi,\;
	\delta=\tfrac{1}{2}\pi$.
	\newline 
	$\sin(\tfrac{1}{5}+\tfrac{4}{f})\pi
	\sin(\tfrac{1}{10}+\tfrac{2}{f})\pi
	=\sin(\tfrac{2}{5}-\tfrac{2}{f})\pi
	\sin(\tfrac{1}{10}-\tfrac{4}{f})\pi$.
\end{itemize}
The equations have no solution for even $f\ge 16$.

\subsubsection*{Case. $\alpha\beta^2,\gamma^2\delta^2$ are vertices}

The angle sums of $\alpha\beta^2,\gamma^2\delta^2$ imply
\[
\alpha=\tfrac{8}{f}\pi,\;
\beta=(1-\tfrac{4}{f})\pi,\;
\gamma+\delta=\pi.
\]
We have $\alpha<\beta$. By Lemma \ref{geometry1}, this implies $\gamma<\delta$.

By $\gamma<\delta$, we get $R(\delta^2)<R(\gamma\delta)<\alpha+\beta$. By $\alpha<\beta$ and the only degree $3$ vertex $\alpha\beta^2$, this implies that, if the remainders have no $\gamma,\delta$, then we get $\gamma\delta\cdots=\alpha^k\gamma\delta$ and $\delta^2\cdots=\alpha^k\delta^2$, with $k\ge 2$. 

By $\gamma^2\delta^2$, and $\gamma<\delta$, a vertex with strictly fewer $\gamma$ than $\delta$ is $\delta^2\cdots$, with no $\gamma,\delta$ in the remainder. Therefore the vertex is $\alpha^k\delta^2$. This implies $\alpha+\delta\le\pi$. On the other hand, $\alpha\beta^2$ implies $\alpha+\beta>\pi$. Therefore $\beta>\delta$. Then by Lemma \ref{geometry5}, we get $\alpha>\gamma$. This implies $\gamma+\delta<\alpha+\delta\le\pi$, contradicting $\gamma+\delta=\pi$.

By applying the counting lemma to $\gamma,\delta$, therefore, we have the same number of $\gamma,\delta$ at any vertex. By $\gamma^2\delta^2$, this means $b$-vertices are $\gamma^2\delta^2$ and $\gamma\delta\cdots$ with no $\gamma,\delta$ in the remainder. We know this $\gamma\delta\cdots=\alpha^k\gamma\delta(k\ge 2)$. Moreover, by $\alpha\beta^2$ and $\alpha<\beta$, the $\hat{b}$-vertices are $\alpha\beta^2,\alpha^k,\alpha^k\beta$. We conclude the list of vertices
\[
\text{AVC}
=\{\alpha\beta^2,\gamma^2\delta^2,\alpha^k,\alpha^k\beta,\alpha^k\gamma\delta\}.
\]

The AVC implies $\thin^{\gamma}\beta^{\alpha}\thin^{\alpha}\beta^{\gamma}\thin\cdots=\alpha\beta^2=\thin^{\alpha}\beta^{\gamma}\thin^{\beta}\alpha^{\delta}\thin^{\gamma}\beta^{\alpha}\thin$. Since this contradicts no $\beta\gamma\cdots$ in the AVC, we do not have the AAD $\thin^{\gamma}\beta^{\alpha}\thin^{\alpha}\beta^{\gamma}\thin$. This implies no $\thin^{\delta}\alpha^{\beta}\thin^{\beta}\alpha^{\delta}\thin$. The AVC also implies no $\beta\delta\cdots$, and therefore no $\thin^{\beta}\alpha^{\delta}\thin^{\beta}\alpha^{\delta}\thin$. Then we get the unique AAD $\thin^{\beta}\alpha^{\delta}\thin^{\delta}\alpha^{\beta}\thin$, which implies no consecutive $\alpha\alpha\alpha$. Then by the only degree $3$ vertex $\alpha\beta^2$, we know $\alpha^k,\alpha^k\beta$ are not vertices, and $k=2$ in $\alpha^k\gamma\delta$. We get the updated list of vertices
\[
\text{AVC}
=\{\alpha\beta^2,\gamma^2\delta^2,\alpha^2\gamma\delta\}.
\]

By no $\beta\delta\cdots$, we get the AAD $\thick^{\delta}\gamma^{\beta}\thin^{\beta}\alpha^{\delta}\thin^{\delta}\alpha^{\beta}\thin^{\alpha}\delta^{\gamma}\thick$ of $\alpha^2\gamma\delta$. This determines $T_1,T_2,T_3,T_4$ in Figure \ref{abb-ccddB}. Then $\delta_3\delta_4\cdots=\gamma^2\delta^2$ determines $T_5,T_6$. Then  $\beta_2\beta_3\cdots=\alpha\beta^2$ and no $\beta\gamma\delta\cdots$ determine $T_7$. Then $\gamma_3\delta_5\delta_7\cdots=\gamma^2\delta^2$ determines $T_8$. Then $\beta_5\beta_6\cdots=\alpha_5\beta_8\cdots=\alpha\beta^2$ determines $T_9$. Then the AAD of $\alpha_8\gamma_9\cdots=\alpha^2\gamma\delta$ determines $T_{10},T_{11}$. Then $\gamma_7\delta_8\delta_{10}\cdots=\gamma^2\delta^2$ determines $T_{12}$. Then we may apply the argument starting with $\alpha_3\alpha_4\gamma_2\delta_1$ to $\alpha_{12}\gamma_1\delta_2\cdots=\alpha^2\gamma\delta$ and get the rest of the tiles. We get the sporadic tiling $S_{16\square}4$.

\begin{figure}[htp]
\centering
\begin{tikzpicture}

\foreach \a in {0,...,5}
{
\begin{scope}[rotate=60*\a]

\draw 
	(0:0.8) -- (60:0.8)
	(0:2.3) -- (60:2.3)
	(0.8,0) -- (2.3,0);

\draw[line width=1.2]
	(0:1.6) -- (60:1.6);

\end{scope}
}

\foreach \a in {1,-1}
{
\begin{scope}[scale=\a]

\draw[line width=1.2]
	(0,0) -- (0.8,0)
	(60:2.3) -- (60:2.7);
	
\node at (0.3,0.55) {\small $\alpha$};
\node at (-0.3,0.5) {\small $\beta$};
\node at (0.5,0.2) {\small $\delta$};
\node at (-0.5,0.2) {\small $\gamma$};

\node at (-0.3,0.85) {\small $\alpha$};
\node at (0.3,0.9) {\small $\beta$};
\node at (-0.5,1.2) {\small $\delta$};
\node at (0.5,1.2) {\small $\gamma$};

\node at (0.9,0.15) {\small $\alpha$};
\node at (0.6,0.7) {\small $\beta$};
\node at (1.27,0.2) {\small $\delta$};
\node at (0.8,1) {\small $\gamma$};

\node at (0.9,-0.15) {\small $\alpha$};
\node at (0.6,-0.7) {\small $\beta$};
\node at (1.27,-0.2) {\small $\delta$};
\node at (0.8,-1.05) {\small $\gamma$};

\node at (0.9,1.85) {\small $\alpha$};
\node at (-0.82,1.8) {\small $\beta$};
\node at (0.7,1.57) {\small $\delta$};
\node at (-0.7,1.57) {\small $\gamma$};

\node at (1.15,1.65) {\small $\alpha$};
\node at (1.95,0.2) {\small $\beta$};
\node at (1,1.4) {\small $\delta$};
\node at (1.7,0.2) {\small $\gamma$};

\node at (1.15,-1.65) {\small $\alpha$};
\node at (2,-0.2) {\small $\beta$};
\node at (1,-1.4) {\small $\delta$};
\node at (1.7,-0.2) {\small $\gamma$};

\node at (2.45,0) {\small $\alpha$};
\node at (-60:2.45) {\small $\beta$};
\node at (1.35,2) {\small $\delta$};
\node at (1.05,2.2) {\small $\gamma$};

\end{scope}
}

\node[draw,shape=circle, inner sep=0.5] at (0,0.3) {\small $1$};
\node[draw,shape=circle, inner sep=0.5] at (0,-0.3) {\small $2$};
\node[draw,shape=circle, inner sep=0.5] at (-30:1.1) {\small $3$};
\node[draw,shape=circle, inner sep=0.5] at (30:1.1) {\small $4$};
\node[draw,shape=circle, inner sep=0.5] at (-30:1.7) {\small $5$};
\node[draw,shape=circle, inner sep=0.5] at (30:1.7) {\small $6$};
\node[draw,shape=circle, inner sep=0.5] at (-90:1.1) {\small $7$};
\node[draw,shape=circle, inner sep=0.5] at (-90:1.7) {\small $8$};
\node[draw,shape=circle, inner sep=0.5] at (-30:2.4) {\small $9$};
\node[draw,shape=circle, inner sep=0] at (210:1.7) {\footnotesize $10$};
\node[draw,shape=circle, inner sep=0] at (150:2.4) {\footnotesize $11$};
\node[draw,shape=circle, inner sep=0] at (210:1.1) {\footnotesize $12$};

\end{tikzpicture}
\caption{Proposition \ref{abb}: Tiling for $\{\alpha\beta^2,\gamma^2\delta^2,\alpha^2\gamma\delta\}$.}
\label{abb-ccddB}
\end{figure}
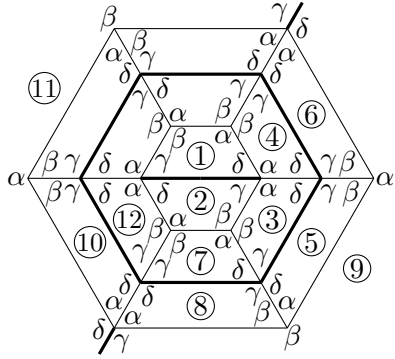

\medskip

\noindent{\em Geometry of Quadrilateral}

\medskip

The tiling $S_{16\square}4$ satisfies
\[
\alpha=\tfrac{1}{2}\pi,\;
\beta=\tfrac{3}{4}\pi,\;
\gamma+\delta=\pi.
\]
By $\gamma<\delta$ and \eqref{coolsaet_eq1}, we find $\gamma$ is determined by 
\[
\tan\gamma=2+\sqrt{2},\quad \gamma<\tfrac{1}{2}\pi.
\]
This gives precise values of the other angles. Then we substitute the precise values into \eqref{coolsaet_eq2}, and get $\cos a=\frac{1}{\sqrt{2}}$. Therefore $a=\frac{1}{4}\pi$. Then we substitute into \eqref{coolsaet_eq4} and use $K=Y(b)^T$ to find $b$ is determined by
\[
\cos b=\tfrac{1}{4}(2\sqrt{2}-1),
\quad b<\pi.
\]
We get the following approximate values
\[
\alpha=\tfrac{1}{2}\pi,\;
\beta=\tfrac{3}{4}\pi,\;
\gamma=0.4093\pi,\;
\delta=0.5906\pi,\;
a=\tfrac{1}{4}\pi,\;
b=0.3488\pi.
\]
By Lemma \ref{geometry8}, the quadrilateral is geometrically suitable for tiling.

\subsubsection*{Case. $\alpha\beta^2,\gamma^3\delta$ are vertices}

The angle sums of $\alpha\beta^2,\gamma^3\delta$ imply
\[
\alpha=(\tfrac{2}{3}+\tfrac{8}{f})\pi-\tfrac{4}{3}\delta,\;
\beta=(\tfrac{2}{3}-\tfrac{4}{f})\pi+\tfrac{2}{3}\delta,\;
\gamma=\tfrac{2}{3}\pi-\tfrac{1}{3}\delta.
\]

By $\delta<\pi$, we get $\beta-\delta=(\frac{2}{3}-\frac{4}{f})\pi-\frac{1}{3}\delta>0$. By Lemma \ref{geometry5}, this implies $\alpha>\gamma$, which means $\delta<\frac{8}{f}\pi\le\frac{1}{2}\pi$. Then by $\gamma-\delta=\frac{2}{3}\pi-\frac{4}{3}\delta>0$ and Lemma \ref{geometry1}, we get $\alpha>\beta$, which means $\delta<\frac{6}{f}\pi$. By $2\gamma-\alpha=(\tfrac{2}{3}-\tfrac{8}{f})\pi+\tfrac{2}{3}\delta>0$ and Lemma \ref{geometry7}, we also get $\beta<2\delta$, which means $\delta>(\tfrac{1}{2}-\tfrac{3}{f})\pi$. Combined with $\delta<\frac{6}{f}\pi$, we get $f<18$. This implies $f=16$, and we have
\[
\alpha=\tfrac{7}{6}\pi-\tfrac{4}{3}\delta,\;
\beta=\tfrac{5}{12}\pi+\tfrac{2}{3}\delta,\;
\gamma=\tfrac{2}{3}\pi-\tfrac{1}{3}\delta.
\]
By $\delta<\frac{6}{16}\pi=\frac{3}{8}\pi$, we get $\alpha+\delta=\tfrac{7}{6}\pi-\tfrac{1}{3}\delta>\pi$. 

By $\alpha\beta^2$ and the counting lemma, we know there is a vertex with strictly more $\alpha$ than $\beta$. By $f=16$ and the first part of Lemma \ref{count-att}, we know the vertex has degree $4$. By $\alpha\beta^2$ and $\alpha>\beta$, the vertex has one or two $\alpha$, and has no $\beta$. Therefore the vertex is $\alpha^2\gamma^2,\alpha^2\gamma\delta,\alpha^2\delta^2$. However, by $\gamma>\delta$ and $\alpha+\delta>\pi$, none of these can be a vertex.

\subsubsection*{Case. $\alpha\beta^2,\gamma\delta^3$ are vertices}

The angle sums of $\alpha\beta^2,\gamma\delta^3$ imply
\[
\alpha=(\tfrac{2}{3}+\tfrac{8}{f})\pi-\tfrac{4}{3}\gamma,\;
\beta=(\tfrac{2}{3}-\tfrac{4}{f})\pi+\tfrac{2}{3}\gamma,\;
\delta=\tfrac{2}{3}\pi-\tfrac{1}{3}\gamma.
\]
We have $\gamma+\delta>\beta$, and $2\alpha+\beta+2\gamma=(2+\frac{12}{f})\pi>2\pi$, and $3\alpha+4\gamma=(2+\frac{24}{f})\pi>2\pi$. 

If $\alpha>\beta$, then we get $\gamma<\frac{6}{f}\pi$. This implies $\delta-\gamma=\tfrac{2}{3}\pi-\tfrac{4}{3}\gamma>0$, contradicting Lemma \ref{geometry1}. Therefore we have $\alpha<\beta$ and $\gamma<\delta$. By $\alpha<\beta$, we get $\gamma>\frac{6}{f}\pi$. This implies $\beta-\delta=\gamma-\frac{4}{f}\pi>0$. We also have $\beta<\gamma+\delta<2\delta$. Then by Lemmas \ref{geometry5} and \ref{geometry7}, we get $\gamma<\alpha<2\gamma$. 

First, we study $\beta\cdots$. Besides $\alpha\beta^2$, the vertex $\beta\cdots$ has degree $\ge 4$. 

By $\alpha\beta^2$ and $\alpha<\beta,2\gamma,2\delta$, we get $\beta^2\cdots=\alpha\beta^2$, and a $\hat{b}$-vertex $\beta\cdots$ is $\alpha\beta^2,\alpha^k\beta$. By $4\alpha+\beta>2\alpha+\beta+2\gamma>2\pi$, we get $\alpha^k\beta=\alpha^3\beta$.

By $\beta^2\cdots=\alpha\beta^2$, a $b$-vertex $\beta\cdots$ is $\beta\gamma^2\cdots,\beta\delta^2\cdots,\beta\gamma\delta\cdots$, with no $\beta$ in the remainders. By $\gamma<\delta$, we get $R(\beta\delta^2)<R(\beta\gamma\delta)<\alpha<2\gamma,2\delta$. This implies $\beta\delta\cdots=\beta\delta^2\cdots,\beta\gamma\delta\cdots$ is not a vertex. By $2\alpha+\beta+2\gamma>2\pi$, we get $R(\beta\gamma^2)<2\alpha<\alpha+2\gamma,4\gamma$. This implies $\beta\gamma^2\cdots=\alpha\beta\gamma^2,\beta\gamma^4$. Combining all the discussions, we get $\beta\cdots=\alpha\beta^2,\alpha^3\beta,\alpha\beta\gamma^2,\beta\gamma^4$.

By no $\beta\delta\cdots$, the AAD implies no $\alpha^k$ for odd $k$, and no consecutive $\gamma\alpha\gamma$. 

The angle sum of one of $\alpha^3\beta,\alpha\beta\gamma^2,\beta\gamma^4$, and the angle sums of $\alpha\beta^2,\gamma\delta^3$, imply the following, including \eqref{coolsaet_eq1}.
\begin{itemize}
\item $\alpha^3\beta$:
	$\alpha=\tfrac{2}{5}\pi,\;
	\beta=\tfrac{4}{5}\pi,\;
	\gamma=(\tfrac{1}{5}+\tfrac{6}{f})\pi,\;
	\delta=(\tfrac{3}{5}-\tfrac{2}{f})\pi$.
	\newline
	$\sin\tfrac{1}{5}\pi
	\sin(\tfrac{1}{5}-\tfrac{2}{f})\pi
	=\sin\tfrac{2}{5}\pi
	\sin\tfrac{6}{f}\pi$. 
\item $\alpha\beta\gamma^2$:
	$\alpha=\tfrac{12}{f}\pi,\;
	\beta=(1-\tfrac{6}{f})\pi,\;
	\gamma=(\tfrac{1}{2}-\tfrac{3}{f})\pi,\;
	\delta=(\tfrac{1}{2}+\tfrac{1}{f})\pi$.
	\newline
	$\sin\tfrac{6}{f}\pi
	\sin\tfrac{4}{f}\pi
	=\sin(\tfrac{1}{2}-\tfrac{3}{f})\pi
	\sin(\tfrac{1}{2}-\tfrac{9}{f})\pi$. 
\item $\beta\gamma^4$:
	$\alpha=(\tfrac{2}{7}+\tfrac{48}{7f})\pi,\;
	\beta=(\tfrac{6}{7}-\tfrac{24}{7f})\pi,\;
	\gamma=(\tfrac{2}{7}+\tfrac{6}{7f})\pi,\;
	\delta=(\tfrac{4}{7}-\tfrac{2}{7f})\pi$.
	\newline
	$\sin(\tfrac{1}{7}+\tfrac{24}{7f})\pi
	\sin(\tfrac{1}{7}+\tfrac{10}{7f})\pi
	=\sin(\tfrac{3}{7}-\tfrac{12}{7f})\pi
	\sin(\tfrac{1}{7}-\tfrac{18}{7f})\pi$.
\end{itemize}
The solutions for even $f\ge 16$ are the following. We also calculate all the possible vertices satisfying the parity lemma. 
\begin{align*}
\alpha\beta\gamma^2 &\colon
	\alpha=\tfrac{1}{2}\pi,\;
	\beta=\tfrac{3}{4}\pi,\;
	\gamma=\tfrac{3}{8}\pi,\;
	\delta=\tfrac{13}{24}\pi, \;
	f=24. \\
&	\text{AVC}=
	\{\alpha\beta^2,\gamma\delta^3,\alpha^4,\alpha\beta\gamma^2\}. 
\\ 
\alpha^3\beta/\beta\gamma^4 & \colon
	\alpha=\tfrac{2}{5}\pi,\;
	\beta=\tfrac{4}{5}\pi,\;
	\gamma=\tfrac{3}{10}\pi,\;
	\delta=\tfrac{17}{30}\pi, \;
	f=60. \\
&	\text{AVC}=
	\{\alpha\beta^2,\gamma\delta^3,\alpha^3\beta,\alpha^2\gamma^4,\beta\gamma^4\}.
\end{align*}
We remark that $\alpha\gamma^4$ for $f=24$ is excluded by no consecutive $\gamma\alpha\gamma$, and $\alpha^5$ for $f=60$ is excluded by no $\alpha^k$ for odd $k$.

For $f=24$, by no $\alpha\delta\cdots,\beta\delta\cdots$, we get the AAD $\thick^{\delta}\gamma^{\beta}\thin^{\beta}\alpha^{\delta}\thin^{\gamma}\beta^{\alpha}\thin^{\beta}\gamma^{\delta}\thick$ of $\alpha\beta\gamma^2$. This determines $T_1,T_2,T_3$ in Figure \ref{abb-cdddA}. Then $\beta_1\beta_2\cdots=\alpha\beta^2$ and no $\alpha\delta\cdots$ determine $T_4$. Then $\gamma_2\delta_4\cdots=\gamma_3\delta_2\cdots=\gamma\delta^3$ implies two $\delta$ in a tile, a contradiction. 

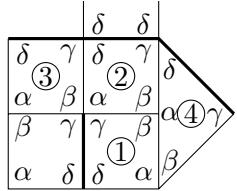
\begin{figure}[htp]
\centering
\begin{tikzpicture}


\draw
	(-1,1) -- (-1,-1) -- (1,-1) -- (2,0)
	(-1,0) -- (1,0)
	(0,0) -- (0,1.5)
	(1,-1) -- (1,1.5);

\draw[line width=1.2]	
	(0,0) -- (0,-1)
	(-1,1) -- (1,1) -- (2,0);

\node at (-0.8,-0.8) {\small $\alpha$};
\node at (-0.2,-0.8) {\small $\delta$};
\node at (-0.8,-0.2) {\small $\beta$};
\node at (-0.2,-0.2) {\small $\gamma$};

\node at (0.8,-0.8) {\small $\alpha$};
\node at (0.2,-0.8) {\small $\delta$};
\node at (0.8,-0.2) {\small $\beta$};
\node at (0.2,-0.2) {\small $\gamma$};

\node at (0.2,0.2) {\small $\alpha$};
\node at (0.8,0.2) {\small $\beta$};
\node at (0.8,0.8) {\small $\gamma$};
\node at (0.2,0.8) {\small $\delta$};

\node at (-0.8,0.2) {\small $\alpha$};
\node at (-0.2,0.2) {\small $\beta$};
\node at (-0.2,0.8) {\small $\gamma$};
\node at (-0.8,0.8) {\small $\delta$};

\node at (1.15,0) {\small $\alpha$};
\node at (1.15,-0.65) {\small $\beta$};
\node at (1.75,-0.05) {\small $\gamma$};
\node at (1.15,0.6) {\small $\delta$};

\node at (0.2,1.2) {\small $\delta$};
\node at (0.8,1.2) {\small $\delta$};

\node[draw,shape=circle, inner sep=0.5] at (0.5,-0.5) {\small $1$};
\node[draw,shape=circle, inner sep=0.5] at (0.5,0.5) {\small $2$};
\node[draw,shape=circle, inner sep=0.5] at (-0.5,0.5) {\small $3$};
\node[draw,shape=circle, inner sep=0.5] at (1.43,0) {\small $4$};

\end{tikzpicture}
\caption{Proposition \ref{abb}: Vertex $\alpha\beta\gamma^2$.}
\label{abb-cdddA}
\end{figure}

For $f=60$, the AAD $\thick^{\delta}\gamma^{\beta}\thin^{\beta}\gamma^{\delta}\thick$ of $\thick\gamma\thin\gamma\thick$ implies a vertex $\thin^{\alpha}\beta^{\gamma}\thin^{\gamma}\beta^{\alpha}\thin\cdots=\alpha\beta^2=\thin^{\gamma}\beta^{\alpha}\thin^{\beta}\alpha^{\delta}\thin^{\alpha}\beta^{\gamma}\thin$, contradicting no $\alpha\delta\cdots$. Therefore $\gamma\thin\gamma\cdots$ is not a vertex. Combined with no consecutive $\gamma\alpha\gamma$, we know $\alpha^2\gamma^4,\beta\gamma^4$ are not vertices. Therefore the AVC is reduced to $\{\alpha\beta^2,\gamma\delta^3,\alpha^3\beta\}$. Applying the counting lemma to $\gamma,\delta$, we get a contradiction. 

We conclude $\alpha^3\beta,\alpha\beta\gamma^2,\beta\gamma^4$ are not vertices, and $\beta\cdots=\alpha\beta^2$. 

The AAD $\thin^{\delta}\alpha^{\beta}\thin^{\beta}\alpha^{\delta}\thin$ implies a vertex $\thin^{\gamma}\beta^{\alpha}\thin^{\alpha}\beta^{\gamma}\thin\cdots=\alpha\beta^2=\thin^{\alpha}\beta^{\gamma}\thin^{\beta}\alpha^{\delta}\thin^{\gamma}\beta^{\alpha}\thin$, contradicting no $\beta\gamma\cdots$. By no $\beta\delta\cdots$, we also know there is no $\thin^{\delta}\alpha^{\beta}\thin^{\delta}\alpha^{\beta}\thin$. Therefore the AAD of $\thin\alpha\thin\alpha\thin$ is $\thin^{\beta}\alpha^{\delta}\thin^{\delta}\alpha^{\beta}\thin$. This implies no consecutive $\alpha\alpha\alpha$. We also recall there is no consecutive $\gamma\alpha\gamma$. 

By $\beta\cdots=\alpha\beta^2$ and the counting lemma, there is a vertex $\alpha\cdots$ without $\beta$. The vertex has degree $\ge 4$. By no consecutive $\alpha\alpha\alpha$, the vertex is not $\alpha^k$. Therefore it is a $b$-vertex. 

By $\alpha+2\gamma+2\delta>2\pi$, we get $R(\alpha\delta^2)<2\gamma<2\alpha,2\delta$. This implies $\alpha\delta^2\cdots=\alpha^2\delta^2$. Therefore a $b$-vertex $\alpha\cdots=\alpha^2\delta^2,\alpha^k\gamma^l,\alpha^k\gamma^l\delta$. By no consecutive $\alpha\alpha\alpha,\gamma\alpha\gamma$, we know $\alpha^k\gamma^l$ is a combination of some $\gamma^2$-fans $\thick\gamma\thin\alpha\thin\alpha\thin\gamma\thick,\thick\gamma\thin\gamma\thick$, and $\alpha^k\gamma^l\delta$ is a combination of one of $\thick\gamma\thin\alpha\thin\alpha\thin\delta\thick$, $\thick\gamma\thin\alpha\thin\delta\thick$, $ \thick\gamma\thin\delta\thick$, and some $\gamma^2$-fans $\thick\gamma\thin\alpha\thin\alpha\thin\gamma\thick$, $ \thick\gamma\thin\gamma\thick$.

If $\alpha^2\gamma^2$ is a vertex, then we get $\gamma=(\frac{24}{f}-1)\pi$. By $2\gamma-\alpha=(\tfrac{72}{f}-4)\pi>0$, we get $f<18$. Therefore $f=16$, and $\beta=\frac{3}{4}\pi$. This implies $\alpha=\frac{1}{2}\pi=\gamma$, a contradiction. Then by $3\alpha+4\gamma>2\pi$ and no $\alpha^2\gamma^2$, we know $\alpha^k\gamma^l$ is a combination of one $\thick\gamma\thin\alpha\thin\alpha\thin\gamma\thick$ with at least one $\thick\gamma\thin\gamma\thick$. The AAD $\thick^{\delta}\gamma^{\beta}\thin^{\beta}\gamma^{\delta}\thick$ of $\thick\gamma\thin\gamma\thick$ implies a vertex $\thin^{\alpha}\beta^{\gamma}\thin^{\gamma}\beta^{\alpha}\thin\cdots=\alpha\beta^2=\thin^{\gamma}\beta^{\alpha}\thin^{\beta}\alpha^{\delta}\thin^{\alpha}\beta^{\gamma}\thin$. This further implies a vertex $\alpha\delta\cdots=\alpha^2\delta^2,\alpha^k\gamma^l\delta$.

By $\alpha+5\gamma+\delta>2\alpha+3\gamma+\delta>2\pi$, we know $\thick\gamma\thin\alpha\thin\alpha\thin\delta\thick$ cannot be combined with $\thick\gamma\thin\alpha\thin\alpha\thin\gamma\thick$ or $ \thick\gamma\thin\gamma\thick$, and $\thick\gamma\thin\alpha\thin\delta\thick$ must be combined with a single $\thick\gamma\thin\gamma\thick$, and $ \thick\gamma\thin\delta\thick$ cannot be combined with $\thick\gamma\thin\alpha\thin\alpha\thin\gamma\thick$. Therefore $\alpha^k\gamma^l\delta=\alpha^2\gamma\delta,\alpha\gamma^3\delta$.

The angle sum of one of $\alpha^2\gamma\delta,\alpha^2\delta^2,\alpha\gamma^3\delta$, and the angle sums of $\alpha\beta^2,\gamma\delta^3$, imply the following, including \eqref{coolsaet_eq1}.
\begin{itemize}
\item $\alpha^2\gamma\delta$:
	$\alpha=(\tfrac{2}{3}-\tfrac{8}{3f})\pi,\;
	\beta=(\tfrac{2}{3}+\tfrac{4}{3f})\pi,\;
	\gamma=\tfrac{8}{f}\pi,\;
	\delta=(\tfrac{2}{3}-\tfrac{8}{3f})\pi$.
	\newline
	$\sin(\tfrac{1}{3}-\tfrac{4}{3f})\pi
	\sin(\tfrac{1}{3}-\tfrac{10}{3f})\pi
	=\sin(\tfrac{1}{3}+\tfrac{2}{3f})\pi
	\sin(\tfrac{28}{3f}-\tfrac{1}{3})\pi$. 
\item $\alpha^2\delta^2$:
	$\alpha=(\tfrac{2}{5}+\tfrac{8}{5f})\pi,\;
	\beta=(\tfrac{4}{5}-\tfrac{4}{5f})\pi,\;
	\gamma=(\tfrac{1}{5}+\tfrac{24}{5f})\pi,\;
	\delta=(\tfrac{3}{5}-\tfrac{8}{5f})\pi$.
	\newline
	$\sin(\tfrac{1}{5}+\tfrac{4}{5f})\pi
	\sin(\tfrac{1}{5}-\tfrac{6}{5f})\pi
	=\sin(\tfrac{2}{5}-\tfrac{2}{5f})\pi
	\sin\tfrac{4}{f}\pi$. 
\item $\alpha\gamma^3\delta$:
	$\alpha=\tfrac{16}{f}\pi,\;
	\beta=(1-\tfrac{8}{f})\pi,\;
	\gamma=(\tfrac{1}{2}-\tfrac{6}{f})\pi,\;
	\delta=(\tfrac{1}{2}+\tfrac{2}{f})\pi$.
	\newline
	$\sin\tfrac{8}{f}\pi
	\sin\tfrac{6}{f}\pi
	=\sin(\tfrac{1}{2}-\tfrac{4}{f})\pi
	\sin(\tfrac{1}{2}-\tfrac{14}{f})\pi$. 
\end{itemize}
The solutions for even $f\ge 16$ are the following. We also calculate all the possible vertices satisfying the parity lemma. 
\begin{align*}
\alpha^2\gamma\delta &\colon
	\alpha=\tfrac{8}{15}\pi,\;
	\beta=\tfrac{11}{15}\pi,\;
	\gamma=\tfrac{2}{5}\pi,\;
	\delta=\tfrac{8}{15}\pi,\;
	f=20. \\
& 	\text{AVC}=
	\{\alpha\beta^2,\gamma\delta^3,\alpha^2\gamma\delta\}. \\
\alpha^2\delta^2/\alpha\gamma^3\delta &\colon
	\alpha=\tfrac{4}{9}\pi,\;
	\beta=\tfrac{7}{9}\pi,\;
	\gamma=\tfrac{1}{3}\pi,\;
	\delta=\tfrac{5}{9}\pi,\;
	f=36.  \\
&	\text{AVC}=
	\{\alpha\beta^2,\gamma\delta^3,\alpha^2\delta^2,\alpha\gamma^3\delta,\gamma^6\}.
\end{align*}
For $f=20$, we get a contradiction by applying the counting lemma to $\gamma,\delta$. 

For $f=36$, by no $\beta\delta\cdots$ and no consecutive $\gamma\alpha\gamma$, we know the AAD of $\alpha\gamma^3\delta$ is $\thick^{\delta}\gamma^{\beta}\thin^{\beta}\alpha^{\delta}\thin^{\alpha}\delta^{\gamma}\thick^{\delta}\gamma^{\beta}\thin^{\beta}\gamma^{\delta}\thick$. Moreover, the AAD $\thin^{\beta}\alpha^{\delta}\thin^{\delta}\alpha^{\beta}\thin$ of $\thin\alpha\thin\alpha\thin$ implies the AAD $\thick^{\gamma}\delta^{\alpha}\thin^{\beta}\alpha^{\delta}\thin^{\delta}\alpha^{\beta}\thin^{\alpha}\delta^{\gamma}\thick$ of $\alpha^2\delta^2$. The AADs of $\gamma\delta^3,\gamma^6$ are the combinations of $\thick^{\delta}\gamma^{\beta}\thin^{\alpha}\delta^{\gamma}\thick$ and $\thick^{\delta}\gamma^{\beta}\thin^{\beta}\gamma^{\delta}\thick$. The AADs determine the tiles around the vertices.

The AAD of $\gamma\delta^3$ determines $T_1,T_2,T_3,T_4$ 
in Figure \ref{abb-cdddB}. Then $\alpha_1\alpha_3\cdots=\alpha^2\delta^2$ determines $T_5,T_6$. Then $\alpha_2\beta_4\cdots=\alpha_5\beta_3\cdots=\alpha_6\beta_1\cdots=\alpha\beta^2$ and no $\beta\gamma\cdots$ determine $T_9,T_8,T_7$. Then $\gamma_3\gamma_8\delta_4\cdots=\alpha\gamma^3\delta$ determines $T_{10},T_{11}$. Then $\beta_{10}\beta_{11}\cdots=\alpha\beta^2$ and no $\beta\gamma\cdots$ determine $T_{12}$, and $\alpha_4\gamma_9\delta_{11}\cdots=\alpha\gamma^3\delta$ determines $T_{13},T_{14}$. Then $\gamma_{11}\delta_{12}\delta_{14}\cdots=\gamma\delta^3$ determines $T_{15}$. Then $\alpha_{14}\alpha_{15}\cdots=\alpha\delta^2$ determines $T_{16},T_{17}$. We have $\alpha_9\beta_2\cdots=\alpha_{13}\beta_{16}\cdots=\alpha\beta^2$ and $\delta_9\delta_{13}\cdots=\alpha^2\delta^2,\gamma\delta^3$. If $\delta_9\delta_{13}\cdots=\gamma\delta^3$, then we find $\beta,\delta$ adjacent in a tile, a contradiction. Therefore $\delta_9\delta_{13}\cdots=\alpha^2\delta^2$, which determines $T_{18},T_{19}$. Then $\gamma_1\gamma_2\gamma_7\gamma_{18}\cdots=\gamma^6$ determines $T_{20},T_{21}$. Then $\delta_{18}\delta_{19}\delta_{20}\cdots=\gamma\delta^3$ determines $T_{22}$. We may repeat the argument with $T_{18},T_{20},T_{19},T_{22}$ in place of $T_1,T_2,T_3,T_4$. One more repetition of the argument gives the tiling $S_{36\square}5$. 

\begin{figure}[htp]
\centering
\begin{tikzpicture}[xscale=-1]

\foreach \a in {0,1,2}
\foreach \b in {1,-1}
{
\begin{scope}[xshift=3.2*\a cm + 0.4*\b cm, scale=\b]

\draw
	(-1.6,0) -- (1.6,0)
	(-1.6,1.2) -- (1.6,1.2)
	(-0.8,0) -- (-0.8,1.2)
	(0.8,0) -- (0.8,1.2)
	(0,1.2) -- (0,2);

\draw[line width=1.2]
	(-1.6,2) -- (-1.6,0)
	(1.6,2) -- (1.6,0)
	(-0.8,1.2) -- (0.8,0);

\node at (0.8,1.35) {\small $\alpha$};
\node at (0.2,1.4) {\small $\beta$};
\node at (1.4,1.4) {\small $\delta$};
\node at (0.8,1.9) {\small $\gamma$};

\node at (-0.8,1.35) {\small $\alpha$};
\node at (-0.2,1.4) {\small $\beta$};
\node at (-1.4,1.4) {\small $\delta$};
\node at (-0.8,1.9) {\small $\gamma$};

\node at (0,1.05) {\small $\alpha$};
\node at (0.65,1) {\small $\beta$};
\node at (-0.3,1.05) {\small $\delta$};
\node at (0.65,0.4) {\small $\gamma$};

\node at (-0.65,0.2) {\small $\alpha$};
\node at (0,0.2) {\small $\beta$};
\node at (-0.65,0.8) {\small $\delta$};
\node at (0.25,0.17) {\small $\gamma$};

\node at (1,0.2) {\small $\alpha$};
\node at (1,1) {\small $\beta$};
\node at (1.4,0.2) {\small $\delta$};
\node at (1.4,1) {\small $\gamma$};

\node at (-1,1) {\small $\alpha$};
\node at (-1,0.2) {\small $\beta$};
\node at (-1.4,1) {\small $\delta$};
\node at (-1.4,0.2) {\small $\gamma$};

\end{scope}
}

\node[draw,shape=circle, inner sep=0.5] at (4.8,1.8) {\small $2$};
\node[draw,shape=circle, inner sep=0.5] at (5.6,1.8) {\small $1$};
\node[draw,shape=circle, inner sep=0.5] at (4.8,0.6) {\small $4$};
\node[draw,shape=circle, inner sep=0.5] at (5.6,0.6) {\small $3$};
\node[draw,shape=circle, inner sep=0.5] at (6.5,0.45) {\small $5$};
\node[draw,shape=circle, inner sep=0.5] at (7.1,0.75) {\small $6$};
\node[draw,shape=circle, inner sep=0.5] at (8,1.8) {\small $7$};
\node[draw,shape=circle, inner sep=0.5] at (6.3,-0.45) {\small $8$};
\node[draw,shape=circle, inner sep=0.5] at (3.9,0.75) {\small $9$};
\node[draw,shape=circle, inner sep=0] at (5.7,-0.75) {\footnotesize $10$};
\node[draw,shape=circle, inner sep=0] at (4.8,-0.6) {\footnotesize $11$};
\node[draw,shape=circle, inner sep=0] at (4.8,-1.8) {\footnotesize $12$};
\node[draw,shape=circle, inner sep=0] at (3.3,0.45) {\footnotesize $13$};
\node[draw,shape=circle, inner sep=0] at (4,-0.6) {\footnotesize $14$};
\node[draw,shape=circle, inner sep=0] at (4,-1.8) {\footnotesize $15$};
\node[draw,shape=circle, inner sep=0] at (3.1,-0.45) {\footnotesize $16$};
\node[draw,shape=circle, inner sep=0] at (2.5,-0.75) {\footnotesize $17$};
\node[draw,shape=circle, inner sep=0] at (2.4,1.8) {\footnotesize $18$};
\node[draw,shape=circle, inner sep=0] at (2.4,0.6) {\footnotesize $19$};
\node[draw,shape=circle, inner sep=0] at (1.6,1.8) {\footnotesize $20$};
\node[draw,shape=circle, inner sep=0] at (-0.8,1.8) {\footnotesize $21$};
\node[draw,shape=circle, inner sep=0] at (1.6,0.6) {\footnotesize $22$};

\end{tikzpicture}
\caption{Proposition \ref{abb}: Tiling for $\{\alpha\beta^2,\gamma\delta^3,\alpha^2\delta^2,\alpha\gamma^3\delta,\gamma^6\}$.}
\label{abb-cdddB}
\end{figure}
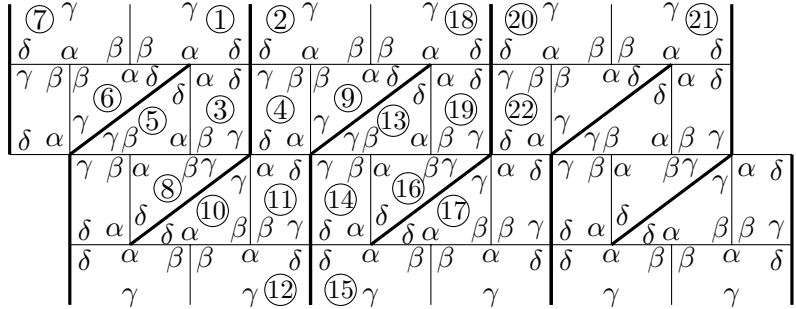

\medskip

\noindent{\em Geometry of Quadrilateral}

\medskip

For the tiling $S_{36\square}5$, we already know all the angle values. Then we use the same routine for $S_{16\square}4$ to find
\begin{align*}
\cos a
&=\tfrac{4}{\sqrt{3}}\sin\tfrac{2}{9}\pi-\sqrt{3}\tan\tfrac{1}{9}\pi,\quad a<\pi, \\
\cos b
&=\tfrac{2}{3}(1-\cos\tfrac{2}{9}\pi)(1+4\cos\tfrac{2}{9}\pi)
-\tfrac{1}{\sqrt{3}}(1-2\cos\tfrac{2}{9}\pi)\cot\tfrac{4}{9}\pi,
\quad b<\pi.
\end{align*}
We note that $\cos a$ and $\cos b$ are the largest (and the only positive) roots of $3t^3+27t^2-3t-19$ and $9t^3+9t^2-9t-1$. We get the approximate values $a=0.1741\pi$ and $b=0.2584\pi$. Then by Lemma \ref{geometry8}, the quadrilateral is geometrically suitable for tiling.
\end{proof}

\begin{proposition}\label{acc}
Tilings of the sphere by congruent almost equilateral quadrilaterals, such that $\alpha\gamma^2$ is the only degree $3$ vertex, are $S_{16\square}2,S_{16\square}3,FS_{16\square}3$.
\end{proposition}

\begin{proof}
By Lemma \ref{count-att}, we know $f\ge 16$. By $\alpha\gamma^2$ and the balance lemma, we know $\delta^2\cdots$ is a vertex. Therefore $\gamma,\delta<\pi$. By $\alpha\gamma^2$, we get $\alpha+2\beta+2\delta=2(\alpha+\beta+\gamma+\delta)-(\alpha+2\gamma)=(2+\frac{8}{f})\pi>2\pi$. 

By the only degree $3$ vertex $\alpha\gamma^2$, we know $\beta,\delta$ do not appear at degree $3$ vertices. By Lemma \ref{deg3miss}, we know one of $\alpha\beta^3,\beta^4,\alpha\beta\delta^2,\beta^2\gamma\delta,\beta^2\delta^2,\gamma\delta^3,\delta^4,\beta^5,\beta^3\delta^2,\beta\delta^4$ is a vertex.

\subsubsection*{Case. $\alpha\gamma^2,\alpha\beta\delta^2$ are vertices}

The angle sums of $\alpha\gamma^2,\alpha\beta\delta^2$ imply
\begin{equation}\label{acc_eq1}
\alpha=2\pi-2\gamma,\;
\beta=\tfrac{8}{f}\pi,\;
\delta=\gamma-\tfrac{4}{f}\pi.
\end{equation}
We have $\gamma>\delta$. By Lemma \ref{geometry1}, this implies $\alpha>\beta$. By $\alpha\beta\delta^2$ and $\alpha>\beta$, we get $\alpha+\delta>\pi>\beta+\delta$. Therefore $\delta<\pi-\beta=(1-\frac{8}{f})\pi$ and $\gamma<(1-\frac{4}{f})\pi$. 

By Lemma \ref{geometry5}, we may further divide the discussion into the case $\alpha>\gamma$ and $\beta>\delta$, and the case $\alpha<\gamma$ and $\beta<\delta$.

\subsubsection*{Subcase. $\alpha>\gamma$ and $\beta>\delta$}

By $\alpha\gamma^2$ and $\alpha>\gamma$, we get $\alpha>\frac{2}{3}\pi>\gamma$. By $\alpha\beta\delta^2$ and $\alpha>\beta$, we get $R(\alpha^2)<R(\alpha\beta)=2\delta$. Then by $\delta<\alpha,\beta,\gamma$, and the only degree $3$ vertex $\alpha\gamma^2$, we know $\alpha^2\cdots$ is not a vertex, and $\alpha\beta\cdots=\alpha\beta\delta^2$. By the only degree $3$ vertex $\alpha\gamma^2$ again, this implies $\alpha\cdots=\alpha\gamma^2,\alpha\beta\delta^2,\alpha\gamma\delta^k(k\ge 3),\alpha\delta^k(k\ge 4)$. Then we get 
\begin{align*}
f
=\#\alpha
&=\#\alpha\gamma^2+\#\alpha\beta\delta^2+\#\alpha\gamma\delta^k+\#\alpha\delta^k \\
&\ge \#\alpha\gamma^2+\#\alpha\beta\delta^2
=\tfrac{1}{2}\#\gamma+\tfrac{1}{2}\#\delta=f.
\end{align*}
Since both sides are equal, we know $\alpha\gamma\delta^k(k\ge 3),\alpha\delta^k(k\ge 4)$ are not vertices, and $\alpha\gamma^2,\alpha\beta\delta^2$ are the only vertices with $\alpha,\gamma,\delta$. Then we get the list of vertices
\begin{equation}\label{acc-abddAVC}
\text{AVC}=\{\alpha\gamma^2,\alpha\beta\delta^2,\beta^k\}.
\end{equation}

Next we find the tiling for the AVC \eqref{acc-abddAVC}. By no $\alpha^2\cdots$, we get the AAD $\thin^{\alpha}\beta^{\gamma}\thin^{\alpha}\beta^{\gamma}\thin\cdots\thin^{\alpha}\beta^{\gamma}\thin$ of $\beta^k$. The AAD $\thin^{\alpha}\beta^{\gamma}\thin^{\alpha}\beta^{\gamma}\thin$ determines $T_1,T_2$ in Figure \ref{accA}. Then $\alpha_1\gamma_2\cdots=\alpha\gamma^2$ determines $T_3$. Then $\beta_3\delta_1\cdots=\alpha\beta\delta^2$ determines $T_4$ and gives $\alpha_5$. There are two ways of arranging $T_5$, given by the two pictures. In the first picture, $\alpha_3\beta_5\cdots=\alpha\beta\delta^2$ determines $T_6,T_7$. Then $\alpha_7\gamma_5\cdots=\alpha\gamma^2$ determines $T_8$. In the second picture, $\alpha_4\beta_5\cdots=\alpha\beta\delta^2$ determines $T_6,T_7$. Then $\alpha_7\gamma_5\cdots=\alpha\gamma^2$ determines $T_8$. 

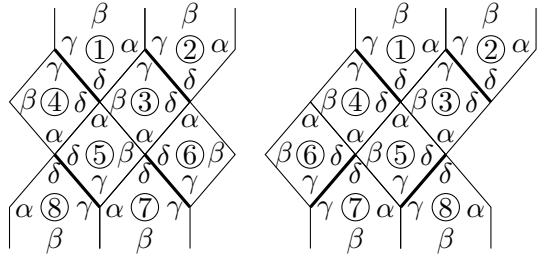
\begin{figure}[htp]
\centering
\begin{tikzpicture}

\foreach \a in {0,1}
{
\begin{scope}[xshift=4*\a cm]

\foreach \b in {0,1}
{
\begin{scope}[xshift=1.2*\b cm]

\draw
	(0,1.6) -- (0,1.05) -- (-0.6,0.35) -- (0,-0.35) -- (1.2,1.05) -- (1.2,1.6);

\draw[line width=1.2]
	(0,1.05) -- (0.6,0.35);

\node at (1,1.1) {\small $\alpha$};
\node at (0.6,1.5) {\small $\beta$};
\node at (0.2,1.1) {\small $\gamma$};
\node at (0.6,0.65) {\small $\delta$};

\node at (0,-0.1) {\small $\alpha$};
\node at (-0.35,0.35) {\small $\beta$};
\node at (0,0.75) {\small $\gamma$};
\node at (0.35,0.35) {\small $\delta$};
	
\end{scope}
}

\node[draw,shape=circle, inner sep=0.5] at (0.6,1.05) {\small $1$};
\node[draw,shape=circle, inner sep=0.5] at (1.8,1.05) {\small $2$};
\node[draw,shape=circle, inner sep=0.5] at (1.2,0.35) {\small $3$};
\node[draw,shape=circle, inner sep=0.5] at (0,0.35) {\small $4$};

\end{scope}
}

\foreach \a/\c in {0/1,5.2/-1}
\foreach \b in {0,1}
{
\begin{scope}[xshift=\a cm, xscale=\c, xshift=1.2*\b cm]

\draw
	(0.6,-1.6) -- (0.6,-1.05) -- (1.2,-0.35) -- (0.6,0.35) -- (-0.6,-1.05) -- (-0.6,-1.6);

\draw[line width=1.2]
	(0,-0.35) -- (0.6,-1.05);
	
\node at (-0.4,-1.1) {\small $\alpha$};
\node at (0,-1.5) {\small $\beta$};
\node at (0.4,-1.1) {\small $\gamma$};
\node at (0,-0.6) {\small $\delta$};

\node at (0.6,0.1) {\small $\alpha$};
\node at (0.95,-0.35) {\small $\beta$};
\node at (0.6,-0.75) {\small $\gamma$};
\node at (0.25,-0.35) {\small $\delta$};

\end{scope}
}

\foreach \a/\c in {0/1,5.2/-1}
{
\begin{scope}[xshift=\a cm, xscale=\c]

\node[draw,shape=circle, inner sep=0.5] at (0.6,-0.35) {\small $5$};
\node[draw,shape=circle, inner sep=0.5] at (1.8,-0.35) {\small $6$};
\node[draw,shape=circle, inner sep=0.5] at (1.2,-1.05) {\small $7$};
\node[draw,shape=circle, inner sep=0.5] at (0,-1.05) {\small $8$};

\end{scope}
}

\end{tikzpicture}
\caption{Proposition \ref{acc}: Tilings for $\{\alpha\gamma^2,\alpha\beta\delta^2,\beta^k\}$.}
\label{accA}
\end{figure}

The argument can be applied to any pair $\thin^{\alpha}\beta^{\gamma}\thin^{\alpha}\beta^{\gamma}\thin$ in $\beta^k$. Moreover, the choice of the arrangement of $T_5$ implies the same choice of the arrangement of $T_6$. Therefore we get two earth map tilings. We also note that $T_1,T_2,T_3,T_4$ and their extensions form the upper half of the earth map tiling, and $T_5,T_6,T_7,T_8$ and their extensions form the lower half of the earth map tiling. The two earth map tilings have the same upper half, and the lower half of the second is the horizontal flip of the lower half of the first.

\subsubsection*{Subcase. $\alpha<\gamma$ and $\beta<\delta$} 

By $\alpha\gamma^2$ and $\alpha<\gamma$, we get $\alpha<\frac{2}{3}\pi<\gamma$. This implies $\gamma+\delta=2\gamma-\frac{4}{f}\pi>(\frac{4}{3}-\frac{4}{f})\pi>\pi$. By $\alpha\gamma^2$, this implies $\alpha<2\delta$. Moreover, by Lemma \ref{geometry3}, we get $(1+\frac{8}{f})\pi=\beta+\pi>\gamma+\delta>(\frac{4}{3}-\frac{4}{f})\pi$. This implies $f<36$. Then $3\beta=\frac{24}{f}\pi>\frac{2}{3}\pi>\alpha$. 

A $b$-vertex is $\gamma^2\cdots,\gamma\delta\cdots,\delta^2\cdots$, and we may assume the remainders of $\gamma\delta\cdots,\delta^2\cdots$ have no $\gamma$. Moreover, the only degree $3$ vertex $\alpha\gamma^2$ implies $\gamma\delta\cdots,\delta^2\cdots$ have degree $\ge 4$.

By $\alpha\gamma^2$, we get $R(\gamma^2)=\alpha<3\beta,\gamma,2\delta$. This implies $\gamma^2\cdots=\alpha\gamma^2,\beta^2\gamma^2$. By $\alpha\beta\delta^2$ and $\gamma>\delta$, we get $R(\gamma\delta)<R(\delta^2)=\alpha+\beta$. Then by $\alpha<2\delta$, and $\beta<\alpha<3\beta$, and $\beta<\delta$, and degree $\ge 4$, this implies $\gamma\delta\cdots,\delta^2\cdots$ (with no $\gamma$ in the remainders) are $\alpha\beta\delta^2,\beta^2\gamma\delta,\beta^3\gamma\delta,\beta^2\delta^2,\beta^3\delta^2,\gamma\delta^3,\delta^4$. By $\beta+\delta<\pi$, we know $\beta^2\delta^2$ is actually not a vertex. 

The angle sum of one of $\beta^2\gamma^2,\beta^2\gamma\delta,\beta^3\gamma\delta,\beta^3\delta^2,\gamma\delta^3,\delta^4$, and the angle sums of $\alpha\gamma^2,\alpha\beta\delta^2$, imply the following, including \eqref{coolsaet_eq1}.
\begin{itemize}
\item $\beta^2\gamma^2/\beta^3\delta^2$:
	$\alpha=\tfrac{16}{f}\pi,\;
	\beta=\tfrac{8}{f}\pi,\;
	\gamma=(1-\tfrac{8}{f})\pi,\;
	\delta=(1-\tfrac{12}{f})\pi$.
	\newline 
	$\sin\tfrac{8}{f}\pi
	\sin(1-\tfrac{16}{f})\pi
	=\sin\tfrac{4}{f}\pi
	\sin(1-\tfrac{16}{f})\pi$. 
\item $\beta^2\gamma\delta$:
	$\alpha=\tfrac{12}{f}\pi,\;
	\beta=\tfrac{8}{f}\pi,\;
	\gamma=(1-\tfrac{6}{f})\pi,\;
	\delta=(1-\tfrac{10}{f})\pi$.
	\newline 
	$\sin\tfrac{6}{f}\pi
	\sin(1-\tfrac{14}{f})\pi
	=\sin\tfrac{4}{f}\pi
	\sin(1-\tfrac{12}{f})\pi$. 
\item $\beta^3\gamma\delta$:
	$\alpha=\tfrac{20}{f}\pi,\;
	\beta=\tfrac{8}{f}\pi,\;
	\gamma=(1-\tfrac{10}{f})\pi,\;
	\delta=(1-\tfrac{14}{f})\pi$.
	\newline 
	$\sin\tfrac{10}{f}\pi
	\sin(1-\tfrac{18}{f})\pi
	=\sin\tfrac{4}{f}\pi
	\sin(1-\tfrac{20}{f})\pi$. 
\item $\gamma\delta^3$:
	$\alpha=(1-\tfrac{6}{f})\pi,\;
	\beta=\tfrac{8}{f}\pi,\;
	\gamma=(\tfrac{1}{2}+\tfrac{3}{f})\pi,\;
	\delta=(\tfrac{1}{2}-\tfrac{1}{f})\pi$.
	\newline 
	$\sin(\tfrac{1}{2}-\tfrac{3}{f})\pi
	\sin(\tfrac{1}{2}-\tfrac{5}{f})\pi
	=\sin\tfrac{4}{f}\pi
	\sin\tfrac{6}{f}\pi$. 
\item $\delta^4$:
	$\alpha=(1-\tfrac{8}{f})\pi,\;
	\beta=\tfrac{8}{f}\pi,\;
	\gamma=(\tfrac{1}{2}+\tfrac{4}{f})\pi,\;
	\delta=\tfrac{1}{2}\pi$.
	\newline 
	$\sin(\tfrac{1}{2}-\tfrac{4}{f})\pi
	\sin(\tfrac{1}{2}-\tfrac{4}{f})\pi
	=\sin\tfrac{4}{f}\pi
	\sin\tfrac{8}{f}\pi$. 
\end{itemize}
The only solutions for even $f\ge 16$ are $f=16$ for $\beta^2\gamma^2$ and $\beta^3\delta^2$, and $f=18$ for $\beta^2\gamma\delta$ and $\gamma\delta^3$. However, $f=16$ implies $\alpha>\gamma$, and $f=18$ implies $\alpha=\gamma$. Both are contradictions. 

Therefore $\alpha\gamma^2,\alpha\beta\delta^2$ are the only $b$-vertices. This implies
\[
f=\tfrac{1}{2}\#\gamma+\tfrac{1}{2}\#\delta
=\#\alpha\gamma^2+\#\alpha\beta\delta^2
\le \#\alpha=f.
\]
Since both sides are equal, we find $\alpha\gamma^2,\alpha\beta\delta^2$ are the only vertices with $\alpha,\gamma,\delta$. Then we get the same AVC \eqref{acc-abddAVC}, and two earth map tilings in Figure \ref{accA}.

\medskip

\noindent{\em Geometry of Quadrilateral}

\medskip

The angles in the tilings in Figure \ref{accA} satisfy \eqref{acc_eq1}. Substituting \eqref{acc_eq1} into \eqref{coolsaet_eq1} and using $\gamma<\pi$, we get
\begin{equation}\label{acc_eq2}
\cos\tfrac{8}{f}\pi\sin\gamma
=2\sin\tfrac{4}{f}\pi(\cos\tfrac{4}{f}\pi-1)\cos\gamma.
\end{equation}
For $f>16$, this implies $\gamma>\frac{1}{2}\pi$. Then by $\alpha\gamma^2$, we get $\alpha<\pi$. By Lemma \ref{geometry3}, we get $\beta+\pi>\gamma+\delta$. This means $\gamma<(\frac{1}{2}+\frac{6}{f})\pi$. Then \eqref{acc_eq2} implies 
\[
\cos\tfrac{8}{f}\pi\cos\tfrac{6}{f}\pi
<2\sin\tfrac{4}{f}\pi(1-\cos\tfrac{4}{f}\pi)\sin\tfrac{6}{f}\pi.
\]
The inequality implies $f<20$. By the only degree $3$ vertex $\alpha\gamma^2$ and the vertex $\beta^{\frac{f}{4}}$, we get $f=16$. In other words, the tiling has four timezones. Then by $\gamma<\pi$ and \eqref{acc_eq2}, we get $\gamma=\frac{1}{2}\pi$, and
\[
f=16\colon \alpha=\pi,\;
\beta=\gamma=\tfrac{1}{2}\pi,\;
\delta=\tfrac{1}{4}\pi.
\]
Therefore the quadrilateral is actually an isosceles triangle with top angle $\delta$ and base angles $\beta,\gamma$. This implies $a=\frac{1}{4}\pi$ and $b=\frac{1}{2}\pi$. 

Figure \ref{accC} gives the two full tilings with four timezones. After exchanging $(\alpha,\delta)$ with $(\beta,\gamma)$, the two earth map tilings become $S_{16\square}3,FS_{16\square}3$. 

\begin{figure}[htp]
\centering
\begin{tikzpicture}[scale=1]

\foreach \a in {0,...,3}
{
\begin{scope}[rotate=90*\a]

\draw
	(0,0) -- (0.8,0) -- (1,1) -- (0,0.8)
	(1.6,0) -- (1,1) -- (0,1.6) -- (1.6,1.6) -- (-1.6,1.6)
	(1.6,1.6) -- (2,2);

\draw[line width=1.2]
	(0.8,0) -- (1,1)
	(1.6,1.6) -- (0,1.6);

\node at (0.2,0.65) {\small $\alpha$};
\node at (0.2,0.2) {\small $\beta$};
\node at (0.65,0.2) {\small $\gamma$};
\node at (0.75,0.75) {\small $\delta$};

\node at (1.4,0) {\small $\alpha$};
\node at (1.07,-0.5) {\small $\beta$};
\node at (1,0) {\small $\gamma$};
\node at (1.1,0.5) {\small $\delta$};

\node at (1.1,1.1) {\small $\alpha$};
\node at (1.4,0.65) {\small $\beta$};
\node at (1.4,1.4) {\small $\gamma$};
\node at (0.65,1.4) {\small $\delta$};

\node at (1.8,1.5) {\small $\alpha$};
\node at (2.2,0) {\small $\beta$};
\node at (1.8,-1.5) {\small $\gamma$};
\node at (1.8,0) {\small $\delta$};

\end{scope}

\begin{scope}[xshift=-5.5cm, rotate=90*\a]

\draw
	(0,0) -- (0.8,0) -- (1,1) -- (0,0.8)
	(1.6,0) -- (1,1) -- (0,1.6) -- (1.6,1.6) -- (-1.6,1.6)
	(1.6,1.6) -- (2,2);

\draw[line width=1.2]
	(0.8,0) -- (1,1)
	(1.6,1.6) -- (1.6,0);

\node at (0.2,0.65) {\small $\alpha$};
\node at (0.2,0.2) {\small $\beta$};
\node at (0.65,0.2) {\small $\gamma$};
\node at (0.75,0.75) {\small $\delta$};

\node at (1.4,0) {\small $\alpha$};
\node at (1.07,-0.5) {\small $\beta$};
\node at (1,0) {\small $\gamma$};
\node at (1.1,0.5) {\small $\delta$};

\node at (1.1,1.1) {\small $\alpha$};
\node at (0.65,1.4) {\small $\beta$};
\node at (1.4,1.4) {\small $\gamma$};
\node at (1.4,0.65) {\small $\delta$};

\node at (1.8,-1.5) {\small $\alpha$};
\node at (2.2,0) {\small $\beta$};
\node at (1.8,1.5) {\small $\gamma$};
\node at (1.8,0) {\small $\delta$};

\end{scope}
}

\end{tikzpicture}
\caption{Proposition \ref{acc}: Earth map tilings have only four timezones.}
\label{accC}
\end{figure}

\subsubsection*{Case. $\alpha\gamma^2$ and one of $\alpha\beta^3,\beta^4,\beta^5$ are vertices}

The angle sum of one of $\alpha\beta^3,\beta^4,\beta^5$, and the angle sum of $\alpha\gamma^2$, imply
\begin{align*}
\alpha\gamma^2,\alpha\beta^3 &\colon
\alpha=2\pi-2\gamma,\;
\beta=\tfrac{2}{3}\gamma,\;
\delta=\tfrac{1}{3}\gamma+\tfrac{4}{f}\pi. \\
\alpha\gamma^2,\beta^4 &\colon
\alpha=2\pi-2\gamma,\;
\beta=\tfrac{1}{2}\pi,\;
\delta=\gamma-(\tfrac{1}{2}-\tfrac{4}{f})\pi. \\
\alpha\gamma^2,\beta^5 &\colon
\alpha=2\pi-2\gamma,\;
\beta=\tfrac{2}{5}\pi,\;
\delta=\gamma-(\tfrac{2}{5}-\tfrac{4}{f})\pi.
\end{align*}

Suppose $\alpha<\beta$ and $\gamma<\delta$ in the $\alpha\beta^3$ case. Then $\gamma<\delta$ implies $\gamma<\frac{6}{f}\pi$, and we get $\beta+2\gamma=\frac{8}{3}\gamma<\frac{16}{f}\pi\le \pi$, contradicting Lemma \ref{geometry4}. In the $\beta^4$ and $\beta^5$ cases, we already have $\gamma>\delta$. By Lemma \ref{geometry1}, therefore, we have $\alpha>\beta$ and $\gamma>\delta$ in all three cases. 

By the only degree $3$ vertex $\alpha\gamma^2$, we know $k+l\ge 4$ in $\hat{b}$-vertex $\alpha^k\beta^l$. For $k\ge 1$, by $\alpha>\beta$, we get $\alpha^k\beta^l=\alpha\beta^3$ for the $\alpha\beta^3$ case, and no $\alpha^k\beta^l$ for the $\beta^4$ case, and $\alpha^k\beta^l=\alpha^4,\alpha^3\beta,\alpha^2\beta^2,\alpha\beta^3$ for the $\beta^5$ case. The angle sum of one of $\alpha^4,\alpha^3\beta,\alpha^2\beta^2$, and the angle sum of $\alpha\gamma^2,\beta^5$, imply the following, including \eqref{coolsaet_eq1}.
\begin{itemize}
\item $\alpha\gamma^2,\beta^5,\alpha^4$:
	$\alpha=\tfrac{1}{2}\pi,\;
	\beta=\tfrac{2}{5}\pi,\;
	\gamma=\tfrac{3}{4}\pi,\;
	\delta=(\tfrac{7}{20}+\tfrac{4}{f})\pi$.
	\newline 
	$\sin\tfrac{1}{4}\pi
	\sin(\tfrac{3}{20}+\tfrac{4}{f})\pi
	=\sin\tfrac{1}{5}\pi
	\sin\tfrac{1}{2}\pi$. 
\item $\alpha\gamma^2,\beta^5,\alpha^3\beta$:
	$\alpha=\tfrac{8}{15}\pi,\;
	\beta=\tfrac{2}{5}\pi,\;
	\gamma=\tfrac{11}{15}\pi,\;
	\delta=(\tfrac{1}{3}+\tfrac{4}{f})\pi$.
	\newline 
	$\sin\tfrac{4}{15}\pi
	\sin(\tfrac{2}{15}+\tfrac{4}{f})\pi
	=\sin\tfrac{1}{5}\pi
	\sin\tfrac{7}{15}\pi$. 
\item $\alpha\gamma^2,\beta^5,\alpha^2\beta^2$:
	$\alpha=\tfrac{3}{5}\pi,\;
	\beta=\tfrac{2}{5}\pi,\;
	\gamma=\tfrac{7}{10}\pi,\;
	\delta=(\tfrac{3}{10}+\tfrac{4}{f})\pi$.
	\newline 
	$\sin\tfrac{3}{10}\pi
	\sin(\tfrac{1}{10}+\tfrac{4}{f})\pi
	=\sin\tfrac{1}{5}\pi
	\sin\tfrac{2}{5}\pi$. 
\end{itemize}
The equations have no solution for even $f\ge 16$. Therefore a $\hat{b}$-vertex $\alpha\cdots$ is $\alpha\beta^3$ for the $\alpha\beta^3$ and $\beta^5$ cases, and there is no $\hat{b}$-vertex $\alpha\cdots$ for the $\beta^4$ case.

Therefore $\alpha^2\cdots$ is a $b$-vertex. By Lemma \ref{geometry4}, we get $\alpha+2\delta>\pi$. Moreover, we have $2\alpha+\gamma+\delta>\alpha+\beta+\gamma+\delta>2\pi$. Then by $\gamma>\delta$, we get $\alpha^2\cdots=\alpha^2\delta^2\cdots$, with no $\gamma,\delta$ in the remainder. Then by $\alpha>\beta$ and $2\alpha+\beta+2\delta>\alpha+2\beta+2\delta>2\pi$, we get $\alpha^2\cdots=\alpha^2\delta^2$. The angle sum of $\alpha^2\delta^2$ further implies the following, including \eqref{coolsaet_eq1}.
\begin{itemize}
\item $\alpha\gamma^2,\alpha\beta^3,\alpha^2\delta^2$:
	$\alpha=(\tfrac{4}{5}-\tfrac{24}{5f})\pi,\;
	\beta=(\tfrac{2}{5}+\tfrac{8}{5f})\pi,\;
	\gamma=(\tfrac{3}{5}+\tfrac{12}{5f})\pi,\;
	\delta=(\tfrac{1}{5}+\tfrac{24}{5f})\pi$.
	\newline
	$\sin(\tfrac{2}{5}-\tfrac{12}{5f})\pi
	\sin\tfrac{4}{f}\pi
	=\sin(\tfrac{1}{5}+\tfrac{4}{5f})\pi
	\sin(\tfrac{1}{5}+\tfrac{24}{5f})\pi$. 
\item $\alpha\gamma^2,\beta^4,\alpha^2\delta^2$:
	$\alpha=(1-\tfrac{8}{f})\pi,\;
	\beta=\tfrac{1}{2}\pi,\;
	\gamma=(\tfrac{1}{2}+\tfrac{4}{f})\pi,\;
	\delta=\tfrac{8}{f}\pi$.
	\newline
	$\sin(\tfrac{1}{2}-\tfrac{4}{f})\pi
	\sin(\tfrac{8}{f}-\tfrac{1}{4})\pi
	=\sin\tfrac{1}{4}\pi
	\sin\tfrac{8}{f}\pi$.  
\item $\alpha\gamma^2,\beta^5,\alpha^2\delta^2$:
	$\alpha=(\tfrac{4}{5}-\tfrac{8}{f})\pi,\;
	\beta=\tfrac{2}{5}\pi,\;
	\gamma=(\tfrac{3}{5}+\tfrac{4}{f})\pi,\;
	\delta=(\tfrac{1}{5}+\tfrac{8}{f})\pi$.
	\newline
	$\sin(\tfrac{2}{5}-\tfrac{4}{f})\pi
	\sin\tfrac{8}{f}\pi
	=\sin\tfrac{1}{5}\pi
	\sin(\tfrac{1}{5}+\tfrac{8}{f})\pi$. 
\end{itemize}
The only solution for even $f\ge 16$ is $f=20$ for $\alpha\gamma^2,\beta^5,\alpha^2\delta^2$. However, this implies $\alpha=\beta$, a contradiction. Therefore $\alpha^2\cdots$ is not a vertex. By Lemma \ref{klem3}, this implies a $\delta^2$-fan has a single $\alpha$. 

By $\alpha\gamma^2$, and no $\alpha^2\cdots$, and $\gamma>\delta$, and $\alpha+2\beta+2\delta>2\pi$, a $b$-vertex $\alpha\cdots$ is $\alpha\gamma^2,\alpha\beta\delta^k,\alpha\gamma\delta^k,\alpha\delta^k$. By the only degree $3$ vertex $\alpha\gamma^2$, we get $k\ge 3$ in $\alpha\gamma\delta^k$ and $k\ge 4$ in $\alpha\delta^k$. Since a $\delta^2$-fan has a single $\alpha$, we know $\alpha\delta^k$ is not a vertex, and $\alpha\beta\delta^k=\alpha\beta\delta^2$, and $\alpha\gamma\delta^k=\alpha\gamma\delta^3$. Therefore a $b$-vertex $\alpha\cdots$ is $\alpha\gamma^2,\alpha\beta\delta^2,\alpha\gamma\delta^3$.

For the case $\alpha\beta^3$ is a vertex, we know $\alpha\cdots=\alpha\gamma^2,\alpha\beta^3,\alpha\beta\delta^2,\alpha\gamma\delta^3$. The angle sum of $\alpha\beta\delta^2$ further implies the following, including \eqref{coolsaet_eq1}.
\begin{align*}
\alpha\gamma^2,\alpha\beta^3,\alpha\beta\delta^2 &\colon
	\alpha=(2-\tfrac{24}{f})\pi,\;
	\beta=\delta=\tfrac{8}{f}\pi,\;
	\gamma=\tfrac{12}{f}\pi. \\
& 	\sin(1-\tfrac{12}{f})\pi
	\sin\tfrac{4}{f}\pi
	=\sin\tfrac{4}{f}\pi
	\sin(\tfrac{24}{f}-1)\pi. 
\end{align*}
The solution for even $f\ge 16$ is $f=18$. However, this implies $\alpha=\gamma=\tfrac{2}{3}\pi$, a contradiction. Moreover, adding the angle sums of $\alpha\gamma^2,\alpha\beta^3,\alpha\gamma\delta^3$ together gives $3(\alpha+\beta+\gamma+\delta)=6\pi$, a contradiction. Therefore $\alpha\cdots=\alpha\beta^3,\alpha\gamma^2$. Then we get 
\[
f=\#\alpha
=\#\alpha\gamma^2+\#\alpha\beta^3
\le \tfrac{1}{2}\#\gamma+\tfrac{1}{3}\#\beta
=\tfrac{1}{2}f+\tfrac{1}{3}f
<f,
\]
a contradiction. 

For the case $\beta^4$ is a vertex, we know $\alpha\cdots=\alpha\gamma^2,\alpha\beta\delta^2,\alpha\gamma\delta^3$. Then we get
\[
f=\#\alpha
=\#\alpha\gamma^2+\#\alpha\beta\delta^2+\#\alpha\gamma\delta^3
\le\tfrac{1}{2}\#\gamma+\tfrac{1}{2}\#\delta=f.
\]
Since both sides are equal, we know $\alpha\gamma\delta^3$ is not a vertex, and $\alpha\gamma^2,\alpha\beta\delta^2$ are the only vertices with $\alpha,\gamma,\delta$. This implies the list of vertices
\[
\text{AVC}=\{\alpha\gamma^2,\alpha\beta\delta^2,\beta^4\}.
\]
This is the case $k=4$ of the AVC \eqref{acc-abddAVC}. Then we get two earth map tilings in Figure \ref{accA} with four timezones. In other words, they are $S_{16\square}3,FS_{16\square}3$ in Figure \ref{accC}.

For the case $\beta^5$ is a vertex, we know $\alpha\cdots=\alpha\gamma^2,\alpha\beta^3,\alpha\beta\delta^2,\alpha\gamma\delta^3$. Since we already proved $\alpha\beta^3$ is not a vertex, we get $\alpha\cdots=\alpha\gamma^2,\alpha\beta\delta^2,\alpha\gamma\delta^3$. This implies
\[
f=\#\alpha
=\#\alpha\gamma^2+\#\alpha\beta\delta^2+\#\alpha\gamma\delta^3
\le\tfrac{1}{2}\#\gamma+\tfrac{1}{2}\#\delta=f.
\]
Since both sides are equal, we get the list of vertices similar to the $\beta^4$ case
\[
\text{AVC}=\{\alpha\gamma^2,\alpha\beta\delta^2,\beta^5\}.
\]
This is the case $k=5$ of the AVC \eqref{acc-abddAVC}. Then we get two earth map tilings in Figure \ref{accA} with five timezones. In fact, further geometrical argument shows the quadrilateral does not exist.

\subsubsection*{Case. $\alpha\gamma^2,\beta^2\delta^2$ are vertices}

The angle sums of $\alpha\gamma^2,\beta^2\delta^2$ imply
\[
\alpha=\tfrac{8}{f}\pi,\;
\beta+\delta=\pi,\;
\gamma=(1-\tfrac{4}{f})\pi.
\]
By Lemma \ref{geometry3}, we get $\beta+\pi>\gamma+\delta$. This implies $\delta<(\frac{1}{2}+\frac{2}{f})\pi$. By $f\ge 16$, we get $\delta<\gamma$. Then by Lemma \ref{geometry1}, we get $\tfrac{8}{f}\pi=\alpha>\beta=\pi-\delta>(\frac{1}{2}-\frac{2}{f})\pi$. This means $f<20$, or $f=16,18$. Then we get the following, including \eqref{coolsaet_eq1}.
\begin{align*}
f=16 &\colon
	\alpha=\tfrac{1}{2}\pi,\;
	\beta+\delta=\pi,\;
	\gamma=\tfrac{3}{4}\pi. \\
& 	\sin\tfrac{1}{4}\pi
	\sin(\pi-\tfrac{3}{2}\beta)
	=\sin\tfrac{1}{2}\beta
	\sin\tfrac{1}{2}\pi. \\
f=18 &\colon
	\alpha=\tfrac{4}{9}\pi,\;
	\beta+\delta=\pi,\;
	\gamma=\tfrac{7}{9}\pi.  \\
& 	\sin\tfrac{2}{9}\pi
	\sin(\pi-\tfrac{3}{2}\beta)
	=\sin\tfrac{1}{2}\beta
	\sin\tfrac{5}{9}\pi.
\end{align*}
The solutions satisfying $0<\beta<\pi$ are $\beta=0.4335\pi$ for $f=16$, and $\beta=0.4142\pi$ for $f=18$. Using the precise values of $\alpha,\gamma$ and approximate values of $\beta,\delta$, we get $\text{AVC}=\{\alpha\gamma^2,\beta^2\delta^2,\alpha^4\}$ for $f=16$, and $\text{AVC}=\{\alpha\gamma^2,\beta^2\delta^2\}$ for $f=18$. 

By applying the counting lemma to $\alpha,\gamma$, we find no tiling for $f=18$, and $\alpha^4$ is a vertex for $f=16$. By $\delta^2\cdots=\beta^2\delta^2=\delta\thick\delta\cdots$, we know $\delta\thin\delta\cdots$ is not a vertex. This implies the AAD of $\alpha^4$ is $\thin^{\beta}\alpha^{\delta}\thin^{\beta}\alpha^{\delta}\thin^{\beta}\alpha^{\delta}\thin^{\beta}\alpha^{\delta}\thin$, which determines $T_1,T_2,T_3,T_4$ in Figure \ref{accB}. Then $\beta_2\delta_1\cdots=\beta^2\delta^2$ and $\gamma_2\cdots=\alpha\gamma^2$ determine $T_5,T_6$. Then $\alpha_5\gamma_6\cdots=\alpha\gamma^2$ determines $T_7$. Similar argument determines all the other tiles. We get the tiling $S_{16\square}2$.

\begin{figure}[htp]
\centering
\begin{tikzpicture}

\foreach \a in {0,1,2,3}
{
\begin{scope}[rotate=90*\a]

\draw
	(0,0) -- (1.8,0) -- (1.8,-1.8)
	(1,0) -- (1,1) -- (2.3,2.3);

\draw[line width=1.2]
	(0,1) -- (1,1)
	(1.8,0) -- (1.8,1.8);
	
\node at (0.2,0.2) {\small $\alpha$};
\node at (0.8,0.2) {\small $\beta$};
\node at (0.2,0.8) {\small $\delta$};
\node at (0.8,0.8) {\small $\gamma$};

\node at (1.2,0.9) {\small $\alpha$};
\node at (1.2,0.2) {\small $\beta$};
\node at (1.6,1.35) {\small $\delta$};
\node at (1.6,0.2) {\small $\gamma$};

\node at (1.2,-0.9) {\small $\gamma$};
\node at (1.2,-0.2) {\small $\delta$};
\node at (1.6,-1.35) {\small $\beta$};
\node at (1.6,-0.2) {\small $\alpha$};

\node at (2.3,0) {\small $\alpha$};
\node at (2,-1.7) {\small $\beta$};
\node at (2,1.7) {\small $\delta$};
\node at (2,0) {\small $\gamma$};

\end{scope}
}

\node[inner sep=0.5, draw, shape=circle] at (0.5,0.5) {\small 1};
\node[inner sep=0.5, draw, shape=circle] at (-0.5,0.5) {\small 2};
\node[inner sep=0.5, draw, shape=circle] at (-0.5,-0.5) {\small 3};
\node[inner sep=0.5, draw, shape=circle] at (0.5,-0.5) {\small 4};
\node[inner sep=0.5, draw, shape=circle] at (0.6,1.4) {\small 5};
\node[inner sep=0.5, draw, shape=circle] at (-0.6,1.4) {\small 6};
\node[inner sep=0.5, draw, shape=circle] at (0.8,2.1) {\small 7};

\end{tikzpicture}
\caption{Proposition \ref{acc}: Tiling for $\{\alpha\gamma^2,\beta^2\delta^2,\alpha^4\}$.}
\label{accB}
\end{figure}

\medskip

\noindent{\em Geometry of Quadrilateral}

\medskip

The tiling $S_{16\square}2$ satisfies
\[
\alpha=\tfrac{1}{2}\pi,\;
\beta+\delta=\pi,\;
\gamma=\tfrac{3}{4}\pi.
\]
Substituting into \eqref{coolsaet_eq1}, we find $\beta$ is determined by 
\[
\cos\beta=\tfrac{1}{2}(\sqrt{2}-1),\quad \beta<\pi.
\]
Then we substitute the precise values into \eqref{coolsaet_eq2} and \eqref{coolsaet_eq4}, and use $K=Y(b)^T$ to get
\[
\cos a=\sqrt{\tfrac{1}{\sqrt{7}}(2\sqrt{2}-1)},\;
\cos b=\sqrt{\tfrac{1}{\sqrt{7}}(22\sqrt{2}-25)},
\quad a,b<\pi.
\]
We get the following approximate values
\[
\alpha=\tfrac{1}{2}\pi,\;
\beta=0.4335\pi,\;
\gamma=\tfrac{3}{4}\pi,\;
\delta=0.5664\pi,\;
a=0.3292\pi,\;
b=0.1158\pi.
\]
By Lemma \ref{geometry8}, the quadrilateral is simple and suitable for tiling.

\subsubsection*{Case. $\alpha\gamma^2,\delta^4$ are vertices}

The angle sums of $\alpha\gamma^2,\delta^4$ imply
\[
\alpha=2\pi-2\gamma,\;
\beta=\gamma-(\tfrac{1}{2}-\tfrac{4}{f})\pi,\;
\delta=\tfrac{1}{2}\pi.
\]

If $\alpha<\beta$ and $\gamma<\delta$, then we get $\alpha<\beta<\delta-(\tfrac{1}{2}-\tfrac{4}{f})\pi=\frac{4}{f}\pi$. Then by Lemma \ref{geometry4}, we get $\pi<\alpha+2\beta<\frac{12}{f}\pi$, contradicting $f\ge 16$. By Lemma \ref{geometry1}, therefore, we get $\alpha>\beta$ and $\gamma>\delta$. 

If $\alpha>\gamma$, then by $\alpha\gamma^2$, we get $\alpha>\frac{2}{3}\pi>\gamma$. By Lemma \ref{geometry5}, this implies $\beta>\delta$, which means $\gamma>(1-\tfrac{4}{f})\pi\ge \frac{3}{4}\pi$, contradicting $\gamma<\frac{2}{3}\pi$. Therefore we have $\alpha<\gamma$. By Lemma \ref{geometry5}, this implies $\beta<\delta$. By $\alpha\gamma^2$ and $\alpha<\gamma$, we get $\alpha<\frac{2}{3}\pi<\gamma$. This implies $\beta>(\tfrac{1}{6}+\tfrac{4}{f})\pi>\tfrac{1}{6}\pi$, and $4\beta>\frac{2}{3}\pi>\alpha$.

The AAD $\thick^{\delta}\gamma^{\beta}\thin^{\beta}\alpha^{\delta}\thin^{\beta}\gamma^{\delta}\thick$ of $\alpha\gamma^2$ implies $\beta\delta\cdots$ is a vertex. By $\gamma>\delta=\frac{1}{2}\pi$, we know $\beta\delta\cdots=\beta\gamma\delta\cdots,\beta\delta^2\cdots$, with no $\gamma,\delta$ in the remainders. We have $R(\beta\gamma\delta)<\alpha$. By $\alpha+2\beta+2\delta>2\pi$, we get $R(\beta\delta^2)<\alpha+\beta$. Then by $\beta<\alpha<4\beta$ and the only degree $3$ vertex $\alpha\gamma^2$, the remainder estimations imply $\beta\delta\cdots=\alpha\beta\delta^2,\beta^2\gamma\delta,\beta^3\gamma\delta,\beta^4\gamma\delta,\beta^2\delta^2,\beta^3\delta^2,\beta^4\delta^2,\beta^5\delta^2$. We proved that the tilings with the only degree $3$ vertex $\alpha\gamma^2$ and the vertex $\alpha\beta\delta^2$ has the AVC $\{\alpha\gamma^2,\alpha\beta\delta^2,\beta^k\}$ (see \eqref{acc-abddAVC} and Figure \ref{accA}). We also proved that the tiling with the only degree $3$ vertex $\alpha\gamma^2$ and the vertex $\beta^2\delta^2$ has the AVC $\{\alpha\gamma^2,\beta^2\delta^2,\alpha^4\}$ (see Figure \ref{accB}). These tilings do not have $\delta^4$. 

The angle sum of one of $\beta^2\gamma\delta,\beta^3\gamma\delta,\beta^4\gamma\delta,\beta^3\delta^2,\beta^4\delta^2,\beta^5\delta^2$, and the angle sums of $\alpha\gamma^2,\delta^4$, imply the following, including \eqref{coolsaet_eq1}.
\begin{itemize}
\item $\beta^2\gamma\delta$:
	$\alpha=(\tfrac{1}{3}+\tfrac{16}{3f})\pi,\;
	\beta=(\tfrac{1}{3}+\tfrac{4}{3f})\pi,\;
	\gamma=(\tfrac{5}{6}-\tfrac{8}{3f})\pi,\;
	\delta=\tfrac{1}{2}\pi$.
	\newline 
	$\sin(\tfrac{1}{6}+\tfrac{8}{3f})\pi
	\sin(\tfrac{1}{3}-\tfrac{2}{3f})\pi
	=\sin(\tfrac{1}{6}+\tfrac{2}{3f})\pi
	\sin(\tfrac{2}{3}-\tfrac{16}{3f})\pi$. 
\item $\beta^3\gamma\delta$:
	$\alpha=(\tfrac{1}{2}+\tfrac{6}{f})\pi,\;
	\beta=(\tfrac{1}{4}+\tfrac{1}{f})\pi,\;
	\gamma=(\tfrac{3}{4}-\tfrac{3}{f})\pi,\;
	\delta=\tfrac{1}{2}\pi$.
	\newline 
	$\sin(\tfrac{1}{4}+\tfrac{3}{f})\pi
	\sin(\tfrac{3}{8}-\tfrac{1}{2f})\pi
	=\sin(\tfrac{1}{8}+\tfrac{1}{2f})\pi
	\sin(\tfrac{1}{2}-\tfrac{6}{f})\pi$. 
\item $\beta^4\gamma\delta$:
	$\alpha=(\tfrac{3}{5}+\tfrac{32}{5f})\pi,\;
	\beta=(\tfrac{1}{5}+\tfrac{4}{5f})\pi,\;
	\gamma=(\tfrac{7}{10}-\tfrac{16}{5f})\pi,\;
	\delta=\tfrac{1}{2}\pi$.
	\newline 
	$\sin(\tfrac{3}{10}+\tfrac{16}{5f})\pi
	\sin(\tfrac{2}{5}-\tfrac{2}{5f})\pi
	=\sin(\tfrac{1}{10}+\tfrac{2}{5f})\pi
	\sin(\tfrac{2}{5}-\tfrac{32}{5f})\pi$. 
\item $\beta^3\delta^2$:
	$\alpha=(\tfrac{1}{3}+\tfrac{8}{f})\pi,\;
	\beta=\tfrac{1}{3}\pi,\;
	\gamma=(\tfrac{5}{6}-\tfrac{4}{f})\pi,\;
	\delta=\tfrac{1}{2}\pi$.
	\newline 
	$\sin(\tfrac{1}{6}+\tfrac{4}{f})\pi
	\sin\tfrac{1}{3}\pi
	=\sin\tfrac{1}{6}\pi
	\sin(\tfrac{2}{3}-\tfrac{8}{f})\pi$. 
\item $\beta^4\delta^2$:
	$\alpha=(\tfrac{1}{2}+\tfrac{8}{f})\pi,\;
	\beta=\tfrac{1}{4}\pi,\;
	\gamma=(\tfrac{3}{4}-\tfrac{4}{f})\pi,\;
	\delta=\tfrac{1}{2}\pi$.
	\newline 
	$\sin(\tfrac{1}{4}+\tfrac{4}{f})\pi
	\sin\tfrac{3}{8}\pi
	=\sin\tfrac{1}{8}\pi
	\sin(\tfrac{1}{2}-\tfrac{8}{f})\pi$. 
\item $\beta^5\delta^2$:
	$\alpha=(\tfrac{3}{5}+\tfrac{8}{f})\pi,\;
	\beta=\tfrac{1}{5}\pi,\;
	\gamma=(\tfrac{7}{10}-\tfrac{4}{f})\pi,\;
	\delta=\tfrac{1}{2}\pi$.
	\newline 
	$\sin(\tfrac{3}{10}+\tfrac{4}{f})\pi
	\sin\tfrac{2}{5}\pi
	=\sin\tfrac{1}{10}\pi
	\sin(\tfrac{2}{5}-\tfrac{8}{f})\pi$. 
\end{itemize}
The equations have no solution for even $f\ge 16$. 

\subsubsection*{Case. $\alpha\gamma^2$ and one of $\beta^2\gamma\delta,\beta^3\delta^2,\beta\delta^4,\gamma\delta^3$ are vertices}

The only degree $3$ vertex $\alpha\gamma^2$ implies $\gamma$ appears twice at every degree $3$ vertex. Then by Lemma \ref{notatdeg4}, we know there is a degree $4$ vertex without $\gamma$. These are $\alpha^4,\alpha^3\beta,\alpha^2\beta^2,\alpha\beta^3,\beta^4,\alpha^2\delta^2,\beta^2\delta^2,\alpha\beta\delta^2,\delta^4$. We have already discussed tilings such that $\alpha\gamma^2$ is the only degree $3$ vertex, and one of $\alpha\beta^3,\beta^4,\alpha\beta\delta^2,\beta^2\delta^2,\delta^4$ is a vertex. The angle sum of one of $\alpha^4,\alpha^3\beta,\alpha^2\beta^2,\alpha^2\delta^2$, and the angle sum of one of $\beta^2\gamma\delta,\beta^3\delta^2,\beta\delta^4,\gamma\delta^3$, and the angle sum of $\alpha\gamma^2$, imply the following, including \eqref{coolsaet_eq1}.
\begin{itemize}
\item $\alpha\gamma^2,\beta^2\gamma\delta,\alpha^4$:
	$\alpha=\tfrac{1}{2}\pi,\;
	\beta=(\tfrac{1}{2}-\tfrac{4}{f})\pi,\;
	\gamma=\tfrac{3}{4}\pi,\;
	\delta=(\tfrac{1}{4}+\tfrac{8}{f})\pi$.
	\newline 
	$\sin\tfrac{1}{4}\pi
	\sin\tfrac{10}{f}\pi
	=\sin(\tfrac{1}{4}-\tfrac{2}{f})\pi
	\sin\tfrac{1}{2}\pi$. 
\item $\alpha\gamma^2,\beta^2\gamma\delta,\alpha^3\beta$:
	$\alpha=(\tfrac{1}{2}+\tfrac{1}{f})\pi,\;
	\beta=(\tfrac{1}{2}-\tfrac{3}{f})\pi,\;
	\gamma=(\tfrac{3}{4}-\tfrac{1}{2f})\pi,\;
	\delta=(\tfrac{1}{4}+\tfrac{13}{2f})\pi$.
	\newline 
	$\sin(\tfrac{1}{4}+\tfrac{1}{2f})\pi
	\sin\tfrac{8}{f}\pi
	=\sin(\tfrac{1}{4}-\tfrac{3}{2f})\pi
	\sin(\tfrac{1}{2}-\tfrac{1}{f})\pi$. 
\item $\alpha\gamma^2,\beta^2\gamma\delta,\alpha^2\beta^2$:
	$\alpha=(\tfrac{1}{2}+\tfrac{2}{f})\pi,\;
	\beta=(\tfrac{1}{2}-\tfrac{2}{f})\pi,\;
	\gamma=(\tfrac{3}{4}-\tfrac{1}{f})\pi,\;
	\delta=(\tfrac{1}{4}+\tfrac{5}{f})\pi$.
	\newline 
	$\sin(\tfrac{1}{4}+\tfrac{1}{f})\pi
	\sin\tfrac{6}{f}\pi
	=\sin(\tfrac{1}{4}-\tfrac{1}{f})\pi
	\sin(\tfrac{1}{2}-\tfrac{2}{f})\pi$.
\item $\alpha\gamma^2,\beta^2\gamma\delta,\alpha^2\delta^2$:
	$\alpha=\tfrac{16}{f}\pi,\;
	\beta=\tfrac{12}{f}\pi,\;
	\gamma=(1-\tfrac{8}{f})\pi,\;
	\delta=(1-\tfrac{16}{f})\pi$.
	\newline 
	$\sin\tfrac{8}{f}\pi
	\sin(1-\tfrac{22}{f})\pi
	=\sin\tfrac{6}{f}\pi
	\sin(1-\tfrac{16}{f})\pi$.
\item $\alpha\gamma^2,\beta^3\delta^2,\alpha^4$:
	$\alpha=\tfrac{1}{2}\pi,\;
	\beta=(\tfrac{1}{2}-\tfrac{8}{f})\pi,\;
	\gamma=\tfrac{3}{4}\pi,\;
	\delta=(\tfrac{1}{4}+\tfrac{12}{f})\pi$.
	\newline 
	$\sin\tfrac{1}{4}\pi
	\sin\tfrac{16}{f}\pi
	=\sin(\tfrac{1}{4}-\tfrac{4}{f})\pi
	\sin\tfrac{1}{2}\pi$.
\item $\alpha\gamma^2,\beta^3\delta^2,\alpha^3\beta$:
	$\alpha=(\tfrac{1}{2}+\tfrac{2}{f})\pi,\;
	\beta=(\tfrac{1}{2}-\tfrac{6}{f})\pi,\;
	\gamma=(\tfrac{3}{4}-\tfrac{1}{f})\pi,\;
	\delta=(\tfrac{1}{4}+\tfrac{9}{f})\pi$.
	\newline 
	$\sin(\tfrac{1}{4}+\tfrac{1}{f})\pi
	\sin\tfrac{12}{f}\pi
	=\sin(\tfrac{1}{4}-\tfrac{3}{f})\pi
	\sin(\tfrac{1}{2}-\tfrac{2}{f})\pi$.
\item $\alpha\gamma^2,\beta^3\delta^2,\alpha^2\beta^2$:
	$\alpha=(\tfrac{1}{2}+\tfrac{4}{f})\pi,\;
	\beta=(\tfrac{1}{2}-\tfrac{4}{f})\pi,\;
	\gamma=(\tfrac{3}{4}-\tfrac{2}{f})\pi,\;
	\delta=(\tfrac{1}{4}+\tfrac{6}{f})\pi$.
	\newline 
	$\sin(\tfrac{1}{4}+\tfrac{2}{f})\pi
	\sin\tfrac{8}{f}\pi
	=\sin(\tfrac{1}{4}-\tfrac{2}{f})\pi
	\sin(\tfrac{1}{2}-\tfrac{4}{f})\pi$.
\item $\alpha\gamma^2,\beta^3\delta^2,\alpha^2\delta^2$:
	$\alpha=\tfrac{24}{f}\pi,\;
	\beta=\tfrac{16}{f}\pi,\;
	\gamma=(1-\tfrac{12}{f})\pi,\;
	\delta=(1-\tfrac{24}{f})\pi$.
	\newline 
	$\sin\tfrac{12}{f}\pi
	\sin(1-\tfrac{32}{f})\pi
	=\sin\tfrac{8}{f}\pi
	\sin(1-\tfrac{24}{f})\pi$.
\item $\alpha\gamma^2,\beta\delta^4,\alpha^4$:
	$\alpha=\tfrac{1}{2}\pi,\;
	\beta=(\tfrac{1}{3}+\tfrac{16}{3f})\pi,\;
	\gamma=\tfrac{3}{4}\pi,\;
	\delta=(\tfrac{5}{12}-\tfrac{4}{3f})\pi$.
	\newline 
	$\sin\tfrac{1}{4}\pi
	\sin(\tfrac{1}{4}-\tfrac{4}{f})\pi
	=\sin(\tfrac{1}{6}+\tfrac{8}{3f})\pi
	\sin\tfrac{1}{2}\pi$.
\item $\alpha\gamma^2,\beta\delta^4,\alpha^3\beta$:
	$\alpha=(\tfrac{4}{7}-\tfrac{16}{7f})\pi,\;
	\beta=(\tfrac{2}{7}+\tfrac{48}{7f})\pi,\;
	\gamma=(\tfrac{5}{7}+\tfrac{8}{7f})\pi,\;
	\delta=(\tfrac{3}{7}-\tfrac{12}{7f})\pi$.
	\newline 
	$\sin(\tfrac{2}{7}-\tfrac{8}{7f})\pi
	\sin(\tfrac{2}{7}-\tfrac{36}{7f})\pi
	=\sin(\tfrac{1}{7}+\tfrac{24}{7f})\pi
	\sin(\tfrac{3}{7}+\tfrac{16}{7f})\pi$.
\item $\alpha\gamma^2,\beta\delta^4,\alpha^2\beta^2$:
	$\alpha=(1-\tfrac{16}{f})\pi,\;
	\beta=\tfrac{16}{f}\pi,\;
	\gamma=(\tfrac{1}{2}+\tfrac{8}{f})\pi,\;
	\delta=(\tfrac{1}{2}-\tfrac{4}{f})\pi$.
	\newline 
	$\sin(\tfrac{1}{2}-\tfrac{8}{f})\pi
	\sin(\tfrac{1}{2}-\tfrac{12}{f})\pi
	=\sin\tfrac{8}{f}\pi
	\sin\tfrac{16}{f}\pi$.
\item $\alpha\gamma^2,\beta\delta^4,\alpha^2\delta^2$:
	$\alpha=(\tfrac{4}{7}+\tfrac{8}{7f})\pi,\;
	\beta=(\tfrac{2}{7}+\tfrac{32}{7f})\pi,\;
	\gamma=(\tfrac{5}{7}-\tfrac{4}{7f})\pi,\;
	\delta=(\tfrac{3}{7}-\tfrac{8}{7f})\pi$.
	\newline 
	$\sin(\tfrac{2}{7}+\tfrac{4}{7f})\pi
	\sin(\tfrac{2}{7}-\tfrac{24}{7f})\pi
	=\sin(\tfrac{1}{7}+\tfrac{16}{7f})\pi
	\sin(\tfrac{3}{7}-\tfrac{8}{7f})\pi$.
\item $\alpha\gamma^2,\gamma\delta^3,\alpha^4$:
	$\alpha=\tfrac{1}{2}\pi,\;
	\beta=(\tfrac{1}{3}+\tfrac{4}{f})\pi,\;
	\gamma=\tfrac{3}{4}\pi,\;
	\delta=\tfrac{5}{12}\pi$.
	\newline 
	$\sin\tfrac{1}{4}\pi
	\sin(\tfrac{1}{4}-\tfrac{2}{f})\pi
	=\sin(\tfrac{1}{6}+\tfrac{2}{f})\pi
	\sin\tfrac{1}{2}\pi$.
\item $\alpha\gamma^2,\gamma\delta^3,\alpha^3\beta$:
	$\alpha=(\tfrac{4}{7}-\tfrac{12}{7f})\pi,\;
	\beta=(\tfrac{2}{7}+\tfrac{36}{7f})\pi,\;
	\gamma=(\tfrac{5}{7}+\tfrac{6}{7f})\pi,\;
	\delta=(\tfrac{3}{7}-\tfrac{2}{7f})\pi$.
	\newline 
	$\sin(\tfrac{2}{7}-\tfrac{6}{7f})\pi
	\sin(\tfrac{2}{7}-\tfrac{20}{7f})\pi
	=\sin(\tfrac{1}{7}+\tfrac{18}{7f})\pi
	\sin(\tfrac{3}{7}+\tfrac{12}{7f})\pi$.
\item $\alpha\gamma^2,\gamma\delta^3,\alpha^2\beta^2$:
	$\alpha=(1-\tfrac{12}{f})\pi,\;
	\beta=\tfrac{12}{f}\pi,\;
	\gamma=(\tfrac{1}{2}+\tfrac{6}{f})\pi,\;
	\delta=(\tfrac{1}{2}-\tfrac{2}{f})\pi$.
	\newline 
	$\sin(\tfrac{1}{2}-\tfrac{6}{f})\pi
	\sin(\tfrac{1}{2}-\tfrac{8}{f})\pi
	=\sin\tfrac{6}{f}\pi
	\sin\tfrac{12}{f}\pi$.
\item $\alpha\gamma^2,\gamma\delta^3,\alpha^2\delta^2$:
	$\alpha=\tfrac{4}{7}\pi,\;
	\beta=(\tfrac{2}{7}+\tfrac{4}{f})\pi,\;
	\gamma=\tfrac{5}{7}\pi,\;
	\delta=\tfrac{3}{7}\pi$.
	\newline 
	$\sin\tfrac{2}{7}\pi
	\sin(\tfrac{2}{7}-\tfrac{2}{f})\pi
	=\sin(\tfrac{1}{7}+\tfrac{2}{f})\pi
	\sin\tfrac{3}{7}\pi$.
\end{itemize}
The solutions for even $f\ge 16$ are $f=16$ for $\beta^3\delta^2,\alpha^4$, and for $\beta\delta^4,\alpha^2\beta^2$. However, we get $\beta=0$ and $\alpha=0$ respectively. Therefore there is no tiling.
\end{proof}

\begin{proposition}\label{add}
Tilings of the sphere by congruent almost equilateral quadrilaterals, such that $\alpha\delta^2$ is the only degree $3$ vertex, are $S_{16\square}1$ and $S_{36\square}6$.
\end{proposition}

The tiling $S_{16\square}1$ obtained in the proof is the four timezone version of the earth map tiling in Figure \ref{add-bbbA}. After exchanging $(\alpha,\delta)$ with $(\beta,\gamma)$, the tiling is $S_{16\square}1$ in Figure \ref{sporatic_tiling}.

The tiling $S_{36\square}6$ is obtained in Figure \ref{add-abbbB}. The picture for $S_{36\square}6$ in Figure \ref{sporatic_tiling} is another way of presenting the same tiling.

\begin{proof}
By Lemma \ref{count-att}, we know $f\ge 16$. By $\alpha\delta^2$ and the balance lemma, we know $\gamma^2\cdots$ is a vertex. Therefore $\gamma,\delta<\pi$. By $\alpha\delta^2$, we get $\alpha+2\beta+2\gamma=2(\alpha+\beta+\gamma+\delta)-(\alpha+2\delta)=(2+\frac{8}{f})\pi>2\pi$. 

By the only degree $3$ vertex $\alpha\delta^2$, we know $\beta,\gamma$ do not appear at degree $3$ vertices. By Lemma \ref{deg3miss}, we know one of $\alpha\beta^3,\beta^4,\alpha\beta\gamma^2,\beta^2\gamma^2,\beta^2\gamma\delta,\gamma^4,\gamma^3\delta,\beta^5,\beta^3\gamma^2,\beta\gamma^4$ is a vertex. 

\subsubsection*{Case. $\alpha\delta^2$ and one of $\alpha\beta^3,\beta^4,\beta^5$ are vertices}

The angle sum of one of $\alpha\beta^3,\beta^4,\beta^5$, and the angle sum of $\alpha\delta^2$, imply
\begin{align*}
\alpha\delta^2,\alpha\beta^3 &\colon
	\alpha=2\pi-2\delta,\;
	\beta=\tfrac{2}{3}\delta,\;
	\gamma=\tfrac{4}{f}\pi+\tfrac{1}{3}\delta. \\
\alpha\delta^2,\beta^4 &\colon
	\alpha=2\pi-2\delta,\;
	\beta=\tfrac{1}{2}\pi,\;
	\gamma=(\tfrac{4}{f}-\tfrac{1}{2})\pi+\delta. \\
\alpha\delta^2,\beta^5 &\colon
	\alpha=2\pi-2\delta,\;
	\beta=\tfrac{2}{5}\pi,\;
	\gamma=(\tfrac{4}{f}-\tfrac{2}{5})\pi+\delta.
\end{align*}

For $\alpha\beta^3$, we have $\beta<\delta$. By Lemma \ref{geometry5}, this implies $\alpha<\gamma$. This means $\delta>(\frac{6}{7}-\frac{12}{7f})\pi$. Therefore $\beta-\gamma=\tfrac{1}{3}\delta-\tfrac{4}{f}\pi>(\frac{2}{7}-\frac{32}{7f})\pi\ge 0$. We get $\alpha<\gamma<\beta<\delta$. 

For $\beta^4,\beta^5$, we have $\gamma<\delta$. By Lemma \ref{geometry1}, this implies $\alpha<\beta$. This means $\delta>\frac{3}{4}\pi>\beta$ for $\beta^4$ and $\delta>\frac{4}{5}\pi>\beta$ for $\beta^5$. Then by Lemma \ref{geometry5}, we get $\alpha<\gamma$. 

We conclude $\alpha<\beta,\gamma$ and $\beta,\gamma<\delta$ in all three cases. Then by $\alpha\delta^2$, this implies $\delta^2\cdots=\alpha\delta^2$, and further implies $\delta\thin\delta\cdots$ is not a vertex. 

By $R(\beta\gamma\delta)<\alpha<\beta,\gamma,\delta$, and the only degree $3$ vertex $\alpha\delta^2$, we know $\beta\gamma\delta\cdots$ is not a vertex. Then by $\delta^2\cdots=\alpha\delta^2$, we know $\beta\delta\cdots$ is not a vertex. 

By no $\beta\delta\cdots,\delta\thin\delta\cdots$, we know the AADs of $\thin\alpha\thin\alpha\thin$ and $\thin\alpha\thin\gamma\thick$ are $\thin^{\delta}\alpha^{\beta}\thin^{\beta}\alpha^{\delta}\thin$  and $\thin^{\delta}\alpha^{\beta}\thin^{\beta}\gamma^{\delta}\thick$. This implies no consecutive $\alpha\alpha\alpha,\alpha\alpha\gamma,\gamma\alpha\gamma$. 

Applying the second part of Lemma \ref{count-att} to $\gamma$, we know there is a degree $4$ vertex $\gamma\cdots$, or a degree $5$ vertex $\gamma^3\cdots$, or a degree $6$ vertex $\gamma^5\cdots$. By $\delta^2\cdots=\alpha\delta^2$, and no $\beta\delta\cdots$, and $\alpha<\beta$, the vertex is \newline $\alpha^2\gamma^2,\alpha^2\gamma\delta,\alpha\beta\gamma^2,\beta^2\gamma^2,\gamma^4,\gamma^3\delta,\alpha\gamma^4,\beta\gamma^4,\alpha\gamma^3\delta,\gamma^6,\gamma^5\delta$. Then by no consecutive $\alpha\alpha\gamma,\gamma\alpha\gamma$, we know $\alpha^2\gamma^2,\alpha^2\gamma\delta,\alpha\gamma^4$ are not vertices. 

For $\alpha\beta^3$, by $\alpha<\gamma<\beta$, we further know $\alpha\beta\gamma^2$ is not a vertex, and $\beta>\frac{1}{2}\pi$. Then $\delta>\frac{3}{4}\pi$, and $5\gamma+\delta=\tfrac{20}{f}\pi+\tfrac{8}{3}\delta>2\pi$. Therefore $\gamma^5\delta$ is not a vertex. Moreover, if $\alpha\gamma^3\delta$ is a vertex, then adding the angle sums of $\alpha\delta^2,\alpha\beta^3,\alpha\gamma^3\delta$ together gives $3(\alpha+\beta+\gamma+\delta)=6\pi$, a contradiction. It remains to consider the combinations of $\alpha\beta^3$ with one of $\beta^2\gamma^2,\gamma^4,\gamma^3\delta,\beta\gamma^4,\gamma^6$.

For $\beta^4$, we have $\alpha+3\gamma+\delta=(\tfrac{1}{2}+\tfrac{12}{f})\pi+2\delta>(2+\tfrac{12}{f})\pi>2\pi$ and $5\gamma+\delta=(\tfrac{20}{f}-\tfrac{5}{2})\pi+6\delta>(2+\tfrac{20}{f})\pi>2\pi$. Therefore $\alpha\gamma^3\delta,\gamma^5\delta$ are not vertices. It remains to consider the combinations of $\beta^5$ with one of $\alpha\beta\gamma^2,\beta^2\gamma^2,\gamma^4,\gamma^3\delta,\beta\gamma^4,\gamma^6$.

For $\beta^5$, we have $\gamma>(\tfrac{2}{5}+\tfrac{4}{f})\pi>\beta$. By $\beta^5$, this implies $\beta\gamma^4,\gamma^6,\gamma^5\delta$ are not vertices. We also have $3\gamma+\delta=(\tfrac{12}{f}-\tfrac{6}{5})\pi+4\delta>(2+\tfrac{12}{f})\pi>2\pi$. Therefore $\gamma^3\delta,\alpha\gamma^3\delta$ are not vertices. It remains to consider the combinations of $\beta^5$ with one of $\alpha\beta\gamma^2,\beta^2\gamma^2,\gamma^4$.

The angle sums of one of the combination pairs listed above, and the angle sum of $\alpha\delta^2$, imply the following, including \eqref{coolsaet_eq1}.
\begin{itemize}
\item $\alpha\delta^2,\alpha\beta^3,\beta^2\gamma^2$:
	$\alpha=\tfrac{8}{f}\pi,\;
	\beta=(\tfrac{2}{3}-\tfrac{8}{3f})\pi,\;
	\gamma=(\tfrac{1}{3}+\tfrac{8}{3f})\pi,\;
	\delta=(1-\tfrac{4}{f})\pi$.
	\newline 
	$\sin\tfrac{4}{f}\pi
	\sin(\tfrac{2}{3}-\tfrac{8}{3f})\pi
	=\sin(\tfrac{1}{3}-\tfrac{4}{3f})\pi
	\sin(\tfrac{1}{3}-\tfrac{4}{3f})\pi$. 
\item $\alpha\delta^2,\alpha\beta^3,\gamma^4$:
	$\alpha=(\tfrac{24}{f}-1)\pi,\;
	\beta=(1-\tfrac{8}{f})\pi,\;
	\gamma=\tfrac{1}{2}\pi,\;
	\delta=(\tfrac{3}{2}-\tfrac{12}{f})\pi$.
	\newline 
	$\sin(\tfrac{12}{f}-\tfrac{1}{2})\pi
	\sin(1-\tfrac{8}{f})\pi
	=\sin(\tfrac{1}{2}-\tfrac{4}{f})\pi
	\sin(1-\tfrac{12}{f})\pi$. 
\item $\alpha\delta^2,\alpha\beta^3,\gamma^3\delta$:
	$\alpha=\tfrac{12}{f}\pi,\;
	\beta=(\tfrac{2}{3}-\tfrac{4}{f})\pi,\;
	\gamma=(\tfrac{1}{3}+\tfrac{2}{f})\pi,\;
	\delta=(1-\tfrac{6}{f})\pi$.
	\newline 
	$\sin\tfrac{6}{f}\pi
	\sin(\tfrac{2}{3}-\tfrac{4}{f})\pi
	=\sin(\tfrac{1}{3}-\tfrac{2}{f})\pi
	\sin(\tfrac{1}{3}-\tfrac{4}{f})\pi$.  
\item $\alpha\delta^2,\alpha\beta^3,\beta\gamma^4$:
	$\alpha=\tfrac{16}{f}\pi,\;
	\beta=(\tfrac{2}{3}-\tfrac{16}{3f})\pi,\;
	\gamma=(\tfrac{1}{3}+\tfrac{4}{3f})\pi,\;
	\delta=(1-\tfrac{8}{f})\pi$.
	\newline 
	$\sin\tfrac{8}{f}\pi
	\sin(\tfrac{2}{3}-\tfrac{16}{3f})\pi
	=\sin(\tfrac{1}{3}-\tfrac{8}{3f})\pi
	\sin(\tfrac{1}{3}-\tfrac{20}{3f})\pi$. 
\item $\alpha\delta^2,\alpha\beta^3,\gamma^6$:
	$\alpha=\tfrac{24}{f}\pi,\;
	\beta=(\tfrac{2}{3}-\tfrac{8}{f})\pi,\;
	\gamma=\tfrac{1}{3}\pi,\;
	\delta=(1-\tfrac{12}{f})\pi$.
	\newline 
	$\sin\tfrac{12}{f}\pi
	\sin(\tfrac{2}{3}-\tfrac{8}{f})\pi
	=\sin(\tfrac{1}{3}-\tfrac{4}{f})\pi
	\sin(\tfrac{1}{3}-\tfrac{12}{f})\pi$.
\item $\alpha\delta^2,\beta^4,\alpha\beta\gamma^2$:
	$\alpha=2\pi-2\delta,\;
	\beta=\tfrac{1}{2}\pi,\;
	\gamma=\delta-\tfrac{1}{4}\pi,\; f=16$.
	\newline 
	$\sin(\pi-\delta)
	\sin(\delta-\tfrac{1}{4}\pi)
	=\sin\tfrac{1}{4}\pi
	\sin(2\delta-\tfrac{5}{4}\pi)$. 
\item $\alpha\delta^2,\beta^4,\beta^2\gamma^2/\gamma^4$:
	$\alpha=\tfrac{8}{f}\pi,\;
	\beta=\tfrac{1}{2}\pi,\;
	\gamma=\tfrac{1}{2}\pi,\;
	\delta=(1-\tfrac{4}{f})\pi$.
	\newline 
	$\sin\tfrac{4}{f}\pi
	\sin(\tfrac{3}{4}-\tfrac{4}{f})\pi
	=\sin\tfrac{1}{4}\pi
	\sin(\tfrac{1}{2}-\tfrac{4}{f})\pi$. 
\item $\alpha\delta^2,\beta^4,\gamma^3\delta$:
	$\alpha=(\tfrac{1}{4}+\tfrac{6}{f})\pi,\;
	\beta=\tfrac{1}{2}\pi,\;
	\gamma=(\tfrac{3}{8}+\tfrac{1}{f})\pi,\;
	\delta=(\tfrac{7}{8}-\tfrac{3}{f})\pi$.
	\newline 
	$\sin(\tfrac{1}{8}+\tfrac{3}{f})\pi
	\sin(\tfrac{5}{8}-\tfrac{3}{f})\pi
	=\sin\tfrac{1}{4}\pi
	\sin(\tfrac{1}{4}-\tfrac{2}{f})\pi$. 
\item $\alpha\delta^2,\beta^4,\beta\gamma^4$:
	$\alpha=(\tfrac{1}{4}+\tfrac{8}{f})\pi,\;
	\beta=\tfrac{1}{2}\pi,\;
	\gamma=\tfrac{3}{8}\pi,\;
	\delta=(\tfrac{7}{8}-\tfrac{4}{f})\pi$.
	\newline 
	$\sin(\tfrac{1}{8}+\tfrac{4}{f})\pi
	\sin(\tfrac{5}{8}-\tfrac{4}{f})\pi
	=\sin\tfrac{1}{4}\pi
	\sin(\tfrac{1}{4}-\tfrac{4}{f})\pi$. 
\item $\alpha\delta^2,\beta^4,\gamma^6$:
	$\alpha=(\tfrac{1}{3}+\tfrac{8}{f})\pi,\;
	\beta=\tfrac{1}{2}\pi,\;
	\gamma=\tfrac{1}{3}\pi,\;
	\delta=(\tfrac{5}{6}-\tfrac{4}{f})\pi$.
	\newline 
	$\sin(\tfrac{1}{6}+\tfrac{4}{f})\pi
	\sin(\tfrac{7}{12}-\tfrac{4}{f})\pi
	=\sin\tfrac{1}{4}\pi
	\sin(\tfrac{1}{6}-\tfrac{4}{f})\pi$. 
\item $\alpha\delta^2,\beta^5,\alpha\beta\gamma^2$:
	$\alpha=2\pi-2\delta,\;
	\beta=\tfrac{2}{5}\pi,\;
	\gamma=\delta-\tfrac{1}{5}\pi,\; f=20$.
	\newline 
	$\sin(\pi-\delta)
	\sin(\delta-\tfrac{1}{5}\pi)
	=\sin\tfrac{1}{5}\pi
	\sin(2\delta-\tfrac{6}{5}\pi)$. 
\item $\alpha\delta^2,\beta^5,\beta^2\gamma^2$:
	$\alpha=\tfrac{8}{f}\pi,\;
	\beta=\tfrac{2}{5}\pi,\;
	\gamma=\tfrac{3}{5}\pi,\;
	\delta=(1-\tfrac{4}{f})\pi$.
	\newline 
	$\sin\tfrac{4}{f}\pi
	\sin(\tfrac{4}{5}-\tfrac{4}{f})\pi
	=\sin\tfrac{1}{5}\pi
	\sin(\tfrac{3}{5}-\tfrac{4}{f})\pi$. 
\item $\alpha\delta^2,\beta^5,\gamma^4$:
	$\alpha=(\tfrac{1}{5}+\tfrac{8}{f})\pi,\;
	\beta=\tfrac{2}{5}\pi,\;
	\gamma=\tfrac{1}{2}\pi,\;
	\delta=(\tfrac{9}{10}-\tfrac{4}{f})\pi$.
	\newline 
	$\sin(\tfrac{1}{10}+\tfrac{4}{f})\pi
	\sin(\tfrac{7}{10}-\tfrac{4}{f})\pi
	=\sin\tfrac{1}{5}\pi
	\sin(\tfrac{2}{5}-\tfrac{4}{f})\pi$.
\end{itemize}
For $\alpha\delta^2,\beta^4,\alpha\beta\gamma^2$, the solution for $\frac{3}{4}\pi<\delta<\pi$ is $\delta=0.7898\pi$. For $\alpha\delta^2,\beta^5,\alpha\beta\gamma^2$, the equation has no solution satisfying $\frac{4}{5}\pi<\delta<\pi$. For all the remaining cases, the only solution for even $f\ge 16$ are $f=36$ for $\alpha\delta^2,\alpha\beta^3,\gamma^3\delta$, and $f=20$ for $\alpha\delta^2,\beta^5,\beta^2\gamma^2$. However, $f=20$ implies $\delta=\frac{4}{5}\pi$, contradicting $\delta>\frac{4}{5}\pi$.

For $\alpha\delta^2,\beta^4,\alpha\beta\gamma^2$ and $f=16$, we use the approximate angles values to get 
\begin{equation}\label{add-abbbavc1}
\text{AVC}=\{\alpha\delta^2,\beta^4,\alpha\beta\gamma^2\}.
\end{equation}
This is the AVC \eqref{add-bbbAVC}, with $\beta^3$ replaced by $\beta^4$. The same argument gives the earth map tiling with four timezones, similar to Figure \ref{add-bbbA}. In the earlier discussion, we also argued the geometrical existence of the tiling. After exchanging $(\alpha,\delta)$ with $(\beta,\gamma)$, the tiling is $S_{16\square}1$.

For $\alpha\delta^2,\alpha\beta^3,\gamma^3\delta$ and $f=36$, we calculate the angles and get the list of vertices
\begin{align}
\alpha\delta^2,\alpha\beta^3,\gamma^3\delta &\colon
	\alpha=\tfrac{1}{3}\pi,\;
	\beta=\tfrac{5}{9}\pi,\;
	\gamma=\tfrac{7}{18}\pi, \;
	\delta=\tfrac{5}{6}\pi. \label{add-abbbavc2} \\
&	\text{AVC}=
	\{\alpha\delta^2,\alpha\beta^3,\gamma^3\delta,\alpha^2\beta\gamma^2\}. \nonumber 
\end{align}
We note that $\alpha^6$ is not included, by no consecutive $\alpha\alpha\alpha$. 

If $\alpha^2\beta\gamma^2$ is not a vertex, then we get
\[
f=\#\alpha
=\#\alpha\delta^2+\#\alpha\beta^3
\le \tfrac{1}{2}\#\delta+\tfrac{1}{3}\#\beta
=\tfrac{1}{2}f+\tfrac{1}{3}f
<f,
\]
a contradiction. Therefore $\alpha^2\beta\gamma^2$ is a vertex. By no consecutive $\alpha\alpha\gamma$, the vertex is $\thick\gamma\thin\alpha\thin\beta\thin\alpha\thin\gamma\thick$. By $\alpha^2\cdots=\alpha^2\beta\gamma^2$, this implies $\alpha\thin\alpha\cdots$ is not a vertex. 

By the AAD $\thin^{\delta}\alpha^{\beta}\thin^{\beta}\gamma^{\delta}\thick$ of $\thin\alpha\thin\gamma\thick$, we know the AAD of $\alpha^2\beta\gamma^2$ is $\thick^{\delta}\gamma^{\beta}\thin^{\beta}\alpha^{\delta}\thin^{\alpha}\beta^{\gamma}\thin^{\delta}\alpha^{\beta}\thin^{\beta}\gamma^{\delta}\thick$. This determines $T_1,T_2,T_3,T_4,T_5$ in Figure \ref{add-abbbA}. Then $\alpha_3\delta_4\cdots=\alpha\delta^2$ determines $T_6$. Then $\alpha_6\delta_3\cdots=\alpha\delta^2$ determines $T_7$. Then $\gamma_3\gamma_7\delta_2\cdots=\gamma^3\delta$ determines $T_8$. Then $\gamma_2\delta_8\cdots=\gamma^3\delta$ determines $T_9,T_{10}$.

We have $\gamma_4\gamma_6\cdots=\alpha^2\beta\gamma^2,\gamma^3\delta$. If the vertex is $\gamma^3\delta$, then we get $T_{11},T_{12}$ in one of the two pictures in Figure \ref{add-abbbA}. In the first picture, $\alpha_7\alpha_{12}\beta_6\cdots=\alpha^2\beta\gamma^2$ determines $T_{13},T_{14}$. Then $\beta_7\beta_8\beta_{14}\cdots=\alpha\beta^3$ and no $\alpha\beta\delta\cdots$ determine $T_{15}$. Then $\alpha_8\beta_9\beta_{15}\cdots=\alpha\beta^3$ and no $\alpha\thin\alpha\cdots$ determine $T_{16}$. Then we find $\alpha_{16}\cdots=\gamma_{16}\cdots=\alpha\gamma\cdots=\alpha^2\beta\gamma^2$. In the second picture, $\gamma_{11}\delta_{12}\cdots=\gamma^3\delta$ determines $T_{13}$. Then $\alpha_{11}\beta_4\beta_5\cdots=\beta_{11}\beta_{13}\cdots=\alpha\beta^3$ determine $T_{14}$. We find $\alpha_5\cdots=\gamma_5\cdots=\alpha\gamma\cdots=\alpha^2\beta\gamma^2$.

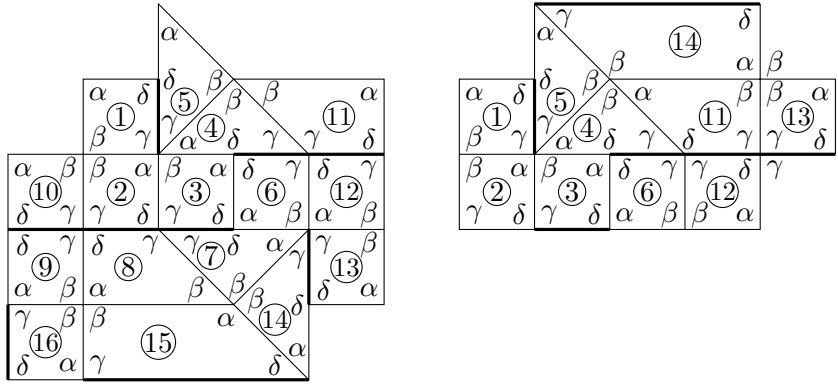
\begin{figure}[htp]
\centering
\begin{tikzpicture}


\draw
	(3,1) -- (-2,1) -- (-2,-2) -- (2,-2) -- (2,1) -- (0,3) -- (0,0) -- (2,-2)
	(-2,-1) -- (1,-1) -- (2,0)
	(0,2) -- (-1,2) -- (-1,-2)
	(3,0) -- (1,0) -- (1,1)
	(2,-1) -- (3,-1) -- (3,2) -- (1,2) -- (0,1)
	;

\draw[line width=1.2]
	(-2,0) -- (1,0)
	(0,1) -- (0,2)
	(1,1) -- (3,1)
	(2,0) -- (2,-1)
	(-1,-2) -- (2,-2)
	(-2,-2) -- (-2,-1);

\node at (0.8,0.8) {\small $\alpha$};
\node at (0.2,0.8) {\small $\beta$};
\node at (0.2,0.2) {\small $\gamma$};
\node at (0.8,0.2) {\small $\delta$};

\node at (-0.2,0.8) {\small $\alpha$};
\node at (-0.8,0.8) {\small $\beta$};
\node at (-0.8,0.2) {\small $\gamma$};
\node at (-0.2,0.2) {\small $\delta$};

\node at (-1.8,0.8) {\small $\alpha$};
\node at (-1.2,0.8) {\small $\beta$};
\node at (-1.2,0.2) {\small $\gamma$};
\node at (-1.8,0.2) {\small $\delta$};

\node at (-1.8,-0.8) {\small $\alpha$};
\node at (-1.2,-0.8) {\small $\beta$};
\node at (-1.2,-0.2) {\small $\gamma$};
\node at (-1.8,-0.2) {\small $\delta$};

\node at (-0.8,-0.8) {\small $\alpha$};
\node at (0.5,-0.8) {\small $\beta$};
\node at (-0.1,-0.2) {\small $\gamma$};
\node at (-0.8,-0.2) {\small $\delta$};

\node at (-0.8,1.8) {\small $\alpha$};
\node at (-0.8,1.2) {\small $\beta$};
\node at (-0.2,1.2) {\small $\gamma$};
\node at (-0.2,1.8) {\small $\delta$};

\node at (1.2,0.2) {\small $\alpha$};
\node at (1.8,0.2) {\small $\beta$};
\node at (1.8,0.8) {\small $\gamma$};
\node at (1.2,0.8) {\small $\delta$};

\node at (1.55,-0.2) {\small $\alpha$};
\node at (1.05,-0.75) {\small $\beta$};
\node at (0.45,-0.2) {\small $\gamma$};
\node at (1,-0.2) {\small $\delta$};

\node at (0.4,1.15) {\small $\alpha$};
\node at (1,1.7) {\small $\beta$};
\node at (1.5,1.2) {\small $\gamma$};
\node at (1,1.2) {\small $\delta$};

\node at (0.15,2.6) {\small $\alpha$};
\node at (0.75,1.95) {\small $\beta$};
\node at (0.15,1.4) {\small $\gamma$};
\node at (0.15,2) {\small $\delta$};

\node[draw,shape=circle, inner sep=0.5] at (-0.5,1.5) {\small $1$};
\node[draw,shape=circle, inner sep=0.5] at (-0.5,0.5) {\small $2$};
\node[draw,shape=circle, inner sep=0.5] at (0.5,0.5) {\small $3$};
\node[draw,shape=circle, inner sep=0.5] at (0.7,1.35) {\small $4$};
\node[draw,shape=circle, inner sep=0.5] at (0.35,1.7) {\small $5$};
\node[draw,shape=circle, inner sep=0.5] at (1.5,0.5) {\small $6$};
\node[draw,shape=circle, inner sep=0.5] at (0.7,-0.35) {\small $7$};
\node[draw,shape=circle, inner sep=0.5] at (-0.4,-0.5) {\small $8$};
\node[draw,shape=circle, inner sep=0.5] at (-1.5,-0.5) {\small $9$};
\node[draw,shape=circle, inner sep=0] at (-1.5,0.5) {\footnotesize $10$};
\node[draw,shape=circle, inner sep=0] at (2.4,1.5) {\footnotesize $11$};
\node[draw,shape=circle, inner sep=0] at (2.5,0.5) {\footnotesize $12$};

\node at (2.8,1.8) {\small $\alpha$};
\node at (1.5,1.8) {\small $\beta$};
\node at (2.05,1.2) {\small $\gamma$};
\node at (2.8,1.2) {\small $\delta$};

\node at (2.2,0.2) {\small $\alpha$};
\node at (2.8,0.2) {\small $\beta$};
\node at (2.8,0.8) {\small $\gamma$};
\node at (2.2,0.8) {\small $\delta$};

\node at (2.8,-0.8) {\small $\alpha$};
\node at (2.8,-0.2) {\small $\beta$};
\node at (2.2,-0.2) {\small $\gamma$};
\node at (2.2,-0.8) {\small $\delta$};

\node at (1.85,-1.6) {\small $\alpha$};
\node at (1.3,-0.95) {\small $\beta$};
\node at (1.85,-0.4) {\small $\gamma$};
\node at (1.85,-1) {\small $\delta$};

\node at (0.9,-1.2) {\small $\alpha$};
\node at (-0.8,-1.2) {\small $\beta$};
\node at (-0.8,-1.8) {\small $\gamma$};
\node at (1.55,-1.8) {\small $\delta$};

\node at (-1.2,-1.8) {\small $\alpha$};
\node at (-1.2,-1.2) {\small $\beta$};
\node at (-1.8,-1.2) {\small $\gamma$};
\node at (-1.8,-1.8) {\small $\delta$};

\node[draw,shape=circle, inner sep=0] at (2.5,-0.5) {\footnotesize $13$};
\node[draw,shape=circle, inner sep=0] at (1.55,-1.2) {\footnotesize $14$};
\node[draw,shape=circle, inner sep=0] at (0,-1.5) {\small $15$};
\node[draw,shape=circle, inner sep=0] at (-1.5,-1.5) {\footnotesize $16$};

\begin{scope}[xshift=5cm]

\draw
	(0,2) -- (-1,2) -- (-1,0) -- (3,0) -- (3,2) -- (3,3)
	(4,1) -- (4,2) -- (1,2) -- (0,1)
	(-1,1) -- (1,1) -- (1,0)
	(2,0) -- (2,1) -- (0,3) -- (0,0);

\draw[line width=1.2]
	(0,0) -- (1,0)
	(0,1) -- (0,2)
	(1,1) -- (4,1)
	(3,3) -- (0,3);

\node at (-0.8,1.8) {\small $\alpha$};
\node at (-0.8,1.2) {\small $\beta$};
\node at (-0.2,1.2) {\small $\gamma$};
\node at (-0.2,1.8) {\small $\delta$};

\node at (-0.2,0.8) {\small $\alpha$};
\node at (-0.8,0.8) {\small $\beta$};
\node at (-0.8,0.2) {\small $\gamma$};
\node at (-0.2,0.2) {\small $\delta$};

\node at (0.8,0.8) {\small $\alpha$};
\node at (0.2,0.8) {\small $\beta$};
\node at (0.2,0.2) {\small $\gamma$};
\node at (0.8,0.2) {\small $\delta$};

\node at (0.4,1.15) {\small $\alpha$};
\node at (1,1.7) {\small $\beta$};
\node at (1.5,1.2) {\small $\gamma$};
\node at (1,1.2) {\small $\delta$};

\node at (0.15,2.6) {\small $\alpha$};
\node at (0.75,1.95) {\small $\beta$};
\node at (0.15,1.4) {\small $\gamma$};
\node at (0.15,2) {\small $\delta$};

\node at (1.2,0.2) {\small $\alpha$};
\node at (1.8,0.2) {\small $\beta$};
\node at (1.8,0.8) {\small $\gamma$};
\node at (1.2,0.8) {\small $\delta$};

\node at (2.8,1.8) {\small $\beta$};
\node at (1.45,1.8) {\small $\alpha$};
\node at (2.05,1.2) {\small $\delta$};
\node at (2.8,1.2) {\small $\gamma$};

\node at (2.8,0.2) {\small $\alpha$};
\node at (2.2,0.2) {\small $\beta$};
\node at (2.2,0.8) {\small $\gamma$};
\node at (2.8,0.8) {\small $\delta$};

\node at (3.2,0.8) {\small $\gamma$};

\node at (3.8,1.8) {\small $\alpha$};
\node at (3.2,1.8) {\small $\beta$};
\node at (3.2,1.2) {\small $\gamma$};
\node at (3.8,1.2) {\small $\delta$};

\node[draw,shape=circle, inner sep=0.5] at (-0.5,1.5) {\small $1$};
\node[draw,shape=circle, inner sep=0.5] at (-0.5,0.5) {\small $2$};
\node[draw,shape=circle, inner sep=0.5] at (0.5,0.5) {\small $3$};
\node[draw,shape=circle, inner sep=0.5] at (0.7,1.35) {\small $4$};
\node[draw,shape=circle, inner sep=0.5] at (0.35,1.7) {\small $5$};
\node[draw,shape=circle, inner sep=0.5] at (1.5,0.5) {\small $6$};
\node[draw,shape=circle, inner sep=0] at (2.4,1.5) {\footnotesize $11$};
\node[draw,shape=circle, inner sep=0] at (2.5,0.5) {\footnotesize $12$};
\node[draw,shape=circle, inner sep=0] at (3.5,1.5) {\footnotesize $13$};
\node[draw,shape=circle, inner sep=0] at (2,2.5) {\footnotesize $14$};

\node at (2.8,2.2) {\small $\alpha$};
\node at (1.1,2.2) {\small $\beta$};
\node at (0.4,2.8) {\small $\gamma$};
\node at (2.8,2.8) {\small $\delta$};

\node at (3.2,2.2) {\small $\beta$};

\end{scope}

\end{tikzpicture}
\caption{Proposition \ref{add}: There is a tile with $\alpha\cdots=\gamma\cdots=\alpha^2\beta\gamma^2$.}
\label{add-abbbA}
\end{figure}

The discussion by Figure \ref{add-abbbA} shows that there is a tile with $\alpha\cdots=\gamma\cdots=\alpha^2\beta\gamma^2$. In Figure \ref{add-abbbB}, therefore, we have the tiles $T_1,T_2,\dots,T_{10}$ like the first of Figure \ref{add-abbbA}, and we may also assume $\alpha_5\cdots=\alpha^2\beta\gamma^2$. Then $\delta_1\delta_5\cdots=\alpha\delta^2$ and the assumption $\alpha_5\cdots=\alpha^2\beta\gamma^2$ determine $T_{11}$. Then the AAD of $\alpha_5\cdots=\alpha^2\beta\gamma^2$ further determines $T_{12},T_{13},T_{14}$. Then we may determine $(T_1),T_{15},T_{16},T_{17},T_{18}$ like $T_6,T_7,T_8,T_9,T_{10}$ in the first of Figure \ref{add-abbbA}. Then $\delta_9\delta_{10}\cdots=\alpha\delta^2$ and $\alpha_{10}\beta_{15}\beta_{16}\cdots=\alpha\beta^3$ determine $T_{19}$. Then $\alpha_{16}\beta_{17}\gamma_{19}\cdots=\alpha^2\beta\gamma^2$ determines $T_{20},T_{21}$. Then $\delta_{17}\delta_{18}\cdots=\alpha_{17}\delta_{21}\cdots=\alpha\delta^2$ determines $T_{22}$. Then $\beta_{12}\beta_{13}\beta_{18}\cdots=\alpha\beta^3$ and no $\alpha\beta\delta\cdots$ determine $T_{23}$. Then $\delta_{13}\delta_{14}\cdots=\alpha_{13}\delta_{23}\cdots=\alpha\delta^2$ determines $T_{24}$. Then $\beta_4\beta_5\beta_{14}\cdots=\alpha\beta^3$ and no $\alpha\beta\delta\cdots$ determine $T_{25}$. Then $\gamma_4\gamma_6\delta_{25}\cdots=\gamma^3\delta$ determines $T_{26}$. Then $\gamma_{25}\delta_{26}\cdots=\gamma^3\delta$ determines $T_{27},T_{28}$. Then we get the rest of the tiling by using the obvious $180^{\circ}$ rotation symmetry. The tiling is $S_{36\square}6$.

\begin{figure}[htp]
\centering
\begin{tikzpicture}

\foreach \a in {1,-1}
{
\begin{scope}[scale=\a]

\draw
	(-0,0) -- (-0,0.8) -- (-1.6,-0.8) -- (-0,-0.8)
	(-0.8,-1.6) -- (-2.4,0) -- (-0.8,1.6) -- (-0,0.8)
	(-0.8,0) -- (-1.6,0.8)
	(-3.4,1.6) -- (3.4,1.6)
	(-2.4,2.6) -- (-2.4,-2.6)
	(3.4,2.6) -- (2.4,2.6) -- (0.8,2.4) -- (-0.8,2.4) -- (-2.4,2.6) -- (-3.4,2.6) -- (-3.4,-2.6)
	(-0.8,1.6) -- (-0.8,2.4)
	(-4.4,2.6) -- (-4.4,3.6)
	(2.4,2.6) to[out=150,in=0] (0,3) to[out=180,in=30] (-2.4,2.6)
	(3.4,2.6) to[out=140,in=0] (0,3.6) to[out=180,in=40] (-3.4,2.6)
	(3.4,-1.6) to[out=60,in=-90] (4,0.5) to[out=90,in=-60] (3.4,2.6)
	(4.4,-2.6) to[out=90,in=-90] (4.4,1) to[out=90,in=40] (3.4,2.6);

\draw[line width=1.2]
	(-1.6,0.8) -- (-0.8,1.6)
	(-0.8,0) -- (-1.6,-0.8) -- (-0.8,-1.6) -- (-0.8,-2.4)
	(-2.4,0) -- (-3.4,0)
	(-4.4,2.6) -- (-2.4,2.6) -- (-0.8,2.4)
	(3.4,2.6) -- (3.4,1.6);

\node at (-0.2,0.4) {\small $\alpha$};
\node at (-0.2,-0.6) {\small $\beta$};
\node at (-1.15,-0.6) {\small $\gamma$};
\node at (-0.7,-0.1) {\small $\delta$};

\node at (-1.6,0.55) {\small $\alpha$};
\node at (-2.1,0.05) {\small $\beta$};
\node at (-1.6,-0.5) {\small $\gamma$};
\node at (-1.1,0) {\small $\delta$};

\node at (-0.8,0.25) {\small $\alpha$};
\node at (-0.3,0.75) {\small $\beta$};
\node at (-0.8,1.3) {\small $\gamma$};
\node at (-1.3,0.85) {\small $\delta$};

\node at (-0.35,1.45) {\small $\alpha$};
\node at (0.1,1) {\small $\beta$};
\node at (1.15,1) {\small $\gamma$};
\node at (0.75,1.4) {\small $\delta$};

\node at (-2.2,0.4) {\small $\alpha$};
\node at (-2.2,1.4) {\small $\beta$};
\node at (-1.2,1.4) {\small $\gamma$};
\node at (-1.7,0.95) {\small $\delta$};

\node at (-2.2,-0.4) {\small $\alpha$};
\node at (-2.23,-1.43) {\small $\beta$};
\node at (-1.25,-1.4) {\small $\gamma$};
\node at (-1.7,-0.9) {\small $\delta$};

\node at (-3.2,1.4) {\small $\alpha$};
\node at (-2.6,1.4) {\small $\beta$};
\node at (-2.6,0.2) {\small $\gamma$};
\node at (-3.2,0.2) {\small $\delta$};

\node at (-3.2,-1.4) {\small $\alpha$};
\node at (-2.6,-1.4) {\small $\beta$};
\node at (-2.6,-0.2) {\small $\gamma$};
\node at (-3.2,-0.2) {\small $\delta$};

\node at (-2.6,1.8) {\small $\alpha$};
\node at (-3.2,1.8) {\small $\beta$};
\node at (-3.2,2.4) {\small $\gamma$};
\node at (-2.6,2.4) {\small $\delta$};

\node at (-2.6,-1.8) {\small $\alpha$};
\node at (-2.6,-2.4) {\small $\beta$};
\node at (-3.2,-2.4) {\small $\gamma$};
\node at (-3.2,-1.8) {\small $\delta$};

\node at (-1,1.8) {\small $\alpha$};
\node at (-2.2,1.8) {\small $\beta$};
\node at (-2.2,2.4) {\small $\gamma$};
\node at (-1,2.2) {\small $\delta$};

\node at (-0.6,2.2) {\small $\alpha$};
\node at (-0.6,1.8) {\small $\beta$};
\node at (0.6,1.8) {\small $\gamma$};
\node at (0.6,2.2) {\small $\delta$};

\node at (2.2,2.4) {\small $\alpha$};
\node at (2.2,1.8) {\small $\beta$};
\node at (1,1.8) {\small $\gamma$};
\node at (1,2.2) {\small $\delta$};

\node at (0.8,2.6) {\small $\alpha$};
\node at (1.65,2.7) {\small $\beta$};
\node at (-1.7,2.7) {\small $\gamma$};
\node at (-0.8,2.6) {\small $\delta$};

\node at (2.95,2.75) {\small $\alpha$};
\node at (2.5,2.8) {\small $\beta$};
\node at (-2.4,2.8) {\small $\gamma$};
\node at (-2.9,2.8) {\small $\delta$};

\node at (-3.55,0) {\small $\alpha$};
\node at (-3.55,0.95) {\small $\beta$};
\node at (-3.55,-2) {\small $\gamma$};
\node at (-3.55,-1.6) {\small $\delta$};

\node at (-3.65,-2.5) {\small $\alpha$};
\node at (-3.6,1.6) {\small $\beta$};
\node at (-3.6,2.4) {\small $\gamma$};
\node at (-4.2,2.4) {\small $\delta$};

\node at (-4.6,2.6) {\small $\alpha$};
\node at (3.45,2.85) {\small $\beta$};
\node at (-3.5,2.8) {\small $\gamma$};
\node at (-4.2,2.8) {\small $\delta$};

\end{scope}
}

\node[draw,shape=circle, inner sep=0.5] at (-0.8,0.8) {\small $1$};
\node[draw,shape=circle, inner sep=0.5] at (0.4,1.2) {\small $2$};
\node[draw,shape=circle, inner sep=0.5] at (0,2) {\small $3$};
\node[draw,shape=circle, inner sep=0.5] at (-1.6,2) {\small $4$};
\node[draw,shape=circle, inner sep=0.5] at (-1.95,1.15) {\small $5$};
\node[draw,shape=circle, inner sep=0.5] at (0,2.7) {\small $6$};
\node[draw,shape=circle, inner sep=0.5] at (1.6,2) {\small $7$};
\node[draw,shape=circle, inner sep=0.5] at (1.95,1.15) {\small $8$};
\node[draw,shape=circle, inner sep=0.5] at (1.6,0) {\small $9$};
\node[draw,shape=circle, inner sep=0] at (0.45,0.35) {\footnotesize $10$};
\node[draw,shape=circle, inner sep=0] at (-1.6,0) {\footnotesize $11$};
\node[draw,shape=circle, inner sep=0] at (-1.95,-1.15) {\footnotesize $12$};
\node[draw,shape=circle, inner sep=0] at (-2.9,-0.8) {\footnotesize $13$};
\node[draw,shape=circle, inner sep=0] at (-2.9,0.8) {\footnotesize $14$};
\node[draw,shape=circle, inner sep=0] at (-0.45,-0.35) {\footnotesize $15$};
\node[draw,shape=circle, inner sep=0] at (-0.4,-1.2) {\footnotesize $16$};
\node[draw,shape=circle, inner sep=0] at (0,-2) {\footnotesize $17$};
\node[draw,shape=circle, inner sep=0] at (-1.6,-2) {\footnotesize $18$};
\node[draw,shape=circle, inner sep=0] at (0.8,-0.8) {\footnotesize $19$};
\node[draw,shape=circle, inner sep=0] at (1.95,-1.15) {\footnotesize $20$};
\node[draw,shape=circle, inner sep=0] at (1.6,-2) {\footnotesize $21$};
\node[draw,shape=circle, inner sep=0] at (0,-2.7) {\footnotesize $22$};
\node[draw,shape=circle, inner sep=0] at (-2.9,-2.1) {\footnotesize $23$};
\node[draw,shape=circle, inner sep=0] at (-3.7,-0.5) {\footnotesize $24$};
\node[draw,shape=circle, inner sep=0] at (-2.9,2.1) {\footnotesize $25$};
\node[draw,shape=circle, inner sep=0] at (0,3.3) {\footnotesize $26$};
\node[draw,shape=circle, inner sep=0] at (-3.9,2) {\footnotesize $27$};
\node[draw,shape=circle, inner sep=0] at (4,3.2) {\footnotesize $28$};

\end{tikzpicture}
\caption{Proposition \ref{add}: Tiling for $\{\alpha\delta^2,\alpha\beta^3,\gamma^3\delta,\alpha^2\beta\gamma^2\}$.}
\label{add-abbbB}
\end{figure}
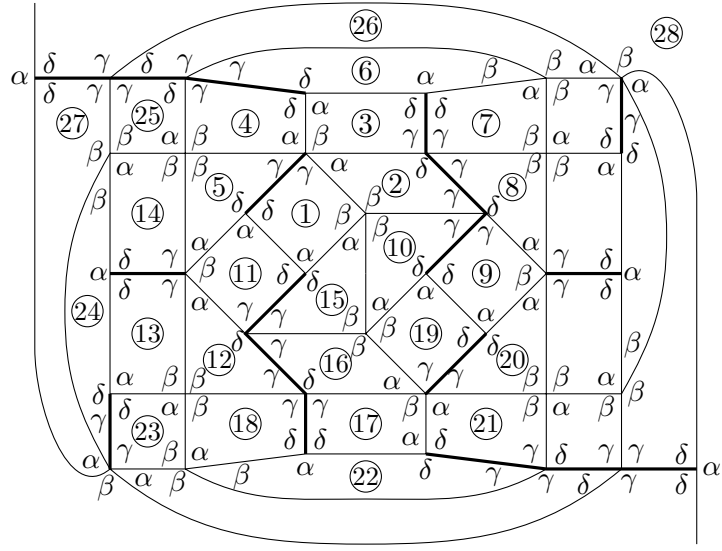

\medskip

\noindent{\em Geometry of Quadrilateral}

\medskip

The tiling $S_{36\square}6$ has angles given by \eqref{add-abbbavc2}. The angle values are obtained by verifying \eqref{coolsaet_eq1}. Then we substitute the values into \eqref{coolsaet_eq2} and \eqref{coolsaet_eq4}, and use $K=Y(b)^T$ to get 
\begin{align*}
\cos a &=4\cos\tfrac{1}{9}\pi-3,\quad a<\pi, \\
\cos b &=6\cos\tfrac{1}{9}\pi+2\sqrt{3}\sin\tfrac{1}{9}\pi-3\sqrt{3}\tan\tfrac{1}{9}\pi-4, \quad b<\pi.
\end{align*}
We note that $\cos a$ and $\cos b$ are the largest (and the only positive) roots of $t^3+9t^2+15t-17$ and $t^3+39t^2+39t-71$. Then we may use Lemma \ref{geometry8} verify that the quadrilateral is simple and suitable for tiling.

\subsubsection*{Case. $\alpha\delta^2,\alpha\beta\gamma^2$ are vertices}

The angle sums of $\alpha\delta^2,\alpha\beta\gamma^2$ imply
\[
\alpha=2\pi-2\delta,\;
\beta=\tfrac{8}{f}\pi,\;
\gamma=\delta-\tfrac{4}{f}\pi.
\]
We have $\gamma<\delta$. By Lemma \ref{geometry1}, this implies $\alpha<\beta$. This means $\delta>(1-\frac{4}{f})\pi$. By Lemma \ref{geometry4}, we get $2\alpha+\beta>\pi$. This implies $\alpha>(\frac{1}{2}-\frac{4}{f})\pi$. Combined with $\alpha<\beta=\tfrac{8}{f}\pi$, we get $f<24$. Therefore we get $f=16,18,20,22$ and the following, including \eqref{coolsaet_eq1}.
\begin{align*}
f=16 &\colon
	\alpha=2\pi-2\delta,\;
	\beta=\tfrac{1}{2}\pi,\;
	\gamma=\delta-\tfrac{1}{4}\pi. \\
	 &\sin(\pi-\delta)
	\sin(\delta-\tfrac{1}{4}\pi)
	=\sin\tfrac{1}{4}\pi
	\sin(2\delta-\tfrac{5}{4}\pi). \\
f=18 &\colon
	\alpha=2\pi-2\delta,\;
	\beta=\tfrac{4}{9}\pi,\;
	\gamma=\delta-\tfrac{2}{9}\pi. \\
	 &\sin(\pi-\delta)
	\sin(\delta-\tfrac{2}{9}\pi)
	=\sin\tfrac{2}{9}\pi
	\sin(2\delta-\tfrac{11}{9}\pi). \\
f=20 &\colon
	\alpha=2\pi-2\delta,\;
	\beta=\tfrac{2}{5}\pi,\;
	\gamma=\delta-\tfrac{1}{5}\pi. \\
	 &\sin(\pi-\delta)
	\sin(\delta-\tfrac{1}{5}\pi)
	=\sin\tfrac{1}{5}\pi
	\sin(2\delta-\tfrac{6}{5}\pi). \\
f=22 &\colon
	\alpha=2\pi-2\delta,\;
	\beta=\tfrac{4}{11}\pi,\;
	\gamma=\delta-\tfrac{2}{11}\pi. \\
	 &\sin(\pi-\delta)
	\sin(\delta-\tfrac{2}{11}\pi)
	=\sin\tfrac{2}{11}\pi
	\sin(2\delta-\tfrac{13}{11}\pi).
\end{align*}
The only solution satisfying $(1-\frac{4}{f})\pi<\delta<\pi$ is $\delta=0.7898\pi$ for $f=16$, and $\delta=0.7942\pi$ for $f=18$. For $f=16$, we use the approximate angle values to get the AVC \eqref{add-abbbavc1}. Then we get $S_{16\square}1$ as before. For the case $f=18$, we use the approximate angle values to get $\text{AVC}=\{\alpha\delta^2,\alpha\beta\gamma^2\}$. Applying the counting lemma to $\alpha,\beta$, we get a contradiction.

\subsubsection*{Case. $\alpha\delta^2,\beta^2\gamma\delta$ are vertices}

The angle sums of $\alpha\delta^2,\beta^2\gamma\delta$  imply
\[
\alpha=\tfrac{4}{f}\pi+\beta,\;
\gamma=(1+\tfrac{2}{f})\pi-\tfrac{3}{2}\beta,\;
\delta=(1-\tfrac{2}{f})\pi-\tfrac{1}{2}\beta.
\]
We have $\alpha>\beta$. By Lemma \ref{geometry1}, this implies $\gamma>\delta$. This means $\beta<\frac{4}{f}\pi$, and implies $\alpha,\beta<\pi$. Then by Lemma \ref{geometry3}, we get $(1+\tfrac{4}{f})\pi+\beta=\alpha+\pi>\gamma+\delta=2\pi-2\beta$. This implies $(1-\tfrac{4}{f})\pi<3\beta<\frac{12}{f}\pi$, contradicting $f\ge 16$.

\subsubsection*{Case. $\alpha\delta^2,\beta^2\gamma^2$ are vertices}

The angle sums of $\alpha\delta^2,\beta^2\gamma^2$ imply
\[
\alpha=\tfrac{8}{f}\pi,\;
\beta+\gamma=\pi,\;
\delta=(1-\tfrac{4}{f})\pi.
\]
By Lemma \ref{geometry3}, we get $\beta+\pi>\gamma+\delta$ and $\gamma+\pi>\beta+\delta$. These imply $(\frac{1}{2}-\frac{2}{f})\pi<\beta,\gamma<(\frac{1}{2}+\frac{2}{f})\pi<\delta$. By Lemmas \ref{geometry5} and \ref{geometry1}, we get $\alpha<\beta,\gamma$.

Similar to the case one of $\alpha\beta^3,\beta^4,\beta^5$ is a vertex, by $\alpha<\beta,\gamma$ and $\beta,\gamma<\delta$, we get $\delta^2\cdots=\alpha\delta^2$, and $\beta\delta\cdots,\delta\thin\delta\cdots$ are not vertices, and there are no consecutive $\alpha\alpha\alpha,\alpha\alpha\gamma,\gamma\alpha\gamma$. 

By $R(\alpha\beta\delta)<\gamma<\delta$, we know $\alpha\beta\delta\cdots$ is not a vertex.

By $\beta^2\gamma^2$ and $\alpha<\gamma$, we know $\alpha^2\beta^2$ is not a vertex. By $\beta>(\frac{1}{2}-\frac{2}{f})\pi$, we get $R(\beta^4)<\frac{8}{f}\pi=\alpha<\beta,\gamma,\delta$. This implies $\beta^4\cdots=\beta^4$. 

Suppose $\gamma\thin\gamma\cdots$ is a vertex. Then we get the AAD $\thick^{\delta}\gamma^{\beta}\thin^{\beta}\gamma^{\delta}\thick$, which determines $T_1,T_2$ in the first of Figure \ref{add-bbddB}. By $\beta^2\gamma^2$, and $\gamma<\delta$, we know $\beta_1\thin\beta_2\cdots$ is $\beta^2\gamma^2$ or a $\hat{b}$-vertex. If it is a $\hat{b}$-vertex, then by $\beta^4\cdots=\beta^4$, and no $\alpha^2\beta^2$, and no consecutive $\alpha\alpha\alpha$, and the only degree $3$ vertex $\alpha\delta^2$, the vertex is $\alpha\beta^3,\alpha^2\beta^3,\alpha^3\beta^3,\alpha^4\beta^3,\beta^4$. 

If $\beta_1\thin\beta_2\cdots=\beta^2\gamma^2$, then we determine $T_3,T_4$. On the other hand, by $\delta^2\cdots=\alpha\delta^2$, we know $\gamma_1\gamma_2\cdots$ is not $\gamma^2\delta^2\cdots$. Therefore one of $\delta_1\cdots,\delta_2\cdots$ is $\delta^2\cdots=\alpha\delta^2$. By symmetry, we may assume $\delta_1\cdots=\alpha\delta^2$, which gives $\alpha_5$. Then by no $\alpha_1\beta_3\delta\cdots$, we determine $T_5$. By $\beta^2\gamma^2$ and $\gamma<\delta$, we know $\beta_3\thin\alpha_1\thin\beta_5\cdots$ is a $\hat{b}$-vertex. Then by $\beta^4\cdots=\beta^4$, and no consecutive $\alpha\alpha\alpha$, and no $\alpha^2\beta^2$, and the only degree $3$ vertex $\alpha\delta^2$, we get $\beta_3\thin\alpha_1\thin\beta_5\cdots=\alpha\beta^3,\alpha^2\beta^3,\alpha^3\beta^2,\alpha^3\beta^3,\alpha^4\beta^3,\alpha^5\beta^3$.

\begin{figure}[htp]
\centering
\begin{tikzpicture}[>=latex,scale=1]


\draw
	(1,-1) -- (2,-1)
	(1,0) -- (2,0);

\draw[line width=1.2]
	(0,0) -- (0,1)
	(-1,-1) -- (1,-1)
	(2,0) -- (2,-1);

\foreach \b in {1,-1}
{
\begin{scope}[xscale=\b]

\draw
	(0,1) -- (1,1) -- (1,-1)
	(1,0) -- (0,0) -- (0,-1);

\draw[line width=1.2]
	(0,0) -- (0,1)
	(-1,-1) -- (1,-1);
	
\node at (0.8,0.8) {\small $\alpha$};
\node at (0.8,0.2) {\small $\beta$};
\node at (0.2,0.8) {\small $\delta$};
\node at (0.2,0.2) {\small $\gamma$};

\node at (0.8,-0.2) {\small $\alpha$};
\node at (0.2,-0.2) {\small $\beta$};
\node at (0.8,-0.8) {\small $\delta$};
\node at (0.2,-0.8) {\small $\gamma$};

\end{scope}
}

\node at (1.2,-0.8) {\small $\alpha$};
\node at (1.2,-0.2) {\small $\beta$};
\node at (1.8,-0.8) {\small $\delta$};
\node at (1.8,-0.2) {\small $\gamma$};

\node at (1,-1.2) {\small $\delta$};

\node[draw,shape=circle, inner sep=0.5] at (0.5,-0.5) {\small $1$};
\node[draw,shape=circle, inner sep=0.5] at (-0.5,-0.5) {\small $2$};
\node[draw,shape=circle, inner sep=0.5] at (0.5,0.5) {\small $3$};
\node[draw,shape=circle, inner sep=0.5] at (-0.5,0.5) {\small $4$};
\node[draw,shape=circle, inner sep=0.5] at (1.5,-0.5) {\small $5$};

\begin{scope}[xshift=5cm]

\foreach \b in {1,-1}
{
\begin{scope}[xscale=\b]

\draw
	(-2,-1) rectangle (2,1)
	(-2,0) -- (2,0)
	(0,-1) -- (0,0)
	(1,1) -- (1,0);

\draw[line width=1.2]	
	(0,0) -- (0,1)
	(1,0) -- (1,-1)
	(2,0) -- (2,1);

\node at (0.8,0.8) {\small $\alpha$};
\node at (0.8,0.2) {\small $\beta$};
\node at (0.2,0.8) {\small $\delta$};
\node at (0.2,0.2) {\small $\gamma$};

\node at (1.2,0.8) {\small $\alpha$};
\node at (1.2,0.2) {\small $\beta$};
\node at (1.8,0.8) {\small $\delta$};
\node at (1.8,0.2) {\small $\gamma$};

\node at (0.2,-0.8) {\small $\alpha$};
\node at (0.2,-0.2) {\small $\beta$};
\node at (0.8,-0.8) {\small $\delta$};
\node at (0.8,-0.2) {\small $\gamma$};

\node at (1.8,-0.8) {\small $\alpha$};
\node at (1.8,-0.2) {\small $\beta$};
\node at (1.2,-0.8) {\small $\delta$};
\node at (1.2,-0.2) {\small $\gamma$};

\end{scope}
}

\node at (0,1.15) {\small $\alpha$};

\node[draw,shape=circle, inner sep=0.5] at (-0.5,0.5) {\small $1$};
\node[draw,shape=circle, inner sep=0.5] at (0.5,0.5) {\small $2$};
\node[draw,shape=circle, inner sep=0.5] at (-0.5,-0.5) {\small $3$};
\node[draw,shape=circle, inner sep=0.5] at (0.5,-0.5) {\small $4$};
\node[draw,shape=circle, inner sep=0.5] at (-1.5,0.5) {\small $5$};
\node[draw,shape=circle, inner sep=0.5] at (1.5,0.5) {\small $6$};

\end{scope}

\end{tikzpicture}
\caption{Proposition \ref{add}: $\thick\gamma\thin\gamma\thick$ and vertex $\thick\gamma\thin\beta\thin\beta\thin\gamma\thick$.}
\label{add-bbddB}
\end{figure}

We conclude that, if $\gamma\thin\gamma\cdots$ is a vertex, then one of $\alpha\beta^3$, $\alpha^2\beta^3$, $\alpha^3\beta^2$, $\alpha^3\beta^3$, $\alpha^4\beta^3$, $\alpha^5\beta^3$, $\beta^4$ is a vertex. We have argued that, if $\alpha\delta^2$ is the only degree $3$ vertex, and $\alpha\beta^3$ or $\beta^4$ is a vertex, then the tiling is $S_{16\square}1$ (see Figure \ref{add-bbbA}) or $S_{36\square}6$ (see Figure \ref{add-abbbB}) . Since both tilings do not have $\beta^2\gamma^2$, we conclude $\alpha\beta^3,\beta^4$ are not vertices.

The angle sum of one of $\alpha^2\beta^3,\alpha^3\beta^2,\alpha^3\beta^3,\alpha^4\beta^3,\alpha^5\beta^3$, and the angle sums of $\alpha\delta^2,\beta^2\gamma^2$, imply the following, including \eqref{coolsaet_eq1}.
\begin{itemize}
\item $\alpha^2\beta^3$:
	$\alpha=\tfrac{8}{f}\pi,\;
	\beta=(\tfrac{2}{3}-\tfrac{16}{3f})\pi,\;
	\gamma=(\tfrac{1}{3}+\tfrac{16}{3f})\pi,\;
	\delta=(1-\tfrac{4}{f})\pi$.
	\newline 
	$\sin\tfrac{4}{f}\pi
	\sin(\tfrac{2}{3}-\tfrac{4}{3f})\pi
	=\sin(\tfrac{1}{3}-\tfrac{8}{3f})\pi
	\sin(\tfrac{1}{3}+\tfrac{4}{3f})\pi$. 
\item $\alpha^3\beta^2$: 
	$\alpha=\tfrac{8}{f}\pi,\;
	\beta=(1-\tfrac{12}{f})\pi,\;
	\gamma=\tfrac{12}{f}\pi,\;
	\delta=(1-\tfrac{4}{f})\pi$.
	\newline 
	$\sin\tfrac{4}{f}\pi
	\sin(\tfrac{1}{2}+\tfrac{2}{f})\pi
	=\sin(\tfrac{1}{2}-\tfrac{6}{f})\pi
	\sin\tfrac{8}{f}\pi$.
\item $\alpha^3\beta^3$:
	$\alpha=\tfrac{8}{f}\pi,\;
	\beta=(\tfrac{2}{3}-\tfrac{8}{f})\pi,\;
	\gamma=(\tfrac{1}{3}+\tfrac{8}{f})\pi,\;
	\delta=(1-\tfrac{4}{f})\pi$.
	\newline 
	$\sin\tfrac{4}{f}\pi
	\sin\tfrac{2}{3}\pi
	=\sin(\tfrac{1}{3}-\tfrac{4}{f})\pi
	\sin(\tfrac{1}{3}+\tfrac{4}{f})\pi$. 
\item $\alpha^4\beta^3$:
	$\alpha=\tfrac{8}{f}\pi,\;
	\beta=(\tfrac{2}{3}-\tfrac{32}{3f})\pi,\;
	\gamma=(\tfrac{1}{3}+\tfrac{32}{3f})\pi,\;
	\delta=(1-\tfrac{4}{f})\pi$.
	\newline 
	$\sin\tfrac{4}{f}\pi
	\sin(\tfrac{2}{3}+\tfrac{4}{3f})\pi
	=\sin(\tfrac{1}{3}-\tfrac{16}{3f})\pi
	\sin(\tfrac{1}{3}+\tfrac{20}{3f})\pi$. 
\item $\alpha^5\beta^3$:
	$\alpha=\tfrac{8}{f}\pi,\;
	\beta=(\tfrac{2}{3}-\tfrac{40}{3f})\pi,\;
	\gamma=(\tfrac{1}{3}+\tfrac{40}{3f})\pi,\;
	\delta=(1-\tfrac{4}{f})\pi$.
	\newline 
	$\sin\tfrac{4}{f}\pi
	\sin(\tfrac{2}{3}+\tfrac{8}{3f})\pi
	=\sin(\tfrac{1}{3}-\tfrac{20}{3f})\pi
	\sin(\tfrac{1}{3}+\tfrac{28}{3f})\pi$. 
\end{itemize}
The only solution for even $f\ge 16$ is $f=20$ for $\alpha^2\beta^3,\alpha^3\beta^2$. However, this implies $\alpha=\beta$, a contradiction.

Therefore $\gamma\thin\gamma\cdots$ is not a vertex. Combined with $\delta^2\cdots=\alpha\delta^2$ and no $\beta\delta\cdots$, we get $\gamma\delta\cdots=\alpha^k\gamma^l\delta$. By no $\gamma\thin\gamma\cdots$ and no consecutive $\alpha\alpha\alpha,\alpha\alpha\gamma,\gamma\alpha\gamma$, we get $\alpha^k\gamma^l\delta=\alpha\gamma\delta$, contradicting the only degree $3$ vertex $\alpha\delta^2$. Therefore $\gamma\delta\cdots$ is not a vertex.

By no $\gamma\delta\cdots$, we know a vertex $\gamma\cdots$ is a combination of $\gamma^2$-fans. By $\beta^2\gamma^2$, and no $\gamma\thin\gamma\cdots$, and no consecutive $\alpha\alpha\alpha,\alpha\alpha\gamma,\gamma\alpha\gamma$, we know the only $\gamma^2$-fans are the vertex $\beta^2\gamma^2$, and $\thick\gamma\thin\beta\thin\gamma\thick,\thick\gamma\thin\alpha\thin\beta\thin\gamma\thick,\thick\gamma\thin\alpha\thin\beta\thin\alpha\thin\gamma\thick$. By $\beta^2\gamma^2$, we get $\beta+2\gamma>\pi$. Therefore a vertex cannot have more than one $\gamma^2$-fan. By $\beta^2\gamma^2$, and $\alpha<\beta$, and the only degree $3$ vertex $\alpha\delta^2$, this implies $\gamma\cdots=\beta^2\gamma^2,\alpha^2\beta\gamma^2$. 

The angle sum of $\alpha^2\beta\gamma^2$ further implies the following, including \eqref{coolsaet_eq1}.
\begin{align*}
\alpha^2\beta\gamma^2 &\colon
	\alpha=\tfrac{8}{f}\pi,\;
	\beta=\tfrac{16}{f}\pi,\;
	\gamma=(1-\tfrac{16}{f})\pi,\;
	\delta=(1-\tfrac{4}{f})\pi. \\
&	\sin\tfrac{4}{f}\pi
	\sin(1-\tfrac{12}{f})\pi
	=\sin\tfrac{8}{f}\pi
	\sin(1-\tfrac{20}{f})\pi.
\end{align*}
The equation has no solution for even $f\ge 16$. 

We conclude $\gamma\cdots=\beta^2\gamma^2$. Combined with $\delta^2\cdots=\alpha\delta^2$, we know $\alpha\delta^2,\beta^2\gamma^2$ are all the $b$-vertices. By no $\alpha\gamma\cdots,\gamma\thin\gamma\cdots$, we get the AAD $\thick^{\delta}\gamma^{\beta}\thin^{\gamma}\beta^{\alpha}\thin^{\alpha}\beta^{\gamma}\thin^{\beta}\gamma^{\delta}\thick$ of $\beta^2\gamma^2$. This determines $T_1,T_2,T_3,T_4$ in the second of Figure \ref{add-bbddB}. Then the AAD of $\beta_1\gamma_3\cdots=\beta_2\gamma_4\cdots=\beta^2\gamma^2$ determines $T_5,T_6$. Then $\delta_1\delta_2\cdots=\alpha\delta^2$ gives an $\alpha$, and this $\alpha$ implies either $\alpha_1\alpha_5\cdots$ or $\alpha_2\alpha_6\cdots$ is $\alpha^2\delta\cdots$, contradicting all the $b$-vertices $\alpha\delta^2,\beta^2\gamma^2$.

\subsubsection*{Case. $\alpha\delta^2,\beta^3\gamma^2$ are vertices}

The angle sums of $\alpha\delta^2,\beta^3\gamma^2$ imply
\[
\alpha=\tfrac{8}{f}\pi+\beta,\;
\gamma=\pi-\tfrac{3}{2}\beta,\;
\delta=(1-\tfrac{4}{f})\pi-\tfrac{1}{2}\beta.
\]
By $\alpha>\beta$ and Lemma \ref{geometry1}, we get $\gamma>\delta$. By $\alpha\delta^2,\beta^3\gamma^2$, this implies $\alpha>3\beta$. This means $\beta<\frac{4}{f}\pi$, and we get $\alpha<\tfrac{12}{f}\pi<\pi$. Then by Lemma \ref{geometry3}, we get $\beta+\pi>\gamma+\delta$. This means $\beta>(\frac{1}{3}-\frac{4}{3f})\pi$. Combined with $\beta<\frac{4}{f}\pi$, we get $f<16$, a contradiction.

\subsubsection*{Case. $\alpha\delta^2,\beta\gamma^4$ are vertices}

The angle sums of $\alpha\delta^2,\beta\gamma^4$ imply
\[
\alpha=(1+\tfrac{8}{f})\pi-\tfrac{3}{2}\beta,\;
\gamma=\tfrac{1}{2}\pi-\tfrac{1}{4}\beta,\;
\delta=(\tfrac{1}{2}-\tfrac{4}{f})\pi+\tfrac{3}{4}\beta.
\]
We have $2\alpha+3\beta=(2+\tfrac{16}{f})\pi>2\pi$. By Lemma \ref{geometry4}, we also have $2\alpha+\beta>\pi$.

By $2\delta=(1-\tfrac{8}{f})\pi+\tfrac{3}{2}\beta>\beta$ and Lemma \ref{geometry7}, we get $\alpha<2\gamma$. This means $\beta>\frac{8}{f}\pi$. Then  $\delta-\gamma=\beta-\frac{4}{f}\pi>0$. By Lemma \ref{geometry1}, this implies $\alpha<\beta$. This means $\beta>(\frac{2}{5}+\frac{16}{5f})\pi$, and further implies $\beta-\gamma=\frac{5}{4}\beta-\frac{1}{2}\pi>0$. 

By $\alpha\delta^2$ and $\alpha<\beta$, we know $\beta\delta^2\cdots$ is not a vertex. By $\alpha<\beta$ and $\alpha+\beta+\gamma+\delta>2\pi$, we know $\beta^2\gamma\delta\cdots$ is not a vertex. Then by $\beta\gamma^4$ and $\gamma<\beta,\delta$, we get $\beta^2\cdots=\alpha^k\beta^l,\alpha^k\beta^l\gamma^2$, with $l\ge 2$. By $\alpha<\beta$, and $2\alpha+\beta>\pi$, and $2\alpha+3\beta>2\pi$, we get $\alpha^k\beta^l=\alpha\beta^3,\alpha^2\beta^2,\beta^4,\alpha^3\beta^2$. By $\alpha<\beta$ and $\alpha+2\beta+2\gamma>2\pi$, we get $\alpha^k\beta^l\gamma^2=\beta^2\gamma^2$. We have argued that, if $\alpha\delta^2$ is the only degree $3$ vertex, and one of $\alpha\beta^3,\beta^4,\beta^2\gamma^2$ is a vertex, then the tiling is $S_{16\square}1$ (see Figure \ref{add-bbbA}) or $S_{36\square}6$ (see Figure \ref{add-abbbB}) . Since both tilings do not have $\beta\gamma^4$, we conclude $\alpha\beta^3,\beta^4,\beta^2\gamma^2$ are not vertices.

The angle sum of one of $\alpha^2\beta^2,\alpha^3\beta^2$ further implies the following, including \eqref{coolsaet_eq1}.
\begin{itemize}
\item $\alpha^2\beta^2 $:
	$\alpha=(1-\tfrac{16}{f})\pi,\;
	\beta=\tfrac{16}{f}\pi,\;
	\gamma=(\tfrac{1}{2}-\tfrac{4}{f})\pi,\;
	\delta=(\tfrac{1}{2}+\tfrac{8}{f})\pi$.
	\newline
	$\sin(\tfrac{1}{2}-\tfrac{8}{f})\pi
	\sin\tfrac{1}{2}\pi
	=\sin\tfrac{8}{f}\pi
	\sin\tfrac{4}{f}\pi$.
\item $\alpha^3\beta^2$:
	$\alpha=(\tfrac{2}{5}-\tfrac{32}{5f})\pi,\;
	\beta=(\tfrac{2}{5}+\tfrac{48}{5f})\pi,\;
	\gamma=(\tfrac{2}{5}-\tfrac{12}{5f})\pi,\;
	\delta=(\tfrac{4}{5}+\tfrac{16}{5f})\pi$.
	\newline
	$\sin(\tfrac{1}{5}-\tfrac{16}{5f})\pi
	\sin(\tfrac{3}{5}-\tfrac{8}{5f})\pi
	=\sin(\tfrac{1}{5}+\tfrac{24}{5f})\pi
	\sin(\tfrac{1}{5}+\tfrac{4}{5f})\pi$.
\end{itemize}
There is no solution for even $f\ge 16$.

\subsubsection*{Case. $\alpha\delta^2,\gamma^4$ are vertices}

The angle sums of $\alpha\delta^2,\gamma^4$ imply
\[
\beta=(\tfrac{1}{2}+\tfrac{4}{f})\pi-\tfrac{1}{2}\alpha,\;
\gamma=\tfrac{1}{2}\pi,\;
\delta=\pi-\tfrac{1}{2}\alpha.
\]
We have $\beta<\delta$. By Lemma \ref{geometry5}, this implies $\alpha<\gamma$. Then we get $\delta>\frac{3}{4}\pi>\gamma$. By Lemma \ref{geometry1}, we get $\alpha<\beta$. This means $\alpha<(\frac{1}{3}+\frac{8}{3f})\pi$, and implies $\alpha+\beta=(\tfrac{1}{2}+\tfrac{4}{f})\pi+\tfrac{1}{2}\alpha<(\frac{2}{3}+\frac{16}{3f})\pi\le\pi$. Then by Lemma \ref{geometry3}, we get $\alpha+\pi>\gamma+\delta$. This means $\alpha>\frac{1}{3}\pi$. 

By $\alpha\delta^2$ and $\alpha<\beta,\gamma,\delta$, we get $\delta^2\cdots=\alpha\delta^2$. This implies $\delta\thin\delta\cdots$ is not a vertex. By $\delta^2\cdots=\alpha\delta^2$ and $R(\beta\gamma\delta)<\alpha<\beta,\gamma,\delta$, we know $\beta\delta\cdots$ is not a vertex. By no $\beta\delta\cdots,\delta\thin\delta\cdots$, we know the AAD of $\thin\alpha\thin\alpha\thin$ is $\thin^{\delta}\alpha^{\beta}\thin^{\beta}\alpha^{\delta}\thin$. This implies no consecutive $\alpha\alpha\alpha$.

By the only degree $3$ vertex $\alpha\delta^2$, the AAD $\thick^{\delta}\gamma^{\beta}\thin^{\beta}\gamma^{\delta}\thick$ at $\gamma^4$ implies a vertex $\beta\thin\beta\cdots$ of degree $\ge 4$. If this is a $b$-vertex, then by no $\beta\delta\cdots$, and $\alpha+2\beta+2\gamma>2\pi$, and $\alpha<\beta,\gamma$, we get $\beta\thin\beta\cdots=\beta^2\gamma^2$. If this is a $\hat{b}$-vertex, then by $\alpha+\beta<\pi$, and no consecutive $\alpha\alpha\alpha$, we know the remainder of $\beta\thin\beta\cdots$ must have $\beta$. By $\beta>\alpha>\frac{1}{3}\pi$, the vertex is $\alpha\beta^3,\alpha\beta^4,\alpha^2\beta^3,\beta^4,\beta^5$. 

We have argued that, if $\alpha\delta^2$ is the only degree $3$ vertex, and one of $\alpha\beta^3,\beta^4,\beta^5,\beta^2\gamma^2$ is a vertex, then the tiling is either the earth map tiling with four timezones similar to Figure \ref{add-bbbA}, or the tiling in Figure \ref{add-abbbB}. Since both tilings do not have $\gamma^4$, we conclude $\alpha\beta^3,\beta^4,\beta^5,\beta^2\gamma^2$ are not vertices.

The angle sum of one of $\alpha\beta^4,\alpha^2\beta^3$, and the angle sums of $\alpha\delta^2,\gamma^4$, imply the following, including \eqref{coolsaet_eq1}. 
\begin{itemize}
\item $\alpha\beta^4$:
	$\alpha=\tfrac{16}{f}\pi,\;
	\beta=(\tfrac{1}{2}-\tfrac{4}{f})\pi,\;
	\gamma=\tfrac{1}{2}\pi,\;
	\delta=(1-\tfrac{8}{f})\pi$.
	\newline
	$\sin\tfrac{8}{f}\pi
	\sin(\tfrac{3}{4}-\tfrac{6}{f})\pi
	=\sin(\tfrac{1}{4}-\tfrac{2}{f})\pi
	\sin(\tfrac{1}{2}-\tfrac{8}{f})\pi$.
\item $\alpha^2\beta^3$:
	$\alpha=(1-\tfrac{24}{f})\pi,\;
	\beta=\tfrac{16}{f}\pi,\;
	\gamma=\tfrac{1}{2}\pi,\;
	\delta=(\tfrac{1}{2}+\tfrac{12}{f})\pi$.
	\newline
	$\sin(\tfrac{1}{2}-\tfrac{12}{f})\pi
	\sin(\tfrac{1}{2}+\tfrac{4}{f})\pi
	=\sin\tfrac{8}{f}\pi
	\sin\tfrac{12}{f}\pi$. 
\end{itemize}
Both equations have no solution for even $f\ge 16$.

\subsubsection*{Case. $\alpha\delta^2,\gamma^3\delta$ are vertices}

The angle sums of $\alpha\delta^2,\gamma^3\delta$ imply
\[
\beta=(\tfrac{2}{3}+\tfrac{4}{f})\pi-\tfrac{2}{3}\alpha,\;
\gamma=\tfrac{1}{3}\pi+\tfrac{1}{6}\alpha,\;
\delta=\pi-\tfrac{1}{2}\alpha.
\]
We have $2\alpha+3\beta=(2+\tfrac{12}{f})\pi>2\pi$.

By $\delta-\beta=(\tfrac{1}{3}-\tfrac{4}{f})\pi+\tfrac{1}{6}\alpha>0$ and Lemma \ref{geometry5}, we get $\alpha<\gamma$. This means $\alpha<\frac{2}{5}\pi$. Then $\beta-\gamma=(\tfrac{1}{3}+\tfrac{4}{f})\pi-\frac{5}{6}\alpha>0$. We conclude $\alpha<\gamma<\beta<\delta$.

Similar to the case $\gamma^4$ is a vertex, by $\alpha<\beta,\gamma$ and $\gamma<\delta$, we know $\beta\delta\cdots,\delta\thin\delta\cdots$ are not vertices. Moreover, we get the unique AAD $\thin^{\delta}\alpha^{\beta}\thin^{\beta}\alpha^{\delta}\thin$ of $\thin\alpha\thin\alpha\thin$, which implies no consecutive $\alpha\alpha\alpha$.

By the only degree $3$ vertex $\alpha\delta^2$, the AAD $\thick^{\delta}\gamma^{\beta}\thin^{\beta}\gamma^{\delta}\thick$ at $\gamma^3\delta$ gives a vertex $\beta\thin\beta\cdots$ of degree $\ge 4$. If the vertex has $\gamma,\delta$, then by no $\beta\delta\cdots$, and $\alpha+2\beta+2\gamma>2\pi$, and $\alpha<\beta,\gamma$, we get $\beta\thin\beta\cdots=\beta^2\gamma^2$. If this is a $\hat{b}$-vertex, then by $\alpha<\beta$, and $2\alpha+3\beta>2\pi$, and no consecutive $\alpha\alpha\alpha$, we know $\beta\thin\beta\cdots=\alpha^2\beta^2,\alpha\beta^3,\beta^4$.

We have argued that, if $\alpha\delta^2$ is the only degree $3$ vertex, and one of $\alpha\beta^3,\beta^4,\beta^2\gamma^2$ is a vertex, then the tiling is $S_{16\square}1$ (see Figure \ref{add-bbbA}) or $S_{36\square}6$ (see Figure \ref{add-abbbB}) . Among the two tilings, only $S_{36\square}6$ has $\gamma^3\delta$. 

The angle sum of $\alpha^2\beta^2$ further implies the following, including \eqref{coolsaet_eq1}.
\begin{align*}
\alpha^2\beta^2 &\colon
	\alpha=(1-\tfrac{12}{f})\pi,\;
	\beta=\tfrac{12}{f}\pi,\;
	\gamma=(\tfrac{1}{2}-\tfrac{2}{f})\pi,\;
	\delta=(\tfrac{1}{2}+\tfrac{6}{f})\pi. \\
& 	\sin(\tfrac{1}{2}-\tfrac{6}{f})\pi
	\sin\tfrac{1}{2}\pi
	=\sin\tfrac{6}{f}\pi
	\sin\tfrac{4}{f}\pi.
\end{align*}
The equation has no solution for even $f\ge 16$. 
\end{proof}

\begin{proposition}\label{acd}
Tilings of the sphere by congruent almost equilateral quadrilaterals, such that $\alpha\gamma\delta$ is a vertex, are the earth map tiling $E_{\square}^A1$ and its modifications $FE_{\square}^A1,RE_{\square}^A1$.
\end{proposition}

\begin{proof}
By $\alpha\gamma\delta$, we get $\beta=\frac{4}{f}\pi<\pi$. By Lemma \ref{fbalance}, we may divide the discussion according to whether both $\gamma^2$-fan and $\delta^2$-fan exist or not.

\subsubsection*{Case. There are no $\gamma^2$-fan and $\delta^2$-fan}

The only fan is the $\gamma\delta$-fan, and we always have the same number of $\gamma$ and $\delta$ at every vertex. By $\alpha\gamma\delta$, we have either $\alpha<\pi$ and $\gamma+\delta>\pi$, or $\alpha\ge \pi$ and $\gamma+\delta\le \pi$.

Suppose $\alpha<\pi$ and $\gamma+\delta>\pi$. Then $\gamma^2\delta^2\cdots$ is not a vertex, and $\alpha\gamma\delta,\alpha^t\beta^k,\beta^s\gamma\delta$ are all the vertices. We have $p\beta=2\pi$ for $p=\tfrac{f}{2}$, and $\alpha=s\beta$, and $k=p-st$. Then we get the list of vertices
\begin{equation}\label{acd_avc1}
\text{AVC}=\{\alpha\gamma\delta,\alpha^t\beta^{p-st},\beta^s\gamma\delta\},\quad 
0\le st\le p=\tfrac{f}{2}.
\end{equation}
We note that $p,s$ are fixed by the quadrilateral, and we may allow $\alpha^t\beta^{p-st}$ for various $t$. If all $\alpha^t\beta^{p-st}$ are not vertices, then $\alpha\gamma\delta,\beta^s\gamma\delta$ are the only vertices. Applying the counting lemma to $\alpha,\delta$, we find $\beta^s\gamma\delta$ is not a vertex. Then $\alpha\gamma\delta$ is the only vertex, a contradiction. Therefore $\alpha^t\beta^{p-st}$ is a vertex for some $t$.

Suppose $\alpha\ge \pi$ and $\gamma+\delta\le \pi$. Then $\alpha^2\cdots$ is not a vertex, and $\alpha\gamma\delta,\alpha\beta^s,\beta^k\gamma^t\delta^t$ are all the vertices. We have $p\beta=2\pi$ for $p=\tfrac{f}{2}$, and $\gamma+\delta=s\beta$, and $k=p-st$. Then we get the list of vertices
\begin{equation}\label{acd_avc2}
\text{AVC}
=\{\alpha\gamma\delta,\alpha\beta^s,\beta^{p-st}\gamma^t\delta^t\},\quad 
0\le st\le p=\tfrac{f}{2}.
\end{equation} 
Again $p,s$ are fixed by the quadrilateral, and we may allow $\beta^{p-st}\gamma^t\delta^t$ for various $t$. Moreover, by the similar reason, we know $\beta^{p-st}\gamma^t\delta^t$ is a vertex for some $t$.

Now we construct the tiling for the AVC \eqref{acd_avc1}. 

The AVC implies $\gamma^2\cdots,\delta^2\cdots,\gamma\thin\delta\cdots$ are not vertices. This implies the AAD of $\alpha^t\beta^{p-st}$ is a combination of $\thin^{\beta}\alpha^{\delta}\thin^{\beta}\alpha^{\delta}\thin$, $\thin^{\beta}\alpha^{\delta}\thin^{\alpha}\beta^{\gamma}\thin$, $\thin^{\alpha}\beta^{\gamma}\thin^{\beta}\alpha^{\delta}\thin$, $\thin^{\alpha}\beta^{\gamma}\thin^{\alpha}\beta^{\gamma}\thin$. In other words, the $\gamma,\delta$ angles are always ``on the same side''.

The AAD $\thin^{\alpha}\beta^{\gamma}\thin^{\alpha}\beta^{\gamma}\thin$ determines $T_1,T_2$ in the first of Figure \ref{acdG}. Then $\alpha_1\gamma_2\cdots=\alpha\gamma\delta$ determines $T_3$. Then $\alpha_3\delta_1\cdots=\alpha\gamma\delta$ determines $T_4$. Therefore consecutive $\beta^s=\thin^{\alpha}\beta^{\gamma}\thin\cdots \thin^{\alpha}\beta^{\gamma}\thin$ determines a partial earth map tiling ${\mc A}_s$, with $\beta^s$ at the other end. In particular, if $\beta^p$ is a vertex (which corresponds to $t=0$), then we actually get the earth map tiling $E_{\square}^A1$.

\begin{figure}[htp]
\centering
\begin{tikzpicture}[>=latex]


\begin{scope}[xshift=-4.5 cm]

\foreach \a in {0,1,2}
\draw[xshift=1*\a cm]
	(0.5,0.9) -- (0.5,0.15) -- (0,-0.15) -- (0,-0.9);

\foreach \a in {0,1}
{
\begin{scope}[xshift=1*\a cm]
	
\draw[line width=1.2]
	(0.5,0.15) -- (1,-0.15);

\node at (1,0.8) {\small $\beta$};
\node at (0.7,0.3) {\small $\gamma$};
\node at (1.3,0.3) {\small $\alpha$};
\node at (1,0.1) {\small $\delta$};

\node at (0.5,-0.85) {\small $\beta$};
\node at (0.8,-0.3) {\small $\gamma$};
\node at (0.2,-0.3) {\small $\alpha$};
\node at (0.5,-0.1) {\small $\delta$};

\end{scope}
}

\node[draw,shape=circle, inner sep=0.5] at (1,0.45) {\small $1$};
\node[draw,shape=circle, inner sep=0.5] at (2,0.45) {\small $2$};
\node[draw,shape=circle, inner sep=0.5] at (1.5,-0.45) {\small $3$};
\node[draw,shape=circle, inner sep=0.5] at (0.5,-0.45) {\small $4$};

\end{scope}


\foreach \a in {1,-1}
\foreach \b in {0,1}
{
\begin{scope}[xshift=6*\b cm, scale=\a]

\draw
	(1,1.6) -- (1,-1.6)
	(-1,0.8) -- (1,0.8)
	(-1,-0.8) -- (0,0.8);

\draw[line width=1.2]
	(-1,0.8) -- (0,0.8);

\node at (0,1.5) {\small $\alpha$};
\node at (0.8,1) {\small $\beta$};
\node at (-0.8,1) {\small $\delta$};
\node at (0,1) {\small $\gamma$};

\node at (0.6,-0.1) {\small $\alpha$};
\node at (0.85,0.3) {\small $\beta$};
\node at (0.3,-0.6) {\small $\delta$};
\node at (0.8,-0.6) {\small $\gamma$};

\end{scope}
}


\draw
 	(1,0.8) -- (3,0.8) -- (3,1.6);

\draw[line width=1.2]
	(1,0.8) -- (2,0.8);

\node at (2,1.5) {\small $\alpha$};
\node at (2.8,1) {\small $\beta$};
\node at (1.2,1) {\small $\delta$};
\node at (2,1) {\small $\gamma$};

\node at (1.2,0.6) {\small $\gamma$};

\node at (0.45,0.55) {\small $\beta^{s-2}$};
\node at (-0.4,-0.55) {\small $\beta^{s-2}$};
\node at (0,0) {\small $P$};

\node[draw,shape=circle, inner sep=0.5] at (0.4,1.2) {\small $1$};
\node[draw,shape=circle, inner sep=0.5] at (2.4,1.2) {\small $2$};
\node[draw,shape=circle, inner sep=0.5] at (0.55,-0.4) {\small $4$};
\node[draw,shape=circle, inner sep=0.5] at (-0.55,0.4) {\small $3$};
\node[draw,shape=circle, inner sep=0.5] at (-0.4,-1.2) {\small $5$};


\begin{scope}[xshift=6cm]

\foreach \a in {-1.5,1.5}
\foreach \b in {1,-1}
{
\begin{scope}[xshift=\a cm, scale=\b]

\draw
	(0.5,-1.6) -- (0.5,1.6);

\draw[line width=1.2]
	(0,0) -- (0.5,-0.8);

\node at (-0.3,-0.8) {\small $\alpha$};
\node at (0,-1.5) {\small $\beta$};	
\node at (-0.35,0.2) {\small $\delta$};
\node at (0.3,-0.85) {\small $\gamma$};

\end{scope}
}

\node at (0.45,0.55) {\small $\beta^{s-2}$};
\node at (-0.4,-0.55) {\small $\beta^{s-2}$};
\node at (0,0) {\small $P$};

\node[draw,shape=circle, inner sep=0.5] at (-1.5,1) {\small $1$};
\node[draw,shape=circle, inner sep=0.5] at (0.4,1.2) {\small $2$};
\node[draw,shape=circle, inner sep=0.5] at (1.5,1) {\small $3$};
\node[draw,shape=circle, inner sep=0.5] at (-0.4,-1.2) {\small $6$};
\node[draw,shape=circle, inner sep=0.5] at (0.55,-0.4) {\small $5$};
\node[draw,shape=circle, inner sep=0.5] at (-0.55,0.4) {\small $4$};

\end{scope}

\end{tikzpicture}
\caption{Proposition \ref{acd}: Tiling for $\{\alpha\gamma\delta,\alpha^t\beta^{p-st},\beta^s\gamma\delta\}$.}
\label{acdG}
\end{figure}

The AAD $\thin^{\beta}\alpha^{\delta}\thin^{\beta}\alpha^{\delta}\thin$ at $\alpha^t\beta^{p-st}$ determines $T_1,T_2$ in the second of Figure \ref{acdG}. By no $\gamma^2\cdots$, we get $\thick^{\gamma}\delta_2^{\alpha}\thin^{\alpha}\beta_1^{\gamma}\thin\cdots=\beta^s\gamma\delta=\thick^{\gamma}\delta_2^{\alpha}\thin^{\alpha}\beta_1^{\gamma}\thin^{\alpha}\beta^{\gamma}\thin\cdots\thin^{\alpha}\beta_4^{\gamma}\thin^{\beta}\gamma^{\delta}\thick$. Then the $\beta^s=\thin^{\alpha}\beta_1^{\gamma}\thin^{\alpha}\beta^{\gamma}\thin\cdots\thin^{\alpha}\beta_4^{\gamma}\thin$ part of the vertex determines a partial earth map tiling ${\mc A}_s$ consisting of  $P,T_1,T_3,T_4,T_5$. This ${\mc A}_s$ is obtained by the first flip in Figure \ref{flip7}. If we flip back this ${\mc A}_s$, then $\alpha_1$ is reverted to $\beta^s$, and the vertex $\alpha^t\beta^{p-st}$ is reverted to $\alpha^{t-1}\beta^{p-s(t-1)}$. If the new vertex $\alpha^{t-1}\beta^{p-s(t-1)}$ still has consecutive $\alpha\alpha$, then we may apply the  flip back again. More flip backs give a tiling with  vertex $\alpha^{t'}\beta^{p-s{t'}}$ that has no consecutive $\alpha\alpha$.  

In a vertex $\alpha^t\beta^{p-st}$ with no consecutive $\alpha\alpha$, any $\alpha$ appears in consecutive $\beta\alpha\beta$. Moreover, the AAD of $\beta\alpha\beta$ is $\thin^{\alpha}\beta^{\gamma}\thin^{\beta}\alpha^{\delta}\thin^{\alpha}\beta^{\gamma}\thin$. This determines $T_1,T_2,T_3$ in the third of Figure \ref{acdG}. By no $\gamma^2\cdots$, we get $\thick^{\delta}\gamma_3^{\beta}\thin^{\alpha}\beta_2^{\gamma}\thin\cdots=\beta^s\gamma\delta=\thick^{\delta}\gamma_3^{\beta}\thin^{\alpha}\beta_2^{\gamma}\thin^{\alpha}\beta^{\gamma}\thin\cdots\thin^{\alpha}\beta_5^{\gamma}\thin^{\alpha}\delta^{\gamma}\thick$. Then the $\beta^s=\thin^{\alpha}\beta_2^{\gamma}\thin^{\alpha}\beta^{\gamma}\thin\cdots\thin^{\alpha}\beta_5^{\gamma}\thin$ part of the vertex determines a partial earth map tiling ${\mc A}_s$ consisting of  $P,T_2,T_4,T_5,T_6$. This ${\mc A}_s$ is again obtained by the first flip in Figure \ref{flip7}. If we flip back this ${\mc A}_s$, then the vertex $\alpha^t\beta^{p-st}$ is reverted to $\alpha^{t-1}\beta^{p-s(t-1)}$. More flip backs give a tiling with vertex $\beta^{p}$, which we know is the earth map tiling $E_{\square}^A1$.  

We conclude that, if $\alpha^t\beta^{p-st}$ is a vertex, then the tiling for the AVC \eqref{acd_avc1} is obtained by repeatedly applying the first flip in Figure \ref{flip7} to $E_{\square}^A1$. The tiling is the flip modification $FE_{\square}^A1$.

Next we construct the tiling for the AVC \eqref{acd_avc2}. 

The AVC implies $\alpha^2\cdots$ is not a vertex. By no $\gamma^2$-fan, we know $\gamma\thin\gamma\cdots$ is not a vertex. By no $\alpha^2\cdots,\gamma\thin\gamma\cdots$, we know the AAD of $\thin\beta\thin\beta\thin$ is $\thin^{\alpha}\beta^{\gamma}\thin^{\alpha}\beta^{\gamma}\thin$. Then by the same argument for the AVC \eqref{acd_avc1}, we get the first of Figure \ref{acdG} as before. Moreover, consecutive $\beta^s=\beta\beta\cdots$ determines the partial earth map tiling ${\mc A}_s$, and a vertex $\beta^p$ implies the earth map tiling $E_{\square}^A1$. 

If $\delta\thick\delta\cdots$ is a vertex, then we get $T_1,T_2$ in the first and second of Figure \ref{acdE}. We know $\delta_1\thick\delta_2\cdots=\beta^{p-st}\gamma^t\delta^t$, and $\delta_1$ is part of a $\gamma\delta$-fan $\thick\delta_1\thin\beta\thin\cdots\thin\beta\thin\gamma\thick$. If the fan has $\beta$, then we get $\beta_3$ in the first of Figure \ref{acdE}. If the fan has no $\beta$, then we get $\gamma_3$ in the second of Figure \ref{acdE}. In the first picture, by $\beta_3$ and no $\alpha^2\cdots$, we determine $T_3$. Then $\alpha_1\gamma_3\cdots=\alpha\gamma\delta$ determines $T_4$. Then $\alpha_4\beta_1\cdots=\alpha\beta^s$ implies an angle $\alpha$ or $\gamma$ just outside $\gamma_1$. Then $\gamma_1\gamma_2\cdots=\alpha\gamma^2\cdots$, or we get a $\gamma^2$-fan $\thick\gamma\thin\gamma\thick$, both contradictions. In the second picture, $\gamma_3$ determines $T_3$. Then by no $\alpha^2\cdots$, we get $\thin^{\beta}\alpha_1^{\delta}\thin^{\gamma}\beta_3^{\alpha}\thin\cdots=\alpha\beta^s=\thin^{\beta}\alpha_1^{\delta}\thin^{\gamma}\beta_3^{\alpha}\thin^{\gamma}\beta^{\alpha}\thin\cdots\thin^{\gamma}\beta^{\alpha}\thin$. This gives an angle $\alpha$ just outside $\beta_1$, similar to $\alpha_4$ in the first picture. Then $\beta_1\cdots=\alpha\beta\cdots=\alpha\beta^s$, and we get the same contradiction as in the first picture.

\begin{figure}[htp]
\centering
\begin{tikzpicture}[>=latex]


\begin{scope}[shift={(-9cm,-0.4cm)}]

\foreach \a in {0,1}
{
\begin{scope}[xshift=3.5*\a cm]

\draw
	(0,0) -- (1.8,0) -- (1.8,1.8)
	(0,1.8) -- (0,-1) -- (1,-1) -- (1,1)
	(0,1) -- (1,1) -- (1.8,1.8);
	
\node at (0.8,0.8) {\small $\alpha$};
\node at (0.2,0.8) {\small $\beta$};
\node at (0.2,0.2) {\small $\gamma$};
\node at (0.8,0.2) {\small $\delta$};

\node at (0.8,-0.8) {\small $\alpha$};
\node at (0.2,-0.8) {\small $\beta$};
\node at (0.2,-0.2) {\small $\gamma$};
\node at (0.8,-0.2) {\small $\delta$};

\node at (-0.38,1) {\small $\beta^{s-1}$};
\node at (-0.35,0.2) {\small $\alpha/\gamma$};

\node[draw,shape=circle, inner sep=0.5] at (0.5,0.5) {\small $1$};
\node[draw,shape=circle, inner sep=0.5] at (0.5,-0.5) {\small $2$};
\node[draw,shape=circle, inner sep=0.5] at (1.4,0.6) {\small $3$};

\end{scope}
}

\draw
	(1.8,1.8) -- (0,1.8);

\draw[line width=1.2]
	(0,0) -- (1,0)
	(1,1) -- (1.8,1.8);

\node at (1.6,0.2) {\small $\alpha$};
\node at (1.2,0.2) {\small $\beta$};
\node at (1.2,0.9) {\small $\gamma$};
\node at (1.6,1.35) {\small $\delta$};

\node at (0.2,1.6) {\small $\beta$};
\node at (0.2,1.2) {\small $\alpha$};
\node at (0.9,1.2) {\small $\delta$};
\node at (1.4,1.6) {\small $\gamma$};

\node[draw,shape=circle, inner sep=0.5] at (0.6,1.4) {\small $4$};

\begin{scope}[xshift=3.5 cm]

\draw[line width=1.2]
	(0,0) -- (1.8,0);

\node at (1.6,0.2) {\small $\delta$};
\node at (1.2,0.2) {\small $\gamma$};
\node at (1.2,0.95) {\small $\beta$};
\node at (1.6,1.4) {\small $\alpha$};

\node at (0.2,1.2) {\small $\alpha$};
\node at (0.85,1.2) {\small $\beta^{s-1}$};

\end{scope}

\end{scope}


\draw
	(-1.5,-1.6) -- (-1.5,1.6)
	(1.5,-1.6) -- (1.5,1.6);

\draw[dotted]
	(-2.5,1.6) -- (-2.5,0.8) -- (-1.5,-0.8);
	
\node at (-1.7,0.8) {\small $\alpha$};
\node at (-2,1.5) {\small $\theta$};

\foreach \b in {1,-1}
{
\begin{scope}[scale=\b]

\draw
	(1.5,-0.8) -- (0.5,0.8) -- (-1.5,0.8);

\draw[line width=1.2]
	(0.5,0.8) -- (0.5,1.6)
	(-0.5,0.8) -- ++(0.2,-0.32);
	
\node at (-0.5,1) {\small $\alpha$};
\node at (-1.3,1) {\small $\beta$};	
\node at (0.3,1) {\small $\delta$};
\node at (-0.5,1.5) {\small $\gamma$};

\node at (1.3,0.8) {\small $\alpha$};
\node at (1.3,-0.1) {\small $\beta$};	
\node at (1,1.5) {\small $\delta$};
\node at (0.7,0.9) {\small $\gamma$};

\node at (0.4,0.6) {\small $\alpha$};
\node at (-0.17,0.63) {\small $\delta$};
\node at (-0.6,0.6) {\small $\gamma$};

\node[rotate=-60] at (-1,0.45) {\small $\beta^{s-2}$};

\end{scope}
}

\node at (0,0) {\small $P$};

\node[draw,shape=circle, inner sep=0.5] at (-2,1) {\small $3$};
\node[draw,shape=circle, inner sep=0.5] at (-0.1,1.2) {\small $2$};
\node[draw,shape=circle, inner sep=0.5] at (-1,-1) {\small $4$};
\node[draw,shape=circle, inner sep=0.5] at (1,1) {\small $1$};
\node[draw,shape=circle, inner sep=0.5] at (0.1,-1.2) {\small $5$};

\end{tikzpicture}
\caption{Proposition \ref{acd}: Tiling for $\{\alpha\gamma\delta,\alpha\beta^s,\beta^{p-st}\gamma^t\delta^t\}$.}
\label{acdE}
\end{figure}

Therefore $\delta\thick\delta\cdots$ is not a vertex. This implies $\beta^{p-st}\gamma^t\delta^t$ is $t$ $\gamma\delta$-fans $\thick\delta\thin\beta\thin\cdots\thin\beta\thin\gamma\thick$ glued together along $\gamma\thick\delta=\thin^{\beta}\gamma^{\delta}\thick^{\gamma}\delta^{\alpha}\thin$. By no $\alpha^2\cdots$, the angle $\gamma$ in the AAD $\thin^{\beta}\gamma^{\delta}\thick^{\gamma}\delta^{\alpha}\thin$ belongs to a $\gamma\delta$-fan $\thick^{\gamma}\delta^{\alpha}\thin^{\gamma}\beta^{\alpha}\thin\cdots\thin^{\gamma}\beta^{\alpha}\thin^{\beta}\gamma^{\delta}\thick$. Therefore the AAD $\thin^{\beta}\gamma^{\delta}\thick^{\gamma}\delta^{\alpha}\thin$ is extended to $\theta^{\alpha}\thin^{\beta}\gamma^{\delta}\thick^{\gamma}\delta^{\alpha}\thin$, where $\theta=\beta,\delta$, depending on whether the $\gamma\delta$-fan has $\beta$. This determines $T_1,T_2,T_3$ in the third of Figure \ref{acdE}. Then by no $\alpha^2\cdots$, we get $\thin\alpha_3^{\theta}\thin^{\gamma}\beta_2^{\alpha}\thin\cdots=\alpha\beta^s=\thin\alpha_3^{\theta}\thin^{\gamma}\beta_2^{\alpha}\thin\cdots\thin^{\gamma}\beta_4^{\alpha}\thin$. Then the $\beta^s$ part of the vertex determines a partial earth map tiling ${\mc A}_s$ consisting of  $P,T_1,T_2,T_4,T_5$. This ${\mc A}_s$ is obtained by the second flip in Figure \ref{flip7}. If we flip back this ${\mc A}_s$, then the vertex $\beta^{p-st}\gamma^t\delta^t$ is reverted to $\beta^{p-s(t-1)}\gamma^{t-1}\delta^{t-1}$. If we apply the flip back to all $t$ $\gamma\thick\delta$ in $\beta^{p-st}\gamma^t\delta^t$, then we get a tiling with vertex $\beta^{p}$, which we know is the earth map tiling $E_{\square}^A1$. Therefore the tiling with the vertex $\beta^{p-st}\gamma^t\delta^t$ is the flip modification $FE_{\square}^A1$.

\subsubsection*{Case. There are $\gamma^2$-fan and $\delta^2$-fan}

The assumption implies both $\gamma^2\cdots,\delta^2\cdots$ are vertices, and further implies $\gamma,\delta<\pi$.

If $\alpha^2\cdots$ is a vertex, then $\alpha<\pi$. By Lemma \ref{geometry3}, we get $\gamma+\delta<\beta+\pi=(1+\frac{4}{f})\pi$. By $\alpha\gamma\delta$, this implies $\alpha>(1-\frac{4}{f})\pi>\beta$. Then by Lemma \ref{geometry1}, we get $\gamma>\delta$. By $\alpha\gamma\delta$ and $\alpha<\pi$, we also get $\gamma+\delta>\pi$. Then by $\alpha\gamma\delta$, and $\gamma>\delta$, and $\gamma+\delta>\pi$, we get $R(\gamma^2)<\alpha,2\gamma,2\delta$. This implies $\gamma^2\cdots=\beta^k\gamma^2$. By the only degree $3$ vertex $\alpha\gamma\delta$, we get $k\ge 2$ in $\beta^k\gamma^2$. This implies $\beta+\gamma\le \pi<\gamma+\delta$. Therefore $\beta<\delta$. Then by Lemma \ref{geometry5}, we get $\alpha<\gamma$. Then we get $\beta+\gamma>\alpha+\beta>(1-\frac{4}{f})\pi+\beta=\pi$, contradicting $\beta^k\gamma^2(k\ge 2)$. 

Therefore $\alpha^2\cdots$ is not a vertex. Then by Lemma \ref{klem3}, we know a $\delta^2$-fan has a single $\alpha$. Then by $\alpha\gamma\delta$, a vertex with $\delta^2$-fan has no $\gamma$. Then by no $\alpha^2\cdots$ again, a vertex with $\delta^2$-fan is a single $\delta^2$-fan, which is $\alpha\beta^k\delta^2$. By $\alpha\gamma\delta$ and $\gamma\ne\delta$, we get $k\ge 1$ in $\alpha\beta^k\delta^2$. The existence of $\delta^2$-fan implies $\alpha\beta^k\delta^2$ is a vertex.

Comparing $\alpha\gamma\delta,\alpha\beta^k\delta^2$, we get $\gamma>\delta$. Then by $\alpha\gamma\delta$, this implies $\alpha\gamma\cdots=\alpha\gamma\delta$. Then by the discussion about $\delta^2$-fan, we know a $b$-vertex $\alpha\cdots=\alpha\gamma\delta,\alpha\beta^k\delta^2$. Then by no $\alpha^2\cdots$, we get $\alpha\cdots=\alpha\gamma\delta,\alpha\beta^{k'},\alpha\beta^k\delta^2$. Then by $\alpha\beta^k\delta^2$ being a vertex, and applying the counting lemma to $\alpha,\delta$, we know $\alpha\beta^{k'}$ is also a vertex. Comparing $\alpha\beta^{k'},\alpha\beta^k\delta^2$, we get $k'>k$, and $\beta\le (k'-k)\beta=2\delta$. By Lemma \ref{geometry7}, this implies $\alpha\le 2\gamma$, and $3\gamma+\delta\ge \alpha+\gamma+\delta=2\pi$. 

By $k\ge 1$ in $\alpha\beta^k\delta^2$, we have $\thin\beta\thin\delta\thick\delta\thin$ at the vertex. By no $\alpha^2\cdots$, the AAD of $\thin\beta\thin\delta\thick\delta\thin$ is $\thin^{\alpha}\beta^{\gamma}\thin^{\alpha}\delta^{\gamma}\thick^{\gamma}\delta^{\alpha}\thin$. This determines $T_1,T_2,T_3$ in the first of Figure \ref{acdB}. Then $\alpha_2\gamma_1\cdots=\alpha\gamma\delta$ determines $T_4$. By $\alpha_4\beta_2\cdots=\alpha\beta^{k'},\alpha\beta^k\delta^2$, we get $\beta$ or $\delta$ just outside $\beta_2$. This implies the angle just outside $\gamma_2$ is not $\beta,\delta$. By $\alpha\gamma\cdots=\alpha\gamma\delta$, this angle is not $\alpha$. Therefore the angle is $\gamma$, and $\gamma_2\gamma_3\cdots=\gamma^3\cdots$. Then by $\gamma>\delta$, and $3\gamma+\delta\ge 2\pi$, we get $\gamma^3\cdots=\gamma^3\delta$. Since $\alpha\beta^k\delta^2$ is a vertex, we know $\gamma^3\delta$ is a vertex.

The AAD $\thin^{\alpha}\beta^{\gamma}\thin^{\alpha}\delta^{\gamma}\thick^{\gamma}\delta^{\alpha}\thin^{\gamma}\beta^{\alpha}\thin$ of $\thin\beta\thin\delta\thick\delta\thin\beta\thin$ would imply another $\beta$ just outside $\delta_3$ in the first of Figure \ref{acdB}. Then the same argument gives another $\gamma$ just outside $\gamma_3$, contradicting $\gamma^3\cdots=\gamma^3\delta$. This proves there is no $\thin\beta\thin\delta\thick\delta\thin\beta\thin$. Therefore $\alpha\beta^k\delta^2=\thick\delta\thin\alpha\thin\beta\thin\cdots\thin\beta\thin\delta\thick$. Then by no $\alpha^2\cdots$, we get the AAD $\thick^{\gamma}\delta^{\alpha}\thin\alpha\thin^{\alpha}\beta^{\gamma}\thin\cdots\thin^{\alpha}\beta^{\gamma}\thin^{\alpha}\delta^{\gamma}\thick$ of the vertex. 

Since a vertex with $\delta^2$-fan is $\alpha\beta^k\delta^2$, a vertex without $\alpha$ has no $\delta^2$-fan. In such a vertex, the number of $\gamma$ is no less than the number of $\delta$. Then by $\gamma^3\cdots=\gamma^3\delta$, a vertex without $\alpha$ is $\gamma^3\delta,\beta^k\gamma^2,\beta^k,\beta^k\gamma\delta,\beta^k\gamma^2\delta^2$. Combined with $\alpha\cdots=\alpha\gamma\delta,\alpha\beta^k,\alpha\beta^k\delta^2$, we get all the vertices. 

By $\alpha\gamma\delta,\gamma^3\delta$, we get $\alpha=2\gamma$. By Lemma \ref{geometry7}, this implies $\beta=2\delta$. This further implies 
\[
\alpha=(\tfrac{4}{3}-\tfrac{4}{3f})\pi,\;
\beta=\tfrac{4}{f}\pi,\;
\gamma=(\tfrac{2}{3}-\tfrac{2}{3f})\pi,\;
\delta=\tfrac{2}{f}\pi.
\]
Then we get the more refined information about all the vertices
\begin{equation}\label{acd_avc3}
\text{AVC}
=\{\alpha\gamma\delta,\gamma^3\delta,\alpha\beta^{q+1},\beta^{3q+2},\alpha\beta^q\delta^2,\beta^{q+1}\gamma^2,\beta^{2q+1}\gamma\delta,\beta^q\gamma^2\delta^2\},\quad
q=\tfrac{f-4}{6}.
\end{equation}
We know $\alpha\gamma\delta,\gamma^3\delta,\alpha\beta^{q+1},\alpha\beta^q\delta^2$ are vertices.

\begin{figure}[htp]
\centering
\begin{tikzpicture}


\draw 
	(1,1) -- (0,1) -- (0,0)
	(1,1) -- (2,0) -- (1,-1) -- (-1,-1) -- (-1,0) -- (1,0) --(1,-1);

\draw[line width=1.2]
	(0,0) -- (0,-1.4)
	(1,1) -- (1,0);

\node at (0.2,-0.8) {\small $\gamma$};
\node at (0.8,-0.8) {\small $\beta$};
\node at (0.2,-0.2) {\small $\delta$};
\node at (0.8,-0.2) {\small $\alpha$};

\node at (0.2,0.8) {\small $\alpha$};
\node at (0.8,0.8) {\small $\delta$};
\node at (0.2,0.2) {\small $\beta$};
\node at (0.8,0.2) {\small $\gamma$};

\node at (-0.8,-0.2) {\small $\alpha$};
\node at (-0.2,-0.2) {\small $\delta$};
\node at (-0.8,-0.8) {\small $\beta$};
\node at (-0.2,-0.8) {\small $\gamma$};

\node at (1.2,0.55) {\small $\gamma$};
\node at (1.75,-0.05) {\small $\beta$};
\node at (1.15,0) {\small $\delta$};
\node at (1.2,-0.55) {\small $\alpha$};

\node at (0.2,-1.2) {\small $\gamma$};
\node at (-0.2,-1.2) {\small $\delta$};
\node at (1,-1.2) {\small $\beta/\delta$};

\node[draw,shape=circle, inner sep=0.5] at (0.5,0.5) {\small $1$};
\node[draw,shape=circle, inner sep=0.5] at (0.5,-0.5) {\small $2$};
\node[draw,shape=circle, inner sep=0.5] at (-0.5,-0.5) {\small $3$};
\node[draw,shape=circle, inner sep=0.5] at (1.45,0) {\small $4$};


\begin{scope}[xshift=3.3cm]

\foreach \a in {0,1,2}
\draw[xshift=1*\a cm]
	(0.5,0.9) -- (0.5,0.15) -- (0,-0.15) -- (0,-0.9);

\foreach \a in {0,1}
{
\begin{scope}[xshift=1*\a cm]

\draw[line width=1.2]
	(1,-0.15) -- (0.5,0.15);

\node at (1,0.8) {\small $\beta$};
\node at (0.7,0.3) {\small $\gamma$};
\node at (1.3,0.3) {\small $\alpha$};
\node at (1,0.1) {\small $\delta$};

\node at (0.5,-0.85) {\small $\beta$};
\node at (0.8,-0.3) {\small $\gamma$};
\node at (0.2,-0.3) {\small $\alpha$};
\node at (0.5,-0.1) {\small $\delta$};

\end{scope}
}

\draw
	(0,-0.15) -- (-0.5,0.15) -- (-0.5,0.9);
	
\draw[line width=1.2]
	(-0.5,0.15) -- (-0.5,0.9);
	
\node at (0,0.8) {\small $\delta$};
\node at (-0.3,0.3) {\small $\gamma$};
\node at (0.3,0.3) {\small $\alpha$};
\node at (0,0.1) {\small $\beta$};
	
\node[draw,shape=circle, inner sep=0.5] at (1,0.45) {\small $1$};
\node[draw,shape=circle, inner sep=0.5] at (2,0.45) {\small $2$};
\node[draw,shape=circle, inner sep=0.5] at (1.5,-0.45) {\small $3$};
\node[draw,shape=circle, inner sep=0.5] at (0.5,-0.45) {\small $4$};
\node[draw,shape=circle, inner sep=0.5] at (0,0.45) {\small $5$};

\end{scope}


\begin{scope}[shift={(7.5cm,1cm)}]

\draw[gray]
	(0.1,0.1) -- (0.1,0.9) -- (2.4,0.9) -- (2.4,-0.9) -- (1.1,-0.9) -- (1.1,0.1) -- (0.1,0.1) 
	(2.4,-1.1) -- (2.4,-2.9) -- (-0.9,-2.9) -- (-0.9,-1.1) -- (2.4,-1.1);

\draw 
	(-1,-3) -- (0,-2) -- (1,-2) -- (1,0) -- (-1,0) -- (-1,-3) -- (2.5,-3) -- (2.5,1) -- (0,1) -- (0,0)
	(2.5,-1) -- (-1,-1);

\draw[line width=1.2]
	(0,0) -- (0,-2)
	(1,0) -- ++(0.3,0.3)
	(2.5,1) -- ++(-0.3,-0.3)
	(1,-2) -- ++(0.3,-0.3)
	(2.5,-3) -- ++(-0.3,0.3);
	
\fill (2.5,-1) circle (0.05);

\node at (0.8,-0.2) {\small $\alpha$};
\node at (0.8,-0.8) {\small $\beta$};
\node at (0.2,-0.2) {\small $\delta$};
\node at (0.2,-0.8) {\small $\gamma$};

\node at (-0.8,-0.2) {\small $\alpha$};
\node at (-0.8,-0.8) {\small $\beta$};
\node at (-0.2,-0.2) {\small $\delta$};
\node at (-0.2,-0.8) {\small $\gamma$};

\node at (0.2,0.8) {\small $\alpha$};
\node at (0.25,0.2) {\small $\beta^q$};
\node at (1.2,-0.8) {\small $\alpha$};
\node at (2.25,-0.8) {\small $\beta^q$};
\node at (2.1,0.8) {\small $\delta$};
\node at (2.3,0.55) {\small $\gamma$};
\node at (1.2,-0.1) {\small $\delta$};
\node at (0.9,0.2) {\small $\gamma$};

\node at (-0.2,0.2) {\small $\alpha$};

\node at (1.75,-0.1) {\small ${\mc A}_q$};

\node at (0.8,-1.8) {\small $\alpha$};
\node at (0.8,-1.2) {\small $\beta$};
\node at (0.2,-1.8) {\small $\delta$};
\node at (0.2,-1.2) {\small $\gamma$};
\node at (0.1,-2.2) {\small $\alpha$};

\node at (-0.8,-1.2) {\small $\alpha$};
\node at (-0.8,-2.6) {\small $\beta$};
\node at (-0.2,-1.2) {\small $\delta$};
\node at (-0.2,-1.9) {\small $\gamma$};

\node at (2.3,-1.2) {\small $\alpha$};
\node at (1.45,-1.25) {\small $\beta^{q-1}$};
\node at (-0.25,-2.75) {\small $\beta^{q-1}$};
\node at (2.05,-2.8) {\small $\delta$};
\node at (2.3,-2.6) {\small $\gamma$};
\node at (1.2,-1.9) {\small $\delta$};
\node at (0.9,-2.2) {\small $\gamma$};

\node at (1.75,-1.9) {\small ${\mc A}_q'$};

\node[draw,shape=circle, inner sep=0.5] at (0.5,-0.5) {\small $2$};
\node[draw,shape=circle, inner sep=0.5] at (-0.5,-0.5) {\small $3$};
\node[draw,shape=circle, inner sep=0.5] at (0.5,-1.5) {\small $5$};
\node[draw,shape=circle, inner sep=0.5] at (-0.5,-1.7) {\small $6$};

\end{scope}

\end{tikzpicture}
\caption{Proposition \ref{acd}: $\thin\beta\thin\delta\thick\delta\thin$ and partial earth map tiling for \eqref{acd_avc3}.}
\label{acdB}
\end{figure}

The AAD $\thin^{\alpha}\beta^{\gamma}\thin^{\alpha}\beta^{\gamma}\thin$ determines $T_1,T_2$ in the second of Figure \ref{acdB}. Then $\alpha_1\gamma_2\cdots=\alpha\gamma\delta$ determines $T_3$. Unlike the first of Figure \ref{acdG}, the current AVC does not imply $T_4$. Therefore $\beta^s=\thin^{\alpha}\beta^{\gamma}\thin\cdots\thin^{\alpha}\beta^{\gamma}\thin$ determines a partial earth map tiling with $\beta^{s-1}=\thin^{\alpha}\beta^{\gamma}\thin\cdots\thin^{\alpha}\beta^{\gamma}\thin$ at the other end. On the other hand, the $\thin^{\alpha}\beta^{\gamma}\thin\cdots\thin^{\alpha}\beta^{\gamma}\thin^{\alpha}\delta^{\gamma}\thick$ part of $\alpha\beta^q\delta^2=\thick^{\gamma}\delta^{\alpha}\thin\alpha\thin^{\alpha}\beta^{\gamma}\thin\cdots\thin^{\alpha}\beta^{\gamma}\thin^{\alpha}\delta^{\gamma}\thick$ determines $T_5$ in the second of Figure \ref{acdB}. Then $\alpha_5\gamma_1\cdots=\alpha\gamma\delta$ determines $T_4$. Therefore $\thin\beta\thin\cdots\thin\beta\thin\delta\thick$ determines the partial earth map tiling ${\mc A}_q$ (in the first of Figure \ref{flip7}) with $\beta^q$ at both ends. 

Back to the vertex $\alpha\beta^q\delta^2$ in the first of Figure \ref{acdB}. The picture is extended to the third of Figure \ref{acdB}. We have $\delta_2\delta_3\cdots=\thick\delta_2\thin\alpha\thin\beta\thin\cdots\thin\beta\thin\delta_3\thick$, and $\gamma_2\gamma_3\cdots=\thick\gamma_2\thin\gamma_6\thick\delta_7\thin\gamma_3\thick$ determines $T_5,T_6$. As explained above, the $\thin\beta\thin\cdots\thin\beta\thin\delta_3\thick$ part of the vertex $\delta_2\delta_3\cdots=\alpha\beta^q\delta^2$ determines the partial earth map tiling ${\mc A}_q$, outlined by the grey lines. 

We have $\beta_2\beta_5\cdots=\alpha\beta^2\cdots=\alpha\beta^{q+1},\alpha\beta^q\delta^2$ ($\alpha$ is from ${\mc A}_q$, just outside $\beta_2$). Since $\thin^{\alpha}\beta_2^{\gamma}\thin^{\gamma}\beta_5^{\alpha}\thin$ is incompatible with the AAD $\thick^{\gamma}\delta^{\alpha}\thin\alpha\thin^{\alpha}\beta^{\gamma}\thin\cdots\thin^{\alpha}\beta^{\gamma}\thin^{\alpha}\delta^{\gamma}\thick$ of $\alpha\beta^q\delta^2$, we get $\alpha\beta_2\beta_5\cdots=\alpha\beta^{q+1}$. By no $\alpha^2\cdots$, the AAD of the vertex is $\thin^{\beta}\alpha^{\delta}\thin^{\alpha}\beta_2^{\gamma}\thin^{\gamma}\beta_5^{\alpha}\thin^{\gamma}\beta^{\alpha}\thin\cdots\thin^{\gamma}\beta^{\alpha}\thin$. The $\beta^q=\thin^{\gamma}\beta_5^{\alpha}\thin^{\gamma}\beta^{\alpha}\thin\cdots\thin^{\gamma}\beta^{\alpha}\thin$ part of the vertex determines a partial earth map tiling with $\beta^q$ at one end and $\beta_{q-1}$ at the other end. If we also include $T_6$, then the partial earth map tiling is extended to another ${\mc A}_q$, that we denote by ${\mc A}'_q$ and outline by the grey lines.

Now we tile beyond the fourth of Figure \ref{acdB}. We redraw the picture as $T_2,T_3,{\mc A}_q,{\mc A}'_q$ in Figure \ref{acdD}. We also have $\alpha$ between $T_3$ and ${\mc A}_q$, as part of $\delta_2\delta_3\cdots=\alpha\beta^q\delta^2$. We also use $\bullet$ to denote one vertex shared by ${\mc A}_q,{\mc A}'_q$ (not the one shared with $T_2$). The $\bullet$-vertex is $\alpha\beta^q\cdots=\alpha\beta^{q+1},\alpha\beta^q\delta^2$. 

In the first and the second of Figure \ref{acdD}, we assume the $\bullet$-vertex is $\alpha\beta^{q+1}$. This gives $\beta_7$, and the two pictures are two ways of arranging $T_7$. In the first of Figure \ref{acdD}, by the angle $\alpha$ between $T_3$ and ${\mc A}_q$, we know $\delta_7\cdots=\alpha\delta\cdots=\alpha\gamma\delta,\alpha\beta^q\delta^2$ is actually $\alpha\beta^q\delta^2$. This determines $T_8$, and the $\thin\beta\thin\cdots\thin\beta\thin\delta\thick$ part of $\delta_7\cdots=\thick^{\gamma}\delta_7^{\alpha}\thin\alpha\thin^{\alpha}\beta^{\gamma}\thin\cdots\thin^{\alpha}\beta^{\gamma}\thin^{\alpha}\delta^{\gamma}\thick$ determines a partial earth map tiling ${\mc A}_q$ that we denote by ${\mc A}''_q$.

In the second of Figure \ref{acdD}, we have $\gamma_7\cdots=\thick\gamma\thin\gamma\thick\delta\thin\cdots=\gamma^3\delta,\beta^q\gamma^2\delta^2$. If $\gamma_7\cdots=\thick\gamma\thin\gamma\thick\cdots=\beta^q\gamma^2\delta^2$, then we get a $\delta^2$-fan without $\alpha$, contradicting the only $\delta^2$-fan $\alpha\beta^q\delta^2$. Therefore $\gamma_7\cdots=\gamma^3\delta$. This determines $T_8$. Then $\beta_8\cdots=\thin^{\beta}\alpha^{\delta}\thin^{\gamma}\beta_8^{\alpha}\thin\cdots=\alpha\beta^{q+1},\alpha\beta^q\delta^2$ ($\alpha$ is from ${\mc A}_q$). Since $\thin^{\beta}\alpha^{\delta}\thin^{\gamma}\beta_8^{\alpha}\thin$ is incompatible with the AAD $\thick^{\gamma}\delta^{\alpha}\thin\alpha\thin^{\alpha}\beta^{\gamma}\thin\cdots\thin^{\alpha}\beta^{\gamma}\thin^{\alpha}\delta^{\gamma}\thick$ of $\alpha\beta^q\delta^2$, we get $\beta_8\cdots=\alpha\beta^{q+1}$. Then by no $\alpha^2\cdots$, the AAD of the vertex is $\thin^{\beta}\alpha^{\delta}\thin^{\gamma}\beta_8^{\alpha}\thin^{\gamma}\beta^{\alpha}\thin\cdots\thin^{\gamma}\beta^{\alpha}\thin$. Then the $\beta^{q+1}=\thin\beta\thin\cdots\thin\beta\thin$ part of the vertex determines a partial earth map tiling with $\beta^{q+1}$ and $\beta^q$ at the two ends. This partial earth map tiling is a copy of ${\mc A}_q$ (which we denote by ${\mc A}''_q$) together with $T_8$.

\begin{figure}[htp]
\centering
\begin{tikzpicture}

\foreach \a in {0,1,2}
\foreach \b in {0,1,2} 
{
\begin{scope}[xshift=4.7*\b cm, rotate=120*\a]

\draw
	(120:1.8) -- (60:1.8) -- (0:1.8) -- (0:0.8) -- (60:0.8) -- (120:0.8);

\draw[line width=1.2]
	(60:0.8) -- (60:1.2)
	(60:1.8) -- (60:1.4);

\node at (0.9,0.2) {\small $\alpha$};
\node at (0.95,-0.25) {\small $\beta^q$};
\node at (1.5,-0.2) {\small $\alpha$};
\node at (1.4,0.22) {\small $\beta^q$};

\node at (-0.9,0.2) {\small $\delta$};
\node at (-0.9,-0.2) {\small $\gamma$};
\node at (-1.5,-0.2) {\small $\delta$};
\node at (-1.5,0.2) {\small $\gamma$};

\end{scope}
}

\foreach \b in {0,1,2} 
{
\begin{scope}[xshift=4.7*\b cm]

\draw[line width=1.2]
	(-0.8,0) -- (0.8,0)
	;

\fill (240:1.8) circle (0.07);
	
\node at (60:0.6) {\small $\alpha$};
\node at (125:0.6) {\small $\beta$};
\node at (0.5,0.2) {\small $\delta$};
\node at (-0.5,0.2) {\small $\gamma$};

\node at (120:-0.6) {\small $\alpha$};
\node at (60:-0.6) {\small $\beta$};
\node at (0.5,-0.2) {\small $\delta$};
\node at (-0.5,-0.2) {\small $\gamma$};

\node at (-30:1.25) {\small ${\mc A}_q$};
\node at (205:1.25) {\small ${\mc A}'_q$};
\node at (85:1.3) {\small ${\mc A}''_q$};

\node[draw,shape=circle, inner sep=0.5] at (0,-0.3) {\small $2$};
\node[draw,shape=circle, inner sep=0.5] at (0,0.3) {\small $3$};

\node[draw,shape=circle, inner sep=0.5] at (210:2) {\small $7$};
\node[draw,shape=circle, inner sep=0.5] at (30:2) {\small $8$};

\end{scope}
}

\foreach \b in {0,1,2} 
{
\begin{scope}[xshift=4.7*\b cm, rotate=120*\b]

\draw[line width=1.2]
	(1.8,0) -- (2.2,0)
	(-1.8,0) -- (-2.2,0)
	;
	
\node at (1.9,-0.2) {\small $\delta$};
\node at (1.9,0.2) {\small $\delta$};
\node at (-1.9,-0.2) {\small $\gamma$};
\node at (-1.9,0.2) {\small $\gamma$};

\node at (60:1.95) {\small $\alpha$};
\node at (-60:1.95) {\small $\alpha$};

\end{scope}
}

\node at (120:2) {\small $\beta$};
\node at (-120:2.1) {\small $\beta$};

\begin{scope}[xshift=4.7 cm] 

\node at (0:2) {\small $\beta$};
\node at (-120:2.1) {\small $\beta$};

\end{scope}

\begin{scope}[xshift=9.4 cm] 

\node at (0:2) {\small $\beta$};
\node at (120:2) {\small $\beta$};

\end{scope}

\end{tikzpicture}
\caption{Proposition \ref{acd}: Tilings for the AVC \eqref{acd_avc3}.}
\label{acdD}
\end{figure}

Finally, in the third of Figure \ref{acdD}, the $\bullet$-vertex is $\alpha\beta^q\delta\cdots=\alpha\beta^q\delta^2$. This determines $T_7,T_8$, similar to $T_2,T_3$. Then we get ${\mc A}_q,{\mc A}''_q$ similar to ${\mc A}_q,{\mc A}'_q$ in the fourth of Figure \ref{acdB}. 

All three tilings in Figure \ref{acdD} are the rearrangement tilings $RE_{\square}^A1$.

\medskip

\noindent{\em Geometry of Quadrilateral}

\medskip

The quadrilateral in $E_{\square}^A1$ satisfies $\beta=\frac{4}{f}\pi$, and the area of the quadrilateral is $\beta$. It is the reduction of the quadrilateral in Proposition \ref{bcd}, and $\alpha,\beta,\gamma,\delta$ in the earlier proposition become $\beta,\gamma,\delta,\alpha$ of the current proposition. The first three of Figure \ref{acdF} are the three possible cases ($\alpha,\gamma<\pi$, and $\gamma>\pi$, and $\alpha>\pi$). They correspond to the first two of Figure \ref{bcdB}, and $B$ in Figure \ref{acdF} is $A$ in Figure \ref{bcdB}. Moreover, the three cases correspond to the three regions in Figure \ref{bcdC}.

Let $\theta=\angle ABD$. Then we get the fourth of Figure \ref{acdF}. The isosceles triangle $\triangle ABD$ has side length $a$, base length $\pi-a$, base angle $\theta$, and top angle $\angle BAD=\pi-\theta$. By spherical trigonometry, we get 
\begin{equation}\label{acdeqC}
\cos a=\frac{\cos\theta}{1+\cos\theta}.
\end{equation}
By $\cos(-\theta)=\cos\theta$, we will take $\theta\ge 0$ for the third of Figure \ref{acdF} (this keeps $\theta$) and take $\theta<0$ for the first and second of Figure \ref{acdF} (this changes $\theta$ to its negative). Then we get
\[
\alpha=\theta+\pi,\quad
\cos a=-\frac{\cos\alpha}{1-\cos\alpha}.
\]

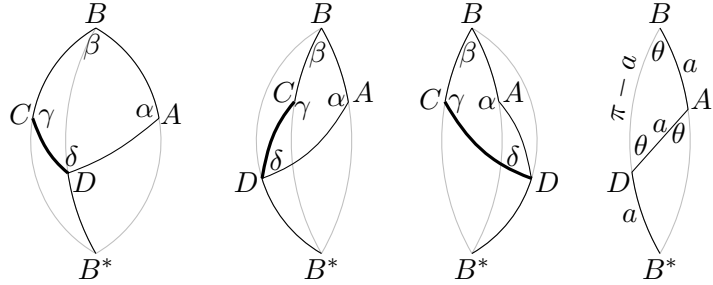
\begin{figure}[htp]
\centering
\begin{tikzpicture}[>=latex]


\draw[gray!50]
	(0,1.5) arc (60:-60:1.732)
	(0,1.5) arc (120:240:1.732)
	(0,1.5) arc (150:210:3);
	
\draw
	(0,1.5) arc (60:10:1.732) 
	(0,1.5) arc (120:170:1.732)
	(0,-1.5) arc (210:188:3)
	(0.84,0.3) to[out=220, in=20] (-0.37,-0.43);

\draw[line width=1.2]
	(-0.37,-0.43) to[out=140, in=-70] (-0.84,0.3);

\node at (1,0.35) {\small $A$};
\node at (0,1.7) {\small $B$};
\node at (0,-1.7) {\small $B^*$};
\node at (-1,0.35) {\small $C$};
\node at (-0.15,-0.55) {\small $D$};

\node at (0.65,0.4) {\small $\alpha$};
\node at (-0.05,1.25) {\small $\beta$};
\node at (-0.65,0.3) {\small $\gamma$};
\node at (-0.32,-0.2) {\small $\delta$};


\begin{scope}[xshift=3cm]

\draw[gray!50]
	(0,1.5) arc (150:210:3)
	(0,1.5) arc (30:-30:3)
	(0,1.5) arc (120:240:1.732);

\draw
	(0,1.5) arc (150:170:3)
	(0,1.5) arc (30:10:3)
	(0,-1.5) arc (240:196:1.732)
	(0.36,0.52) to[out=240, in=20] (-0.79,-0.5);

\draw[line width=1.2]
	(-0.36,0.52) to[out=230,in=80] (-0.79,-0.5);
		
\node at (0.55,0.65) {\small $A$};
\node at (0,1.7) {\small $B$};
\node at (0,-1.7) {\small $B^*$};
\node at (-0.5,0.65) {\small $C$};
\node at (-1,-0.55) {\small $D$};

\node at (-0.05,1.15) {\small $\beta$};
\node at (-0.25,0.4) {\small $\gamma$};
\node at (-0.6,-0.25) {\small $\delta$};
\node at (0.2,0.55) {\small $\alpha$};

\end{scope}


\begin{scope}[xshift=5cm]

\draw[gray!50]
	(0,1.5) arc (150:210:3)
	(0,1.5) arc (30:-30:3)
	(0,1.5) arc (60:-60:1.732);

\draw
	(0,1.5) arc (150:170:3)
	(0,1.5) arc (30:10:3)
	(0,-1.5) arc (-60:-16:1.732)
	(0.36,0.52) to[out=-50,in=100] (0.79,-0.5);

\draw[line width=1.2]	
	(-0.36,0.52) to[out=-60, in=160] (0.79,-0.5);
		
\node at (0.55,0.65) {\small $A$};
\node at (0,1.7) {\small $B$};
\node at (0,-1.7) {\small $B^*$};
\node at (-0.55,0.55) {\small $C$};
\node at (1,-0.55) {\small $D$};

\node at (0.2,0.5) {\small $\alpha$};
\node at (-0.05,1.15) {\small $\beta$};
\node at (-0.2,0.5) {\small $\gamma$};
\node at (0.55,-0.2) {\small $\delta$};

\end{scope}


\begin{scope}[xshift=7.5cm]

\draw[gray!50]
	(0,1.5) arc (150:210:3)
	(0,1.5) arc (30:-30:3);

\draw
	(0,1.5) arc (30:8:3) -- (0,0)
	(0,-1.5) arc (210:188:3) -- (0,0);
		
\node at (0.55,0.55) {\small $A$};
\node at (0,1.7) {\small $B$};
\node at (0,-1.7) {\small $B^*$};
\node at (-0.55,-0.55) {\small $D$};

\node at (-0.02,1.2) {\small $\theta$};
\node at (-0.25,-0.05) {\small $\theta$};
\node at (0.25,0.1) {\small $\theta$};

\node at (0.4,1) {\small $a$};
\node at (-0.4,-1) {\small $a$};
\node at (0,0.2) {\small $a$};
\node[rotate=78] at (-0.5,0.7) {\small $\pi-a$};

\end{scope}
	
\end{tikzpicture}
\caption{Proposition \ref{acd}: Geometry of quadrilateral.}
\label{acdF}
\end{figure}

Next we find the range of $\theta$, such that the corresponding quadrilateral is suitable for tiling. This means the point $D$ determined by $DB^*=a$ with $a$ given by \eqref{acdeqC} lies in the region in Figure \ref{bcdC}. The first region corresponds to $\theta\in(-\beta,0)$. We require $0<a<\pi$ in this range, which means $|\cos a|<1$. By \eqref{acdeqC}, this is the same as $\theta\in (-\frac{2}{3}\pi,\frac{2}{3}\pi)$. Then we get the range $\theta\in(-\beta,0)\cap(-\frac{2}{3}\pi,\frac{2}{3}\pi)=(-\beta,0)$ for the first region. The second region corresponds to $\theta\in (-\beta-\pi,-\beta]$. We require $0<a<\frac{1}{2}\pi$ in this range, which means $0<\cos a<1$, and is the same as $\theta\in (-\frac{1}{2}\pi,\frac{1}{2}\pi)$. Then we get the range $\theta\in(-\beta-\pi,-\beta]\cap(-\frac{1}{2}\pi,\frac{1}{2}\pi)=(-\frac{1}{2}\pi,-\beta]$ for the second region. We note that the range is non-empty only for $f\ge 10$. The third region corresponds to $\theta\in [0,\pi)$. We require $0<a<\frac{1}{2}\pi$ in this range, which is the same as $\theta\in (-\frac{1}{2}\pi,\frac{1}{2}\pi)$. Then we get the range $\theta\in[0,\pi)\cap(-\frac{1}{2}\pi,\frac{1}{2}\pi)=[0,\frac{1}{2}\pi)$ for the third region.

Taking the union of three ranges, we conclude the range $\theta\in(-\frac{1}{2}\pi,\frac{1}{2}\pi)$ for $f\ge 8$, and the range $\theta\in(-\beta,\frac{1}{2}\pi)$ for $f=6$. In terms of $\alpha=\theta+\pi$, the range is $\alpha\in (\frac{1}{2}\pi,\frac{3}{2}\pi)$ for $f\ge 8$, and $\alpha\in(\frac{1}{3}\pi,\frac{3}{2}\pi)$ for $f=6$. We remark that the quadrilateral is a rhombus, i.e., $a=b$, when $\theta=-\frac{1}{2}\beta=-\frac{2}{f}\pi$. This corresponds to $\alpha=(1-\frac{2}{f})\pi$.

The flip modification $FE_{\square}^A1$ requires $\alpha=s\beta$ or $\gamma+\delta=s\beta$. This means
\begin{align*}
\alpha=s\beta &\colon
\alpha=\tfrac{4s}{f}\pi,\;
\beta=\tfrac{4}{f}\pi,\;
\gamma+\delta=(2-\tfrac{4s}{f})\pi. \\
\gamma+\delta=s\beta &\colon
\alpha=(2-\tfrac{4s}{f})\pi,\;
\beta=\tfrac{4}{f}\pi,\;
\gamma+\delta=\tfrac{4s}{f}\pi.
\end{align*}
The quadrilateral is uniquely determined by $f,s$. Moreover, $t$ such flips require $ts\beta\le 2\pi$. For $f\ge 8$, by $\frac{1}{2}\pi<\alpha<\frac{3}{2}\pi$, the exact conditions for $s,t$ are
\begin{equation}\label{acdeqB}
\tfrac{1}{8}f<s<\tfrac{3}{8}f,\quad
ts\le \tfrac{1}{2}f.
\end{equation}
The two conditions imply $t\le 3$. For any $s,t$ satisfying the conditions, we may pick any $t$ non-overlapping copies of consecutive $s$ timezones ${\mc A}_s$ and flip them simultaneously. 

For $f=6$, by $\alpha\ne \beta$, the flip modifications require $\alpha=2\beta$. Since both flips produce degree $2$ vertex $\alpha\beta$, there is no flip modification for $f=6$.

Finally, the rearrangement tiling $RE_{\square}^A1$ requires very specific quadrilateral. In the first of Figure \ref{acdC}, we form a hexagon using six triangles with angles $\frac{1}{3}\pi$, $\frac{1}{2}\beta=\frac{2}{f}\pi$ and $\gamma=(\tfrac{2}{3}-\tfrac{2}{3f})\pi$. The quadrilateral is obtained by dividing the hexagon into two equal halves. The second of Figure \ref{acdC} is the earth map tiling $E_{\square}^A1$ constructed from this tile. The tiling has the underlying tiling by congruent triangles as indicated by the grey lines, and the triangular tiling is not edge-to-edge.
\end{proof}

\begin{figure}[htp]
\centering
\begin{tikzpicture}[scale=0.8]

\begin{scope}[yshift=-0.5cm]

\foreach \a in {0,1,2}
{
\begin{scope}[rotate=120*\a]

\draw[gray!50]
	(0,-0.8) -- (0,2)
	(0,2) -- (30:0.8) -- (-30:2);

\end{scope}
}

\draw[line width=1.2]
	(0,-0.8) -- (0,2);

\draw
	(0,2) -- (30:0.8) -- (-30:2)
	(0,-0.8) -- (-30:2);	

\node at (30:0.6) {\small $\alpha$};
\node at (-30:1.4) {\small $\beta$};
\node at (0.2,-0.6) {\small $\gamma$};	
\node at (0.2,1.2) {\small $\delta$};

\end{scope}

\foreach \a in {0,1,2}
{
\begin{scope}[scale=0.8, xshift=5 cm+1.6*\a cm]

\foreach \c in {-1,1}
\draw[gray!50, scale=\c]
	(0,0) -- (1.2,0.6) -- (1.2,1.8)
	(0.4,0.2) -- (0.4,1.8)
	(0.4,0.2) -- (-0.4,0.6) -- (-0.4,1.8)
	(1.2,0.6) -- (0.4,-0.6);

\draw
	(-0.4,1.8) -- (-0.4,0.6) -- (-1.2,-0.6) -- (-1.2,-1.8)
	(1.2,1.8) -- (1.2,0.6) -- (0.4,-0.6) -- (0.4,-1.8)
	(1.2,0.6) -- (-1.2,-0.6);

\draw[line width=1.2]
	(1.2,0.6) -- (-1.2,-0.6);

\node at (-0.15,0.55) {\small $\alpha$};
\node at (0.15,-0.55) {\small $\alpha$};
\node at (0.4,1.7) {\small $\beta$};
\node at (-0.4,-1.7) {\small $\beta$};
\node at (0.9,0.75) {\small $\gamma$};
\node at (-0.9,-0.75) {\small $\gamma$};
\node at (0.6,0.06) {\small $\delta$};
\node at (-0.6,-0.03) {\small $\delta$};

\end{scope}
}

\end{tikzpicture}
\caption{Proposition \ref{acd}: Triangular tiling underlying the tiling for \eqref{acd_avc3}.}
\label{acdC}
\end{figure}
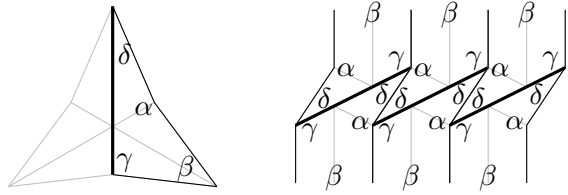


\begin{thebibliography}{1}

\bibitem{aehj}
C. Adams, C. Edgar, P. Hollander, L. Jacoby.
\newblock The rest of the tilings of the sphere by regular polygons.
\newblock {\em preprint}, arXiv:2101.10743, 2021.

\bibitem{akama1}
Y.~Akama.
\newblock Classification of spherical tilings by congruent quadrangles over pseudo-double wheels I: a special tiling by congruent concave quadrangles. 
\newblock {\em Hiroshima Math. J.}, 43(3):285--304, 2013.

\bibitem{akama2}
Y.~Akama.
\newblock Classification of spherical tilings by congruent quadrangles over pseudo-double wheels II: the isohedral case. 
\newblock {\em Hiroshima Math. J.}, 49(1):1--34, 2019.


\bibitem{ac}
Y.~Akama., N.~van Cleemput.
\newblock Spherical tilings by congruent quadrangles: Forbidden cases and substructures.
\newblock {\em Ars Math. Contemp.}, 8:297--318, 2015. 


\bibitem{ay1}
Y.~Akama, M.~Yan.
\newblock On deformed dodecahedron tilings.
\newblock {\em Austral. J. Combin.}, 85(1):1--14, 2023.

\bibitem{awy}
Y.~Akama, E.~X.~Wang, M.~Yan.
\newblock Tilings of the sphere by congruent pentagons III: edge combination $a^5$.
\newblock {\em Adv. in Math.}, 394:107881, 2022.

\bibitem{cl}
H.~M.~Cheung, H.~P.~Luk.
\newblock Rational angles and tilings of the sphere by congruent quadrilaterals.
\newblock {\em preprint}, arXiv:2204.02748, 2022.

\bibitem{cly}
H.~M.~Cheung, H.~P.~Luk, M. Yan.
\newblock Tilings of the sphere by congruent pentagons IV: edge combination $a^4b$.
\newblock {\em preprint}, arXiv:2307.11453, 2023.

\bibitem{cj}
J. H. Conway, A. J. Jones. 
\newblock Trigonometric Diophantine equations (On vanishing sums of roots of unity). 
\newblock {\em Acta Arithmetica}, 30:229-240, 1976.
 
\bibitem{coolsaet}
K.~Coolsaet.
\newblock Spherical quadrangles with three equal sides and rational angles.
\newblock {\em Ars Math. Contemp.}, 12:415--424, 2017.


\bibitem{cromwell}
P.~Cromwell. 
\newblock Polyhedra. 
\newblock Cambridge University Press, 1999.

\bibitem{davisH}
H. L.~Davies.
\newblock Packings of spherical triangles and tetrahedra.
\newblock In: 1967 Proc. Coll. on Convexity (Copenhagen, 1965), 42-51. Kobenhavns Univ. Mat. Inst., Copenhagen 2017.


\bibitem{dawson1}
R.~Dawson.
\newblock An isosceles triangle that tiles the sphere in exactly three ways.
\newblock {\em Discrete Comput. Geom.}, 30(3):459--466, 2003.

\bibitem{dawson2}
R.~Dawson.
\newblock Tilings of the sphere with isosceles triangles.
\newblock {\em Discrete Comput. Geom.}, 30(3):467--487, 2003.

\bibitem{dd1}
R.~Dawson, B.~Doyle.
\newblock Tilings of the sphere with right triangles I. The asymptotically right families.
\newblock {\em Elec. J. Combi.}, 13:\#R1.48, 2006.

\bibitem{dd2}
R.~Dawson, B.~Doyle.
\newblock Tilings of the sphere with right triangles II. The $(1,3,2),(0,2,n)$ subfamily.
\newblock {\em Elec. J. Combi.}, 13:\#R1.49, 2006.

\bibitem{dd3}
R.~Dawson, B.~Doyle.
\newblock Tilings of the sphere with right triangles III. The asymptotically obtuse families.
\newblock {\em Elec. J. Combi.}, 14:\#R1.48, 2007.

\bibitem{gsy}
H.~H.~Gao, N.~Shi, M.~Yan.
\newblock Spherical tiling by $12$ congruent pentagons.
\newblock {\em J. Combinatorial Theory Ser. A}, 120(4):744--776, 2013.

\bibitem{johnson}
N. Johnson.
\newblock Convex solids with regular faces.
\newblock {\em Canadian J. Math.}, 18:169-200, 1966.

\bibitem{lqwx1} 
Y. Liao, P. Qian, E. Wang, Y. Xu. 
\newblock Tilings of the sphere by congruent quadrilaterals I: edge combination $a^2bc$. \newblock {\em preprint}, arXiv:2110.10087, 2021. 

\bibitem{lw}
Y. Liao, E. Wang. 
\newblock Tilings of the sphere by congruent quadrilaterals II: edge combination $a^3b$ with rational angles. 
\newblock {\em preprint}, arXiv:2205.14936, 2022.

\bibitem{lqwx2}
Y. Liao, P. Qian, E. Wang, Y. Xu. 
\newblock Tilings of the sphere by congruent quadrilaterals III: edge combination $a^3b$ with general angles. 
\newblock {\em preprint}, arXiv:2206.15342, 2022.

\bibitem{myerson}
G. Myerson.
\newblock Rational products of sines of rational angles.
\newblock {\em Aequationes Math.}, 45:70-82, 1993.

\bibitem{rao}
M.~Rao.
\newblock Exhaustive search of convex pentagons which tile the plane.
\newblock {\em preprint}, arXiv:1708.00274, 2017.

\bibitem{sakano-akama}
Y.~Sakano, Y.~Akama.
\newblock Anisohedral spherical triangles and classification of spherical tilings by congruent kites, darts, and rhombi. 
\newblock {\em Hiroshima Math. J.}, 45(3):309--339, 2015.

\bibitem{so}
D.~M.~Y.~Sommerville.
\newblock Division of space by congruent triangles and tetrahedra.
\newblock {\em Proc. Royal Soc. Edinburgh}, 43:85--116, 1923.

\bibitem{ua}
Y.~Ueno, Y.~Agaoka. 
\newblock Classification of tilings of the 2-dimensional sphere by congruent triangles.
\newblock {\em Hiroshima Math. J.}, 32(3):463--540, 2002.

\bibitem{ua2}
Y.~Ueno, Y.~Agaoka. 
\newblock Examples of spherical tilings by congruent quadrangles.
\newblock {\em Math. Inform. Sci., Fac. Integrated Arts and Sci., Hiroshima Univ., Ser. IV}, 27:135--144, 2001.

\bibitem{wy1}
E.~X.~Wang, M.~Yan.
\newblock Tilings of the sphere by congruent pentagons I: edge combinations $a^2b^2c$ and $a^3bc$.
\newblock {\em Adv. in Math.}, 394:107866, 2022.

\bibitem{wy2}
E.~X.~Wang, M.~Yan.
\newblock Tilings of the sphere by congruent pentagons II: edge combination $a^3b^2$.
\newblock {\em Adv. in Math.}, 394:107867, 2022.

\bibitem{wy3}
E.~X.~Wang, M.~Yan.
\newblock Moduli of pentagonal subdivision tiling.
\newblock preprint, arXiv: 1907.08776, 2019.

\bibitem{yan}
M.~Yan.
\newblock Combinatorial tilings of the sphere by pentagons.
\newblock {\em Elec. J. Combi.}, 20:\#P1.54, 2013.

\bibitem{yan2}
M.~Yan.
\newblock Pentagonal subdivision.
\newblock {\em Elec. J. Combi.}, 26:\#P4.19, 2019.

\bibitem{zal}
V. Zalgaller.
\newblock Convex polyhedra with regular faces.
\newblock {\em Zap. Nauchn. Semin. Leningr. Otd. Mat. Inst. Steklova} (in Russian) 2:1–221, 1967.

\bibitem{zong}
C. M.~Zong.
\newblock Can you pave the plane with identical tiles?
\newblock {\em Notice AMS}, 67(5):635--646, 2020.

\end{thebibliography}
\end{document}